\documentclass[oneside]{article}

\usepackage[english]{babel}
\usepackage[utf8]{inputenc}
\usepackage[T1]{fontenc}

\usepackage{csquotes}
\usepackage[backend=biber, style=alphabetic]{biblatex}
\usepackage{amsmath, amstext, amsthm, mathrsfs,dsfont,amssymb,amsbsy,enumitem}
\usepackage{tikz}
\usepackage{graphicx}
\usepackage{caption}
\usepackage{subcaption}
\usepackage[all]{xy}

\usepackage{longtable}

\usepackage{lipsum}
\allowdisplaybreaks
 
\usepackage{hyperref}
\hypersetup{colorlinks = true, urlcolor = blue, citecolor=blue, bookmarksopen = true}

\usepackage{xcolor}      % For color options
\usepackage[most]{tcolorbox}

\tcbuselibrary{breakable}

\definecolor{dartmouthgreen}{rgb}{0.05, 0.5, 0.06}

\numberwithin{equation}{section}

\theoremstyle{plain}
\newtheorem{theorem}{Theorem}
\newtheorem{hypo}{Hypothesis}
\newtheorem{prop}{Proposition}
\newtheorem{lemma}{Lemma}[section]

\theoremstyle{definition}

\theoremstyle{remark}

\newtheorem{remark}{Remark}
\renewenvironment{proof}{\noindent\textbf{Proof }}{\hspace*{\fill}$\Box$\medskip}

\def \ind {\mathds{1}}
\def \ppp {\ldots}
\def \lc {\left\lbrace}
\def \rc {\right\rbrace}
\def \uzet {\underline{\zeta}}
\def \Im {\mathrm{Im} \,}

\def \A {\mathbb{A}}
\def \C {\mathbb{C}}
\def \D {\mathbb{D}}

\def \N {\mathbb{N}}

\def \R {\mathbb{R}}
\def \S {\mathbb{S}}

\def \U {\mathbb{U}}
\def \Z {\mathbb{Z}}

\def \Cc {\mathcal{C}}

\def \Ec {\mathcal{E}}
\def \Fc {\mathcal{F}}
\def \Gc {\mathcal{G}}

\def \Lc {\mathcal{L}}
\def \Mc {\mathcal{M}}
\def \Nc {\mathcal{N}}
\def \Oc {\mathcal{O}}
\def \Rc {\mathcal{R}}

\def \Tc {\mathcal{T}}
\def \Uc {\mathcal{U}}

\def \Ccc {\mathscr{C}}

\def \Gcc {\mathscr{G}}

\def \Lcc {\mathscr{L}}

\def \rg {\textbf{\textit{r}}}
\def \lg {\textbf{\textit{l}}}
\def \Pg {\textbf{\textit{P}}}
\def \xg {\textbf{\textit{x}}}
\def \lambg {\boldsymbol{\lambda}}
\def \nug {\boldsymbol{\nu}}

\def \dsp {\overline{u}^s}

\begin{document}
	
	\begin{center}
		{\large \bfseries Linear stability of discrete shock profiles for systems of conservation laws}
	\end{center}
	
	\begin{center}
		Lucas \textsc{Coeuret}\footnote{Dipartimento di Matematica “Tullio Levi-Civita”, Università di Padova, Via Trieste 63, 35121 Padova, Italy. ORCID: 0009-0009-6746-4786. E-mail: lucas.coeuret@math.unipd.it}
	\end{center}
	
	\vspace{5mm}
	
	\begin{center}
		\textbf{Abstract}
	\end{center}
	We prove the linear orbital stability of spectrally stable stationary discrete shock profiles for conservative finite difference schemes applied to systems of conservation laws. The proof relies on an accurate description of the pointwise asymptotic behavior of the Green's function associated with those discrete shock profiles, improving on the result of Lafitte-Godillon \cite{Godillon}. The main novelty of this stability result is that it applies to a fairly large family of schemes that introduce some artificial possibly high-order viscosity. The result is obtained under a sharp spectral assumption rather than by imposing a smallness assumption on the shock amplitude.
	
	\vspace{4mm}
	
	
	
	\textbf{Acknowledgements:} Research of the author was supported by the Agence Nationale de la Recherche project Indyana (ANR-21-CE40-0008), by the Labex Centre International de Mathématiques et Informatique de Toulouse under grant agreement ANR-11-LABX-0040, as well as by the PRIN PNRR P2022XJ9SX of the European Union — Next Generation EU. The author thanks Jean-François Coulombel and Grégory Faye for several fruitful discussions on the subject and their proofreading. Their help has been essential for this paper. The author is also thankful to the referee for meaningful comments to improve the article.
	
	\tableofcontents
	
	\section*{Notations}
	
	Throughout this article, we define the following sets:
	$$\U:=\lc z\in \C, |z|>1\rc,\quad \D:=\lc z\in \C, |z|<1\rc, \quad \S^1:=\lc z\in \C, |z|=1\rc,$$
	$$\overline{\U}:=\S^1\cup \U,\quad  \overline{\D}:=\S^1\cup \D.$$
	
	For $z\in \C$ and $r>0$, we let $B(z,r)$ denote the open ball in $\C$ centered at $z$ with radius $r$. We also introduce the Kronecker symbol $\delta_{i,j}$ which equals $1$ if $i=j$ and $0$ when $i\neq j$.
	
	For $E$ a Banach space, we denote $\Lc(E)$ the space of bounded operators acting on $E$ and $\left\|\cdot\right\|_{\Lc(E)}$ the operator norm. For $T$ in $\Lc(E)$, the notation $\sigma(T)$ stands for the spectrum of the operator $T$ and $\rho(T)$ denotes the resolvent set of $T$.
	
	We let $\Mc_{n,k}(\C)$ denote the space of complex valued $n\times k$ matrices and we use the notation $\Mc_n(\C)$ when $n=k$. For an element $M$ of $\Mc_{n,k}(\C)$, the notation $M^T$ stands for the transpose of $M$. For a square matrix $M$, $\mathrm{com}(M)$ corresponds to the cofactor matrix associated with $M$. 
	
	We use the notation $\lesssim$ to express an inequality up to a multiplicative constant. Eventually, we let $C$ (resp. $c$) denote some large (resp. small) positive constants that may vary throughout the text (sometimes within the same line). Furthermore, we use the usual Landau notation $O(\cdot)$ to introduce a term uniformly bounded with respect to the argument. For more clarity, we will also occasionally use the notation $O_\C(\cdot)$ to precise the fact that the term is a complex scalar and $O_{\Mc_{1,d}}\left(\cdot\right)$ to say that it belongs to $\Mc_{1,d}(\C)$.
	
	We let $\mathrm{Res}(f,a)$ denote the residue of a meromorphic function $f$ at the point $a$.

		The following table of symbols summarizes most of the essential and recurring notations introduced in the article. We separate them by section.

			\begin{center}
			
			\vspace{0.1cm}
			
			\textbf{Section \ref{sec:Intro}: Introduction}
			
			\vspace{0.1cm}
			
			\begin{longtable}{p{0.2\textwidth}p{0.75\textwidth}}
				\raggedleft$d$, $f$, $\Uc$ & dimension, flux and space of states of the system of conservation laws \eqref{def:EDP}\\
				\raggedleft$u^-,u^+$ & states of the stationary shock \eqref{def:Choc}\\
				\raggedleft$\lambg^\pm_k,\rg^\pm_k,\lg^\pm_k$& eigenvalues and right and left eigenvectors of $df(u^\pm)$\\
				\raggedleft$\Pg^\pm$ & matrix \eqref{def:P} composed with the vectors $\rg^\pm_k$\\
				\raggedleft$I$ & index associated with the Lax shock hypothesis (Hypothesis \ref{H:Lax})\\
				\raggedleft$\nug$ &ratio between the space and time steps $\Delta x$ and $\Delta t$\\
				\raggedleft$\Nc$& nonlinear operator defined by \eqref{def:evol} corresponding to the evolution operator of the finite difference scheme \eqref{def:SchemeNum} \\
				\raggedleft$p,q$& integers determining the stencil of the finite difference scheme scheme \eqref{def:SchemeNum} \\
				\raggedleft$\dsp$ & stationary discrete shock profile associated with the shock \eqref{def:Choc}\\
				\raggedleft$\Lcc^\pm$ & linearization \eqref{def:linearizedSchemeEndState} of the operator $\Nc$ about the constant state $u^\pm$\\
				\raggedleft$A^\pm_k,B^\pm_k$ & matrices \eqref{def:Ak} and \eqref{def:Bk} appearing in the definition of the operator $\Lcc^\pm$\\
				\raggedleft$\Lcc$ & linearization \eqref{def:linearizedScheme} of the operator $\Nc$ about the discrete shock profile $\dsp$\\
				\raggedleft$A_{j,k},B_{j,k}$ & matrices \eqref{def:Ajk} and \eqref{def:Bjk} appearing in the definition of the operator $\Lcc$\\
				\raggedleft$\Fc_l^\pm$& eigenvalue of the amplification matrix $\sum_{k=-p}^q\kappa^kA_k^\pm$, defined by \eqref{def:Fcl}\\
				\raggedleft$\alpha_l^\pm,\beta_l^\pm,\mu$& constants defined by \eqref{eg:FcFin} and \eqref{F} which characterize the speed and viscosity  parameters of the different waves in the decomposition \eqref{decompoGreen}\\
				\raggedleft$\Oc$ & unbounded connected component of $\C\backslash\left(\sigma(\Lcc^+)\cup\sigma\left(\Lcc^-\right)\right)$ (see Figure \ref{fig:spec})\\
				\raggedleft$\mathrm{Ev}$& Evans function defined by \eqref{def:Ev} which intervenes in Hypothesis \ref{H:Evans} \\
				\raggedleft$V$& sequence defined by \eqref{def:V} spanning the eigenspace of $\Lcc$ associated with the eigenvalue $1$, satisfying \eqref{egV} and \eqref{decExpoV}\\
				\raggedleft$\Gcc(n,j_0,j)$ & temporal Green's function defined by \eqref{defGreenTempo}\\
				\raggedleft$G(z,j_0,j)$ & spatial Green's function defined by \eqref{defGreenSpatial}\\
				\raggedleft$H_{2\mu}(\beta;\cdot),E_{2\mu}(\beta;\cdot)$& generalized Gaussian and generalized Gaussian error function defined by \eqref{def:H2mu_et_E2mu}\\
				\raggedleft$S_l^\pm(n,j_0,j)$ &  outgoing/incoming waves in the decomposition \eqref{decompoGreen}, defined by \eqref{def:S+} and \eqref{def:S-}\\
				\raggedleft$R_{l^\prime,l}^\pm(n,j_0,j)$ & reflected waves in the decomposition \eqref{decompoGreen}, defined by \eqref{def:R+} and \eqref{def:R-}\\
				\raggedleft$T_{l^\prime,l}^\pm(n,j_0,j)$ & transmitted waves in the decomposition \eqref{decompoGreen}, defined by \eqref{def:T+} and \eqref{def:T-}\\
				\raggedleft$E_{l^\prime}^\pm(n,j_0)$ & vectors defined by \eqref{def:E+} and \eqref{def:E-}, associated in the decomposition \eqref{decompoGreen} with the activation of the sequence $V$ \\
				\raggedleft$\Rc(n,j_0,j)$& fast decaying residual term in the decomposition \eqref{decompoGreen}, defined by \eqref{def:ResTh}
			\end{longtable}
		
		\end{center}

		\begin{center}

			\vspace{0.1cm}
			
			\textbf{Section \ref{sec:GSloin}: Local exponential bounds on the spatial Green's function for $z$ far from $1$}
			
			\vspace{0.1cm}
			
			\begin{longtable}{p{0.25\textwidth}p{0.7\textwidth}}
				\raggedleft$\A_{j,k}(z),\A^\pm_k(z)$& matrices defined by \eqref{def:AAjk}\\
				\raggedleft$M_j(z),M^\pm(z)$ & matrices defined by \eqref{def:MjM} which intervene in the dynamical systems \eqref{syst_dyn} and \eqref{systDyn+-}\\
				\raggedleft$X_j(z)$& fundamental matrices of the dynamical system \eqref{syst_dyn}, defined by \eqref{mat_fond}\\
				\raggedleft$E^\pm(z),E_0^\pm(z)$ & sets defined by \eqref{def:EpmE0pm} which correspond respectively to the solutions of the  dynamical system \eqref{syst_dyn} tending towards $0$ at $\pm\infty$ and their traces at $0$\\
				\raggedleft$M^\pm_l(z)$ & matrices defined by \eqref{def:Mlpm}\\
				\raggedleft$E^s(M^\pm(z) ),E^u(M^\pm(z))$ & strictly stable and unstable subspaces of $M^\pm(z)$ which verify \eqref{decompo:C^d(p+q)}\\
				\raggedleft$P_s^\pm(z),P_u^\pm(z)$ & projectors associated with the decomposition \eqref{decompo:C^d(p+q)} of $\C^{d(p+q)}$\\
				\raggedleft$Q_U^\pm(z),Q(z)$ & local and extended geometric dichotomies respectively defined in Lemmas \ref{lem_geo_dich} and \ref{lemGeoDichStronger}\\
				\raggedleft$\Oc_\rho,\Oc_\sigma$ & sets defined by \eqref{def:Oc_rho_sigma} corresponding to the intersection of $\Oc$ with the resolvent set and spectrum of the operator $\Lcc$\\
				\raggedleft$\Pi$ & linear operator \eqref{def:Pi} which extracts the "center" values of a vector of size $d(p+q)$\\
				\raggedleft$W(z,j_0,j,\textbf{e})$& vector defined by \eqref{def:WGreenspatial} and composed of $G(z,j_0,j+q-1)\textbf{e},\hdots,G(z,j_0,j-p)\textbf{e}$
			\end{longtable}
			
		\end{center}

		\begin{center}
			
			\vspace{0.1cm}
			
			\textbf{Section \ref{sec:GSnear1}: Extension of the spatial Green's function near $1$}
			
			\vspace{0.1cm}
			
			\begin{longtable}{p{0.26\textwidth}p{0.69\textwidth}}
				\raggedleft$\delta_0,\delta_1,\delta_2$& radii defined respectively by the construction of the functions $\zeta_m^\pm$ and Lemmas  \ref{lem_choice_base} and \ref{lem:Delta}\\
				\raggedleft$\uzet_m^\pm,\zeta_m^\pm(z)$& eigenvalues respectively of $M^\pm(1)$ and $M^\pm(z)$ which verify \eqref{prop:uzet} and \eqref{inZeta}\\
				\raggedleft$I_{ss}^\pm,I_{cs}^\pm,I_{cu}^\pm,I_{su}^\pm$ & partitions \eqref{def:Iss,cs,cu,su^pm} of the set $\lbrace1,\hdots,d(p+q)\rbrace$\\
				\raggedleft$I_{ss},I_{cs},I_{cu},I_{su}$ & sets defined by \eqref{def:Iss,cs,cu,su}\\
				\raggedleft$R_m^\pm(z),L_m^\pm(z)$ & right and left eigenvectors associated with the eigenvalue $\zeta_m^\pm(z)$ for the matrix $M^\pm(z)$, defined respectively by \eqref{def:Rm} and \eqref{def:Lm}\\
				\raggedleft$W_m^\pm(z,j):=\zeta_m^\pm(z)^jV_m^\pm(z,j)$ & solution of the dynamical system \eqref{syst_dyn} defined in Lemma \ref{lem_choice_base}\\
				\raggedleft$\theta_{s,m},\theta_{u,m}$ & coefficients defined by \eqref{egWssWsu} and \eqref{def:theta_su}\\
				\raggedleft$\Phi_m(z,j)$ & solution of the dynamical system \eqref{syst_dyn} defined by \eqref{defPhi}\\
				\raggedleft$\Gc^\Phi(z,j) ,\Gc^\pm(z,j),\widetilde{\Gc}^\pm(z,j)$ & matrices defined by \eqref{def:Mat_et_det}\\
				\raggedleft$D^\Phi(z) ,D^\pm(z)$ & determinant of matrices defined by \eqref{def:Mat_et_det} and satisfy \eqref{eg:Ev}\\
				\raggedleft$\Pi_m(z)$ & projector \eqref{def:Pim} on $\mathrm{Span}(\Phi_m(z,0))$ along $\mathrm{Span}(\Phi_{\nu}(z,0))_{\nu\in\lc1,\ppp,d(p+q)\rc\backslash\lc m\rc}$\\
				\raggedleft$\widehat{D}_m(z,j_0,\textbf{e})$ & complex valued functions defined by \eqref{defDChap}\\
				\raggedleft$\Ccc_m^\pm(z,j_0,\textbf{e}),\widetilde{\Ccc}_m^\pm(z,j_0,\textbf{e})$& complex valued functions defined by \eqref{def:C_Ctilde}\\
				\raggedleft$g_{m^\prime,m}^\pm(z),\widetilde{g}_{m^\prime,m}^\pm(z)$ & complex valued functions defined respectively by \eqref{defM} and \eqref{def:gtilde}\\
				\raggedleft$\Delta_m^\pm(z,j_0,\textbf{e})$ & complex valued functions defined by \eqref{def:NDelta}
			\end{longtable}
			
		\end{center}

		\begin{center}
			
			\vspace{0.1cm}
			
			\textbf{Section \ref{sec:GT}: Temporal Green's function and proof of Theorem \ref{th:Green}}
			
			\vspace{0.1cm}
			
			\begin{longtable}{p{0.23\textwidth}p{0.72\textwidth}}
				\raggedleft$\varepsilon_0^*,\varepsilon_1^*,\varepsilon_2^*$& radii defined respectively by \eqref{def:varepsilon_0^*}, \eqref{def:varpi} and Lemma \ref{lem:BorneVarphiVarpi} \\
				\raggedleft$\varepsilon,\eta$& radius and width defined by the conditions \eqref{condEta}, \eqref{condEta2} and \eqref{defIextr}\\
				\raggedleft$r_\varepsilon$& function defined by \eqref{defreps} which determines the imaginary part of the extremities of $-\eta+i\R\cap B(0,\varepsilon)$\\
				\raggedleft$\varpi_l^\pm(\tau)$& complex function \eqref{def:varpi} corresponding to the logarithm of $\zeta_m^\pm\left(e^\tau\right)$ for $m\in I_{cs}^\pm\cup I^\pm_{cu}$ \\
				\raggedleft$\varphi_l^\pm(\tau),\xi_l^\pm(\tau)$& complex functions defined by \eqref{def:varphi} and \eqref{eg:LienVarpiVarphi} corresponding respectively to the principal part and the rest of the asymptotic expansion of $\varpi_l^\pm$ at $\tau=0$ \\
				\raggedleft$\Gamma_{out}(\eta),\Gamma_{in}^\pm(\eta),\Gamma_{in}^0(\eta),$ & paths defined by \eqref{def:Paths} and represented in Figure \ref{FigChem}\\
				\raggedleft$\Gamma_{in}(\eta),\Gamma(\eta),\Gamma_{d}(\eta)$&\\
				\raggedleft$X$ & set of the paths going from $-\eta-ir_\varepsilon(\eta)$ to $-\eta +ir_\varepsilon(\eta)$ whilst remaining in $B(0,\varepsilon)$\\
				\raggedleft$\widehat{\Pi}$ & projector on the vector space $\mathrm{Span}\left(\sum_{j\in\Z}V(j)\right)$ along the vector space spanned by the vectors $\rg_1^-,\hdots,\rg_{I-1}^-,\rg_{I+1}^+,\hdots,\rg_{d}^+$
			\end{longtable}
			
		\end{center}

	\section{Introduction}\label{sec:Intro}
	
	\subsection{Context}
	
	A fundamental issue on the subject of systems conservation laws is understanding how discontinuities that can arise in solutions are handled by conservative finite difference schemes. At the center of this question stands the notion of discrete shock profiles which are defined as solutions of the numerical scheme that also are traveling waves linking two states. They correspond to numerical approximations of shocks. A desirable feature of the numerical scheme should be that stable shock waves for the system of conservation laws should yield stable discrete shock profiles (or a family of them) for the numerical scheme. For a general introduction on the questions of existence and stability of discrete shock profiles, we highly encourage the interested reader to take a close look at \cite{Serre}.
	
	In the present paper, we will consider conservative finite difference schemes that \textit{introduce numerical viscosity} and will focus on the study of the discrete shock profiles associated with \textit{standing Lax shocks}. We can assume that there exists a continuous one-parameter family of discrete shock profiles associated with such a shock. Such an existence result has been proved for instance in \cite{Majda-Ralston,Michelson} under a weakness assumption on the shock, i.e. when the difference between the two states is sufficiently small. Let us introduce two notions of stability for the family of discrete shock profiles:
	\begin{itemize}
		\item Spectral stability amounts to asking for the operators obtained by linearizing the numerical scheme about the discrete shock profiles to have no unstable or marginally stable eigenvalues except for $1$ which is always an eigenvalue because of the existence of the continuous one-parameter family of discrete shock profiles. Furthermore, we ask for $1$ to be a simple eigenvalue (in the sense of Hypothesis \ref{H:Evans} related to the so-called Evans function) of the linearized operator. Spectral stability thus corresponds to Hypotheses \ref{H:spec} and \ref{H:Evans} introduced below. 
		
		\item Nonlinear orbital stability signifies that for initial conditions of the numerical scheme which are small enough perturbations of a discrete shock profile, the solution of the numerical scheme that ensues stays close to the manifold of the discrete shock profiles with respect to some well-chosen norm of a Banach space. This is essentially a stronger and more concrete stability property.
	\end{itemize}
	
	There are some results surrounding nonlinear stability that have already been proven. Most of them introduce a weakness assumption on the amplitude of the underlying shocks and/or focus on fairly specific schemes or situations. For instance, \cite{Liu-Xin,Liu-Xin2,Ying} focus on proving a nonlinear orbital stability result for discrete shock profiles moving with rational speeds associated with weak Lax shocks for the Lax-Friedrichs scheme. Another example is \cite{Smyrlis} which also proves a nonlinear orbital stability result on some stationary discrete shock profiles for the Lax-Wendroff scheme without any weakness assumption. Two of the main results that we can point out are the following:
	\begin{itemize}
		\item In \cite{Michelson2}, Michelson proves the nonlinear orbital stability of the family of discrete shock profiles associated with weak standing Lax shocks for schemes of any odd order under an assumption of stability of the viscous shock profiles associated with some scalar problem.
		
		\item In \cite{Jennings}, Jennings focuses on the particular case of \textit{monotone} schemes for \textit{scalar} conservation laws. The main results are the existence and uniqueness of continuous one-parameter family of discrete shock profiles with rational speeds and a proof of nonlinear orbital stability for them when they are associated with Lax shocks. In this paper, no weakness assumption on the associated shocks is introduced.
	\end{itemize}
	Compared with the nonlinear stability theory for viscous shock profiles \cite{ZH} or for semi-discrete shock profiles \cite{BenzoniHuotRousset,BeckHupkesSandstedeZumbrun}, one can hope to prove that spectrally stable discrete shock profiles associated with stationary Lax shocks verify nonlinear orbital stability in some well-chosen Banach space, at least when the numerical scheme introduces some numerical viscosity. This new result would generalize the previously cited articles by proving a result of nonlinear stability that holds for \emph{systems} of conservation laws, for a fairly large family of finite difference schemes, while avoiding to introduce a weakness assumption on the shocks. This would answer Open problem 5.3 of \cite{Serre}.

		Just like in \cite{ZH,MasciaZumbrun,BenzoniHuotRousset,BeckHupkesSandstedeZumbrun}, proving that spectral stability implies nonlinear orbital stability can essentially be separated in three steps:
		\begin{enumerate}
			\item We first need to precisely describe the long time behavior of the temporal Green's function $\Gcc(n,j_0,j)$ (defined below by \eqref{defGreenTempo}) associated with the operator $\Lcc$ (defined below by \eqref{def:linearizedScheme}) obtained by linearizing the numerical scheme about the spectrally stable discrete shock profile. We also need an accurate description of the discrete derivative $\Gcc(n,j_0,j)-\Gcc(n,j_0-1,j)$ of the Green's function.
			\item One can then use this description of the Green's function and its discrete derivative to prove decay estimates in suitable Banach spaces (for instance weighted $\ell^r$-spaces) for the families of operators $(\Lcc^n)_{n\in\N}$ and $\left(\Lcc^n\left(Id-\Tc\right)\right))_{n\in\N}$ where  the shift operator $\Tc$ is defined by \eqref{def:Tc}. This amount to proving a linear stability result for the studied discrete shock profile.
			\item Finally, one can use the decay estimates obtained in the second point to hopefully conclude a nonlinear orbital stability result in suitable norms.
		\end{enumerate}
		
		The main result of the present article (Theorem \ref{th:Green} below) provides an accurate description of the long time behavior of the Green's function and its discrete derivative for spectrally stable discrete shock profiles associated with \textit{stationary Lax shocks} when the numerical scheme \textit{introduces numerical viscosity} (and under a few additional technical hypotheses). As a consequence, we also prove decay estimates on the semi-group $(\Lcc^n)_{n\in\N}$ when it acts on $\ell^r$-spaces (see Theorem \ref{thStab} below) as an example of linear stability result. Though the decay estimates of Theorem \ref{thStab} might not be sufficient to conclude a nonlinear stability result, we claim that the statement of Theorem \ref{th:Green} should be precise enough to find a suitable functional framework to achieve a nonlinear stability result for these discrete shock profiles. This was achieved in the particular case of \textit{scalar} conservation laws by the author in a following paper \cite{Coeuret2024} using Theorem \ref{th:Green}. A generalization in the general case of \textit{systems} of conservation laws could be investigated in a future paper.

	Let us now focus on the study of the Green's function and the statement of Theorem \ref{th:Green}. It can be seen as an improvement on the result of \cite{Godillon} that highly influenced the analysis performed in the present paper. In \cite{Godillon,Godillon_these}, Lafitte-Godillon generalizes in the fully discrete setting several tools introduced in \cite{ZH} that are necessary to study the Green’s function for the linearized operator. More precisely, she constructs the Evans function for this problem and introduces in her thesis \cite{Godillon_these} the notion of geometric dichotomies (an equivalent version of the exponential dichotomies for the discrete dynamical systems). Those tools will be redefined and used intensively in the present paper. Lafitte-Godillon then attempts to obtain precise estimates on the Green's function of the linearized operator. However, the result of \cite{Godillon} has two limitations:
		\begin{itemize}
				\item The proof is done specifically for the modified Lax-Friecrichs scheme. This is not a strong limitation as it is quite clear that the content of the paper \cite{Godillon} can be generalized to a larger class of numerical schemes (at least for odd ordered schemes).
				
				\item The estimates on the Green's function proved in \cite[Theorem 1.1]{Godillon} are not sufficient to conclude on the nonlinear stability as they are only local with respect to the initial localization of the Dirac mass associated with the Green’s function (the parameter $l$ in \cite[Theorem 1.1]{Godillon} which corresponds to the parameter $j_0$ in Theorem \ref{th:Green}). This is a consequence of the analysis on the so-called spatial Green's function (defined below by \eqref{defGreenSpatial}) done in \cite{Godillon} which is not precise enough for linear/nonlinear stability purposes.
			\end{itemize}
		In the present paper, we solve those issues by describing precisely the leading order of the Green’s function and proving sharp and \emph{uniform} estimates on the remainder. We also consider schemes of any odd order, in particular with only few restrictions on the size of the stencil of the scheme.

		Before finishing this introduction section, we wish to point out to the reader that, even though the aim of this article is to obtain a result for the study of the Green's function of discrete shock profiles with a choice of context and a proof that are as general as possible, we still have to introduce fairly restrictive assumptions. This also opens up several direction to improve the analysis of the present article.
		\begin{itemize}
			\item The stability analysis has been restricted to discrete shock profiles associated with \textit{stationary Lax shocks}. It would interesting to investigate the case of discrete shock profiles associated with under- and overcompressive shocks (like in \cite{Godillon}) as well as the case of moving shocks, and more precisely of moving discrete shock profiles with \textit{rational speeds}. We point out that the latter case might be more difficult to study. 
			\item Results on the verification of the spectral stability of discrete shock profiles (Hypotheses \ref{H:spec} and \ref{H:Evans} below), which is the starting point of our analysis, are scarce. Known results are usually proved in very restrictive contexts (see for instance \cite[Chapter 3]{Coulombel2024} for such an investigation specifically for the Lax-Wendroff scheme and scalar conservation laws). We also highlight the Open Question 5.1 \cite{Serre} which tackles a necessary criterion for the spectral stability of weak discrete shock profiles, which stills remains to be investigated.
			\item The proof of the main result of the present article relies on additional structural assumptions related to the choice of finite difference schemes. For instance, we introduce a joint reduction assumption (Hypothesis \ref{H:VPAk} below) that plays a central role in Section \ref{sec:GSnear1} to construct the Evans function and prove a precise and useful expression of the spatial Green's function. Similarly, the finite difference scheme must satisfy a viscosity assumption (Hypothesis \ref{H:F} below) which is verified for first order schemes and possibly for higher odd ordered schemes. In might be possible (and most importantly desirable) to relax or even remove those assumptions to generalize the result of the present paper for a larger family of finite difference schemes.
		\end{itemize}

	\subsection{Definition of stationary discrete shock profiles (SDSP)}
	
	We consider a one-dimensional system of conservation laws 
	\begin{equation}\label{def:EDP}
		\begin{array}{c}
			\partial_t u+\partial_x f(u)=0, \quad t\in\R_+, x\in\R,\\
			u:\R_+\times\R\rightarrow \Uc,
		\end{array}
	\end{equation}
	where $d\in\N\backslash\lc0\rc$ corresponds to the number of unknown $u=:(u_1,\ppp,u_d)$ of \eqref{def:EDP}, the space of states $\Uc$ is an open set of $\R^d$ and the flux $f:\Uc\rightarrow \R^d$ is a $\Ccc^\infty$ function. We will suppose that the system of conservation laws is hyperbolic, meaning that for all $u\in\Uc$, the jacobian matrix $df(u)$ is diagonalisable with real eigenvalues.
	
	We fix two states $u^-,u^+\in\Uc$ such that
	\begin{equation}\label{eq:RankineHugioniot}
		f(u^-)=f(u^+).
	\end{equation}
	This is the well-known Rankine-Hugoniot condition which allows to state that the standing shock defined by
	\begin{equation}\label{def:Choc}
		\forall t\in\R_+,\forall x\in\R,\quad u(t,x):=\lc\begin{array}{cc}u^- & \text{ if }x<0, \\ u^+ & \text{ else},\end{array}\right.
	\end{equation}
	is a weak solution of \eqref{def:EDP}.
	
	Since the system of conservation laws we consider is hyperbolic at the states $u^\pm$, we introduce the eigenvalues $\lambg_1^\pm,\ppp,\lambg_d^\pm\in\R$ and a basis of nonzero eigenvectors $\rg_1^\pm,\ppp,\rg_d^\pm\in\R^d$ of $df(u^\pm)\in\Mc_d(\R)$ associated with those eigenvalues. We also define the invertible matrix 
	\begin{equation}\label{def:P}
		\Pg^\pm:=\begin{pmatrix}
			\rg_1^\pm& |\ppp| & \rg_d^\pm
		\end{pmatrix}\in \Mc_d(\R)
	\end{equation}
	and the dual basis $\lg_1^\pm,\ppp,\lg_d^\pm\in\R^d$ associated with the eigenvectors $\rg_1^\pm,\ppp,\rg_d^\pm$:
	\begin{equation}\label{def:l}
		\begin{pmatrix}
			\lg_1^\pm& |\ppp| & \lg_d^\pm
		\end{pmatrix}^T:=\left({\Pg^\pm}\right)^{-1}.
	\end{equation}
	The vectors $\lg_l^\pm$ are then eigenvectors of $df(u^\pm)^T$ associated with the eigenvalues $\lambg_l^\pm$. We organize the eigenvalues so that
	$$\lambg_1^\pm\leq\ppp\leq \lambg_d^\pm.$$
	In this paper, we focus our attention on Lax shocks.
	
	\begin{hypo}[Lax shock]\label{H:Lax}
		We assume that the shock is non-characteristic, i.e.: $$0\notin\sigma(df(u^\pm))=\lc\lambg_1^\pm,\ppp,\lambg_d^\pm\rc.$$
		Furthermore, we assume that there exists an index $I\in\lc1,\ppp,d\rc$ such that
		$$\begin{array}{c}
			\lambg_{I}^+<0<\lambg_{I+1}^+,\\
			\lambg_{I-1}^-<0<\lambg_{I}^-,
		\end{array}$$
		where we use the convention $\lambg_0^\pm:=-\infty$ and $\lambg_{d+1}^\pm:=+\infty$ if $I=1$ or $I=d$.
	\end{hypo}
	
	We consider a constant $\nug>0$ and introduce a space step $\Delta x>0$ and a time step $\Delta t:=\nug \Delta x>0$. The constant $\nug$ then corresponds to the ratio between the space and time steps. We impose the following CFL condition on the constant $\nug$:\footnote{Up to considering that the space of state $\Uc$ is a close neighborhood of the SDSP defined underneath in Hypothesis \ref{H:SDSP}, we should be able to satisfy such a condition.}
	\begin{equation}\label{cond:CFL}
		\forall u\in\Uc,\quad \nug\min\sigma(df(u))>-q \quad \text{and} \quad \nug\max\sigma(df(u))< p.
	\end{equation}
	We introduce the discrete evolution operator $\Nc: \Uc^\Z \rightarrow (\R^d)^\Z$ defined for $u=(u_j)_{j\in\Z}\in\Uc^\Z$ as
	\begin{equation}\label{def:evol}
		\forall j\in \Z, \quad (\Nc (u))_j := u_j- \nug\left(F\left(\nug;u_{j-p+1},\ppp, u_{j+q}\right)-F\left(\nug;u_{j-p},\ppp, u_{j+q-1}\right)\right),
	\end{equation}
	where $p,q\in\N\backslash\lc0\rc$ and the numerical flux $F:(\nu;u_{-p},\ppp, u_{q-1})\in]0,+\infty[\times \Uc^{p+q}\rightarrow \R^d$ is a $\Cc^1$ function. We are interested in solutions of the conservative one-step explicit finite difference scheme defined by 
	\begin{equation}\label{def:SchemeNum}
		\forall n\in\N,\quad u^{n+1}=\Nc (u^n)
	\end{equation}
	where $u^0\in\Uc^\Z$. 
	
	We assume that the numerical scheme satisfies the following consistency condition with regards to the PDE \eqref{def:EDP}
	\begin{equation}\label{cond:consistency}
		\forall \nu\in]0,+\infty[, \forall u \in\Uc, \quad F(\nu;u,\ppp,u)=f(u).
	\end{equation}
	This condition is often a consequence of the condition \eqref{cond:CFL} verified by $\nug$.
	 
	Traveling wave solutions of the numerical scheme \eqref{def:SchemeNum} linking two end states of a shock wave are the so-called discrete shock profiles. Since we are considering stationary shocks \eqref{def:Choc} in the present paper, the associated discrete shock profiles will also be stationary and will thus correspond to fixed points of the operator $\Nc$.
	
	\begin{hypo}[Existence of a stationary discrete shock profile (SDSP)]\label{H:SDSP}
		We suppose that there exists a sequence $\dsp=(\dsp_j)_{j\in\Z}\in\Uc^\Z$ that satisfies
		$$\Nc(\dsp)=\dsp \quad \text{and} \quad \dsp_j\underset{j\rightarrow \pm\infty}\rightarrow u^\pm.$$
	\end{hypo}
	
	Let us point out that in \cite{Serre}, it is proved that the existence of a SDSP implies that the Rankine-Hugoniot \eqref{eq:RankineHugioniot} is verified. However, the existence of a SDSP for all admissible and physically significant standing shock is not fully answered. Existence results tend to actually prove the existence of a continuous one-parameter family of discrete shock profiles. The main results tackling the issue of existence of SDSP would be \cite{Majda-Ralston,Michelson,Jennings}:
	\begin{itemize}
		\item In \cite{Jennings}, Jennings focuses on discrete shock profiles for monotone conservative schemes applied to scalar conservation laws. In this context, he proves the existence and uniqueness of a continuous one-parameter family of discrete shock profiles associated with shocks of any strength for rational speeds. He also proves nonlinear orbital stability for such DSPs.
		
		\item In \cite{Majda-Ralston}, Majda and Ralston tackle the case of system of conservation laws and prove the existence of a continuous one-parameter family of DSPs with rational speeds. They introduce two limitations though: They consider schemes of order 1 (this corresponds to the case where $\mu=1$ in Hypothesis \ref{H:F} below) and they only consider weak shocks, i.e. shocks where the difference between the two states must be small enough. The result is generalized in \cite{Michelson} for schemes of order 3 (i.e. $\mu=2$ in Hypothesis \ref{H:F} below).
	\end{itemize}
 	
 	The following assumption on the convergence of the SDSP $\overline{u}^s$ towards its limit state is important in our work as it it used to construct some of the main tools needed to carry out the analysis (for instance to prove the geometric dichotomy in Section \ref{subsec:GeoDich} or for the proof of Lemma \ref{lem_choice_base}).
 	
 	\begin{hypo}[Exponential convergence of the SDSP towards its limit states]\label{H:CVexpo}
 		There exist some constants $C,c>0$ such that
 		\begin{equation}\label{in:CVexpoSDSP}
 			\forall j \in\N, \quad \begin{array}{c}  |\overline{u}^s_j-u^+|\leq Ce^{-cj}, \\
 				|\overline{u}^s_{-j}-u^-|\leq Ce^{-cj}.
 			\end{array}
 		\end{equation}
 	\end{hypo}
 	
 	Hypothesis \ref{H:CVexpo} can most likely be proved to be a consequence of the shock being non-characteristic (Hypothesis \ref{H:Lax}). We refer to \cite[Corollary 1.2]{ZH} for a proof of this fact in the continuous setting and \cite[Lemma 1.1]{BeckHupkesSandstedeZumbrun} in the semi-discrete case.
 	
 	\subsection{Linearized scheme about the end states \texorpdfstring{$u^\pm$}{u+-}}
 	
 	Let us now introduce some hypotheses on the end states $u^+$ and $u^-$ and on the considered numerical scheme. To summarize briefly the main assumptions, we mainly ask for the numerical schemes we consider to introduce numerical viscosity and to have linear $\ell^r$-stability at the states $u^+$ and $u^-$ for any $r\in[1,+\infty]$.
 	
 	We linearize the discrete evolution operator $\Nc$ about the constant states $u^-$ and $u^+$ and thus introduce the bounded operators $\Lcc^\pm$ acting on $\ell^r(\Z,\C^d)$ with $r\in[1,+\infty]$ defined by
 	\begin{equation}\label{def:linearizedSchemeEndState}
 		\forall h\in \ell^r(\Z,\C^d),\forall j\in \Z, \quad (\Lcc^\pm h)_j := \sum_{k=-p}^{q}A^\pm_kh_{j+k},
 	\end{equation}
 	where for $k\in\lc-p,\ppp,q-1\rc$, we first define
 	\begin{equation}\label{def:Bk}
 		B_k^\pm:=\displaystyle\nug \partial_{u_k}F\left(\nug;u^\pm,\ppp,u^\pm\right) \in\Mc_d(\C)
 	\end{equation}
 	and then for $k\in\lc-p,\ppp,q\rc$, we let
 	\begin{equation}\label{def:Ak}
	 	A_k^\pm:=
	 	\lc\begin{array}{cc}
	 		-B_{q-1}^\pm & \text{ if }k=q,\\
	 		B_{-p}^\pm & \text{ if }k=-p,\\
	 		\delta_{k,0}Id +B_k^\pm - B_{k-1}^\pm & \text{ else.}
	 	\end{array}\right.
 	\end{equation}
 	
 	We start by introducing the following assumption on the matrices $B_k^\pm$ and $A_k^\pm$.
 	\begin{hypo}\label{H:VPAk}
 		For $k\in\lc-p,\ppp,q-1\rc$, the eigenvectors $\rg_1^\pm,\ppp,\rg_d^\pm$ of $df(u^\pm)$ are also eigenvectors of the matrices $B_k^\pm$ defined by \eqref{def:Bk}. If the system of conservation laws \eqref{def:EDP} is strictly hyperbolic at the states $u^\pm$, this is equivalent to the matrices $df(u^\pm)$ and $B_k^\pm$ commuting.
 	\end{hypo} 
 	
 	Hypothesis \ref{H:VPAk} implies that the eigenvectors $\rg_1^\pm,\ppp,\rg_d^\pm$ of $df(u^\pm)$ are also eigenvectors of the matrices $A^\pm_k$ defined by \eqref{def:Ak} for $k\in\lc-p,\ppp,q\rc$. Hypothesis \ref{H:VPAk} is fairly usual and is not that far fetched since the consistency condition \eqref{cond:consistency} links the numerical flux $F$ and the flux $f$ and that, most of the time, the matrices $B_k^\pm$ defined by \eqref{def:Bk} are expressed using $df(u^\pm)$. We can then introduce the notation for $k\in\lc-p,\ppp,q\rc$:
 	\begin{equation}\label{def:lambda_k}
 		\begin{pmatrix}
 			\lambda_{1,k}^\pm & &\\
 			&\ddots &\\
 			& & \lambda_{d,k}^\pm
 		\end{pmatrix}:= {\Pg^\pm}^{-1}A_k^\pm \Pg^\pm.
 	\end{equation}
 	For $l\in\lc1,\ppp,d\rc$, we then define the meromorphic function $\Fc^\pm_l$ on $\C\backslash\lc0\rc$ by:
 	\begin{equation}\label{def:Fcl}
 		\forall \kappa\in\C\backslash\lc0\rc, \quad \Fc_l^\pm(\kappa):=\sum_{k=-p}^q \lambda_{l,k}^\pm\kappa^k\in\C.
 	\end{equation}
 	
 		By \eqref{def:lambda_k} and \eqref{def:Fcl}, the functions $\Fc_l^\pm$ correspond to the eigenvalues of the amplification matrix $\sum_{k=-p}^q\kappa^kA_k^\pm$ of the linearization of the numerical scheme at the states $u^\pm$. 
 	 
 	They will thus allow us to characterize the spectrum of the operators $\Lcc^\pm$. We refer for instance to \cite{Thomee,D-S,R-S,C-F, CoeuLLT} for a study in the scalar case of similar convolution operators as $\Lcc^\pm$. Fourier analysis and in particular the well-known Wiener theorem \cite{Newman} imply that:
 	\begin{equation}\label{spec_Li}
 		\sigma(\Lcc^\pm)=\bigcup_{\kappa\in\S^1}\sigma\left(\sum_{k=-p}^q\kappa^kA_k^\pm\right)=\bigcup_{l=1}^d\Fc_l^\pm(\S^1).
 	\end{equation}
 	
 	The definition \eqref{def:Ak} of the matrices $A_k^\pm$ and the consistency condition \eqref{cond:consistency} imply that:
 	$$\sum_{k=-p}^qA_k^\pm=Id \quad \text{ and } \quad \sum_{k=-p}^q kA_k^\pm=-\nug df(u^\pm)$$
 	which translates into having:
 	\begin{equation}\label{eg:FcFin}
 		\forall l\in\lc1,\ppp,d\rc,\quad \Fc_l^\pm(1)=1 \quad \text{and} \quad \alpha_l^\pm:=-{\Fc_l^\pm}^\prime(1)=\nug \lambg_l^\pm\neq 0
 	\end{equation}
 	with $\lambg_l^\pm$ defined in Hypothesis \ref{H:Lax}.
 	
 	The following assumption is linked to the linear $\ell^r$-stability of the numerical scheme \eqref{def:SchemeNum} at the end states which corresponds to the $\ell^r$-power boundedness of the operators $\Lcc^\pm$.
 	
 	\begin{hypo}\label{H:F}
 		For all $l\in\lc1,\ppp,d\rc$, we have:
 		$$\forall \kappa\in \S^1\backslash \lc1\rc, \quad |\Fc^\pm_l(\kappa)|<1.\quad \text{(Dissipativity condition)}$$
 		Moreover, we suppose that there exists an integer $\mu\in \N\backslash\lc0\rc$ and for all $l\in\lc1,\ppp,d\rc$, there exists a complex number $\beta^\pm_l$ with positive real part such that:
 		\begin{equation}\label{F}
 			\Fc^\pm_l(e^{i\xi})\underset{\xi\rightarrow0}= \exp(-i\alpha_l^\pm\xi - \beta_l^\pm \xi^{2\mu} + O(|\xi|^{2\mu+1})).\quad \text{(Diffusivity condition)}
 		\end{equation}
 	\end{hypo}

 		Hypothesis \ref{H:F} implies that the linearization of the numerical scheme at the states $u^\pm$ is dissipative with an order of dissipation $2\mu$ (see \cite[Section IV.1.3]{Godlewski}), i.e. there exists a constant $\delta>0$ such that:
 		\begin{equation}\label{def:Dissip}
 			\forall l\in\lbrace1,\hdots,d\rbrace,\forall \xi\in[-\pi,\pi],\quad \left|\Fc^\pm_l(e^{i\xi})\right|\leq 1-\delta |\xi|^{2\mu}.
 		\end{equation} 
 		Let us point out that, reciprocally, the dissipation property \eqref{def:Dissip} with $\mu=1$ also immediately implies Hypothesis \ref{H:F}. Most finite difference schemes of order of accuracy $1$ satisfying some CFL condition will verify \eqref{def:Dissip} with $\mu=1$. For instance, this is true for monotone schemes in the case of scalar conservation laws.
 		
 		However, when considering the case where the dissipation property \eqref{def:Dissip} is verified for $\mu>1$, it is still possible to have terms of the form $i c\xi^k$ with $c\in \R\backslash\lbrace0\rbrace$ and $k\in\lbrace2,\hdots,2\mu-1\rbrace$ that appear in the asymptotic expansion of the logarithm of $\Fc_l^\pm(e^{i\xi})$ at $\xi=0$ and Hypothesis \ref{H:F} would thus not be verified. This is for instance the case of the Lax-Wendroff scheme which is dissipative of order $4$ but a term of order $3$ appears in the asymptotic expansion of the logarithm of $\Fc_l^\pm(e^{i\xi})$. The possibly additional terms appearing in this asymptotic expansion are linked with the order of accuracy of the numerical scheme (see \cite{Hedstrom1975} for details in the case of finite difference schemes for the scalar transport equation). Verifying Hypothesis \ref{H:F} and thus having the expansion \eqref{F} essentially translates the fact that the modified equation truncated at order $2\mu$ and associated with the linearization of the numerical scheme about the states $u^\pm$ introduces only a numerical viscosity term of order $2\mu$ (see \cite[(1.7),(1.8)]{Hedstrom1975}).
 		
 		From the standpoint of the power-boundedness of $\Lcc^\pm$, simply asking for the $\ell^2$-power boundedness of the operator $\Lcc^\pm$ would be equivalent to asking that $\sigma(\Lcc^\pm)\subset\overline{\D}$ (Von Neumann condition). The stronger Hypothesis \ref{H:F} is thus also inspired by the fundamental contribution \cite{Thomee} due to Thomée and has much further consequences, as the asymptotic expansion \eqref{F} assures the $\ell^r$-power boundedness of the operator $\Lcc^\pm$ for every $r$ in $[1,+\infty]$ (see \cite[Theorem 1]{Thomee} which focuses in the scalar case on the $\ell^\infty$-power boundedness but also studies the $\ell^r$-power boundedness as a consequence).

 	We conclude this section by defining the open set $\Oc$ which corresponds to the unbounded connected component of $\C\backslash(\sigma(\Lcc^+)\cup\sigma(\Lcc^-))$ represented on Figure \ref{fig:spec}. Hypothesis \ref{H:F} implies that $\overline{\U}\backslash\lc1\rc\subset \Oc$. 
 	
 	\begin{figure}
 		\begin{center}
 			\begin{tikzpicture}[scale=2]
 				\fill[color=gray!20] (-1.5,-1.5) -- (-1.5,1.5) -- (1.5,1.5) -- (1.5,-1.5) -- cycle;
 				\fill[color=white] plot [samples = 100, domain=0:2*pi] ({1/4+3*cos(\x r)/4},{sin(\x r)/3}) -- cycle ;
 				\fill[color=white] plot [samples = 100, domain=0:2*pi] ({1/2.5+1.5*cos(\x r)/2.5},{sin(\x r)/2}) -- cycle ;
 				
 				\draw (0,0) circle (1);
 				
 				\draw[thick,red] plot [samples = 100, domain=0:2*pi] ({(1/4+3*cos(\x r)/4},{sin(\x r)/3});
 				\draw[thick,red] plot [samples = 100, domain=0:2*pi] ({1/2.5+1.5*cos(\x r)/2.5},{sin(\x r)/2});
 				\draw[dartmouthgreen] (-0.7,0.2) node {$\bullet$};
 				\draw[dartmouthgreen] (0.5,0.6) node {$\bullet$};
 				\draw[dartmouthgreen] (-0.1,-0.7) node {$\bullet$};
 				
 				\draw (45:1.2) node {$\S^1$};
 				\draw[color=black!70] (-1,1.2) node {$\Oc$};
 				\draw[red] (-0.2,0.55) node {$\sigma(\Lcc^\pm)$};
 				\draw[dartmouthgreen] (1,0) node {$\bullet$} node[right] {$1$};
 			\end{tikzpicture}
 			\caption{In red, we have the spectrum of the operators $\Lcc^\pm$ which corresponds to the union of the curves $\Fc_l^\pm(\S^1)$. Here we chose $d=1$, i.e. one curve corresponds to the spectrum of $\Lcc^+$ and the other corresponds to the spectrum of $\Lcc^-$. In gray, we represent the set $\Oc$ which corresponds to the unbounded component of $\C\backslash(\sigma(\Lcc^+)\cup\sigma(\Lcc^-))$. The elements of the set $\Oc$ are either eigenvalues of the operator $\Lcc$ (represented by green dots) or belong to the resolvent set $\rho(\Lcc)$. We know that $1$ is an eigenvalue of $\Lcc$ and Hypothesis \ref{H:spec} implies that the eigenvalues of $\Lcc$ in $\Oc$ are located within the open unit disk.}
 			\label{fig:spec}
 		\end{center}
 	\end{figure}
 	
 	\subsection{Linearized scheme about the SDSP \texorpdfstring{$\dsp$}{dsp}}
 	
 	We now linearize the discrete evolution operator $\Nc$ about the discrete shock profile $\dsp$ and thus define the bounded operator $\Lcc$ acting on $\ell^r(\Z,\C^d)$ with $r\in[1,+\infty]$ defined by:
 	\begin{equation}\label{def:linearizedScheme}
 		\forall h\in \ell^r(\Z,\C^d),\forall j\in \Z, \quad (\Lcc h)_j := \sum_{k=-p}^{q}A_{j,k}h_{j+k},
 	\end{equation}
 	where for $j\in\Z$ and $k\in\lc-p,\ppp,q-1\rc$, we first define the matrix
 	\begin{equation}\label{def:Bjk}
 		B_{j,k}:=\displaystyle\nug \partial_{u_k}F\left(\nug;\dsp_{j-p},\ppp,\dsp_{j+q-1}\right)\in\Mc_d(\C) 
 	\end{equation}
 	and for $j\in\Z$ and $k\in\lc-p,\ppp,q\rc$:
 	\begin{equation}\label{def:Ajk}
 		A_{j,k}:=
 		\lc\begin{array}{cc}
 			-B_{j+1,q-1} & \text{ if }k=q,\\
 			B_{j,-p} & \text{ if }k=-p,\\
 			\delta_{k,0}Id +B_{j,k} - B_{j+1,k-1} & \text{ otherwise.}
 		\end{array}\right.
 	\end{equation}
 	
 	We observe that since the SDSP $(\dsp_j)_{j\in\Z}$ converges exponentially fast towards its limit states $u^\pm$, we have that the matrices $A_{j,k}$ (resp. $B_{j,k}$) converge exponentially fast towards the matrices $A_k^\pm$ (resp. $B_k^\pm$) defined by \eqref{def:Ak} (resp. \eqref{def:Bk}) as $j$ tends towards $\pm\infty$.
 	
 	We will now focus on the spectral properties of the operator $\Lcc$ when it acts on $\ell^2(\Z,\C^d)$. The following proposition which localizes the essential spectrum of the operator $\Lcc$ is central.
 	\begin{prop}\label{prop:SpectreEssentielLcc}
 		We have that
 		$$\sigma_{ess}(\Lcc)\cap \Oc = \emptyset.$$
 	\end{prop} 	
 	Proposition \ref{prop:SpectreEssentielLcc} allows us to conclude that for $z\in \Oc$, $zId_{\ell^2}-\Lcc$ is a Fredholm operator of index $0$ and thus that $z$ either belongs to the resolvent set of $\Lcc$ or is an eigenvalue of $\Lcc$. Proposition \ref{prop:SpectreEssentielLcc} is proved for instance in \cite[Theorem 4.1]{Serre} using the so-called geometric dichotomy developped in the thesis of Lafitte-Godillon \cite[Section III.1.5]{Godillon_these}. We will have to reintroduce the geometric dichotomy in Section \ref{sec:GSloin} and we will thus provide the reader with the proof of Proposition \ref{prop:SpectreEssentielLcc} (see Lemma \ref{lem_spec_ess} hereafter).
 	
 	We will now introduce the spectral stability assumption that we impose on our SDSP $\dsp$. It can be separated in two parts.
 	\begin{hypo}\label{H:spec}
 		The operator $\Lcc$ has no eigenvalue of modulus equal or larger than $1$ other than $1$.
 	\end{hypo}
 	Combining Hypothesis \ref{H:spec} with Proposition \ref{prop:SpectreEssentielLcc}, we can then conclude that the set $\overline{\U}\backslash\lc1\rc$ is included in the resolvent set of $\Lcc$. 
 	
 	The second part of the spectral stability assumption and the last hypothesis we will introduce on the spectrum of the operator $\Lcc$ has to do with the so-called Evans function $\mathrm{Ev}$ defined later on in the article by \eqref{def:Ev}. This is a complex holomorphic function defined in a neighborhood of $1$ that vanishes at the eigenvalues of $\Lcc$. The Evans function plays the role of a characteristic polynomial for the operator $\Lcc$.
 	
 	We will show that under the previous hypotheses, $1$ is an eigenvalue of the operator $\Lcc$ and thus that the Evans function $\mathrm{Ev}$ vanishes at $1$. This is the consequence of the existence of a differentiable one-parameter family of SDSPs associated with the one we are studying. In continuous and semi-discrete settings as in \cite{ZH,BeckHupkesSandstedeZumbrun}, there is an underlying regular profile that describes the differentiable one-parameter family of traveling waves studied in those papers as translations of said profile. The one-parameter family then corresponds to an invariance by translation. The derivative of the regular profile then belongs to the kernel of the linearized operator in these settings. In our present fully discrete setting, the existence of such a regular underlying profile that describes the different SDSPs of the one-parameter family is not clear.
 	
 	Here, we will make a strong hypothesis on the behavior of the Evans function at $1$.
 	\begin{hypo}\label{H:Evans}
 		We have that $1$ is a simple zero of the Evans function $\mathrm{Ev}$ defined below by \eqref{def:Ev}, i.e.
 		$$\mathrm{Ev}(1)=0 \quad \text{ and } \quad \frac{\partial \mathrm{Ev}}{\partial z}(1)\neq 0.$$
 	\end{hypo}
 	
	 	We will show that Hypothesis \ref{H:Evans} implies that the eigenspace of $\Lcc$ associated with the eigenvalue $1$ is of dimension $1$. More precisely, we will define a sequence $V\in \ell^2(\Z,\C^d)\backslash\lc0\rc$ below in \eqref{def:V} such that:
	 
 	\begin{equation}\label{egV}
 		\ker(Id_{\ell^2}-\Lcc)=\mathrm{Span} V
 	\end{equation}
 	and we will prove that this sequence $V$ converges exponentially fast towards $0$ at infinity, i.e. there exist two positive constants $C,c$ such that
 	\begin{equation}\label{decExpoV}
 		\forall j\in \Z, \quad |V(j)|\leq Ce^{-c|j|}.
 	\end{equation}
 	Coming back to the above discussion, if there exists a regular profile that allows us to describe the one-parameter family of SDSPs as this profile up to translations, then the sequence $V$ and the derivative of the said profile would be collinear.
 	
 	We finalize this section by introducing two last (technical) hypotheses.
 	\begin{hypo}\label{H:inv}
 		The matrices $A_{j,-p}=B_{j,-p}$ and $A_{j,q}=-B_{j+1,q-1}$ are invertible for all $j\in\Z$ and the matrices $A_{-p}^\pm=B_{-p}^\pm$ and $A_q^\pm=-B_{q-1}^\pm$ are also invertible. 
 	\end{hypo}
 	This hypothesis is usually a consequence of the CFL condition \eqref{cond:CFL}. Hypothesis \ref{H:inv} serves us in the article to express the eigenvalue problem associated with the operators $\Lcc$, $\Lcc^+$ and $\Lcc^-$ as dynamical systems (see Section \ref{subsec:GSloin:ExpEigen} consecrated to the so-called "spatial dynamics"). Finally, we impose for the following assumption to be verified.
 	\begin{hypo}\label{H:Mpm1}
 		For all $l\in \lc1, \ppp,d\rc$, the equation 
 		\begin{equation}\label{eg:Fcl}
 			\Fc_l^\pm(\kappa)=1
 		\end{equation}
 		has $p+q$ distinct solutions $\kappa\in \C\backslash\lc0\rc$.
 	\end{hypo}
 	Hypothesis \ref{H:Mpm1} will be used to prove that the matrix $M_l^\pm(1)$ defined by \eqref{def:Mlpm} is diagonalizable with simple eigenvalues. This will allow us to study the eigenvalues and eigenvectors of the matrix $M^\pm(1)$ defined by \eqref{def:M} in Section \ref{sec:GSnear1}. Let us observe that the expression \eqref{def:Fcl} of $\Fc_l^\pm$ implies that searching for solutions $\kappa\in \C^{p+q}\backslash\lc0\rc$ of \eqref{eg:Fcl} is equivalent to searching for zeroes of the polynomial:
 	$$\kappa\in\C\backslash\lc0\rc\mapsto \kappa^p-\sum_{k=-p}^q\lambda_{l,k}^\pm \kappa^{k+p}.$$
 	The function above has $p+q$ zeroes counted with multiplicity and Hypothesis \ref{H:Mpm1} is then just equivalent to asking for the zeroes of the above function to be simple. 
 	
 	\subsection{Temporal and spatial Green's functions}
 	
 	For $j_0\in\Z$, we define the temporal Green's function recursively as:
 	\begin{equation}\label{defGreenTempo}
 		\begin{array}{cc}
 			&\Gcc(0,j_0,\cdot) : = \delta_{j_0}\\
 			\forall n\in\N, & \Gcc(n+1,j_0,\cdot) : = \Lcc  \Gcc(n,j_0,\cdot) ,
 		\end{array}
 	\end{equation}
 	where the sequence $\delta_{j_0}$ is defined by
 	$$\delta_{j_0}:=\left(\delta_{j_0,j}Id\right)_{j\in\Z}\in\ell^2(\Z, \Mc_d(\C)).$$
 	For $z \in \rho(\Lcc)$, we also define the spatial Green's function $G(z,j_0,\cdot)$ as the only element of $\ell^2(\Z, \Mc_d(\C))$ such that:
	\begin{equation}\label{defGreenSpatial}
		(zId_{\ell^2}-\Lcc)G(z,j_0,\cdot)= \delta_{j_0}.
	\end{equation}

		The following lemma proved via a simple recurrence is a direct consequence of the definition \eqref{defGreenTempo} of the temporal Green's function and the finite speed propagation of the linearized scheme.
		\begin{lemma}\label{lem:GreenTempoOutside}
			For all $n\in \N$, $j_0,j\in \Z$, we have that
			$$j-j_0\notin \lc-nq,\ppp,np\rc\Rightarrow \Gcc(n,j_0,j)=0.$$
		\end{lemma}
		
		The main result of the present article (Theorem \ref{th:Green}) is a precise description of the temporal Green's function $\Gcc(n,j_0,j)$ and of its discrete derivative $\Gcc(n,j_0,j)-\Gcc(n,j_0-1,j)$ respectively when $j-j_0\in\lbrace-nq,\hdots,np\rbrace$ and when $j-j_0\in\lbrace-nq-1,\hdots,np\rbrace$. This result is \textit{essential} in order to hopefully conclude the nonlinear stability analysis of the discrete shock profile $\dsp$. Indeed, as exemplified in the following paper by the author \cite{Coeuret2024} or in \cite[Chapter 5]{Coulombel2024}, the nonlinear stability analysis about the discrete shock profile $\dsp$ relies crucially on proving sharp decay estimates on the families of operators $(\Lcc^n)_{n\in\N}$ and $(\Lcc^n(Id-\Tc))_{n\in\N}$ acting on well-chosen weighted $\ell^r$-spaces, where the shift operator $\Tc$ is defined by:
		\begin{equation}\label{def:Tc}
			\Tc: (h_j)_{j\in\Z}\in\C^\Z\mapsto (h_{j+1})_{j\in\Z}\in\C^\Z.
		\end{equation}
		Yet, for all $h\in \ell^r(\Z,\C^d)$ with $r\in[1,+\infty]$, we have that:
		\begin{align}
			\forall n\in\N,\forall j\in\Z,\quad & (\Lcc^nh)_j = \sum_{j_0\in\Z}\Gcc(n,j_0,j)h_{j_0},\label{eg:LnGreen}\\
			\forall n\in\N,\forall j\in\Z,\quad & (\Lcc^n(Id-\Tc)h)_j = \sum_{j_0\in\Z}\left(\Gcc(n,j_0,j)-\Gcc(n,j_0-1,j)\right)h_{j_0}.\label{eg:Ln(Id-T)GreenDer}
		\end{align}
		Therefore, Theorem \ref{th:Green} below should allow us to prove the required sharp bounds on the families of operators $(\Lcc^n)_{n\in\N}$ and $(\Lcc^n(Id-\Tc))_{n\in\N}$. As an example, we prove in Theorem \ref{thStab} such decay estimates for the semi-group $\left(\Lcc^n\right)_{n\in\N}$ acting on $\ell^r(\Z)$. More useful bounds when the operators $\Lcc^n$ and $\Lcc^n(Id-\Tc)$ act on polynomially weighted $\ell^r$-spaces are proved in \cite[Proposition 2]{Coeuret2024} using Theorem \ref{th:Green} whilst focusing on the case of \textit{scalar} conservation laws.
		
		Let us briefly present how the rest of the introduction will be structured. We start by defining the different terms \eqref{def:Termes_decompo_Green} appearing in the decomposition \eqref{decompoGreen} of the temporal Green's function in Theorem \ref{th:Green}. We will then state Theorem \ref{th:Green} describing precisely the long time behavior of the temporal Green's function and its discrete derivative and follow it up with a more informal and legible description of the statement of Theorem \ref{th:Green}. Finally, in accordance with the discussion in the paragraph above, we will state Theorem \ref{thStab} which tackles decay estimates of the semi-group $\left(\Lcc^n\right)_{n\in\N}$ acting on $\ell^r$-spaces.

	Let us start by defining the functions $H_{2\mu},E_{2\mu}: \R\rightarrow \C$ such that for $\beta \in \C$ with positive real part, we have:
	\begin{align}\label{def:H2mu_et_E2mu}
		\begin{split}
			\forall x \in \R,\quad &H_{2\mu}(\beta;x) := \frac{1}{2\pi} \int_\R e^{ixu}e^{-\beta u^{2\mu}}du,\\
			\forall x \in \R,\quad &E_{2\mu}(\beta;x) := \int_x^{+\infty} H_{2\mu}(\beta;y)dy,
		\end{split}
	\end{align}
	where we recall that the integer $\mu$ is defined in Hypothesis \ref{H:F}. Let us point out that Lemma \ref{lemHE} implies that the function $E_{2\mu}$ is well-defined. We call the functions $H_{2\mu}$ generalized Gaussians and the functions $E_{2\mu}$ generalized Gaussian error functions since for $\mu=1$, we have 
	$$\forall x \in \R,\quad H_{2}(\beta;x)=\frac{1}{\sqrt{4\pi\beta}}e^{-\frac{x^2}{4\beta}}.$$
	Noticing that the function $H_{2\mu}$ is an even function and that it is the inverse Fourier transform of $u\mapsto e^{-\beta u^{2\mu}}$, we observe that
	\begin{subequations}
		\begin{align}
			\lim_{x\rightarrow -\infty}E_{2\mu}(\beta;x)=\int_{-\infty}^{+\infty} H_{2\mu}(\beta;y)dy= 1 \label{eq:E_en_-infty}\\
			\forall x\in \R, \quad E_{2\mu}(\beta,-x)= 1-E_{2\mu}(\beta,x).
		\end{align}
	\end{subequations}
	
	The following lemma introduces some useful inequalities on the functions $H_{2\mu}$ and $E_{2\mu}$ defined by \eqref{def:H2mu_et_E2mu}.
	\begin{lemma}\label{lemHE}
		Let us consider a compact subset $A$ of $\lc z\in\C,\Re(z)>0\rc$ and integers $\mu,m\in\N\backslash\lc0\rc$. There exist two positive constants $C,c$ such that for all $\beta\in A$ 
		\begin{subequations}
			\begin{align}
				\forall x\in\R,\quad & |\partial_x^mH_{2\mu}(\beta;x)|\leq C\exp(-c|x|^\frac{2\mu}{2\mu-1}),\label{inH}\\
				\forall x\in]0,+\infty[,\quad & |E_{2\mu}(\beta;x)|\leq C\exp(-c|x|^\frac{2\mu}{2\mu-1}),\label{inE+}\\
				\forall x\in]-\infty,0[, \quad& |1-E_{2\mu}(\beta;x)|\leq C\exp(-c|x|^\frac{2\mu}{2\mu-1}).\label{inE-}
			\end{align}
		\end{subequations}
	\end{lemma}
	The interested reader can find a proof of \eqref{inH} when the subset $A$ is a one point set in \cite[Lemma 9]{CoeuLLT} or in \cite[Proposition 5.3]{Rob} for a more general point of view. By observing that the constants $C,c$ constructed in those proofs depend continuously on $\beta$, we can then conclude the proof \eqref{inH} for general sets $A$. Inequalities \eqref{inE+} and \eqref{inE-} for the function $E_{2\mu}$ are directly deduced by integrating the function $H_{2\mu}$ and using \eqref{eq:E_en_-infty} and \eqref{inH}.

		Let us now introduce the different terms that will describe the leading order of the pointwise asymptotic behavior of the Green's function $\Gcc(n,j_0,j)$ in \eqref{decompoGreen}. A more informal description of those terms will be given below Theorem \ref{th:Green}. We define for $n\in\N\backslash\lc0\rc$ and $j,j_0\in\Z$ the following functions:
	
	\begin{subequations}\label{def:Termes_decompo_Green}
		\begin{itemize}
			\item For $j_0\geq 0$ and $l\in\lc 1,\ppp,d\rc$
			\begin{equation}\label{def:S+}
				S_l^+(n,j_0,j):= \ind_{j\geq0}\ind_{\frac{j-j_0}{\alpha_l^+}\in\left[\frac{n}{2},2n\right]}\frac{1}{n^\frac{1}{2\mu}}H_{2\mu}\left(\beta_l^+,\frac{n\alpha_l^++j_0-j}{n^\frac{1}{2\mu}}\right)\rg_l^+{\lg_l^+}^T,
			\end{equation}
			
			\item For $j_0\geq 0$, $l^\prime\in\lc 1,\ppp,I\rc$ and $l\in\lc I+1,\ppp d\rc$ (i.e. such that $\alpha_{l^\prime}^+<0$ and $\alpha_l^+>0$)
			\begin{equation}\label{def:R+}
				R_{l^\prime,l}^+(n,j_0,j):= \ind_{j\geq0}\ind_{\frac{j}{\alpha_l^+}-\frac{j_0}{\alpha_{l^\prime}^+}\in\left[\frac{n}{2},2n\right]}\frac{1}{n^\frac{1}{2\mu}}H_{2\mu}\left(\frac{j}{n\alpha_l^+}\beta_l^+ -\frac{j_0}{n\alpha_{l^\prime}^+} \beta_{l^\prime}^+ \left(\frac{\alpha_l^+}{\alpha_{l^\prime}^+}\right)^{2\mu},\frac{n\alpha_l^++j_0\frac{\alpha_l^+}{\alpha_{l^\prime}^+}-j}{n^\frac{1}{2\mu}}\right)\rg_l^+{\lg_{l^\prime}^+}^T,
			\end{equation}
			
			\item For $j_0\geq 0$, $l^\prime\in\lc 1,\ppp,I\rc$ and $l\in\lc 1,\ppp I-1\rc$ (i.e. such that $\alpha_{l^\prime}^+<0$ and $\alpha_l^-<0$)
			\begin{equation}\label{def:T+}
				T_{l^\prime,l}^+(n,j_0,j):= \ind_{j<0}\ind_{\frac{j}{\alpha_l^-}-\frac{j_0}{\alpha_{l^\prime}^+}\in\left[\frac{n}{2},2n\right]}\frac{1}{n^\frac{1}{2\mu}}H_{2\mu}\left(\frac{j}{n\alpha_l^-}\beta_l^- -\frac{j_0}{n\alpha_{l^\prime}^+} \beta_{l^\prime}^+ \left(\frac{\alpha_l^-}{\alpha_{l^\prime}^+}\right)^{2\mu},\frac{n\alpha_l^-+j_0\frac{\alpha_l^-}{\alpha_{l^\prime}^+}-j}{n^\frac{1}{2\mu}}\right)\rg_l^-{\lg_{l^\prime}^+}^T,
			\end{equation}
			
			\item For $j_0\geq 0$ and $l^\prime\in\lc 1,\ppp,I\rc$ (i.e. such that $\alpha_{l^\prime}^+<0$)
			\begin{equation}\label{def:E+}
				E_{l^\prime}^+(n,j_0):= E_{2\mu}\left(\beta_{l^\prime}^+;\frac{n\alpha_{l^\prime}^+ +j_0}{n^\frac{1}{2\mu}}\right){\lg_{l^\prime}^+}^T,
			\end{equation}
			
			\item For $j_0< 0$ and $l\in\lc 1,\ppp,d\rc$
			\begin{equation}\label{def:S-}
				S_l^-(n,j_0,j):= \ind_{j\leq0}\ind_{\frac{j-j_0}{\alpha_l^-}\in\left[\frac{n}{2},2n\right]}\frac{1}{n^\frac{1}{2\mu}}H_{2\mu}\left(\beta_l^-,\frac{n\alpha_l^-+j_0-j}{n^\frac{1}{2\mu}}\right)\rg_l^-{\lg_l^-}^T,
			\end{equation}
			
			\item For $j_0< 0$, $l^\prime\in\lc I,\ppp,d\rc$ and $l\in\lc 1,\ppp I-1\rc$ (i.e. such that $\alpha_{l^\prime}^->0$ and $\alpha_l^-<0$)
			\begin{equation}\label{def:R-}
				R_{l^\prime,l}^-(n,j_0,j):= \ind_{j\leq0}\ind_{\frac{j}{\alpha_l^-}-\frac{j_0}{\alpha_{l^\prime}^-}\in\left[\frac{n}{2},2n\right]}\frac{1}{n^\frac{1}{2\mu}}H_{2\mu}\left(\frac{j}{n\alpha_l^-}\beta_l^- -\frac{j_0}{n\alpha_{l^\prime}^-} \beta_{l^\prime}^- \left(\frac{\alpha_l^-}{\alpha_{l^\prime}^-}\right)^{2\mu},\frac{n\alpha_l^-+j_0\frac{\alpha_l^-}{\alpha_{l^\prime}^-}-j}{n^\frac{1}{2\mu}}\right)\rg_l^-{\lg_{l^\prime}^-}^T,
			\end{equation}
			
			\item For $j_0< 0$, $l^\prime\in\lc I,\ppp,d\rc$ and $l\in\lc I+1,\ppp d\rc$ (i.e. such that $\alpha_{l^\prime}^->0$ and $\alpha_l^+>0$)
			\begin{equation}\label{def:T-}
				T_{l^\prime,l}^-(n,j_0,j):= \ind_{j>0}\ind_{\frac{j}{\alpha_l^+}-\frac{j_0}{\alpha_{l^\prime}^-}\in\left[\frac{n}{2},2n\right]}\frac{1}{n^\frac{1}{2\mu}}H_{2\mu}\left(\frac{j}{n\alpha_l^+}\beta_l^+ -\frac{j_0}{n\alpha_{l^\prime}^-} \beta_{l^\prime}^- \left(\frac{\alpha_l^+}{\alpha_{l^\prime}^-}\right)^{2\mu},\frac{n\alpha_l^++j_0\frac{\alpha_l^+}{\alpha_{l^\prime}^-}-j}{n^\frac{1}{2\mu}}\right)\rg_l^+{\lg_{l^\prime}^-}^T,
			\end{equation}
			
			\item For $j_0<0$, $l^\prime\in\lc I,\ppp,d\rc$ (i.e. such that $\alpha_{l^\prime}^->0$)
			\begin{equation}\label{def:E-}
				E_{l^\prime}^-(n,j_0):= E_{2\mu}\left(\beta_{l^\prime}^-;\frac{-n\alpha_{l^\prime}^- -j_0}{n^\frac{1}{2\mu}}\right){\lg_{l^\prime}^-}^T.
			\end{equation}
		\end{itemize}
	\end{subequations}

		We observe that the functions $S_l^\pm$, $R_{l^\prime,l}^\pm$ and $T_{l^\prime,l}^\pm$ are valued in $\Mc_d(\C)$ and the functions $E_{l^\prime}^\pm$ are valued in $\Mc_{1,d}(\C)$. 
		
		Let us now state Theorem \ref{th:Green} which gives an accurate description of the temporal Green's function and its discrete derivative. Below the statement of Theorem \ref{th:Green}, the reader will find a more legible explanation of Theorem \ref{th:Green} as well as one of its consequence (Theorem \ref{thStab}). We recall that the Landau notations $O$, $O_{\Mc_{1,d}}$ and $O_\C$ are described in the notations section at the beginning of the paper and are used respectively to differentiate matrices, line vectors and scalars.
		
		\begin{theorem}\label{th:Green}
			Let us assume that Hypotheses \ref{H:Lax}-\ref{H:Mpm1} are verified. Then, there exist:
			\begin{itemize}
				\item A sequence $V\in\ell^2(\Z,\C^d)\backslash\lc0\rc$ defined by \eqref{def:V} satisfying \eqref{egV} and \eqref{decExpoV},
				\item Families of complex scalars $(C^{E,\pm}_{l^\prime})_{l^\prime}$, $(C^{R,\pm}_{l^\prime,l})_{l^\prime,l}$ and $(C^{T,\pm}_{l^\prime,l})_{l^\prime,l}$,
			\end{itemize}
			such that for all $n\in\N\backslash\lc0\rc$, $j_0\in\N$ and $j\in\Z$ which verify $j-j_0\in\lc-nq,\ppp,np\rc$, we have that
			\begin{multline}\label{decompoGreen}
				\Gcc(n,j_0,j) = \sum_{l=1}^d S_l^+(n,j_0,j) +\sum_{l^\prime=1}^I\left(\sum_{l=I+1}^d C^{R,+}_{l^\prime,l}R_{l^\prime,l}^+(n,j_0,j) + \sum_{l=1}^{I-1} C^{T,+}_{l^\prime,l}T_{l^\prime,l}^+(n,j_0,j)\right)\\
				+V(j)\left(\sum_{l^\prime=1}^IC^{E,+}_{l^\prime}E_{l^\prime}^+(n,j_0) \right)+ \Rc(n,j_0,j)
			\end{multline}
			where $\Rc(n,j_0,j)$ is a faster decaying residual of the following form for some constant $c>0$ independent from $n,j_0,j$:

			\begin{subequations}\label{def:ResTh}
				$\bullet$ For $j\geq0$ such that $j-j_0\in\lc-nq,\ppp,np\rc$:
				\begin{align}
					\begin{split}
						&\Rc(n,j_0,j) = O(e^{-cn})\\
						& +\sum_{l=1}^d  \exp\left(-c\left(\frac{\left|n-\left(\frac{j-j_0}{\alpha_l^+}\right)\right|}{n^\frac{1}{2\mu}}\right)^\frac{2\mu}{2\mu-1}\right)\left(O\left(\frac{e^{-c|j|}}{n^\frac{1}{2\mu}}\right)+\rg_l^+O_{\Mc_{1,d}}\left(\frac{e^{-c|j_0|}}{n^\frac{1}{2\mu}}\right)+ O_\C\left(\frac{1}{n^\frac{1}{\mu}}\right)\rg_l^+{\lg_l^+}^T\right)\\
						& + \sum_{l^\prime=1}^I\sum_{l=I+1}^d \exp\left(-c\left(\frac{\left|n-\left(\frac{j}{\alpha_l^+}-\frac{j_0}{\alpha_{l^\prime}^+}\right)\right|}{n^\frac{1}{2\mu}}\right)^\frac{2\mu}{2\mu-1}\right)\left(O\left(\frac{e^{-c|j|}}{n^\frac{1}{2\mu}}\right)+\rg_l^+O_{\Mc_{1,d}}\left(\frac{e^{-c|j_0|}}{n^\frac{1}{2\mu}}\right) + O_\C\left(\frac{1}{n^\frac{1}{\mu}}\right)\rg_l^+{\lg_{l^\prime}^+}^T\right)\\
						& + \sum_{l^\prime=1}^I O\left(\frac{e^{-c|j|}}{n^\frac{1}{2\mu}}\exp\left(-c\left(\frac{\left|n+\frac{j_0}{\alpha_{l^\prime}^+}\right|}{n^\frac{1}{2\mu}}\right)^\frac{2\mu}{2\mu-1}\right)\right)  +\sum_{l=I+1}^d \rg_l^+O_{\Mc_{1,d}}\left(\frac{e^{-c|j_0|}}{n^\frac{1}{2\mu}}\exp\left(-c\left(\frac{\left|n-\frac{j}{\alpha_l^+}\right|}{n^\frac{1}{2\mu}}\right)^\frac{2\mu}{2\mu-1}\right)\right),
					\end{split}\label{def:ResTh1}
				\end{align}
				
				$\bullet$ For $j<0$ such that $j-j_0\in\lc-nq,\ppp,np\rc$:
				\begin{align}
					\begin{split}
						&\Rc(n,j_0,j) =O(e^{-cn})\\
						& +\sum_{l^\prime=1}^I\sum_{l=1}^{I-1}  \exp\left(-c\left(\frac{\left|n-\left(\frac{j}{\alpha_l^-}-\frac{j_0}{\alpha_{l^\prime}^+}\right)\right|}{n^\frac{1}{2\mu}}\right)^\frac{2\mu}{2\mu-1}\right)\left(O\left(\frac{e^{-c|j|}}{n^\frac{1}{2\mu}}\right)+\rg_l^-O_{\Mc_{1,d}}\left(\frac{e^{-c|j_0|}}{n^\frac{1}{2\mu}}\right)+ O_\C\left(\frac{1}{n^\frac{1}{\mu}}\right)\rg_l^-{\lg_{l^\prime}^+}^T\right) \\
						& + \sum_{l^\prime=1}^I O\left(\frac{e^{-c|j|}}{n^\frac{1}{2\mu}}\exp\left(-c\left(\frac{\left|n+\frac{j_0}{\alpha_{l^\prime}^+}\right|}{n^\frac{1}{2\mu}}\right)^\frac{2\mu}{2\mu-1}\right)\right) + \sum_{l=1}^{I-1} \rg_l^-O_{\Mc_{1,d}}\left(\frac{e^{-c|j_0|}}{n^\frac{1}{2\mu}}\exp\left(-c\left(\frac{\left|n-\frac{j}{\alpha_l^-}\right|}{n^\frac{1}{2\mu}}\right)^\frac{2\mu}{2\mu-1}\right)\right).
					\end{split}\label{def:ResTh2}
				\end{align}
			\end{subequations}
			
			The coefficients $C_{l^\prime,l}^{R,+}$, $C_{l^\prime,l}^{T,+}$ and $C_{l^\prime}^{E,+}$ are defined by \eqref{coeff}. There is a similar result when $j_0\in-\N\backslash\lc0\rc$ using the families of complex scalars $(C^{E,-}_{l^\prime})_{l^\prime}$, $(C^{R,-}_{l^\prime,l})_{l^\prime,l}$ and $(C^{T,-}_{l^\prime,l})_{l^\prime,l}$ and the functions $E_{l^\prime}^-$, $S_l^-$, $R_{l^\prime,l}^-$ and $T_{l^\prime,l}^-$.
			
			Furthermore, we can choose the constant $c>0$ such that the discrete derivative with regards to the parameter $j_0$ of the temporal Green's function verifies that, for all $n\in\N\backslash\lc0\rc$, $j_0\in\N$ and $j\in\Z$ which verify $j-j_0\in\lc-nq-1,\ppp,np\rc$:
			
			\begin{subequations}\label{decompoDerGreen}
				$\bullet$ For $j\geq0$:
				\begin{align}
					\begin{split}
						&\Gcc(n,j_0,j)- \Gcc(n,j_0-1,j) = O(e^{-cn})\\
						& +\sum_{l=1}^d  \exp\left(-c\left(\frac{\left|n-\left(\frac{j-j_0}{\alpha_l^+}\right)\right|}{n^\frac{1}{2\mu}}\right)^\frac{2\mu}{2\mu-1}\right)\left(O\left(\frac{e^{-c|j|}}{n^\frac{1}{2\mu}}\right)+\rg_l^+O_{\Mc_{1,d}}\left(\frac{e^{-c|j_0|}}{n^\frac{1}{2\mu}}\right)+ O_\C\left(\frac{1}{n^\frac{1}{\mu}}\right)\rg_l^+{\lg_l^+}^T\right)\\
						& + \sum_{l^\prime=1}^I\sum_{l=I+1}^d \exp\left(-c\left(\frac{\left|n-\left(\frac{j}{\alpha_l^+}-\frac{j_0}{\alpha_{l^\prime}^+}\right)\right|}{n^\frac{1}{2\mu}}\right)^\frac{2\mu}{2\mu-1}\right)\left(O\left(\frac{e^{-c|j|}}{n^\frac{1}{2\mu}}\right)+\rg_l^+O_{\Mc_{1,d}}\left(\frac{e^{-c|j_0|}}{n^\frac{1}{2\mu}}\right) + O_\C\left(\frac{1}{n^\frac{1}{\mu}}\right)\rg_l^+{\lg_{l^\prime}^+}^T\right)\\
						& + \sum_{l^\prime=1}^I O\left(\frac{e^{-c|j|}}{n^\frac{1}{2\mu}}\exp\left(-c\left(\frac{\left|n+\frac{j_0}{\alpha_{l^\prime}^+}\right|}{n^\frac{1}{2\mu}}\right)^\frac{2\mu}{2\mu-1}\right)\right)  +\sum_{l=I+1}^d \rg_l^+O_{\Mc_{1,d}}\left(\frac{e^{-c|j_0|}}{n^\frac{1}{2\mu}}\exp\left(-c\left(\frac{\left|n-\frac{j}{\alpha_l^+}\right|}{n^\frac{1}{2\mu}}\right)^\frac{2\mu}{2\mu-1}\right)\right),
					\end{split}
				\end{align}
				
				$\bullet$ For $j<0$:
				\begin{align}
					\begin{split}
						&\Gcc(n,j_0,j)- \Gcc(n,j_0-1,j) = O(e^{-cn})\\
						& +\sum_{l^\prime=1}^I\sum_{l=1}^{I-1}  \exp\left(-c\left(\frac{\left|n-\left(\frac{j}{\alpha_l^-}-\frac{j_0}{\alpha_{l^\prime}^+}\right)\right|}{n^\frac{1}{2\mu}}\right)^\frac{2\mu}{2\mu-1}\right)\left(O\left(\frac{e^{-c|j|}}{n^\frac{1}{2\mu}}\right)+\rg_l^-O_{\Mc_{1,d}}\left(\frac{e^{-c|j_0|}}{n^\frac{1}{2\mu}}\right)+ O_\C\left(\frac{1}{n^\frac{1}{\mu}}\right)\rg_l^-{\lg_{l^\prime}^+}^T\right) \\
						& + \sum_{l^\prime=1}^I O\left(\frac{e^{-c|j|}}{n^\frac{1}{2\mu}}\exp\left(-c\left(\frac{\left|n+\frac{j_0}{\alpha_{l^\prime}^+}\right|}{n^\frac{1}{2\mu}}\right)^\frac{2\mu}{2\mu-1}\right)\right) + \sum_{l=1}^{I-1} \rg_l^-O_{\Mc_{1,d}}\left(\frac{e^{-c|j_0|}}{n^\frac{1}{2\mu}}\exp\left(-c\left(\frac{\left|n-\frac{j}{\alpha_l^-}\right|}{n^\frac{1}{2\mu}}\right)^\frac{2\mu}{2\mu-1}\right)\right).
					\end{split}
				\end{align}
			\end{subequations}
			A similar description of the discrete derivative of the temporal Green's function holds for $j_0\in-\N\backslash\lc0\rc$.
		\end{theorem}

	\begin{figure}
		\begin{center}
			\begin{tikzpicture}[scale=1]
				\draw[->,thick] (-3,0) --(7,0) node[right] {$j$};
				\draw[->,thick] (0,-0.1) -- (0, 5) node[above] {$n$};
				\draw (3,0) node[below] {$j_0$};
				
				\draw[thick,blue] (3,0) -- (0,2.5);
				\draw[thick,dartmouthgreen] (0,2.5) -- (-1.5,5);
				\draw[thick,dartmouthgreen] (0,2.5) -- (-3,5);
				\draw[thick,purple] (0,2.5) -- (7,4.25);
				
				\draw[blue] (3,1) node {$S_l^+(n,j_0,j)$};
				\draw[dartmouthgreen] (-2.5,3.5) node {$T_{l^\prime,l}^+(n,j_0,j)$};
				\draw[purple] (4.5,4.1) node {$R_{l^\prime,l}^+(n,j_0,j)$};
				\draw[red] (1.25,4.25) node {$E_{l^\prime}^+(n,j_0)V(j)$};
				
				\draw[thick,blue] (3,0) -- (5,0.5);
				\draw[thick, dashed,blue] (5,0.5) -- (7,1);
				
				\draw[thick,blue] (3,0) -- (2,0.5);
				\draw[thick, dashed,blue] (2,0.5) -- (0,1.5);
				
				\draw[thick,blue] (3,0) -- (2,0.2);
				\draw[thick, dashed,blue] (2,0.2) -- (0,0.6);
				
				\draw[thick,red] plot [samples = 100, domain=-pi:pi] ({\x*0.3/pi},{4+0.2*cos(\x r)});
			\end{tikzpicture}
		\end{center}
		\caption{A schematic representation of the result of Theorem \ref{th:Green} on the Green's function $\Gcc(n,j_0,j)$. Here we represent a case where $j_0\in\N$, $d=4$ and $I=3$. We recall that the integer $I$ is defined by Hypothesis \ref{H:Lax}. We have $d$ generalized Gaussian waves (in blue) arising from the Dirac mass at $j_0$ which travel along the characteristics of the right state $u^+$. The ones reaching the shock location are decomposed into reflected waves (in purple), transmitted waves (in green) and the activation of the component of the Green's function along the vector space $\ker(Id_{\ell^2}-\Lcc)$ (in red). We only represent this decomposition for one of the incoming waves.}
		\label{FigGreen}
	\end{figure}
		
	Let us describe more clearly what the decomposition \eqref{decompoGreen} of the temporal Green's function $\Gcc(n,j_0,j)$ conveys for $j_0\geq 0$. The same description can be done for $j_0<0$ as well as for the discrete derivative of the temporal Green's function using \eqref{decompoDerGreen}. Figure \ref{FigGreen} is a schematic representation of the decomposition \eqref{decompoGreen}. The first term on the right hand-side using the function $S_{l}^+$  of \eqref{decompoGreen} corresponds to $d$ generalized Gaussian waves arising from the Dirac mass at $j_0$ which travel along the characteristics of the right state $u^+$ (the blue waves in Figure \ref{FigGreen}). The generalized Gaussian behavior of the different waves originates from the smearing effect caused by the diffusivity condition in Hypothesis \ref{H:F} which corresponds to the introduction of artificial viscosity at the states $u^\pm$. These waves correspond to the leading order of the Green's function of the operator $\Lcc^+$ associated with the "right" state $u^+$. Recalling that we are considering a Lax shock under Hypothesis \ref{H:Lax}, we observe the following distinction:
	\begin{itemize}
		\item The first $I$ generalized Gaussian waves follow the characteristics incoming the shock since $\alpha_l^+<0$ for $l\in\lc1,\ppp,I\rc$ and will eventually reach the shock (located at $j=0$).
		\item The last $d-I$ generalized Gaussian waves follow the outgoing characteristics with respect to the shock since $\alpha_l^+>0$ for $l\in\lc I+1,\ppp,d\rc$ and travel towards $+\infty$ without interaction with the shock.
	\end{itemize}
	When the generalized Gaussian waves following incoming characteristics reach the shock location, we observe that they lead to different behaviors for three new types of waves:
	\begin{itemize}
		\item There are reflected generalized Gaussian waves along the outgoing characteristics of the state $u^+$ (i.e. the purple waves in Figure \ref{FigGreen}). It corresponds to the second term using the function $R_{l^\prime,l}^+$ in \eqref{decompoGreen}. This explains the restriction of $l\in\lc I+1,\ppp,d\rc$ which implies that $\alpha_l^+>0$.
		\item There are transmitted generalized Gaussian waves along the outgoing characteristics of the state $u^-$ (i.e. the green waves in Figure \ref{FigGreen}). It corresponds to the third term using the function $T_{l^\prime,l}^+$ in \eqref{decompoGreen}. This explains the restriction of $l\in\lc 1,\ppp,I-1\rc$ which implies that $\alpha_l^-<0$.
		\item Because of the properties of the function $E_{2\mu}$ defined by \eqref{def:H2mu_et_E2mu}, we have that the vectors $E^+_{l^\prime}(n,j_0)$ are close to $0$ for small times $n$ and converge towards ${\lg_{l^\prime}^+}^T$ as $n$ tends towards $+\infty$.  Thus, the last term in the decomposition \eqref{decompoGreen} could be described as the progressive construction of the component of the Green's function along the vector subspace $\ker(Id_{\ell^2}-\Lcc)$ (i.e. the red wave in Figure \ref{FigGreen}). Each wave with speed $\alpha_{l^\prime}^+$ activates part of this profile as they reach the shock. Since $V$ is an eigenvector of $\Lcc$ for the eigenvalue $1$, this part of the Green's function does not decay as $n$ tends towards $+\infty$.
	\end{itemize}

		Before stating Theorem \ref{thStab}, we want to address a slight observation that could be made on Theorem \ref{th:Green}. One might expect that it could be possible to obtain the decomposition \eqref{decompoDerGreen} of the discrete derivative of the temporal Green's function $\Gcc(n,j_0,j)- \Gcc(n,j_0-1,j)$ using the difference of the decomposition \eqref{decompoGreen} of the temporal Green's function for $\Gcc(n,j_0,j)$ and $\Gcc(n,j_0-1,j)$ since \eqref{decompoDerGreen} and \eqref{decompoGreen} look fairly similar. However, we claim that this is non trivial since the reflected and transmitted waves $R_{l^\prime,l}^\pm(n,j_0,j)$ and $T_{l^\prime,l}^\pm(n,j_0,j)$, which are defined in \eqref{def:Termes_decompo_Green}, use the function $H_{2\mu}$ with varying viscosity parameters in their definitions. The known estimates \eqref{inH} of the functions $H_{2\mu}$ would not allow to conclude the proof \eqref{decompoDerGreen} this way. To prove the decomposition \eqref{decompoDerGreen} of the discrete derivative of the temporal Green's function presented in the article, we will rather slightly adapt the proof of the decomposition \eqref{decompoGreen} for the temporal Green's function (see Section \ref{sec:GT}).
		
		As explained at the beginning of this section, the main use of Theorem \ref{th:Green} is to prove decay estimates on the families of operators $\left(\Lcc^n\right)_{n\in\N}$ and $\left(\Lcc^n(Id-\Tc)\right)_{n\in\N}$ on well chosen Banach spaces. Of course, an important step remains to choose the correct Banach spaces on which those operators must act to conclude a nonlinear orbital stability result. This has been done in the scalar case in \cite{Coeuret2024} by considering polynomially weighted $\ell^r$-spaces. Investigating the case of systems of conservation laws remains to be done. However, to display an example of the usefulness and versatility of Theorem \ref{th:Green}, we prove the following decay estimates of the semi-group $(\Lcc^n)_{n\in\N}$ when it acts on $\ell^r(\Z,\C^d)$.

	\begin{theorem}\label{thStab}
		Let us assume that Hypotheses \ref{H:Lax}-\ref{H:Mpm1} are verified. For $r_1,r_2\in[1,+\infty]$ such that $r_1\leq r_2$, there exists positive constant $C$ such that
		$$\forall h\in\ell^{r_1}(\Z,\C^d),\forall n\in\N,\quad \min_{V\in\ker(Id_{\ell^2}-\Lcc)} \left\|\Lcc^n h - V\right\|_{\ell^{r_2}}\leq \frac{C}{n^{\frac{1}{2\mu}\left(\frac{1}{r_1}-\frac{1}{r_2}\right)}}\left\|h\right\|_{\ell^{r_1}}.$$
	\end{theorem}
	
	Let us point out that such decay estimates are similar to the one associated with the heat equation for instance. This is a consequence of the diffusive nature of the numerical scheme introduced in Hypothesis \ref{H:F}.
	
	\subsection{Plan of the paper}
	
	Firstly, Section \ref{sec:Th1<-Th2} will be dedicated to the proof of Theorem \ref{thStab} using Theorem \ref{th:Green}. The main part of the article however (from Sections \ref{sec:GSloin} to \ref{sec:GT}) will concern the proof of Theorem \ref{th:Green}:
	\begin{itemize}
		\item In Section \ref{sec:GSloin}, we will prove Proposition \ref{prop:SpectreEssentielLcc} which describes the spectrum of the operator $\Lcc$ in the set $\Oc$ and will allow to define the spatial Green's function defined by \eqref{defGreenSpatial} on $\overline{\U}\backslash\lc1\rc$. We will then prove Proposition \ref{GreenSpatialFar} which implies exponential bounds on the spatial Green's function in the neighborhood of any point of $\overline{\U}\backslash\lc1\rc$.
		
		\item In Section \ref{sec:GSnear1}, we prove that Proposition \ref{prop:GS_near_1} which claims that the spatial Green's function can be meromorphically extended in a neighborhood of $1$ through the essential spectrum of the operator $\Lcc$. We will show that it has a pole of order $1$ at $z=1$ and find precise expressions \eqref{ExpGs: l geq 0 : j geq l+1}-\eqref{ExpGs: l geq 0 : j < 0} on it that will be essential in Section \ref{sec:GT}.
		
		\item In Section \ref{sec:GT}, we express the temporal Green's function defined by \eqref{defGreenTempo} with the spatial Green's function \eqref{defGreenSpatial} using the inverse Laplace Transform. Using the different results proved on the spatial Green's function in Sections \ref{sec:GSloin} and \ref{sec:GSnear1}, we will conclude the proof of Theorem \ref{th:Green}.
	\end{itemize}
	Section \ref{sec:Num} will be dedicated to a concrete example. We will first show that some of the previously introduced hypotheses are verified for the modified Lax-Friedrichs scheme and then present a numerical example of the result of Theorem \ref{th:Green} when the conservation law we consider is the Burgers equation. Section \ref{sec:Appendix} is the Appendix and contains the proof of some technical lemmas used throughout the paper. 
	
	The author feels like it is important to point out that the proofs of some lemmas  in the present paper are done in other papers. However, the author feels like they needed to be reproved either to correct mistakes or because the way they are presented here is fairly different from the statement in other papers. For instance, the geometric dichotomy (Lemma \ref{lem_geo_dich} below) uses the same proof with slight variations as in \cite[Section III.1.5]{Godillon_these} where it is first introduced.
		
	\section{Proof of linear stability (Theorem \ref{thStab})}\label{sec:Th1<-Th2}
	
	The goal of this section is to prove Theorem \ref{thStab} using the description of the temporal Green's function obtained in Theorem \ref{th:Green}. We recall that for $r_1,r_2\in[1,+\infty]$ such that $r_1\leq r_2$, we have that for all $n\in\N$ and $h\in \ell^{r_1}(\Z,\C^d)$
	\begin{equation}\label{lienLccGreen}
		\Lcc^nh= \left(\sum_{j_0\in\Z}\Gcc(n,j_0,j)h_{j_0}\right)_{j\in \Z}\in\ell^{r_2}(\Z,\C^d).
	\end{equation}
	
	\textbf{\underline{Step 1:}} Let us prove that for all integers $l\in\lc1,\ppp,d\rc$ and for all couples $r_1,r_2\in[1,+\infty]$ such that $r_1\leq r_2$, the operators 
	\begin{equation}\label{def:ope}
		\begin{array}{ccc}
			\ell^{r_1}(\Z,\C^d) & \rightarrow & \ell^{r_2}(\Z,\C^d)\\ 
			h & \mapsto & \left(\sum_{j_0\in \N} \ind_{j-j_0\in\lc-nq,\ppp,np\rc} S_l^+(n,j_0,j)h_{j_0}\right)_{j\in\Z}
		\end{array}
	\end{equation}
	for $n\in\N\backslash\lc0\rc$ are well-defined and that there exists a positive constants $C$ such that
	\begin{equation}\label{in:ope}
		\forall n\in\N\backslash\lc0\rc, \forall h\in \ell^{r_1}(\Z,\C^d),\quad  \left\| \left(\sum_{j_0\in \N} \ind_{j-j_0\in\lc-nq,\ppp,np\rc}S_l^+(n,j_0,j)h_{j_0}\right)_{j\in\Z}\right\|_{\ell^{r_2}}\leq \frac{C}{n^{\frac{1}{2\mu}\left(\frac{1}{r_1}-\frac{1}{r_2}\right)}}\left\|h\right\|_{l^{r_1}}.
	\end{equation}
	We fix an integer $l\in\lc1,\ppp,d\rc$. Using \eqref{inH}, we have that there exist two positive constants $C,c$ such that
	\begin{multline}\label{th2:1}
		\forall n\in\N\backslash\lc0\rc,\forall h\in\ell^{\infty}(\Z,\C^d),\forall j\in\Z,\\ \left|\sum_{j_0\in\N} \ind_{j-j_0\in\lc-nq,\ppp,np\rc}S_l^+(n,j_0,j)h_{j_0}\right|\leq \sum_{j_0\in\N} \frac{C}{n^\frac{1}{2\mu}} \exp\left(-c\left(\frac{|n\alpha_l^++j_0-j|}{n^\frac{1}{2\mu}}\right)^\frac{2\mu}{2\mu-1}\right)\left|h_{j_0}\right|.
	\end{multline}
	We also observe that there exists a constant $\tilde{C}>0$ such that
	\begin{equation}\label{th2:2}
		\forall n\in\N\backslash\lc0\rc,\quad \sum_{j\in\Z} \frac{C}{n^\frac{1}{2\mu}}\exp\left(-c\left(\frac{|n\alpha_l^+-j|}{n^\frac{1}{2\mu}}\right)^\frac{2\mu}{2\mu-1}\right)\leq\tilde{C}.
	\end{equation}
	Let us now prove that, for different choices of couples $r_1,r_2$, the operator defined by \eqref{def:ope} is well-defined and that \eqref{in:ope} is verified.
	\begin{itemize}
		\item Using \eqref{th2:1} and \eqref{th2:2}, we have that:
		$$\forall n\in\N\backslash\lc0\rc,\forall h\in\ell^{\infty}(\Z,\C^d), \forall j\in\Z,\quad \left|\sum_{j_0\in\N} \ind_{j-j_0\in\lc-nq,\ppp,np\rc}S_l^+(n,j_0,j)h_{j_0}\right|\leq \tilde{C}\left\|h\right\|_{\ell^\infty}.$$
		Thus, for $r_1,r_2=+\infty$, the operator defined by \eqref{def:ope} is well-defined and \eqref{in:ope} is verified.
		\item Using \eqref{th2:1}, we have that:
		$$\forall n\in\N\backslash\lc0\rc,\forall h\in\ell^{1}(\Z,\C^d), \forall j\in\Z,\quad \left|\sum_{j_0\in\N} \ind_{j-j_0\in\lc-nq,\ppp,np\rc}S_l^+(n,j_0,j)h_{j_0}\right|\leq \frac{C}{n^\frac{1}{2\mu}}\left\|h\right\|_{\ell^1}.$$
		Thus, for $r_1=1$ and $r_2=+\infty$, the operator defined by \eqref{def:ope} is well-defined and \eqref{in:ope} is verified.
		\item Using \eqref{th2:1}, Fubini-Tonelli's theorem and \eqref{th2:2}, we have that for $n\in\N\backslash\lc0\rc$ and $h\in\ell^{1}(\Z,\C^d)$:
		\begin{multline*}
			\sum_{j\in\Z}\left|\sum_{j_0\in\N} \ind_{j-j_0\in\lc-nq,\ppp,np\rc}S_l^+(n,j_0,j)h_{j_0}\right| \\ \leq \sum_{j_0\in\N} \left(\sum_{j\in\Z}\frac{C}{n^\frac{1}{2\mu}} \exp\left(-c\left(\frac{|n\alpha_l^++j_0-j|}{n^\frac{1}{2\mu}}\right)^\frac{2\mu}{2\mu-1}\right)\right)\left|h_{j_0}\right|\leq \tilde{C}\left\|h\right\|_{\ell^1}.
		\end{multline*}
		Thus, for $r_1=1$ and $r_2=1$, the operator defined by \eqref{def:ope} is well-defined and \eqref{in:ope} is verified.
		\item Up until now, we have proved that for all couple $(r_1,r_2)\in\lc(1,1),(1,+\infty),(+\infty,+\infty)\rc$ the operator defined by \eqref{def:ope} is well-defined and that \eqref{in:ope} is verified. Using Riesz-Thorin Theorem, we can then conclude that, for all couple $(r_1,r_2)\in[1,+\infty]^2$ such that $r_1\leq r_2$, the operator defined by \eqref{def:ope} is well-defined and that \eqref{in:ope} is verified.
	\end{itemize}

	\textbf{\underline{Step 2:}} We claim that, using an identical proof as in Step 1, for all couple $r_1,r_2\in[1,+\infty]$ such that $r_1\leq r_2$, the operators from $\ell^{r_1}(\Z,\C^d)$ to $\ell^{r_2}(\Z,\C^d)$ defined by
	\begin{align}
		\begin{split}
			h\mapsto \left(\sum_{j_0\in \N} \ind_{j-j_0\in\lc-nq,\ppp,np\rc}R_{l^\prime,l}^+(n,j_0,j)h_{j_0}\right)_{j\in\Z} \quad &\quad h \mapsto \left(\sum_{j_0<0} \ind_{j-j_0\in\lc-nq,\ppp,np\rc}R_{l^\prime,l}^-(n,j_0,j)h_{j_0}\right)_{j\in\Z}\\
			h \mapsto \left(\sum_{j_0\in \N} \ind_{j-j_0\in\lc-nq,\ppp,np\rc}T_{l^\prime,l}^+(n,j_0,j)h_{j_0}\right)_{j\in\Z} \quad &\quad h \mapsto \left(\sum_{j_0<0} \ind_{j-j_0\in\lc-nq,\ppp,np\rc}T_{l^\prime,l}^-(n,j_0,j)h_{j_0}\right)_{j\in\Z}\\
			h \mapsto \left(\sum_{j_0<0} \ind_{j-j_0\in\lc-nq,\ppp,np\rc}S_{l}^-(n,j_0,j)h_{j_0}\right)_{j\in\Z}\quad &\quad h\in  \mapsto \left(\sum_{j_0\in\Z} \ind_{j-j_0\in\lc-nq,\ppp,np\rc}\Rc(n,j_0,j)h_{j_0}\right)_{j\in\Z}
		\end{split}\label{th2:3}
	\end{align}
	for $n\in\N\backslash\lc0\rc$ are well-defined and satisfy similar bounds as \eqref{in:ope}. Let us now introduce the linear operators $\Psi_n$ from $\ell^{r_1}(\Z,\C^d)$ to $\ell^{r_2}(\Z,\C^d)$ defined by:
	
		\begin{multline}\label{def:Psi}
			\forall n\in\N\backslash\lc0\rc,\forall h\in\ell^{r_1}(\Z,\C^d),\forall j\in\Z,\\ (\Psi_n h)_j:=\sum_{j_0\in\N}\ind_{j-j_0\in\lc-nq,\ppp,np\rc}\left(\sum_{l^\prime=1}^I {C^E_{l^\prime}}^+E^+_{l^\prime}(n,j_0)h_{j_0}\right)V(j) \\  
			+\sum_{j_0<0}\ind_{j-j_0\in\lc-nq,\ppp,np\rc}\left(\sum_{l^\prime=I}^d {C^E_{l^\prime}}^-E^-_{l^\prime}(n,j_0)h_{j_0} \right) V(j).
		\end{multline}
	
	We recall that the Green's function $\Gcc(n,j_0,j)$ is equal to $0$ when $j-j_0\notin\lc-nq,\ppp,np\rc$. Thus, using the equality \eqref{lienLccGreen} and the decomposition of the Green's function $\Gcc(n,j_0,j)$ when $j-j_0\in\lc-nq,\ppp,np\rc$ given by Theorem \ref{th:Green}, the previous results \eqref{in:ope} we proved on the operators \eqref{def:ope} and \eqref{th2:3} allow us to prove that for all $r_1,r_2\in[1,+\infty]$ such that $r_1\leq r_2$, we have that there exists a positive constant $C$ such that:
	\begin{equation}\label{th2:4}
		\forall h\in \ell^{r_1}(\Z,\C^d), \forall n\in\N\backslash\lc0\rc, \quad \left\|\Lcc^n h -\Psi_nh \right\|_{\ell^{r_2}}\leq \frac{C}{n^{\frac{1}{2\mu}\left(\frac{1}{r_1}-\frac{1}{r_2}\right)}}\left\|h\right\|_{\ell^{r_1}}
	\end{equation}
	
	\textbf{\underline{Step 3:}} Let us prove that for all integers $l^\prime\in\lc1,\ppp,I\rc$, the operators 
	\begin{equation}\label{def:ope2}
		\begin{array}{ccc}
			\ell^{\infty}(\Z,\C^d) & \rightarrow & \ell^{1}(\Z,\C^d)\\ 
			h & \mapsto & \left(\sum_{j_0\in\N} \ind_{j-j_0\notin\lc-nq,\ppp,np\rc} C_{l^\prime}^{E,+} E_{l^\prime}^+(n,j_0)h_{j_0}V(j)\right)_{j\in\Z}
		\end{array}
	\end{equation}
	for $n\in\N\backslash\lc0\rc$ are well-defined and that there exist two positive constants $C,c$ such that
	\begin{equation}\label{in:ope2}
		\forall n\in\N\backslash\lc0\rc, \forall h\in \ell^{\infty}(\Z,\C^d),\quad  \left\| \left(\sum_{j_0\in\N} \ind_{j-j_0\notin\lc-nq,\ppp,np\rc} C_{l^\prime}^{E,+} E_{l^\prime}^+(n,j_0)h_{j_0}V(j)\right)_{j\in\Z}\right\|_{\ell^{1}}\leq Ce^{-cn}\left\|h\right\|_{\ell^{\infty}}.
	\end{equation}
	We recall that \eqref{decExpoV} implies that there exist two positive constants $C_V,c_V$ such that:
	\begin{equation}\label{th2:inV}
		\forall j\in\Z,\quad |V(j)|\leq C_Ve^{-c_V|j|}.
	\end{equation}
	Furthermore, for $l^\prime\in\lc1,\ppp,I\rc$, the CFL condition \eqref{cond:CFL} implies that
	\begin{equation}\label{consCFL}
		0<-\alpha_{l^\prime}^+<\frac{q-\alpha_{l^\prime}^+}{2}<q.
	\end{equation}
	We fix an integer $l^\prime\in\lc1,\ppp,I\rc$. For $h\in\ell^{\infty}(\Z,\C^d)$, $n\in\N\backslash\lc0\rc$ and $j\in\Z$, we have that:
	\begin{multline}\label{th2:6}
		\sum_{j_0\in\N} \left|\ind_{j-j_0\notin\lc-nq,\ppp,np\rc} C_{l^\prime}^{E,+} E_{l^\prime}^+(n,j_0)h_{j_0}V(j)\right|\\ \leq \left|C_{l^\prime}^{E,+}\right|C_V \left[\sum_{j_0\in\left[0,n\frac{q-\alpha_{l^\prime}^+}{2}\right]\cap\N}\ind_{j-j_0\notin\lc-nq,\ppp,np\rc}e^{-c_V|j|} |E_{l^\prime}^+(n,j_0)||h_{j_0}|\right.\\\left.+\sum_{j_0\in\left[n\frac{q-\alpha_{l^\prime}^+}{2},+\infty\right[\cap\N}\ind_{j-j_0\notin\lc-nq,\ppp,np\rc}e^{-c_V|j|} |E_{l^\prime}^+(n,j_0)||h_{j_0}|\right] .
	\end{multline}
	
	\begin{itemize}
		\item Using \eqref{consCFL}, we observe that there exists a positive constant $c>0$ such that for all $j_0\in \left[0,n\frac{q-\alpha_{l^\prime}^+}{2}\right]\cap\N$ and $j\in\Z$ which verify $j-j_0\notin\lc-nq,\ppp,np\rc$, then:
		$$|j|\geq cn.$$
		Therefore, since the function $E_{2\mu}$ is bounded, we have that there exists a positive constant $C$ such that for all $h\in\ell^{\infty}(\Z,\C^d)$ and $n\in\N\backslash\lc0\rc$:
		$$\sum_{j\in\Z} \sum_{j_0\in\left[0,\frac{q-\alpha_{l^\prime}^+}{2}n\right]\cap\N}\ind_{j-j_0\notin\lc-nq,\ppp,np\rc}e^{-c_V|j|} |E_{l^\prime}^+(n,j_0)||h_{j_0}|\leq C\left(\sum_{j\in\Z}e^{-\frac{c_V}{2}|j|}\right)ne^{-\frac{c_V}{2}cn}\left\|h\right\|_{\ell^\infty}. $$
		Therefore, there exist two new positive constants $C,c$ such that for all $h\in\ell^{\infty}(\Z,\C^d)$ and $n\in\N\backslash\lc0\rc$:
		\begin{equation}\label{th2:7}
			\sum_{j\in\Z} \sum_{j_0\in\left[0,\frac{q-\alpha_{l^\prime}^+}{2}n\right]\cap\N}\ind_{j-j_0\notin\lc-nq,\ppp,np\rc}e^{-c_V|j|} |E_{l^\prime}^+(n,j_0)||h_{j_0}|\leq Ce^{-cn}\left\|h\right\|_{\ell^\infty}. 
		\end{equation}
		\item Using \eqref{inE+}, there exist two positive constants $C,c$ such that for all $h\in\ell^{\infty}(\Z,\C^d)$ and $n\in\N\backslash\lc0\rc$:
		\begin{multline*}
			\sum_{j\in\Z} \sum_{j_0\in\left[\frac{q-\alpha_{l^\prime}^+}{2}n,+\infty\right[\cap\N}\ind_{j-j_0\notin\lc-nq,\ppp,np\rc}e^{-c_V|j|} |E_{l^\prime}^+(n,j_0)||h_{j_0}| \\ \leq C\left\|h\right\|_{\ell^\infty} \left(\sum_{j\in\Z}e^{-c_V|j|}\right)\sum_{j_0\in\left[\frac{q-\alpha_{l^\prime}^+}{2}n,+\infty\right[\cap\N}\exp\left(-c\left(\frac{|n\alpha_{l^\prime}^++j_0|}{n^\frac{1}{2\mu}}\right)^\frac{2\mu}{2\mu-1}\right). 
		\end{multline*}
		
		Therefore, using \eqref{consCFL},  there exist two new positive constants $C,c$ such that for all $h\in\ell^{\infty}(\Z,\C^d)$ and $n\in\N\backslash\lc0\rc$:
		\begin{equation}\label{th2:8}
			\sum_{j\in\Z} \sum_{j_0\in\left[\frac{q-\alpha_{l^\prime}^+}{2}n,+\infty\right[\cap\N}\ind_{j-j_0\notin\lc-nq,\ppp,np\rc}e^{-c_V|j|} |E_{l^\prime}^+(n,j_0)||h_{j_0}|\leq Ce^{-cn}\left\|h\right\|_{\ell^\infty}. 
		\end{equation}
	\end{itemize}
	
	Thus, combining \eqref{th2:6}, \eqref{th2:7} and \eqref{th2:8}, we can conclude the proof of \eqref{in:ope2}.

			\textbf{\underline{Step 4:}} Using a similar proof as in Step 3 for the operator from $\ell^\infty(\Z,\C^d)$ to $\ell^1(\Z,\C^d)$ defined by:
		\begin{align*}
			h  \mapsto  \left(\sum_{j_0<0} \ind_{j-j_0\notin\lc-nq,\ppp,np\rc} C_{l^\prime}^{E,-} E_{l^\prime}^-(n,j_0)h_{j_0}V(j)\right)_{j\in\Z},
		\end{align*}
		we can obtain similar bounds as \eqref{in:ope2}. We can then prove that there exist two positive constants $C,c$ such that: 
		\begin{multline}\label{th2:5}
			\forall h\in \ell^{r_1}(\Z,\C^d), \forall n\in\N\backslash\lc0\rc, \\ \left\|\Psi_n h - \left( \sum_{j_0\in\N}\sum_{l^\prime=1}^I {C^E_{l^\prime}}^+E^+_{l^\prime}(n,j_0)h_{j_0} +\sum_{j_0<0}\sum_{l^\prime=I}^d {C^E_{l^\prime}}^-E^-_{l^\prime}(n,j_0) h_{j_0} \right) V\right\|_{\ell^{r_2}} \\ \leq \frac{C}{n^{\frac{1}{2\mu}\left(\frac{1}{r_1}-\frac{1}{r_2}\right)}}\left\|h\right\|_{\ell^{r_1}}
		\end{multline} 
		where the operator $\Psi_n$ is defined by \eqref{def:Psi}. Using \eqref{th2:4} and \eqref{th2:5}, we can conclude that there exists a positive constant $C>0$ such that:
		\begin{multline*}
			\forall h\in \ell^{r_1}(\Z,\C^d), \forall n\in\N\backslash\lc0\rc, \\ \left\|\Lcc^n h - \left( \sum_{j_0\in\N}\sum_{l^\prime=1}^I {C^E_{l^\prime}}^+E^+_{l^\prime}(n,j_0)h_{j_0} +\sum_{j_0<0}\sum_{l^\prime=I}^d {C^E_{l^\prime}}^-E^-_{l^\prime}(n,j_0) h_{j_0} \right) V\right\|_{\ell^{r_2}}\\\leq \frac{C}{n^{\frac{1}{2\mu}\left(\frac{1}{r_1}-\frac{1}{r_2}\right)}}\left\|h\right\|_{\ell^{r_1}}.
		\end{multline*}
		Since the sequence $V\in \ell^1(\Z,\C^d)$ defined in Theorem \ref{th:Green} verifies that:
		$$\ker(Id_{\ell^2}-\Lcc)=\mathrm{Span} V,$$
		this allows us to conclude the proof of Theorem \ref{thStab}.

	 \section{Local exponential bounds on the spatial Green's function for \texorpdfstring{$z$}{z} far from \texorpdfstring{$1$}{1}}\label{sec:GSloin}
	 
	 In this section, the goal is twofold:
	 
	 \begin{itemize}
	 	\item In order to determine where the spatial Green's function is defined, we want to study the spectrum of the operator $\Lcc$ in the set $\Oc$ (i.e. outside of the curves representing the spectrum of the limit operators $\Lcc^\pm$). More precisely, we will prove Lemma \ref{lem_spec_ess} below that characterizes the eigenvalues of $\Lcc$ in the set $\Oc$ and states that there is no essential spectrum of the operator $\Lcc$ which lies in the set $\Oc$. This result was already proved in \cite[Theorem 4.1]{Serre}. As a direct consequence, we will have proved that the elements of the set $\Oc$ are either in the resolvent set of the operator $\Lcc$ or are eigenvalues of $\Lcc$. Using Hypothesis \ref{H:spec}, we can thus deduce that the set $\overline{\U}\backslash\lc1\rc$ is included in the resolvent set of $\Lcc$ and that the spatial Green's function can be defined in a neighborhood of any point of $\overline{\U}\backslash\lc1\rc$.
	 	\item We will prove Proposition \ref{GreenSpatialFar} below that introduces locally uniform exponential bounds on the spatial Green's function $G(z,j_0,\cdot)$ when $z$ belongs to $\overline{\U}\backslash\lc1\rc$ and $j_0\in\Z$. We will see later on in Section \ref{sec:GSnear1} that the study of the spatial Green's function for $z$ near $1$ will require some special care and that it is a more refined analysis of the case where $z$ is in $\overline{\U}\backslash\lc1\rc$. It might be important to keep in mind that a lot of the tools we will introduce in Section \ref{sec:GSloin} will also be useful in Section \ref{sec:GSnear1} to deal with the case where $z$ is close to $1$.
	 \end{itemize}
	 
	 The main ideas of this section will be to characterize the solutions of the eigenvalue problem associated with the operator $\Lcc$ by using solutions of a discrete dynamical system of finite dimension.
	 
	 We will then define a central tool for our analysis : the \textit{geometric dichotomy} introduced by Lafitte-Godillon in her thesis \cite{Godillon_these} and based on the exponential dichotomy coined by Coppel in \cite{Coppel}. We will take some time to rewrite the proofs of some lemmas even though most of the ideas can already be found in the previously cited texts. This will motivate the introduction of several notations and will make several crucial bounds precise.
	 
	 \subsection{Rewriting the eigenvalue problem as a dynamical system}\label{subsec:GSloin:ExpEigen}
	 
	As we explained, one of our objectives is to study the spectrum of the operator $\Lcc$. In this section, we express the eigenvalue problem $(zId_{\ell^2}-\Lcc)u=0$ as a discrete dynamical system. We will define a few important mathematical objects that will appear often throughout this article.

		For $z \in \C$, $j\in\Z$ and $k\in\lc-p,\ppp,q\rc$,  we define the matrices 
		\begin{equation}
				\A_{j,k}(z)=z\delta_{k,0}Id-A_{j,k},\quad \text{ and } \quad  \A_k^\pm(z)=z\delta_{k,0}Id-A_k^\pm.\label{def:AAjk}
		\end{equation}
		We recall that Hypothesis \ref{H:inv} implies that for all $z \in \C$ and $j\in \Z$, the matrices $\A_{j,-p}(z)$, $\A_{j,q}(z)$, $\A_{-p}^\pm(z)$ and $\A_q^\pm(z)$ are invertible and we can thus introduce the matrices :
		\begin{subequations}\label{def:MjM}
			\begin{equation}\label{def:Mj}
				\forall j\in\Z, \forall z \in \C, \quad M_j(z):=\begin{pmatrix}-\A_{j,q}(z)^{-1}\A_{j,q-1}(z) & \ppp & \ppp & -\A_{j,q}(z)^{-1}\A_{j,-p}(z) \\
				Id & 0 &\ppp   &  0\\
				0 & \ddots  &\ddots & \vdots \\
				0 &  0 & Id & 0\end{pmatrix}\in\Mc_{d(p+q)}(\C)
			\end{equation}
			and :
			\begin{equation}\label{def:M}
				\forall z \in\C, \quad M^\pm(z):=\begin{pmatrix}-\A_q^\pm(z)^{-1}\A_{q-1}^\pm(z) & \ppp & \ppp & -\A_q^\pm(z)^{-1}\A_{-p}^\pm(z) \\
				Id & 0 &\ppp   &  0\\
				0 & \ddots  &\ddots & \vdots \\
				0 &  0 & Id & 0\end{pmatrix}\Mc_{d(p+q)}(\C).
			\end{equation}
		\end{subequations}
		which are well-defined and are furthermore invertible. We observe that for all $z\in \C$ the family of matrices $(M_j(z))_{j\in \Z}$ converges towards $M^\pm(z)$ as $j$ tends towards $\pm\infty$. If we define for every $z \in \C$ and $j\in \Z$ the matrices
		$$\Ec_j^\pm(z):= M_j(z)-M^\pm(z),$$
		then Hypothesis \ref{H:CVexpo} implies the following lemma.

	\begin{lemma}
		There exists a constant $\alpha>0$ such that for every bounded set $U$ of $\C$, there exists a constant $C>0$ such that
		\begin{equation}\label{CV_expo_mat}
			\forall z \in U, \forall j\in\N, \quad \begin{array}{c}  |\Ec_j^+(z)|\leq Ce^{-\alpha j}, \\
				|\Ec_{-j}^-(z)|\leq Ce^{-\alpha j}.
			\end{array}
		\end{equation}
	\end{lemma}
	
	\begin{proof}
		We have that
		$$\Ec^\pm_j(z)= \begin{pmatrix}
			\varepsilon_{j,q-1}^\pm(z) & \ppp & \varepsilon_{j,-p}^\pm(z) \\
			0 & \ppp & 0 \\
			\vdots & \vdots & \vdots \\
			0 & \ppp & 0
		\end{pmatrix}$$
		where 
		$$\varepsilon_{j,k}^\pm(z):= \lc\begin{array}{cc}
			(A_q^\pm)^{-1}A_k^\pm -(A_{j,q})^{-1}A_{j,k} & \text{ if }k\in\lc-p,\ppp,q-1\rc\backslash \lc0\rc, \\ 
			(A_q^\pm)^{-1}A_0^\pm -(A_{j,q})^{-1}A_{j,0} - z ((A_q^\pm)^{-1}-(A_{j,q})^{-1}) & \text{ if }k=0.
		\end{array}\right.$$
		Using Hypothesis \ref{H:CVexpo}, we conclude the proof of \eqref{CV_expo_mat} and observe that the constant $\alpha$ can be taken uniformly on $\C$ but the constant $C$ must depend on $z$.
	\end{proof}
	
	Let us consider $u\in \ell^2(\Z,\C^d)$ such that 
	$$(zId_{\ell^2}-\Lcc) u=0.$$
	If we define for all $j\in \Z$ the vector
	$$W_j=\begin{pmatrix}u_{j+q-1}\\ \vdots \\ u_{j-p}\end{pmatrix},$$
	we observe that $(W_j)_{j\in\Z} \in\ell^2\left(\Z,\C^{d(p+q)}\right)$ which implies that the vector $W_j$ converges towards $0$ as $j$ tends towards $\pm\infty$. We also observe that 
	\begin{equation}\label{syst_dyn}
		\forall j \in \Z, \quad W_{j+1} = M_j(z)W_j.
	\end{equation}
	
	To study the solutions of the dynamical system \eqref{syst_dyn}, we define the family $(X_j(z))_{j\in\Z}\in\Mc_{d(p+q)}(\C)$ of fundamental matrices defined by
	\begin{equation}\label{mat_fond}
		\begin{array}{c} \forall j \in \Z, \quad X_{j+1}(z)= M_j(z)X_j(z),\\ X_0(z) =Id.\end{array} 
	\end{equation}
	We observe that a solution $(W_j)_{j\in\Z}$ of the dynamical system \eqref{syst_dyn} thus verifies:
	$$\forall j\in \Z,\quad W_j = X_j(z)W_0.$$
	To find the eigenvalues of the operator $\Lcc$, we thus need to search for the solutions $(W_j)_{j\in\Z}$ of the dynamical system \eqref{syst_dyn} which belong to $\ell^2$ and therefore converge towards $0$ when $j$ tends to $\pm\infty$. We thus introduce the following sets for any $z\in \C$:
	\begin{subequations}\label{def:EpmE0pm}
		\begin{align}
			E^\pm(z)&:= \lc(W_j)_{j\in \Z}\in \left(\C^{(p+q)d}\right)^\Z \text{ solution of }\eqref{syst_dyn} \text{ such that } W_j\underset{j\rightarrow \pm\infty}\rightarrow0\rc \label{def:Epm},\\
			E^\pm_0(z)&:= \lc W_0\in \C^{(p+q)d}, \quad (X_j(z)W_0)_{j\in \Z}\in E^\pm(z)\rc. \label{def:E0pm}
		\end{align}
	\end{subequations}
	The sets correspond to the solutions of the dynamical system \eqref{syst_dyn} which converge towards $0$ as $j$ tends towards $\pm\infty$ and to their traces at $j=0$.

	\subsection{Spectral splitting: study of the spectrum of \texorpdfstring{$M^\pm(z)$}{M+-(z)}}\label{subsec:GSloin:SpecSpl}
	
	Since the matrices $M_j(z)$ converge towards $M^\pm(z)$ as $j$ converges towards $\pm\infty$, the dynamical system \eqref{syst_dyn} can be considered to be perturbations respectively for $j\in \N$ and $j\in -\N$ of the dynamical systems
	\begin{subequations}\label{systDyn+-}
		\begin{align}
			\forall j\in\N, \quad &W_{j+1} = M^+(z) W_j,  \label{systDyn+}\\
			\forall j\in-\N, \quad &W_{j+1} = M^-(z) W_j.  \label{systDyn-}
		\end{align}
	\end{subequations}
	To study the solutions of \eqref{syst_dyn} which converge towards $0$ as $j$ tends to $\pm\infty$, we will study solutions converging towards $0$ of the dynamical systems \eqref{systDyn+} and \eqref{systDyn-}. This relies on studying the spectrum of the matrices $M^\pm(z)$. 
	
	Using the eigenvalues $\lambda_{l,k}^\pm$ of the matrix $A_k^\pm$ defined by \eqref{def:lambda_k}, we introduce the scalar quantities:
	\begin{equation}\label{def:Lambda}
		\forall z \in \C, \forall l\in\lc 1 ,\ppp, d\rc , \forall k \in \lc-p,\ppp,q\rc, \quad \Lambda_{l,k}^\pm(z):= z \delta_{k,0}-\lambda_{l,k}^\pm,
	\end{equation}
	and the matrices:
	\begin{equation}\label{def:Mlpm}
		\forall z \in \C, \forall l\in\lc 1 ,\ppp, d\rc , \quad M^\pm_l(z):=\begin{pmatrix}-\Lambda_{l,q}^\pm(z)^{-1}\Lambda_{l,q-1}^\pm(z) & \ppp & \ppp & -\Lambda_{l,q}^\pm(z)^{-1}\Lambda_{l,-p}^\pm(z) \\
			1 & 0 &\ppp   &  0\\
			0 & \ddots  &\ddots & \vdots \\
			0 &  0 & 1 & 0\end{pmatrix}.
	\end{equation}
	Hypothesis \ref{H:inv} implies that the matrices $A_{-p}^\pm$ and $A_q^\pm$ are invertible so $\Lambda_{l,-p}^\pm(z),\Lambda_{l,q}^\pm(z)\neq 0$. Thus, the matrices $M^\pm_l(z)$ are well-defined and invertible. We have the following result.
	
	\begin{lemma}\label{decomp_mat}
		There exist invertible matrices $Q^\pm\in \Mc_{(p+q)d}(\C)$ such that
		$$\forall z \in \C, \quad M^\pm(z) = \left(Q^\pm\right)^{-1}\begin{pmatrix}
			M_1^\pm(z) & & \\ & \ddots & \\ & & M_d^\pm(z)
		\end{pmatrix}Q^\pm.$$
	\end{lemma}
	
	\begin{proof}
		Recalling the definition \eqref{def:P} of the matrices $\Pg^\pm$, we observe that
		\begin{equation}\label{eq_mat}\begin{pmatrix}
				\Pg^\pm & & \\
				&\ddots & \\
				& & \Pg^\pm
			\end{pmatrix}^{-1} M^\pm(z)\begin{pmatrix}
				\Pg^\pm & & \\
				&\ddots & \\
				& & \Pg^\pm
			\end{pmatrix}= \begin{pmatrix}-\D_{q}^\pm(z)^{-1}\D_{q-1}^\pm(z) & \ppp & \ppp & -\D_{q}^\pm(z)^{-1}\D_{-p}^\pm(z) \\
				Id & 0 &\ppp   &  0\\
				0 & \ddots  &\ddots & \vdots \\
				0 &  0 & Id & 0\end{pmatrix}\end{equation}
		where 
		$$\forall k \in\lc-p,\ppp,q\rc,\quad \D^\pm_k(z)=\begin{pmatrix}\Lambda^\pm_{1,k}(z)& & \\ & \ddots & \\ & &\Lambda^\pm_{d,k}(z) \end{pmatrix}.$$
		Then, in the right hand term's matrix of \eqref{eq_mat}, by reassembling the first columns of each blocks, then the second columns, \ppp and then doing the same for the lines, we prove that the matrix $M^\pm(z)$ is similar to the block diagonal matrix
		$$\begin{pmatrix}
			M_1^\pm(z) & & \\
			& \ddots & \\
			& & M_d^\pm(z)
		\end{pmatrix}.$$
	\end{proof}
	
	The following lemma is due to Kreiss (see \cite{Kreiss}) and describes precisely the spectrum of the matrix $M_l^\pm(z)$ as $z$ belongs to $\Oc\cup\lc1\rc$ and $l\in\lc1,\ppp,d\rc$.
	\begin{lemma}[Spectral Splitting]\label{lem:SpecSpl}
		\begin{itemize}
			\item For $z\in \C$ and $l\in\lc1,\ppp,d\rc$, $\kappa\in\C$ is an eigenvalue of $M_l^\pm(z)$ if and only if $\kappa\neq 0$ and 
			$$\Fc_l^\pm(\kappa)=z.$$
			\item Let $z\in \Oc$ and $l\in\lc1,\ppp,d\rc$. Then the companion matrix $M_l^\pm(z)$ has
			\begin{itemize}
				\item no eigenvalue on $\S^1$,
				\item $p$ eigenvalues in $\D\backslash\lc0\rc$ (that we call stable eigenvalues),
				\item  $q$ eigenvalues in $\U$ (that we call unstable eigenvalues).
			\end{itemize}
			\item We also have that
			\begin{itemize}
				\item if $\alpha_l^\pm>0$, $M_l^\pm(1)$ has $1$ as a simple eigenvalue, $p-1$ eigenvalues in $\D\backslash\lc0\rc$ and $q$ eigenvalues in $\U$.
				\item if $\alpha_l^\pm<0$, $M_l^\pm(1)$ has $1$ as a simple eigenvalue, $p$ eigenvalues in $\D\backslash\lc0\rc$ and $q-1$ eigenvalues in $\U$.
			\end{itemize}
		\end{itemize}
	\end{lemma}
	
	Lemma \ref{lem:SpecSpl} is proved in \cite[Lemma 1]{C-F} (see also \cite{Kreiss}). For $z\in\Oc$, combining the consequences of Lemmas \ref{decomp_mat} and \ref{lem:SpecSpl}, the matrix $M^\pm(z)$ only has eigenvalues in $\D$ or $\U$. Also, if we define the space $E^s(M^\pm(z))$ (resp. $E^u(M^\pm(z))$) which is the strictly stable (resp. strictly unstable) subspace of $M^\pm(z)$ corresponding to the subspace spanned by the generalized eigenvectors of $M^\pm(z)$ associated with eigenvalues in $\D$ (resp. $\U$), then we have $\dim E^s(M^\pm(z))= dp$, $\dim E^u(M^\pm(z))= dq$ and
	\begin{equation}\label{decompo:C^d(p+q)}
		\C^{(p+q)d} = E^s(M^\pm(z)) \oplus E^u(M^\pm(z)).
	\end{equation}
	
		We consider $P_s^\pm(z)$ and $P_u^\pm(z)$ the associated projectors in $\C^{d(p+q)}$. Those projectors can be expressed as contour integrals (see \cite[Chapter 2 Section 1.4]{Kato}). For instance, we have:
	
	$$P_s^\pm(z) = \frac{1}{2i\pi}\int_{\gamma} (t Id - M^\pm(z))^{-1} dt$$
	where $\gamma$ is a simple closed positively oriented contour which surrounds the stable eigenvalues of $M^\pm(z)$ and not the unstable ones ($\S^1$ is a good candidate). Therefore, the projectors $P_s^\pm(z)$ and $P_u^\pm(z)$ depend holomorphically on $z\in \Oc$.

	\subsection{"Local" geometric dichotomy}\label{subsec:GeoDich}
	
	The conclusion of the study of the spectrum of the matrices $M^\pm(z)$ done in Section \ref{subsec:GSloin:SpecSpl} is that the vector space of solutions of \eqref{systDyn+} (resp. \eqref{systDyn-}) converging towards $0$ as $j$ tends towards $+\infty$ (resp. $-\infty$) has dimension $dp$ (resp. $dq$) and can be characterized by using the spectral projector $P^+_s(z)$ (resp. $P^-_u(z)$). We recall that \eqref{syst_dyn} is a perturbation of the dynamical systems \eqref{systDyn+} and \eqref{systDyn-}. Thus, we could expect for the vector spaces $E^+(z)$ and $E^+_0(z)$ (resp. $E^-(z)$ and $E^-_0(z)$) defined by \eqref{def:Epm} and \eqref{def:E0pm} to be of dimension $dp$ (resp. $dq$) and we would want some way to characterize their elements. 
	
	In the present section, our goal is to construct projectors which will play for the dynamical system \eqref{syst_dyn} a similar role as the spectral projectors $P^+_s(z)$ and $P^-_u(z)$ for the dynamical systems \eqref{systDyn+} and \eqref{systDyn-}. This is the aim of the following lemma.
	
	\begin{lemma}[Geometric dichotomy]\label{lem_geo_dich}
		For any bounded open set $U$ such that $\overline{U}\subset \Oc$, there exist two holomorphic functions $Q^\pm_U:U\rightarrow \Mc_{(p+q)d}(\C)$ such that
		\begin{itemize}
			\item For all $z \in U$, $Q^\pm_U(z)$ is a projector and we have
			$$\dim \Im Q^+_U(z) = \dim\ker Q^-_U(z) = dp \quad \text{ and } \quad \dim \Im Q^-_U(z) = \dim\ker Q^+_U(z) = dq.$$
			\item There exist two positive constants $C,c$ such that for all $z\in U$, there holds:
			\begin{subequations}
				\begin{align}
					\forall j\geq k\geq 0, &\quad \left|X_j(z)Q^+_U(z)X_k(z)^{-1}\right|\leq C e^{-c|j-k|},\label{in_geo_dich_+_1}\\
					\forall k\geq j\geq 0, &\quad \left|X_j(z)(Id-Q^+_U(z))X_k(z)^{-1}\right|\leq C e^{-c|j-k|},\label{in_geo_dich_+_2}\\
					\forall j\leq k \leq 0,& \quad \left|X_j(z)Q^-_U(z)X_k(z)^{-1}\right|\leq C e^{-c|j-k|},\label{in_geo_dich_-_1}\\
					\forall k\leq j \leq 0, &\quad \left|X_j(z)(Id-Q^-_U(z))X_k(z)^{-1}\right|\leq C e^{-c|j-k|}\label{in_geo_dich_-_2}.
				\end{align}
			\end{subequations}
		\end{itemize}
	\end{lemma}
	
	Lemma \ref{lem_geo_dich} has been developped in the thesis of Lafitte-Godillon \cite[Section III.1.5]{Godillon_these} and is inspired by the exponential dichotomy discussed by Coppel in \cite{Coppel}. As it is explained in \cite[Chapter 2]{Coppel}, to better understand the meaning of this lemma, it is interesting to see that the inequalities \eqref{in_geo_dich_+_1} and \eqref{in_geo_dich_+_2} imply that for all vector $\xi\in \C^{(p+q)d}$, there holds:
	\begin{align*}
		\forall j\geq 0, &\quad \left|X_j(z)Q^+_U(z)\xi\right|\leq C e^{-cj}\left|\xi\right|,\\
		\forall k\geq 0, &\quad \left|(Id-Q^+_U(z))\xi\right|\leq C e^{-ck}\left|X_k(z)\xi\right|.
	\end{align*}
	The first inequality implies that there exists a $dp$-dimensional subspace of solutions $(W_j)_{j\in \Z}$ of the dynamical system \eqref{syst_dyn} which converge exponentially fast toward $0$ at $+\infty$ (and which thus belong to $\ell^2(\N)$). The second inequality translates to the fact that there exists a supplementary of the previous subspace of solutions for which the solutions explode exponentially at $+\infty$. Thus, $Q_U^+(z)$ plays a similar role for the dynamical system \eqref{syst_dyn} as any projector for which the range is $E^+_s(z)$ (for instance the spectral projector $P_s^+(z)$) for the dynamical system \eqref{systDyn+}. The same kind of conclusion can be achieved with $Q^-_U$ for the behavior at $-\infty$. 
	
	Thus, the construction of those two projectors $Q^\pm_U$ is fundamental to study the solutions $(W_j)_{j\in\Z}$ of \eqref{syst_dyn} that converge toward $0$ as $j$ tends to $\pm\infty$, i.e. the elements of the set $E^\pm(z)$ defined by \eqref{def:Epm}. We will see in Lemma \ref{lem_id_E0} that the projectors $Q_U^\pm(z)$ allow to characterize the elements of the vector spaces $E^\pm_0(z)$.
	
	We give below the proof of Lemma \ref{lem_geo_dich} which is overall the same as in \cite[Section III.1.5]{Godillon_these}. The main modification is that this new iteration of the proof pinpoints more clearly the holomorphicity of the projectors $Q_U^\pm$ and most importantly why the projectors $Q_U^\pm$ we construct depend on the relatively compact subset $U$ of $\Oc$ which is considered. This last point is not really addressed in \cite{Godillon_these,Godillon} and we feel that it is quiet important. For instance, for two sets $U_1$ and $U_2$ that satisfy the conditions of Lemma \ref{lem_geo_dich}, the construction of the proof of Lemma \ref{lem_geo_dich} does not imply that the projectors $Q^\pm_{U_1}$ and $Q^\pm_{U_2}$ are equal on $U_1\cap U_2$, even if $U_1\subset U_2$. Therefore, we cannot immediately construct two functions $Q^\pm$ that are defined on $\Oc$ which would verify similar properties as $Q^\pm_U$. However, it turns out that the ranges $\Im Q^\pm_{U_1}(z)$ and $\Im Q^\pm_{U_2}(z)$ coincide for $z\in U_1\cap U_2$. We will prove this fact later on and use it to extend uniformly the geometric dichotomy on a large part of $\Oc$ (see Lemma \ref{lemGeoDichStronger} below). \newline
	
	\begin{proof}\textbf{of Lemma \ref{lem_geo_dich}}
			
		The construction of both functions $Q^\pm_U$ is similar so we focus here on the construction of $Q^+_U$. The proof will be separated in four steps. In the first step, we will construct the function $Q_U^+$ using a fixed point argument. The second step will be dedicated to proving that for all $z\in U$, $Q_U^+(z)$ is a projector for which the kernel and the range are respectively of dimension $dq$ and $dp$. The third and fourth steps concern the proof of the inequalities \eqref{in_geo_dich_+_1} and \eqref{in_geo_dich_+_2}.
		
		\textbf{\underline{Step 1:}} Construction of $Q_U^+$.
		
		We set for $z\in \Oc$ that
		$$\begin{array}{c}
			\eta_s^\pm(z):=\max\big\{\ln(|\zeta|), \quad \zeta\in\sigma(M^\pm(z))\cap \D\big\},\\
			\eta_u^\pm(z):= \min\big\{\ln(|\zeta|), \quad \zeta\in\sigma(M^\pm(z))\cap \U\big\}.
		\end{array}$$
		The functions $\eta_s^+$ and $\eta_u^+$ are continuous on $\Oc$ and verify that 
		$$\forall z \in \Oc, \quad \eta_s^+(z)<0\quad \text{ and }\quad \eta_u^+(z)>0.$$
		The set $\overline{U}$ being compact and included in $\Oc$, there exists a constant $c_H$ such that
		$$\max_{z\in\overline{U}}\eta_s^+(z)<-c_H<0\quad \text{ and } \quad 0<c_H<\min_{z\in \overline{U}}\eta_u^+(z).$$
		We will also ask that $c_H<\alpha$ where $\alpha$ is the constant appearing in \eqref{CV_expo_mat}. By definition of $\eta_s^+$ and $\eta_u^+$, there exists a positive constant $C_H$ such that
		\begin{equation}\label{in:expo}
			\forall z \in U, \forall j \in \N, \quad \begin{array}{c} \left|M^+(z)^jP_s^+(z)\right|\leq C_H e^{-c_H j}, \\ \left|M^+(z)^{-j}P_u^+(z)\right|\leq C_H e^{-c_H j}.\end{array} 
		\end{equation}
		Furthermore, using \eqref{CV_expo_mat}, since $U$ is bounded, there exists a positive constant $C_\Ec$ such that
		\begin{equation}\label{in:Ec}
			\forall z\in U, \forall j\in \N, \quad |\Ec_j^+(z)|\leq C_\Ec e^{-\alpha j}.
		\end{equation}
		
		We consider an integer $J\in\N$ and we will make a more precise choice later. We define the Banach space 
		$$\ell^\infty_J:= \lc \left(Y_j\right)_{j\geq J}\in\Mc_{(p+q)d}(\C)^{\lc j\in \N, j\geq J\rc},\quad \sup_{j\geq J} |Y_j|<+\infty\rc $$
		with the norm
		$$\left\|Y\right\|_{\infty,J}:=\sup_{j\geq J} |Y_j|.$$
		Furthermore, for $z\in U$, we define the linear map $\varphi(z) \in \Lc\left(\ell^\infty_J\right)$	and $T(z):\ell^\infty_J \rightarrow \ell^\infty_J$ such that for $Y\in\ell^\infty_J$ and $j\geq J$, we have
		$$\left(\varphi(z)Y\right)_j:= \sum_{k=J}^{j-1} M^+(z)^{j-1-k}P_s^+(z) \Ec_k^+(z)Y_k -  \sum_{k=j}^{+\infty} M^+(z)^{j-1-k}P_u^+(z) \Ec_k^+(z)Y_k$$
		and
		$$T(z)Y:= \left(M^+(z)^{j-J}P^+_s(z)\right)_{j\geq J} + \varphi(z)Y .$$
		
		We observe that 
		$$\forall Y\in\ell^\infty_J,\forall j\geq J,\quad (\varphi(z)Y)_{j+1} = M^+(z) (\varphi(z)Y)_j +\Ec_j^+(z)Y_j  ,$$
		and thus
		\begin{equation}\label{Tuj+1}
			\forall Y\in\ell^\infty_J,\forall j\geq J,\quad (T(z)Y)_{j+1} = M^+(z) (T(z)Y)_j +\Ec_j^+(z)Y_j  .
		\end{equation}
		Our goal will be to find a fixed point of $T(z)$. It will be a solution of the dynamical system \eqref{syst_dyn} for $j\geq J$. To do so, we will have to prove that there exists $J$ large enough so that 
		$$\left\|\varphi(z)\right\|_{\Lc(\ell^\infty_J)}<1.$$
		
		We begin by proving that the applications $\varphi(z)$ and $T(z)$ are well-defined. We consider $Y\in\ell^\infty_J$ and $j\geq J$. Using \eqref{in:expo} and \eqref{in:Ec}, we have the estimates:
		\begin{align*}
			\left|(\varphi(z)Y)_j\right| & \leq\sum_{k=J}^{j-1}C_HC_\Ec e^{-c_H|j-1-k|}e^{-\alpha k}|Y_k| + \sum_{k=j}^{+\infty}C_HC_\Ec e^{-c_H|j-1-k|}e^{-\alpha k}|Y_k|\\
			& \leq \left\|Y\right\|_{\infty,J}C_HC_\Ec e^{-\alpha J}\sum_{k=J}^{+\infty}e^{-c_H|j-1-k|}\\
			& \leq \left\|Y\right\|_{\infty,J}C_HC_\Ec e^{-\alpha J}\frac{1+e^{-c_H}}{1-e^{-c_H}}.
		\end{align*}
		
		If we set 
		\begin{equation}\label{def:theta}
			\theta:=C_HC_\Ec e^{-\alpha J}\frac{1+e^{-c_H}}{1-e^{-c_H}},
		\end{equation}
		then we have just proved that the operator $\varphi(z)$ is well-defined, bounded and
		$$\left\|\varphi(z)\right\|_{\Lc(\ell^\infty_J)}\leq \theta.$$
		We also observe that \eqref{in:expo} implies that $(M^+(z)^{j-J}P_s^+(z))_{j\geq J}\in \ell_J^\infty$. Therefore, $T(z)$ is well-defined.
		
		Let us choose the integer $J$ large enough so that $\theta<1$. For $z\in U$, we have that $Id-\varphi(z)$ is invertible. Thus, we can define
		$$Y(z):=(Id-\varphi(z))^{-1}\left(M^+(z)^{j-J}P^+_s(z)\right)_{j\geq J}.$$
		This sequence $Y(z)\in\ell^\infty_J$ is the only fixed point of $T(z)$ and it depends holomorphically on $z$. We observe that \eqref{Tuj+1} implies that
		\begin{equation}\label{eg:YU}
			\forall z\in U, \forall j\geq J,\quad Y_{j+1}(z)=M_j(z)Y_j(z).
		\end{equation}
		We now define 
		\begin{equation}\label{def:QU}
			\forall z\in U, \quad Q^+_U(z) :=  X_{J}(z)^{-1}Y_{J}(z)X_{J}(z).
		\end{equation}
		Since $Y$ depends holomorphically on $z$ and is bounded on $U$, $Q_U^+(z)$ also depends holomorphically on $z$ for $z\in U$ and is bounded on $U$. 
		
		\textbf{\underline{Step 2:}}  We now show that $Q_U^+$ is a projector.
		
		We are going to prove that for all $z\in U$ the matrix $Q_U^+(z)$ we have just constructed is a projector such that
		$$\ker Q_U^+(z) = X_{J}(z)^{-1} E^u(M^+(z)) \quad \text{ and } \quad \dim \Im Q_U^+(z)=dp.$$
		
		By observing that $P_s^+(z)^2=P_s^+(z)$, we can prove that $(Y_j(z) P_s^+(z))_{j\geq J}$ is another fixed point $T(z)$. Since $Y(z)$ is the only fixed point of $T(z)$ in $\ell^\infty_J$, we thus have that:
		\begin{equation}\label{eq_proj_1}
			Y_J(z)P_s^+(z)=Y_J(z).
		\end{equation}
		
		Using that $Y(z)$ is a fixed point of $T(z)$, we also have that
		\begin{align*}
			P_s^+(z) Y_J(z) &= P_s^+(z) (T(z)Y(z))_{J} \\
			& = P_s^+(z)\left(P_s^+(z)- \sum_{k=J}^{+\infty} M^+(z)^{J-1-k}P_u^+(z)\Ec_k^+(z)Y_k(z)\right).
		\end{align*}
		Because $P_s^+(z)$ commutes with $M^+(z)$, $P_s^+(z)^2=P_s^+(z)$ and $P_s^+(z)P_u^+(z)=0$, we have proved
		\begin{equation}\label{eq_proj_2}
			P_s^+(z)=P_s^+(z)Y_J(z).
		\end{equation}
		Using \eqref{eq_proj_2}, we prove that $(Y_j(z) Y_J(z))_{j\geq J}$ is a fixed point $T(z)$. Since $Y(z)$ is the only fixed point of $T(z)$ in $\ell^\infty_J$, we have in particular that:
		$${Y_{J}(z)}^2=Y_J(z)$$
		which means that $Y_J(z)$ is a projector. The definition \eqref{def:QU} of $Q_U^+(z)$ then implies that $Q_U^+(z)$ is a projector. The equalities \eqref{eq_proj_1} and \eqref{eq_proj_2} allow us to prove that $\ker Y_J(z) = \ker P_s^+(z) = E^u(M^+(z))$ which implies that:
		$$\Im Q_U^+(z) = X_{J}(z)^{-1}\Im Y_J(z)\quad \text{ and } \quad \ker Q_U^+(z) =X_{J}(z)^{-1} E^u(M^+(z)) .$$
		
		\textbf{\underline{Step 3:}}  We now show that $Q_U^+$ satisfies the inequalities \eqref{in_geo_dich_+_1} and \eqref{in_geo_dich_+_2}.
		
		First, we are going to prove the inequality \eqref{in_geo_dich_+_1} for $j\geq k\geq J$ and the inequality \eqref{in_geo_dich_+_2} for $k\geq j\geq J$. We observe that \eqref{eg:YU} implies that
		$$\forall z\in U,\forall j\geq J,\quad Y_{j+1}(z)= M_j(z)Y_j(z).$$
		and thus
		\begin{subequations}\label{def:YZ}
			\begin{equation}
				\forall z\in U,\forall j\geq J,\quad Y_j(z) = X_j(z)X_{J}(z)^{-1}Y_{J}(z)=X_j(z)Q_U^+(z)X_{J}(z)^{-1}.
			\end{equation}
			We introduce
			\begin{equation}
				\forall z\in U,\forall j\geq J, \quad Z_j(z):=X_j(z)X_{J}(z)^{-1}\left(Id-Y_{J}(z)\right)=X_j(z)\left(Id-Q_U^+(z)\right)X_{J}(z)^{-1}.
			\end{equation}
		\end{subequations}
		We have the following lemma.
		\begin{lemma}\label{lemYuZu}
			We have that
			\begin{align*}
				\forall j\geq k\geq J, \quad Y_j(z)= &M^+(z)^{j-k}P_s^+(z)Y_k(z)+\sum_{l=k}^{j-1}M^+(z)^{j-1-l}P_s^+(z)\Ec_l^+(z)Y_l(z)\\ & - \sum_{l=j}^{+\infty}M^+(z)^{j-1-l}P_u^+(z)\Ec_l^+(z)Y_l(z),
			\end{align*}
			and 
			\begin{align*}
				\forall k\geq j\geq J, \quad Z_j(z) =& M^+(z)^{j-k}P_u^+(z)Z_k(z) +\sum_{l=J}^{j-1}M^+(z)^{j-1-l}P_s^+(z)\Ec_l^+(z)Z_l(z)\\
				&-\sum_{l=j}^{k-1}M^+(z)^{j-1-l}P_u^+(z)\Ec_l^+(z)Z_l(z).
			\end{align*}
		\end{lemma}
		\begin{proof}\textbf{of Lemma \ref{lemYuZu}}
			\begin{itemize}
				\item Since we have that
				$$\forall j \geq J, \quad Y_{j+1}(z)=M_j(z)Y_j(z) = (M^+(z)+\Ec_j^+(z))Y_j(z),$$
				using the Duhamel formula, we find that
				$$\forall  k\geq J, \quad Y_k(z) = M^+(z)^{k-J}Y_J(z) +\sum_{l=J}^{k-1}M^+(z)^{k-1-l}\Ec_l^+(z)Y_l(z).$$
				Knowing that $Y(z)$ is a fixed point of $T(z)$ and that $P_s^+(z)Y_J(z)=P_s^+(z)$, we have for $j\geq k\geq J$
				\begin{align*}
					&Y_j(z) \\
					=& (T(z)Y(z))_j \\
					=& M^+(z)^{j-J}P_s^+(z)+\sum_{l=J}^{j-1}M^+(z)^{j-1-l}P_s^+(z)\Ec_l^+(z)Y_l(z) - \sum_{l=j}^{+\infty}M^+(z)^{j-1-l}P_u^+(z)\Ec_l^+(z)Y_l(z)\\
					=& M^+(z)^{j-k} P_s^+(z)\left(M^+(z)^{k -J}Y_J(z)+\sum_{l=J}^{k-1}M^+(z)^{k-1-l}\Ec_l^+(z)Y_l(z)\right) \\ 
					&+ \sum_{l=k}^{j-1}M^+(z)^{j-1-l}P_s^+(z)\Ec_l^+(z)Y_l(z) - \sum_{l=j}^{+\infty}M^+(z)^{j-1-l}P_u^+(z)\Ec_l^+(z)Y_l(z)\\
					=&  M^+(z)^{j-k}P_s^+(z)Y_k(z)+\sum_{l=k}^{j-1}M^+(z)^{j-1-l}P_s^+(z)\Ec_l^+(z)Y_l(z) - \sum_{l=j}^{+\infty}M^+(z)^{j-1-l}P_u^+(z)\Ec_l^+(z)Y_l(z)
				\end{align*}
				which corresponds to the statement of Lemma \ref{lemYuZu}.
				\item We now turn to the equation of $Z_j(z)$. Since we have that
				$$\forall j \geq J, \quad Z_{j+1}(z)=M_j(z)Z_j(z) = (M^+(z)+\Ec_j^+(z))Z_j(z),$$
				using the Duhamel formula, we find that for $k\geq j \geq J$
				\begin{align}
					Z_j(z) &= M^+(z)^{j-J}Z_J(z) +\displaystyle\sum_{l=J}^{j-1}M^+(z)^{j-1-l}\Ec_l^+(z)Z_l(z) \label{ZUj} \\
					Z_k(z) &= M^+(z)^{k-j}Z_j(z) +\displaystyle\sum_{l=j}^{k-1}M^+(z)^{k-1-l}\Ec_l^+(z)Z_l(z).\label{ZUk}
				\end{align}
				Using \eqref{ZUj} and knowing that $P_s^+(z)Z_J(z)=P_s^+(z)(Id-Y_J(z))=0$, we have that
				$$P_s^+(z)Z_j(z) = \sum_{l=J}^{j-1}M^+(z)^{j-1-l}P_s^+(z)\Ec_l^+(z)Z_l(z).$$
				Furthermore, \eqref{ZUk} implies that
				$$ P_u^+(z)Z_j(z)=M^+(z)^{j-k}P_u^+(z)Z_k(z) - \sum_{l=j}^{k-1}M^+(z)^{j-1-l}P_u^+(z)\Ec_l^+(z)Z_l(z).$$
				We end the proof of Lemma \ref{lemYuZu} by observing that $Z_j(z)=P_s^+(z)Z_j(z)+P_u^+(z)Z_j(z)$.
			\end{itemize}
		\end{proof}
		
		We introduce the constant
		\begin{equation}\label{def:ThetaTilde}
			\Theta:=\theta \frac{1-e^{-c_H}}{1+e^{-c_H}}=C_HC_\Ec e^{-\alpha J}>0
		\end{equation}
		where the constant $\theta$ is defined by \eqref{def:theta}. Using Lemma \ref{lemYuZu} and \eqref{in:expo}, we obtain that for any vector $\xi\in \C^{(p+q)d}$:
		\begin{align*}
			\forall j\geq k \geq J, \quad \left|Y_j(z)\xi\right| \leq& C_H e^{-c_H(j-k)}\left|Y_k(z)\xi\right| + \sum_{l=k}^{+\infty} C_HC_\Ec e^{-c_H|j-1-l|} e^{-\alpha l}|Y_l(z)\xi| \\
			 \leq& C_H e^{-c_H(j-k)}\left|Y_k(z)\xi\right|+ \Theta\sum_{l=k}^{+\infty}e^{-c_H|j-1-l|}\left|Y_l(z)\xi\right|,
		\end{align*}
		and similarly:
		\begin{equation*}
			\forall k\geq j \geq J, \quad \left|Z_j(z)\xi\right|\leq C_H e^{-c_H(k-j)}\left|Z_k(z)\xi\right|+ \Theta\sum_{l=J}^{k-1}e^{-c_H|j-1-l|}\left|Z_l(z)\xi\right|.
		\end{equation*}
		The following lemma corresponding to \cite[Lemma 1.5.1, Section III.1.5]{Godillon_these} will allow us to obtain clearer bounds on $|Y_j(z)\xi|$ and $|Z_j(z)\xi|$.
		
		\begin{lemma}\label{lemDecExp}
			Let us consider positive constants $C_H$, $c_H$ and $\Theta$ such that 
			\begin{equation}\label{def:theta2}
				\theta:= \Theta \frac{1+e^{-c_H}}{1-e^{-c_H}}=\Theta\frac{\cosh\left(\frac{c_H}{2}\right)}{\sinh\left(\frac{c_H}{2}\right)} \in]0,1[.
			\end{equation}
			For any real valued sequence $y\in \ell^\infty(\N)$ with non negative coefficients that satisfies :
			\begin{equation}\label{inY}
				\forall j\in \N, \quad y_j\leq C_He^{-c_H j}+\Theta\sum_{k=0}^{+\infty}e^{-c_H|j-1-k|}y_k,
			\end{equation}
			we have that:
			$$\forall j\in \N, \quad y_j\leq \rho r^j,$$
			where 
			\begin{equation}\label{def:rRho}
				r:=\cosh(c_H)-2\sinh\left(\frac{c_H}{2}\right)\sqrt{\cosh^2\left(\frac{c_H}{2}\right)-\theta}\in]e^{-c_H},1[ \quad \text{ and }\quad \rho:=\frac{C_H}{\Theta}(r-e^{-c_H})>0.
			\end{equation}
		\end{lemma}
		
		The proof can be found in the Appendix (Section \ref{sec:Appendix}). We will now use Lemma \ref{lemDecExp} to prove that
		\begin{subequations}
			\begin{equation}\label{inYU}
				\forall j\geq k\geq J, \quad |Y_j(z)\xi|\leq \rho r^{j-k}|Y_k(z)\xi|,
			\end{equation}
			\begin{equation}\label{inZU}
				\forall k\geq j\geq J, \quad |Z_j(z)\xi|\leq \rho r^{k-j}|Z_k(z)\xi|,
			\end{equation}
		\end{subequations}
		where $r$ and $\rho$ are defined by \eqref{def:rRho}.
		
		We consider $k\geq J$. If $Y_k(z)\xi\neq 0$, then by applying Lemma \ref{lemDecExp} to the bounded sequence $y:=\left(\frac{\left|Y_{k+j}(z)\xi\right|}{\left|Y_k(z)\xi\right|}\right)_{j\in\N}$, we obtain \eqref{inYU} with $r$ and $\rho$ defined by \eqref{def:rRho}. Else, if $Y_k(z)\xi=0$, then for $j\geq k$, we have:
		$$Y_j(z)\xi = X_j(z)X_k(z)^{-1}Y_k(z)\xi=0.$$
		Thus, \eqref{inYU} is also verified in this case.
		
		The proof of \eqref{inZU} is similar. If $Z_k(z)\xi\neq 0$, then we apply Lemma \ref{lemDecExp} to the sequence $y$ defined by
		$$\forall j\in \N, \quad y_j:=\lc\begin{array}{cc}\frac{\left|Z_{k-j}(z)\xi\right|}{\left|Z_k(z)\xi\right|} & \text{ if } j\in \lc0,\ppp,k\rc, \\ 0 & \text{ else.}\end{array}\right.$$
		This proves \eqref{inZU} in this case. If $Z_k(z)\xi= 0$, then since 
		$$\forall j\in\lc J,\ppp,k\rc, \quad Z_j(z)\xi = X_j(z)X_k(z)^{-1}Z_k(z)\xi=0,$$
		\eqref{inZU} is also trivially verified in this case.
		
		Using \eqref{def:YZ}, \eqref{inYU} and \eqref{inZU}, we proved that:
		\begin{subequations}
			\begin{align}
				\forall j\geq k\geq J, &\quad |X_j(z)Q_U^+(z)X_k(z)^{-1}|\leq \rho r^{j-k}|X_k(z)Q_U^+(z)X_k(z)^{-1}|,\label{inQ1}\\
				\forall k\geq j\geq J, &\quad |X_j(z)(Id-Q_U^+(z))X_k(z)^{-1}|\leq \rho r^{k-j}|X_k(z)(Id-Q_U^+(z))X_k(z)^{-1}|.\label{inQ2}
			\end{align}
		\end{subequations}
		If we prove that the families $(X_k(z)Q_U^+(z)X_k(z)^{-1})_{k\geq J}$ and $(X_k(z)(Id-Q_U^+(z))X_k(z)^{-1})_{k\geq J}$ are uniformly bounded for $z\in U$, we will have proved \eqref{in_geo_dich_+_1} and \eqref{in_geo_dich_+_2} respectively for $j\geq k\geq J$ and $k\geq j\geq J$.
		
		Using Lemma \ref{lemYuZu}, we prove that for $j\geq J$:
		\begin{align*}
			P_u^+(z)(Y(z))_j &= -\sum_{l=j}^{+\infty}M^+(z)^{j-1-l}P_u^+(z)\Ec_l^+(z)(Y(z))_l,\\
			P_s^+(z)(Z(z))_j &= \sum_{l=J}^{j-1}M^+(z)^{j-1-l}P_s^+(z)\Ec_l^+(z)(Z(z))_l.
		\end{align*}
		Thus, we have using \eqref{def:YZ} that:
		\begin{align*}
			P_u^+(z)X_j(z)Q_U^+(z)X_j(z)^{-1}&= -\sum_{l=j}^{+\infty}M^+(z)^{j-1-l}P_u^+(z)\Ec_l^+(z)X_l(z)Q_U^+(z)X_j(z)^{-1},\\
			P_s^+(z)X_j(z)(Id-Q_U^+(z))X_j(z)^{-1}&= \sum_{l=J}^{j-1}M^+(z)^{j-1-l}P_s^+(z)\Ec_l^+(z)X_l(z)(Id-Q_U^+(z))X_j(z)^{-1}.
		\end{align*}
		Using \eqref{in:expo}, the definitions \eqref{def:theta}, \eqref{def:ThetaTilde} and \eqref{def:rRho} of the constants $\theta$, $\Theta$ and $\rho$, as well as \eqref{inQ1}, we have
		\begin{align*}
			|P_u^+(z)X_j(z)Q_U^+(z)X_j(z)^{-1}|&\leq \Theta\sum_{l=j}^{+\infty} e^{-c_H(l-(j-1))}|X_l(z)Q_U^+(z)X_j(z)^{-1}|\\
			& \leq \Theta e^{-c_H}\rho \sum_{l=j}^{+\infty}(re^{-c_H})^{l-j} |X_j(z)Q_U^+(z)X_j(z)^{-1}|\\
			& = \Theta \frac{e^{-c_H}}{1-re^{-c_H}}\rho |X_j(z)Q_U^+(z)X_j(z)^{-1}|\\
			& = C_H \frac{r-e^{-c_H}}{e^{c_H}-r} |X_j(z)Q_U^+(z)X_j(z)^{-1}|.
		\end{align*}
		Similarly, using \eqref{inQ2}, we have
		\begin{align*}
			|P_s^+(z)X_j(z)(Id-Q_U^+(z))X_j(z)^{-1}|&\leq \Theta\sum_{l=J}^{j-1} e^{-c_H(j-1-l)}|X_l(z)(Id-Q_U^+(z))X_j(z)^{-1}|\\
			& \leq \Theta e^{c_H}\rho \sum_{l=J}^{j-1}(re^{-c_H})^{j-l} |X_j(z)(Id-Q_U^+(z))X_j(z)^{-1}|\\
			&  \leq \Theta \rho r\frac{1}{1-re^{-c_H}} |X_j(z)(Id-Q_U^+(z))X_j(z)^{-1}|\\
			& = C_H re^{c_H}\frac{r-e^{-c_H}}{e^{c_H}-r} |X_j(z)(Id-Q_U^+(z))X_j(z)^{-1}|.
		\end{align*}
		Therefore, if we define $\eta:= C_H\frac{e^{c_H}}{e^{c_H}-1}(r-e^{-c_H})$, we have for all $j\geq J$
		\begin{align}\label{in:lemGeo}
			\begin{split}
				|P_u^+(z)X_j(z)Q_U^+(z)X_j(z)^{-1}|&\leq \eta |X_j(z)Q_U^+(z)X_j(z)^{-1}|,\\
				|P_s^+(z)X_j(z)(Id-Q_U^+(z))X_j(z)^{-1}| & \leq \eta|X_j(z)(Id-Q_U^+(z))X_j(z)^{-1}|.
			\end{split}
		\end{align}
		Using the definition \eqref{def:rRho} of $r$, we observe that
		\begin{equation}\label{def:eta}
			\eta = C_H\frac{e^{c_H}}{e^{c_H}-1} 2\sinh\left(\frac{c_H}{2}\right)\left(\cosh\left(\frac{c_H}{2}\right) - \sqrt{\cosh\left(\frac{c_H}{2}\right)^2-\theta}\right).
		\end{equation}
		We already supposed that $J$ was taken large enough so that the number $\theta$ in \eqref{def:theta} satisfies $\theta<1$. We will now also suppose that we took $J$ large enough so that $\theta$ is close enough to $0$ so that the above number $\eta$ in \eqref{def:eta} satisfies $\eta<\frac{1}{2}$. To conclude this step of the proof, we observe that
		$$X_j(z)Q_U^+(z)X_j(z)^{-1}-P_s^+(z)= P_u^+(z)X_j(z)Q_U^+(z)X_j(z)^{-1}-P_s^+(z)X_j(z)(Id-Q_U^+(z))X_j(z)^{-1}$$
		and
		$$X_j(z)(Id-Q_U^+(z))X_j(z)^{-1}-P_u^+(z)= P_s^+(z)X_j(z)(Id-Q_U^+(z))X_j(z)^{-1}-P_u^+(z)X_j(z)Q_U^+(z)X_j(z)^{-1}.$$
		Thus, using \eqref{in:expo} to bound $P^+_s$ and $P^+_u$ and \eqref{in:lemGeo}, we have
		$$|X_j(z)Q_U^+(z)X_j(z)^{-1}|\leq C_H + \eta \left(|X_j(z)Q_U^+(z)X_j(z)^{-1}| + |X_j(z)(Id-Q_U^+(z))X_j(z)^{-1}|\right)$$
		and 
		$$|X_j(z)(Id-Q_U^+(z))X_j(z)^{-1}|\leq C_H + \eta \left(|X_j(z)Q_U^+(z)X_j(z)^{-1}| + |X_j(z)(Id-Q_U^+(z))X_j(z)^{-1}|\right).$$
		This implies that :
		$$\forall j \geq J, \quad \begin{array}{c}|X_j(z)Q_U^+(z)X_j(z)^{-1}|\leq \displaystyle\frac{2C_H}{1-2\eta},\\ |X_j(z)(Id-Q_U^+(z))X_j(z)^{-1}|\leq \displaystyle\frac{2C_H}{1-2\eta}.\end{array}$$
		
		Therefore, we have proved that for all $z\in U$, we have:
		\begin{subequations}
			\begin{align}
				\forall j\geq k \geq J, \quad & |X_j(z)Q_U^+(z)X_k(z)^{-1}|\leq \rho \frac{2C_H}{1-2\eta} r^{j-k},\label{inQ1b}\\
				\forall k\geq j \geq J, \quad & |X_j(z)(Id-Q_U^+(z))X_k(z)^{-1}|\leq \rho \frac{2C_H}{1-2\eta} r^{k-j}.\label{inQ2b}
			\end{align} 
		\end{subequations}
		
		\textbf{\underline{Step 4:}}  $Q_U^+$ satisfies the inequalities \eqref{in_geo_dich_+_1} and \eqref{in_geo_dich_+_2} respectively for all $j\geq k \geq 0$ and $k\geq j \geq 0$
		
		We will only finish the proof of \eqref{in_geo_dich_+_1} since the proof for \eqref{in_geo_dich_+_2} is similar. We have proved \eqref{in_geo_dich_+_1} for $j\geq k \geq J$. We consider a constant $C>0$ such that
		$$\forall z \in U, \quad \begin{array}{c} C>r^{-J}\max_{j\in\lc0,\ppp,J-1\rc}|X_j(z)Q_U^+(z)X_{J}(z)^{-1}| \\ C>r^{-J}\max_{j\in\lc0,\ppp,J-1\rc}|X_{J}(z)Q_U^+(z)X_j(z)^{-1}|.\end{array}$$
		This can be done since the projector $Q_U^+$ defined by \eqref{def:QU} is bounded on $U$.
		\begin{itemize}
			\item If $j\geq J >k\geq 0$, we have 
			\begin{align*}
				|X_j(z)Q_U^+(z)X_k(z)^{-1}|&\leq|X_j(z)Q_U^+(z)X_{J}(z)^{-1}||X_{J}(z)Q_U^+(z)X_k(z)^{-1}| \\
				&\leq \rho \frac{2C_H}{1-2\eta} r^{j-J} C r^{J} \\
				&\leq C\rho \frac{2C_H}{1-2\eta} r^{j-k}.
			\end{align*}
			\item If $J> j\geq k\geq 0$, we have 
			\begin{align*}
				|X_j(z)Q_U^+(z)X_k(z)^{-1}|&\leq|X_j(z)Q_U^+(z)X_{J}(z)^{-1}||X_{J}(z)Q_U^+(z)X_k(z)^{-1}| \\
				&\leq  C^2 r^{2J} \\
				&\leq C^2 r^{j-k}.
			\end{align*}
		\end{itemize}
		Therefore, there exist two constants $C,c>0$ such that for all $z\in U$, \eqref{in_geo_dich_+_1} is verified.
	\end{proof}
	
	\subsection{Spectrum of \texorpdfstring{$\Lcc$}{L} and extended geometric dichotomy}
	
	Now that we have proved the geometric dichotomy, Lemma \ref{lem_geo_dich} let us go back on studying the vector spaces $E_0^\pm(z)$ which characterize the solutions of the dynamical system \eqref{syst_dyn} converging towards $0$ as $j$ tends towards $\pm\infty$. The previous section about the geometric dichotomy allows us to prove the following lemma.
	
	\begin{lemma}\label{lem_id_E0}
		For any open bounded set $U$ such that $\overline{U}\subset \Oc$, we have
		$$\forall z\in U,\quad E^+_0(z)=\Im(Q^+_U(z)) \quad \text{ and }\quad E^-_0(z)=\Im(Q^-_U(z)) .$$
		Therefore, for all $z\in \Oc$, $\dim E_0^+(z)=dp$ and $\dim E_0^-(z)=dq$. Also, for $W_0\in\C^{d(p+q)}$, we have that
		$$W_0\in E^+_0(z)\cap E^-_0(z) \Leftrightarrow(X_j(z)W_0)_{j\in\Z}\in \ell^2(\Z, \C^{(p+q)d})$$
		where $X_j(z)$ is defined by \eqref{mat_fond}.
	\end{lemma}
	
	\begin{proof}
		We prove the first set equality on $E_0^+(z)$. The second one on $E_0^-(z)$ would be proved similarly. 
		\begin{itemize}
			\item For $W_0\in \Im(Q^+_U(z))$, we have for $j\in \N$ using \eqref{in_geo_dich_+_1}
			$$X_j(z) W_0 = X_j(z)Q^+_U(z)X_0(z)^{-1} W_0 \underset{j\rightarrow +\infty}\rightarrow 0.$$
			Thus, we have that $W_0\in E_0^+(z)$.
			\item For $W_0\in E^+_0(z)$, we have for $j\in \N$ using \eqref{in_geo_dich_+_2}
			$$(Id-Q^+_U(z))W_0 = X_0(z)(Id-Q^+_U(z))X_j(z)^{-1}X_j(z)W_0 \underset{j\rightarrow +\infty}\rightarrow 0.$$
			Thus, $W_0$ belongs to the kernel of $Id-Q^+_U(z)$, i.e. $W_0\in \Im(Q^+_U(z))$.
		\end{itemize}
		Therefore, we have proved that 
		$$E_0^\pm(z)=\Im(Q^\pm_U(z)).$$
		
		For $W_0\in\C^{d(p+q)}$, we immediately have that if the family $(X_j(z)W_0)_{j\in\Z}$ belongs to $\ell^2(\Z,\C^{d(p+q)})$, then $W_0$ belongs to $E_0^+(z)\cap E_0^-(z)$. We now consider $W_0\in E^+_0(z)\cap E^-_0(z) = \Im(Q^+_U(z))\cap \Im(Q^-_U(z))$. Since we have
		$$\forall j\in\Z, \quad X_j(z)W_0 = X_j(z) Q^+_U(z)X_0(z)^{-1}W_0 = X_j(z) Q^-_U(z)X_0(z)^{-1}W_0,$$
		the inequalities \eqref{in_geo_dich_+_1} and \eqref{in_geo_dich_-_1} imply that $(X_j(z)W_0)_{j\in\Z}\in \ell^2(\Z, \C^{d(p+q)})$.
	\end{proof}
	
	Let us now come back to the heart of the matter: the study of the spectrum of the operator $\Lcc$. We introduced the dynamical system \eqref{syst_dyn} to study the solutions of the eigenvalue problem 
	$$(zId_{\ell^2}-\Lcc)u=0.$$
	The following lemma, for which the main part is proved in \cite[Theorem 4.1]{Serre}, is deduced by using the geometric dichotomy.
	
	\begin{lemma}\label{lem_spec_ess}
		For $z\in \Oc$, we have that
		\begin{equation}\label{dimEigSpaceLcc}
			\dim \ker(zId_{\ell^2}-\Lcc)=\dim E_0^+(z)\cap E_0^-(z).
		\end{equation}
		Furthermore, $zId_{\ell^2}-\Lcc$ is a Fredholm operator of index $0$, i.e.
		$$\sigma_{ess}(\Lcc)\cap\Oc=\emptyset.$$		
	\end{lemma}
	
	Before proving Lemma \ref{lem_spec_ess}, let us thus introduce the linear map which extracts the center values of a vector of size $d(p+q)$
	\begin{equation}\label{def:Pi}
		\begin{array}{cccc}
			\Pi : &\C^{d(p+q)} &\rightarrow  & \C^d \\ & (x_j)_{j\in\lc1, \ppp,d(p+q)\rc} & \mapsto &(x_j)_{j\in\lc d(q-1)+ 1, \ppp,dq\rc}
		\end{array}.
	\end{equation}
	We now give the proof of Lemma \ref{lem_spec_ess}. Let us point out that the proof of the fact that the essential spectrum of $\Lcc$ does not intersect $\Oc$ is exactly the same proof as in \cite[Theorem 4.1]{Serre}.
	
	\begin{proof}\textbf{of Lemma \ref{lem_spec_ess}}
		
		We consider $z\in \Oc$ and start by proving the relation \eqref{dimEigSpaceLcc}.
		
		$\bullet$ For $w\in \ker(zId_{\ell^2}-\Lcc)$, if we introduce 
		$$W_0:= \begin{pmatrix}
			w_{q-1}\\ \vdots \\ w_{-p}
		\end{pmatrix}\in\C^{d(p+q)},$$
		then we have that
		$$\forall j\in \Z, \quad X_j(z)W_0 = \begin{pmatrix}
			w_{j+q-1}\\ \vdots \\ w_{j-p}
		\end{pmatrix}.$$
		Since $w$ belongs to $\ell^2(\Z,\C^d)$, we have that $W_0\in E_0^+(z)\cap E_0^-(z)$. This implies that the linear map:
		$$\begin{array}{cccc}\varphi: &\ker(zId_{\ell^2}-\Lcc)  & \rightarrow & E_0^+(z)\cap E_0^-(z)\\ & w & \mapsto & \begin{pmatrix}
				w_{q-1}\\ \vdots \\ w_{-p}
		\end{pmatrix}\end{array}$$
		is well-defined.
		
		$\bullet$ We consider $W_0\in E_0^+(z)\cap E_0^-(z)$ and define for $j\in \Z$
		$$w_j:= \Pi (X_j(z)W_0)\in\C^d$$
		where the operator $\Pi$ is defined by \eqref{def:Pi}. Lemma \ref{lem_id_E0} implies that the sequence $w:=(w_j)_{j\in \Z}$ belongs to $\ell^2(\Z,\C^d)$. Furthermore, since $(X_j(z)W_0)_{j\in\Z}$ is a solution of \eqref{syst_dyn}, we have that
		$$(zId_{\ell^2}-\Lcc)w=0.$$
		Therefore, the linear map:
		$$\begin{array}{cccc}\psi: &E_0^+(z)\cap E_0^-(z)& \rightarrow & \ker(zId_{\ell^2}-\Lcc)  \\ & W_0 & \mapsto & (\Pi (X_j(z)W_0))_{j\in\Z}\end{array}$$
		is also well-defined and we can easily verify the equalities:
		$$\varphi\circ \psi = Id_{E_0^+(z)\cap E_0^-(z)}\quad \text{ and } \quad \psi\circ \varphi = Id_{\ker(zId_{\ell^2}-\Lcc)}.$$
		This concludes the proof of \eqref{dimEigSpaceLcc}.
		
		We now focus on the second part of Lemma \ref{lem_spec_ess} which consists in proving that for any $z\in \Oc$ we have that $zId_{\ell^2}-\Lcc$ is a Fredholm operator of index $0$. This part of the proof is the same as \cite[Theorem 4.1]{Serre}. Our first goal is to prove that $zId_{\ell^2}-\Lcc$ is a Fredholm operator. We have already that
		$$\dim \ker(zId_{\ell^2}-\Lcc)= \dim E_0^+(z)\cap E_0^-(z)<+\infty.$$
		There remains to prove that $\Im(zId_{\ell^2}-\Lcc)$ is closed and that
		$$\mathrm{codim}\, \Im(zId_{\ell^2}-\Lcc)<+\infty.$$
		We now fix a bounded open neighborhood $U$ of $z\in \Oc$ such that $\overline{U}\subset \Oc$. We consider $h\in \ell^2(\Z,\C^d)$. The sequence $h$ belongs to the range of $zId_{\ell^2}-\Lcc$ if and only if there exists $v\in \ell^2(Z,\C^d)$ such that, if we define the vectors:
		$$\forall j \in \Z, \quad W_j=\begin{pmatrix}
			v_{j+q-1}\\ \vdots \\ v_{j-p}
		\end{pmatrix}\in\C^{d(p+q)},\quad H_j=\begin{pmatrix}
			\A_{j,q}(z)^{-1}h_j\\ 0 \\ \vdots \\ 0
		\end{pmatrix}\in\C^{d(p+q)},$$
		then there holds the recurrence relations
		\begin{equation}\label{syst_dyn_inh}
			\forall j \in \Z, \quad W_{j+1} = M_j(z)W_j + H_j.
		\end{equation}
		For $j\geq 0$, we define
		$$Z^+_j:= \sum_{k=0}^jX_j(z)Q_U^+(z)X_k(z)^{-1}H_{k-1} - \sum_{k=j+1}^{+\infty}X_j(z)(Id-Q_U^+(z))X_k(z)^{-1}H_{k-1},$$
		where the matrix $Q_U^+(z)$ is defined in Lemma \ref{lem_geo_dich}. Those vectors $Z_j^+$ are well-defined and verify that
		$$\forall j \geq 0, \quad Z^+_{j+1} = M_j(z)Z^+_j+H_j.$$
		Furthermore, using inequalities \eqref{in_geo_dich_+_1} and \eqref{in_geo_dich_+_2}, there exist two positive constants $C,c$ such that
		$$\forall j\in\N,\quad |Z_j^+|\leq C\sum_{k=0}^{+\infty}e^{-c|j-k|}|H_{k-1}|.$$
		Using Young's convolution inequality, the sequence $(Z_j^+)_{j\in\N}$ belongs to $\ell^2(\N)$ and satisfies the estimates
		$$\left(\sum_{j\geq 0} |Z^+_j|^2\right)^\frac{1}{2}\leq C \left(\sum_{j\in \Z} |H_j|^2\right)^\frac{1}{2}$$
		where the positive $C$ is independent of $h$. At this stage, we have found a particular solution of \eqref{syst_dyn_inh} on $\N$ and any sequence $(\tilde{Z}_j)_{j\geq 0}$ solution of \eqref{syst_dyn_inh} on $\N$ can be written as:
		$$\forall j\geq 0, \quad \tilde{Z}_j = Z^+_j +X_j(z)V^+$$
		where $V^+\in E_0^+(z)$. We prove in a similar way that the sequences $(\tilde{Z}_j)_{j\leq 0}\in\ell^2(-\N)$ solution of \eqref{syst_dyn_inh} on $-\N$ are the sequences defined by
		$$\forall j\leq 0, \quad \tilde{Z}_j = Z^-_j +X_j(z)V^-$$
		where $V^-$ is a vector in the finite dimension space $E_0^-(z)$ and the vectors $Z_j^-$ are defined by:
		$$\forall j\leq 0, \quad Z^-_j:= \sum_{k=-\infty}^jX_j(z)(Id-Q^-_U(z))X_k(z)^{-1}H_{k-1}- \sum_{k=j+1}^1X_j(z)Q_U^-(z)X_k(z)^{-1}H_{k-1}.$$
		We also have that 
		$$\left(\sum_{j\leq 0} |Z^-_j|^2\right)^\frac{1}{2}\leq C \left(\sum_{j\in \Z} |H_j|^2\right)^\frac{1}{2}$$
		where the positive constant $C$ is independent from $h$. Using all those information, we conclude that a sequence $h$ belongs to the range of $zId_{\ell^2}-\Lcc$ if and only if there exists a couple of vectors $(V^+,V^-)\in E_0^+(z)\times E_0^-(z)$ such that
		$$Z_0^+-Z_0^-=V^--V^+.$$
		If we now define the bounded operator 
		\begin{equation}\label{def:nu}
			\nu:h\in \ell^2(\Z,\C^d)\mapsto Z_0^+-Z_0^-\in\C^{d(p+q)}
		\end{equation}
		and 
		\begin{equation}\label{def:phi}
			\varphi : (V^+,V^-)\in E_0^+(z)\times E_0^-(z)\mapsto V^--V^+\in \C^{d(p+q)}
		\end{equation}
		which is an operator from a finite dimension vector space to another finite dimension vector space, then we have proved that
		$$\Im(zId_{\ell^2}-\Lcc)=\nu^{-1}(\Im \varphi).$$
		Therefore, the range $\Im(zId_{\ell^2}-\Lcc)$ is closed.
		
		We now want to prove that $\mathrm{codim}\, \Im(zId_{\ell^2}-\Lcc)<+\infty.$ We consider $N\geq 1$ such that there exists $(h_1,\ppp,h_N)$ a linearly independent family of $\ell^2(\Z,\C^d)$ such that 
		$$\Im(zId_{\ell^2}-\Lcc) \cap Span(h_1,\ppp,h_N)=\lc 0\rc.$$
		We are going to prove that the family $(\nu(h_1),\ppp,\nu(h_N))$ is linearly independent in $\C^{d(p+q)}$ and therefore that $N\leq d(p+q)$. Here and below, the operator $\nu$ is the one defined in \eqref{def:nu}. We consider scalars $\lambda_1, \ppp,\lambda_N\in \C$ such that
		$$0=\sum_{i=1}^N\lambda_i\nu(h_i)=\nu\left(\sum_{i=1}^N\lambda_ih_i\right).$$
		We therefore have that 
		$$\nu\left(\sum_{i=1}^N\lambda_ih_i\right)\in \Im\varphi$$
		with $\varphi$ defined in \eqref{def:phi} and thus
		$$\sum_{i=1}^N\lambda_ih_i\in \Im(zId_{\ell^2}-\Lcc) \cap Span(h_1,\ppp,h_N).$$
		This implies that 
		$$\sum_{i=1}^N\lambda_ih_i=0$$
		and the linear independency of $(h_1,\ppp,h_N)$ allows us to conclude that $\lambda_1=\ppp=\lambda_N=0$. We have thus proved that $zId_{\ell^2}-\Lcc$ is a Fredholm operator for all $z\in \Oc$. We also know that, since $\Oc$ is unbounded, there exists $z\in \Oc$ such that $zId_{\ell^2}-\Lcc$ is an isomorphism. The set $\Oc$ being connected and by continuity of the Fredholm index, the statement of the lemma is true.
	\end{proof}
	
	We now introduce the sets 
	\begin{equation}\label{def:Oc_rho_sigma}
		\Oc_\rho := \Oc\cap\rho(\Lcc) \quad \text{ and }\quad \Oc_\sigma := \Oc\cap\sigma(\Lcc).
	\end{equation}
	Lemma \ref{lem_spec_ess} implies that the set $\Oc_\sigma$ only contains eigenvalues of $\Lcc$. Then, because of Hypothesis \ref{H:spec}, we have that 
	$$\overline{\U}\backslash\lc 1\rc\subset \Oc_\rho. $$
	We also observe that Lemma \ref{lem_id_E0} gives us the dimension of the subspaces $E_0^\pm(z)$. Then \eqref{dimEigSpaceLcc} implies that
	\begin{equation}\label{decomp_E0}
		\forall z\in\Oc_\rho,\quad E_0^+(z)\oplus E_0^-(z)=\C^{(p+q)d}.
	\end{equation}
	Thus, for $z\in\Oc_\rho$, we can define the unique projector $Q(z)$ from $\C^{d(p+q)}$ to $E_0^+(z)$ such that
	$$\Im Q(z) =E_0^+(z) \quad \text{ and } \quad \ker Q(z) =E_0^-(z).$$
	
		The function $z\in \Oc_\rho\mapsto Q(z)$ is holomorphic (see \cite[Chapter 2 Section 4.2]{Kato} to construct bases of $E_0^\pm(z)$ that are locally holomorphic). We will now prove that the function $Q$ is fundamental to the study of \eqref{syst_dyn} by extending the geometric dichotomy. The following lemma is once again very much inspired by \cite[Section III.1.5]{Godillon_these} and \cite[Chapter 2]{Coppel}.

	\begin{lemma}[Extended geometric dichotomy]\label{lemGeoDichStronger}
		For any bounded open set $U$ such that $\overline{U}\subset \Oc_\rho$, there exist two positive constants $C,c>0$ such that for all $z\in U$, the projector $Q(z)$ associated with the decomposition \eqref{decomp_E0} satisfies:
		\begin{subequations}
			\begin{align}
				\forall  j \geq k, &\quad \left|X_j(z)Q(z)X_k(z)^{-1}\right|\leq C e^{-c|j-k|},\label{in_geo_dich_1}\\
				\forall k \geq j, &\quad \left|X_j(z)(Id-Q(z))X_k(z)^{-1}\right|\leq C e^{-c|j-k|}.\label{in_geo_dich_2}
			\end{align}
		\end{subequations}
	\end{lemma}
	
	\begin{proof}
		We begin by assuming that we proved the existence of two constants $C,c>0$ such that for all $z \in U$
		\begin{subequations}
			\begin{align}
				\forall  j\geq k\geq 0, &\quad \left|X_j(z)Q(z)X_k(z)^{-1}\right|\leq C e^{-c|j-k|},\label{in_geo_dich_1_a}\\
				\forall k\geq j\geq 0, &\quad \left|X_j(z)(Id-Q(z))X_k(z)^{-1}\right|\leq C e^{-c|j-k|},\label{in_geo_dich_2_a}\\
				\forall   k \leq j \leq 0, &\quad \left|X_j(z)Q(z)X_k(z)^{-1}\right|\leq C e^{-c|j-k|},\label{in_geo_dich_1_b}\\
				\forall  j \leq k \leq 0, &\quad \left|X_j(z)(Id-Q(z))X_k(z)^{-1}\right|\leq C e^{-c|j-k|}.\label{in_geo_dich_2_b}
			\end{align}
		\end{subequations}
		Then, we observe that to prove the assertion \eqref{in_geo_dich_1}, there would only remain to prove \eqref{in_geo_dich_1} in the case where $j\geq 0\geq k$. Using \eqref{in_geo_dich_1_a} and \eqref{in_geo_dich_1_b}, we have
		$$\left|X_j(z)Q(z)X_k(z)^{-1}\right|\leq\left|X_j(z)Q(z)X_0(z)^{-1}\right|\left|X_0(z)Q(z)X_k(z)^{-1}\right|\leq C^2e^{-c(j-k)}.$$
		Hence, the assertion \eqref{in_geo_dich_1} follows from \eqref{in_geo_dich_1_a} and \eqref{in_geo_dich_1_b}. Similarly, \eqref{in_geo_dich_2} follows from \eqref{in_geo_dich_2_a} and \eqref{in_geo_dich_2_b}.
		
		Therefore, there only remains to prove the existence of $C,c>0$ such that \eqref{in_geo_dich_1_a}-\eqref{in_geo_dich_2_b} are true for all $z\in U$. We will prove \eqref{in_geo_dich_1_a} and \eqref{in_geo_dich_1_b}. The proof for \eqref{in_geo_dich_2_a} and \eqref{in_geo_dich_2_b} can be dealt with similarly. First, we need to consider a bounded open set $V$ such that $\overline{V}\subset \Oc_\rho$ and $\overline{U}\subset V$ and the projector $Q_V^+$ provided by Lemma \ref{lem_geo_dich}. This will be useful later on to bound the difference $Q^+_V-Q$. For $z\in U$, Lemma \ref{lem_id_E0} implies that $E_0^+(z) =\Im Q^+_V(z) = \Im Q(z)$, i.e.
		$$Q^+_V(z)Q(z)=Q(z) \quad \text{ and }\quad Q(z)Q^+_V(z)=Q^+_V(z).$$
		This allows us to prove that 
		$$Q^+_V(z)-Q(z) =Q^+_V(z)(Q^+_V(z) -Q(z))(Id-Q^+_V(z)).$$
		Therefore, for $j,k\in \N$, we have
		$$X_j(z) (Q^+_V(z)-Q(z))X_k(z)^{-1} = X_j(z) Q^+_V(z)X_0(z)^{-1}(Q^+_V(z)-Q(z))X_0(z) (Id-Q^+_V(z))X_k(z)^{-1}.$$
		Thus, because of the inequalities \eqref{in_geo_dich_+_1} and \eqref{in_geo_dich_+_2}, we have the estimate
		\begin{equation}\label{ine_proj}
			\left|X_j(z) (Q^+_V(z)-Q(z))X_k(z)^{-1}\right| \leq C^2e^{-c(j+k)}|Q^+_V(z)-Q(z)|.
		\end{equation}
		Using the inequalities \eqref{in_geo_dich_+_1}, \eqref{in_geo_dich_+_2} and \eqref{ine_proj}, we can thus prove that
		$$\forall j\geq k\geq 0,\quad |X_j(z)Q(z)X_k(z)^{-1}|\leq Ce^{-c(j-k)}+ C^2e^{-c(j+k)}|Q^+_V(z)-Q(z)|$$
		and
		$$\forall k\geq j\geq 0,\quad |X_j(z)(Id-Q(z))X_k(z)^{-1}|\leq Ce^{-c(k-j)}+ C^2e^{-c(j+k)}|Q^+_V(z)-Q(z)|.$$
		Since $z\in V\mapsto|Q^+_V(z)-Q(z)|$ is continuous and $\overline{U}\subset V$, we can uniformly bound $|Q^+_V(z)-Q(z)|$ for $z\in U$. We can then deduce the existence of two positive constants $C,c$ to verify the inequalities \eqref{in_geo_dich_1_a} and \eqref{in_geo_dich_1_b}.
	\end{proof}
	
	\subsection{Bounds on the spatial Green's function far from \texorpdfstring{$1$}{1}}\label{subsec:BoundsGsLoin}
	
	For $z\in \Oc_\rho$ and $j_0\in \Z$, since $z$ is in the resolvent set of $\Lcc$, the spatial Green's function $G(z,j_0,\cdot)$ defined by \eqref{defGreenSpatial} is well-defined. We observe that the function $z\in\Oc_\rho\mapsto G(z,j_0,\cdot)$ is holomorphic.
	
	We consider $\textbf{e}\in \C^d$. We then observe that the vector valued sequence $G(z,j_0,\cdot)\textbf{e}$ belongs to $\ell^2(\Z,\C^d)$ and satisfies:
	$$\forall z\in \Oc_\rho,\forall j_0\in\Z,\forall \textbf{e}\in\C^d,\quad zG(z,j_0,\cdot)\textbf{e}- \Lcc G(z,j_0,\cdot)\textbf{e}=\delta_{j_0}\textbf{e},$$
	i.e.
	$$\forall z\in \Oc_\rho,\forall j_0,j\in\Z,\forall \textbf{e}\in\C^d, \quad \sum_{k=-p}^q \A_{j,k}(z)G(z,j_0,j+k)\textbf{e}=\delta_{j_0,j}\textbf{e}.$$
	Thus, if we define the vectors:
	\begin{equation}\label{def:WGreenspatial}
		\forall z\in \Oc_\rho, \forall j_0,j\in\Z,\forall \textbf{e}\in\C^d,\quad W(z,j_0,j,\textbf{e}) := \begin{pmatrix}
			G(z,j_0,j+q-1)\textbf{e}\\\vdots\\G(z,j_0,j-p)\textbf{e}
		\end{pmatrix},
	\end{equation}
	then we have that:
	\begin{equation}\label{systDynWj}
		\forall z\in \Oc_\rho, \forall j_0,j\in\Z,\forall \textbf{e}\in\C^d, \quad W(z,j_0,j+1,\textbf{e})= M_j(z)W(z,j_0,j,\textbf{e}) -\begin{pmatrix}
			A_{j,q}^{-1}\delta_{j_0,j}\textbf{e}\\0\\\vdots\\0
		\end{pmatrix}
	\end{equation}
	We will now prove the following proposition using the extended geometric dichotomy (Lemma \ref{lemGeoDichStronger}).
	
	\begin{prop}[Bounds far from $1$]\label{GreenSpatialFar}
		For $U$ a bounded open set such that $\overline{U}\subset \Oc_\rho$, there exist two constants $C,c>0$ such that
		$$\forall z\in U, \forall \textbf{e}\in\C^d, \forall j,j_0\in \Z, \quad |W(z,j_0,j,\textbf{e})|\leq C|\textbf{e}|e^{-c|j-j_0|}.$$
	\end{prop}
	
	In particular, the result of Proposition \ref{GreenSpatialFar} holds true in a neighborhood of any point $z\in\overline{\U}\backslash\lc1\rc$.
	A direct consequence of Proposition \ref{GreenSpatialFar} on the spatial Green's function \eqref{defGreenSpatial} is that for any $U$ a bounded open set such that $\overline{U}\subset \Oc_\rho$, there exist two constants $C,c>0$ such that
	$$\forall z\in U, \forall j,j_0\in \Z, \quad |G(z,j_0,j)|\leq Ce^{-c|j-j_0|}.$$

	\begin{proof}\textbf{of Proposition \ref{GreenSpatialFar}}
		
		We consider $z\in U$, $j,j_0\in \Z$ and $\textbf{e}\in \C^d$ such that $|\textbf{e}|\leq 1$. The equality \eqref{systDynWj} implies the following results:
		\begin{itemize}
			\item We have
			$$\forall j \geq j_0+1, \quad W(z,j_0,j+1,\textbf{e})= M_j(z)W(z,j_0,j,\textbf{e}),$$
			i.e.
			\begin{equation}\label{WjR}
				\forall j \geq j_0+1, \quad W(z,j_0,j,\textbf{e})= X_j(z)X_{j_0+1}(z)^{-1}W(z,j_0,j_0+1,\textbf{e}).
			\end{equation}
			Also, since $G(z,j_0,\cdot)\textbf{e}\in \ell^2(\Z,\C^d)$, we have that $X_{j_0+1}(z)^{-1}W(z,j_0,j_0+1,\textbf{e})\in E_0^+(z)$.
			
			\item We have
			$$\forall j \leq j_0-1, \quad W(z,j_0,j+1,\textbf{e})= M_j(z)W(z,j_0,j,\textbf{e}),$$
			i.e.
			\begin{equation}\label{WjL}
				\forall j  \leq j_0, \quad W(z,j_0,j,\textbf{e})= X_j(z)X_{j_0}(z)^{-1}W(z,j_0,j_0,\textbf{e}).
			\end{equation}
			Also, since $G(z,j_0,\cdot)\textbf{e}\in \ell^2(\Z,\C^d)$, we have that $X_{j_0}(z)^{-1}W(z,j_0,j_0,\textbf{e})\in E_0^-(z)$.
			
			\item We have
			$$W(z,j_0,j_0+1,\textbf{e})= M_{j_0}(z)W(z,j_0,j_0,\textbf{e}) -\begin{pmatrix}
				A_{j_0,q}^{-1}\textbf{e}\\0\\\vdots\\0
			\end{pmatrix}$$
			i.e.
			$$X_{j_0+1}(z)^{-1}W(z,j_0,j_0+1,\textbf{e})- X_{j_0}(z)^{-1}W(z,j_0,j_0,\textbf{e})= -X_{j_0+1}(z)^{-1}\begin{pmatrix}
				A_{j_0,q}^{-1}\textbf{e}\\0\\\vdots\\0
			\end{pmatrix}.$$
		\end{itemize}
		
		Since $Q(z)$ is the projection on $E_0^+(z)$ with respect to $E_0^-(z)$, we have that
		$$\begin{array}{c}
			X_{j_0+1}(z)^{-1}W(z,j_0,j_0+1,\textbf{e})=-Q(z)X_{j_0+1}(z)^{-1}\begin{pmatrix}
				A_{j_0,q}^{-1}\textbf{e}\\0\\\vdots\\0
			\end{pmatrix},\\
			X_{j_0}(z)^{-1}W(z,j_0,j_0,\textbf{e})=(Id-Q(z))X_{j_0+1}(z)^{-1}\begin{pmatrix}
				A_{j_0,q}^{-1}\textbf{e}\\0\\\vdots\\0
			\end{pmatrix}.
		\end{array}$$
		Using \eqref{WjR} and \eqref{WjL}, we thus have the formulas:
		\begin{subequations}
			\begin{align}
				\forall j \geq j_0+1, \quad& W(z,j_0,j,\textbf{e})= -X_j(z)Q(z)X_{j_0+1}(z)^{-1}\begin{pmatrix}
					A_{j_0,q}^{-1}\textbf{e}\\0\\\vdots\\0
				\end{pmatrix}, \label{Wr}\\
				\forall j  \leq j_0, \quad& W(z,j_0,j,\textbf{e})= X_j(z)(Id-Q(z))X_{j_0+1}(z)^{-1}\begin{pmatrix}
					A_{j_0,q}^{-1}\textbf{e}\\0\\\vdots\\0
				\end{pmatrix}.\label{Wl}
			\end{align}
		\end{subequations}
		We now apply the inequalities \eqref{in_geo_dich_1} and \eqref{in_geo_dich_2} and obtain:
		$$\forall z \in U, \forall j,j_0\in\Z,\quad |W(z,j_0,j,\textbf{e})|\leq Ce^{-c|j-(j_0+1)|}\left|\begin{pmatrix}
			A_{j_0,q}^{-1}\textbf{e}\\0\\\vdots\\0
		\end{pmatrix}\right|.$$
		The result of Proposition \ref{GreenSpatialFar} follows.
	\end{proof}
	
	\section{Extension of the spatial Green's function near \texorpdfstring{$1$}{1}}\label{sec:GSnear1}

		The analysis of the spatial Green's function done in the previous section does not hold in a neighborhood of $1$. Indeed, the spectrum of the limit operators $\Lcc^\pm$ in \eqref{spec_Li} should belong to the essential spectrum of the operator $\Lcc$. This corresponds to the matrices $M^\pm(z)$ having central eigenvalues for some $z$ near $1$. Thus, the geometric dichotomy presented in the previous section cannot be applied immediately. To circumvent these issues, the strategy will be to refine the analysis of \eqref{syst_dyn} near $1$ by finding a particular basis of $E_0^\pm(z)$ and using this basis to express the spatial Green's function. In some sense, it amounts at using the projections on a specific basis of solutions of \eqref{syst_dyn} rather than the projection associated with the geometric dichotomy. This corresponds to adapt in a fully discrete setting the same strategy as in \cite{ZH,MasciaZumbrun,BenzoniHuotRousset,BeckHupkesSandstedeZumbrun} which tackle continuous or semi-discrete problems.

	\subsection{Right and left eigenvectors of \texorpdfstring{$M^\pm(z)$}{M+-(z)} for \texorpdfstring{$z$}{z} near \texorpdfstring{$1$}{1}}\label{subsec:EtudeSpectreMPrèsde1}
	
	To study the spatial Green's function for $z$ near $1$, we will need to study the solutions of the dynamical system \eqref{syst_dyn} with more accuracy. The first step is to have a better understanding of the eigenvalues and eigenvectors of $M^\pm(z)$ when $z$ is close to $1$.
	
	First, let us make some observations on the eigenvalues of $M_l^\pm(1)$ defined by \eqref{def:Mlpm} for $l\in\lc1,\ppp,d\rc$. Using Lemma \ref{lem:SpecSpl}, we know that the eigenvalues $\kappa\in\C\backslash\lc0\rc$ of $M_l^\pm(1)$ are the solutions of 
	$$\Fc_l^\pm(\kappa)=1.$$
	Hypothesis \ref{H:Mpm1} allows us to conclude that the matrix $M_l^\pm(1)$ only has simple eigenvalues. Furthermore, Lemma \ref{lem:SpecSpl} implies that $1$ is a simple eigenvalue of $M_l^\pm(1)$ and that the rest of the eigenvalues are in $\D\backslash\lc0\rc$ or $\U$ and we know the number of eigenvalues in each set depending on the sign of $\alpha_l^\pm$. Thus, we can define a family $(\uzet^\pm_m)_{m\in\lc1,\ppp,d(p+q)\rc}\in\C^{d(p+q)}$ such that 
	$$\forall l \in\lc1,\ppp,d\rc, \quad \sigma(M_l^\pm(1))=\lc\uzet_l^\pm,\uzet_{l+d}^\pm,\ppp,\uzet_{l+(p+q-1)d}^\pm\rc.$$
	Furthermore, using Hypothesis \ref{H:Lax} to determine the sign of $\alpha_l^\pm$ defined by \eqref{eg:FcFin} and Lemma \ref{lem:SpecSpl}, we can index them in order to have the following fact.
	\begin{itemize}
		\item For all $l\in\lc1, \ppp, I\rc$, since $\alpha_l^+<0$, we choose
		$$\uzet_l^+, \ppp, \uzet_{l+d(p-1)}^+ \in \D, \quad \uzet_{l+dp}^+=1, \quad \uzet_{l+d(p+1)}^+,\ppp,\uzet_{l+d(p+q-1)}^+\in\U.$$
		\item For all $l\in\lc I+1, \ppp,d\rc$, since $\alpha_l^+>0$, we choose
		$$\uzet_l^+, \ppp, \uzet_{l+d(p-2)}^+ \in \D, \quad \uzet_{l+d(p-1)}^+=1, \quad \uzet_{l+dp}^+,\ppp,\uzet_{l+d(p+q-1)}^+\in\U.$$
		\item For all $l\in\lc1, \ppp, I-1\rc$, since $\alpha_l^-<0$, we choose
		$$\uzet_l^-, \ppp, \uzet_{l+d(p-1)}^- \in \D, \quad \uzet_{l+dp}^-=1, \quad \uzet_{l+d(p+1)}^-,\ppp,\uzet_{l+d(p+q-1)}^-\in\U.$$
		\item For all $l\in\lc I, \ppp, d\rc$, since $\alpha_l^->0$, we choose
		$$\uzet_l^-, \ppp, \uzet_{l+d(p-2)}^- \in \D, \quad \uzet_{l+d(p-1)}^-=1, \quad \uzet_{l+dp}^-,\ppp,\uzet_{l+d(p+q-1)}^-\in\U.$$
	\end{itemize}
	
	We indexed the eigenvalues to separate the stable, central and unstable eigenvalues of the matrices $M^\pm_l(1)$. More precisely, we observe that if we introduce the sets
	\begin{equation}\label{def:Iss,cs,cu,su^pm}
		\begin{array}{ccc}
			I_{ss}^+:=\lc1,\ppp,d(p-1)+I\rc,& \quad & I_{ss}^-:=\lc1,\ppp,d(p-1)+I-1\rc,\\
			I_{cs}^+:=\lc d(p-1)+I+1,\ppp,dp\rc,& \quad & I_{cs}^-:=\lc d(p-1)+I,\ppp,dp\rc,\\
			I_{cu}^+:=\lc dp+1,\ppp, dp+I\rc,& \quad & I_{cu}^-:=\lc dp+1,\ppp , dp+I-1\rc,\\
			I_{su}^+:=\lc dp+I+1,\ppp, d(p+q)\rc, & \quad & I_{su}^-:=\lc dp+I,\ppp , d(p+q)\rc,
		\end{array}
	\end{equation}
	then we have that
	\begin{align}
		\begin{split}
			\forall m\in I_{ss}^\pm, &\quad \uzet_m^\pm\in \D,\\
			\forall m\in I_{cs}^\pm\cup I_{cu}^\pm, &\quad  \uzet_m^\pm=1,\\
			\forall m\in I_{su}^\pm, & \quad  \uzet_m^\pm\in \U.
		\end{split}\label{prop:uzet}
	\end{align}
	
	Since those are simple eigenvalues of $M_l^\pm(1)$, we are able to extend them holomorphically in a neighborhood of $1$. We consider $\delta_0>0$ a radius such that for each $m=l+(k-1)d\in\lc1,\ppp,d(p+q)\rc$ with $k\in\lc1,\ppp,p+q\rc$ and $l\in\lc1,\ppp,d\rc$, there exists a holomorphic function $\zeta_m^\pm:B(1,\delta_0)\rightarrow \C$ such that $\zeta_{m}^\pm(1)=\uzet_m^\pm$ and for all $z\in B(1,\delta_0)$, $\zeta_{m}^\pm(z)$ is a simple eigenvalue of $M_l^\pm(z)$. We will also separate the different types of eigenvalues by assuming that we chose $\delta_0$ small enough so that there exists a constant $c_*>0$ such that for all $z\in B(1,\delta_0)$
	\begin{subequations}\label{inZeta}
		\begin{align}
			\forall m\in I_{ss}^\pm, &\quad |\zeta_m^\pm(z)|\leq \exp(-2c_*), \label{inZetaSs} \\
			\forall m\in I_{cs}^\pm\cup I_{cu}^\pm, &\quad \exp(-c_*)\leq |\zeta_m^\pm(z)|\leq \exp(c_*) \label{inZetaC} \\
			\forall m\in I_{su}^\pm, & \quad  \exp(2c_*)\leq |\zeta_m^\pm(z)|. \label{inZetaSu}
		\end{align} 	
	\end{subequations}
	When we will study the temporal Green's function $\Gcc(n,j_0,j)$ later on in Section \ref{sec:GT}, we will have to bound terms of the form
	$$|\zeta_m^\pm(z)|^j|\zeta_{m^\prime}^\pm(z)|^{-j_0}.$$
	The inequalities \eqref{inZetaSs}-\eqref{inZetaSu} will allow us in a lot of cases to obtain exponential bounds for some of those terms. 
	
	Using Lemma \ref{decomp_mat}, we thus have a complete description of the eigenvalues of $M^\pm(z)$ for $z$ in a neighborhood of $1$. The following lemma also allows us to introduce a basis of eigenvectors for the matrices $M^\pm(z)$.
	
	\begin{lemma}\label{lemRm}
		For $m=l+(k-1)d\in\lc1,\ppp,d(p+q)\rc$ with $k\in\lc1,\ppp,p+q\rc$ and $l\in\lc1,\ppp,d\rc$, the vector
		\begin{equation}\label{def:Rm}
			R_m^\pm(z):=\begin{pmatrix}
				\zeta_m^\pm(z)^{q-1}\rg_l^\pm \\ \vdots \\ \zeta_m^\pm(z)^{-p}\rg_l^\pm
			\end{pmatrix}\in\C^{d(p+q)}
		\end{equation}
		is an eigenvector of $M^\pm(z)$ associated with the eigenvalue $\zeta_m^\pm(z)$. Furthermore, for all $z\in B(1,\delta_0)$, the family $(R_m^\pm(z))_{m\in\lc1,\ppp,d(p+q)\rc}$ is a basis of $\C^{d(p+q)}$. 
	\end{lemma}
	
	\begin{proof}
		Let us start by proving that the vector $R_m^\pm(z)$ defined by \eqref{def:Rm} is an eigenvector of $M^\pm(z)$ associated with the eigenvalue $\zeta_m^\pm(z)$. We have that $\zeta_m^\pm(z)$ is an eigenvalue of $M^\pm_l(z)$ so Lemma \ref{lem:SpecSpl} implies that 
		$$\Fc_l^\pm(\zeta_m^\pm(z))=z.$$
		We use the definition \eqref{def:Lambda} of the functions $\Lambda_{l,k}^\pm$ and the definition \eqref{def:Fcl} of the function $\Fc_l^\pm$ to prove that
		\begin{align*}
			-\sum_{k=-p}^{q-1}\A_q^\pm(z)^{-1}\A_k^\pm(z)\zeta_m^\pm(z)^{k}\rg_l^\pm &= -\sum_{k=-p}^{q-1}\Lambda_{l,q}^\pm(z)^{-1}\Lambda_{l,k}^\pm(z)\zeta_m^\pm(z)^{k}\rg_l^\pm\\
			&=\left(\zeta_m^\pm(z)^{q}+\Lambda_{l,q}^\pm(z)^{-1}\left(z-\Fc_l^\pm(\zeta_m^\pm(z))\right)\right) \rg_l^\pm \\
			& =\zeta_m^\pm(z)^{q} \rg_l^\pm.
		\end{align*}
		This allows us to conclude that the vector $R_m^\pm(z)$ is an eigenvector of $M^\pm(z)$ associated with the eigenvalue $\zeta_m^\pm(z)$.
		
		We now consider $z\in B(1,\delta_0)$ and a family of complex numbers $(\lambda_m)_{m\in\lc1,\ppp,d(p+q)\rc}$ such that:
		$$0=\sum_{m=1}^{d(p+q)}\lambda_mR_m^\pm(z).$$
		Separating the blocks of coefficients of size $d$ in the previous equality and observing that the family $(\rg_l^\pm)_{l\in\lc1,\ppp,d\rc}$ is linearly independent, we have for all $l\in \lc1,\ppp,d\rc$:
		$$\forall j\in\lc-p,\ppp,q-1\rc, \quad 0=\sum_{k=1}^{p+q}\lambda_{l+(k-1)d}\zeta_{l+(k-1)d}^\pm(z)^j.$$
		We have that, for each integer $k\in\lc1,\ppp,p+q\rc$, $\zeta_{l+(k-1)d}^\pm(z)$ is a simple eigenvalue of $M_l^\pm(z)$ for all $z\in B(1,\delta_0 )$. Therefore, the complex values $(\zeta_{l+(k-1)d}^\pm(z))_{k\in\lc1,\ppp,p+q\rc}$ are distinct and thus
		$$\forall k\in \lc1,\ppp,p+q\rc,\quad \lambda_{l+d(k-1)}=0.$$
		Since this is true for all $l\in \lc1,\ppp,d\rc$, we proved that the family $(R_m^\pm(z))_{m\in\lc1,\ppp,d(p+q)\rc}$ is linearly independent and is thus a basis of $\C^{d(p+q)}$.
	\end{proof}
	
	Thus, we have a characterization of the eigenvalues and eigenvectors of $M^\pm(z)$ for $z\in B(1,\delta_0)$. Lemma \ref{lem:SpecSpl} implies that, for all $z\in \Oc\cap B(1,\delta_0)$, we have that $|\zeta_m^\pm(z)|<1$ for $m\in \lc1,\ppp,dp\rc$ and $|\zeta_m^\pm(z)|>1$ for $m\in \lc dp+1,\ppp,d(p+q)\rc$. Thus, for $z\in \Oc\cap B(1,\delta_0)$
	$$E^s(M^\pm(z))=\mathrm{Span}\lc R^\pm_m(z), \quad m\in\lc1,\ppp,dp\rc\rc$$
	and
	$$E^u(M^\pm(z))=\mathrm{Span}\lc R^\pm_m(z), \quad m\in\lc dp+1,\ppp,d(p+q)\rc\rc.$$
	This equality implies that we can extend holomorphically the definitions of $E^s(M^\pm(z))$ and $E^u(M^\pm(z))$ for $z\in B(1,\delta_0)$.
	
	We now conclude this section by studying the dual basis associated with the basis $(R_m^\pm(z))_{m\in\lc1,\ppp,d(p+q)\rc}$. We introduce the invertible matrix
	\begin{equation}\label{def:Ninf}
		\forall z\in B(1,\delta_0),\quad N^{\pm,\infty}(z):= \begin{pmatrix}
			R_1^\pm(z) &|\ppp| & R_{d(p+q)}^\pm(z)
		\end{pmatrix}\in\Mc_{d(p+q)}(\C)
	\end{equation}
	and the vectors $L_1^\pm(z),\ppp,L_{d(p+q)}^\pm(z)\in\C^{d(p+q)}$ defined by
	\begin{equation}\label{def:Lm}
		\forall z\in B(1,\delta_0),\quad \begin{pmatrix}
			L_1^\pm(z) &|\ppp| & L_{d(p+q)}^\pm(z)
		\end{pmatrix}^T:= N^{\pm,\infty}(z)^{-1}.
	\end{equation}
	We observe that
	\begin{equation}\label{eg:LmRm}
		\forall z\in B(1,\delta_0),\forall m,\widetilde{m}\in\lc1,\ppp,d(p+q)\rc, \quad L_m^\pm(z)^TR_{\widetilde{m}}^\pm(z)=\delta_{m,\widetilde{m}}
	\end{equation}
	and
	\begin{equation}\label{eg:LmVp}
		\forall z\in B(1,\delta_0),\forall m\in\lc1,\ppp,d(p+q)\rc, \quad M^\pm(z)^TL_m^\pm(z)=\zeta_m^\pm(z)L_m^\pm(z).
	\end{equation}
	
	We will now prove the following lemma which gives a more precise description of the vectors $L_m^\pm(z)$ of the dual basis.
	
	\begin{lemma}\label{lem:dualBasis}
		We consider $m=l+(k-1)d\in\lc1,\ppp,d(p+q)\rc$ with $k\in\lc1,\ppp,p+q\rc$ and $l\in\lc1,\ppp,d\rc$. For all $z\in\C$, there exist coefficients $x_1^\pm(z),\ppp,x_{p+q}^\pm(z)\in\C$ such that
		\begin{equation}\label{exp:Lm}
			L_m^\pm(z):=\begin{pmatrix}
				x_1^\pm(z)\lg_l^\pm \\ \vdots \\ x_{p+q}^\pm(z)\lg_l^\pm
			\end{pmatrix}.
		\end{equation}
		Furthermore, we have
		\begin{equation}\label{eg:x1}
			\forall z\in B(1,\delta_0),\quad x_1^\pm(z) =\lambda_{l,q}^\pm \frac{d \zeta_m^\pm}{dz}(z)
		\end{equation}
		where $\lambda_{l,q}^\pm$ is defined by \eqref{def:lambda_k}.
	\end{lemma}
	
	In the proof of Lemma \ref{lem:dualBasis}, we also find the expressions of the coefficients $x_2^\pm(z),\ppp,x_{p+q}^\pm(z)$ but, contrarily to $x_1^\pm(z)$, they will not be used later on in the paper.
	
	\begin{proof}
		The proof of Lemma \ref{lem:dualBasis} uses calculations similar to those done at the end of \cite[Lemma 2.4]{CoeuLLT}. We consider $z\in B(1,\delta_0)$. We begin by introducing the vectors $\xg_1^\pm(z),\ppp,\xg_{p+q}^\pm(z)\in\C^d$ defined by
		$$\begin{pmatrix}
			\xg_1^\pm(z) \\ \vdots \\ \xg_{p+q}^\pm(z)
		\end{pmatrix}:=L_m^\pm(z).$$
		We consider $\widetilde{l}\in\lc1,\ppp,d\rc\backslash\lc l\rc$. Using the definition \eqref{def:Rm} and the linear independence of the vectors $R_{\widetilde{m}}^\pm(z)$, we have that
		$$\mathrm{Span}\lc R_{\widetilde{l}+(\widetilde{k}-1)d}^\pm(z),\quad \widetilde{k}\in\lc1,\ppp,p+q\rc\rc= \lc\begin{pmatrix}
		y_1\rg_{\widetilde{l}}^\pm\\ \vdots \\ y_{p+q}\rg_{\widetilde{l}}^\pm
		\end{pmatrix},\quad y_1,\ppp,y_{p+q}\in\C\rc.$$
		Using \eqref{eg:LmRm}, we can then prove that:
		$$\forall y_1,\ppp,y_{p+q}\in\C,\quad \sum_{j=1}^{p+q} y_j\xg_j^\pm(z)^T\rg_{\tilde{l}}^\pm=0$$
		and thus:
		$$\forall j\in\lc1,\ppp,p+q\rc,\quad  \xg_j^\pm(z)^T\rg_{\widetilde{l}}^\pm=0.$$
		Since this is true for all $\widetilde{l}\in\lc1,\ppp,d\rc\backslash\lc l\rc$, we have that $\xg_j^\pm(z)\in \mathrm{Span}\; \lg_l^\pm$ for all $j\in\lc1,\ppp,p+q\rc$. 
		
		Now that we know that we can express the vector $L_m^\pm(z)$ as \eqref{exp:Lm}, let us prove \eqref{eg:x1}. Using \eqref{def:lambda_k} and the definitions \eqref{def:l} and \eqref{def:Lambda} respectively of the vectors $\lg_l^\pm$ and of the functions $\Lambda_{l,k}^\pm$, we have looking at the $j$-th block of size $d$ of \eqref{eg:LmVp} that:
		$$\forall j\in\lc1,\ppp,p+q-1\rc,\quad \zeta_m^\pm(z) x_j^\pm(z) = x_{j+1}^\pm(z)-\Lambda^\pm_{l,q}(z)^{-1}\Lambda_{l,q-j}^\pm(z)x_1^\pm(z),$$
		and:
		$$\zeta_m^\pm(z) x_{p+q}^\pm(z) = -\Lambda^\pm_{l,q}(z)^{-1}\Lambda_{l,-p}^\pm(z)x_1^\pm(z).$$
		Thus, we have:
		$$\forall j\in\lc1,\ppp,p+q\rc,\quad x_j^\pm(z) = -\left(\sum_{k=-p}^{q-j} \frac{\Lambda_{l,k}^\pm(z)}{ \zeta_m^\pm(z)^{q-j-k+1}}\right) \Lambda^\pm_{l,q}(z)^{-1}x_1^\pm(z).$$
		
		We now have an expression of each $x_j^\pm(z)$ depending on $x_1^\pm(z)$. We also recall that 
		$${\lg_l^\pm}^T\rg_l^\pm=1.$$
		Using the expressions \eqref{def:Rm} and \eqref{exp:Lm} respectively of the vectors $R_m^\pm(z)$ and $L_m^\pm(z)$ as well as \eqref{eg:LmRm}, we have
		$$1 = L_m^\pm(z)^TR_m^\pm(z) = \sum_{j=1}^{p+q}x_j^\pm(z)\zeta_m^\pm(z)^{q-j} = -\left(\sum_{j=1}^{p+q}\sum_{k=-p}^{q-j} \Lambda_{l,k}^\pm(z)\zeta_m^\pm(z)^{k-1} \right)\Lambda^\pm_{l,q}(z)^{-1}x_1^\pm(z).$$
		Using the definitions \eqref{def:Lambda} and \eqref{def:Fcl} of the functions $\Lambda_{l,k}^\pm$ and $\Fc_l^\pm$, we have
		\begin{align*}
			1 & = - \left(\sum_{k=-p}^{q} (q-k)\Lambda_{l,k}^\pm(z)\zeta_m^\pm(z)^{k-1} \right)\Lambda^\pm_{l,q}(z)^{-1}x_1^\pm(z) \\
			& = - \left(q\zeta_m^\pm(z)^{-1}z - \sum_{k=-p}^{q} (q-k)\lambda_{l,k}^\pm\zeta_m^\pm(z)^{k-1} \right)\Lambda^\pm_{l,q}(z)^{-1}x_1^\pm(z)\\
			& = -\left( q\zeta_m^\pm(z)^{-1}(z-\Fc_l^\pm(\zeta_m^\pm(z)))+ \frac{d\Fc_l^\pm}{d\kappa}(\zeta_m^\pm(z)) \right)\Lambda^\pm_{l,q}(z)^{-1}x_1^\pm(z).
		\end{align*}
		We observe that since $\zeta_m^\pm(z)$ is an eigenvalue $M_l^\pm(z)$, Lemma \ref{lem:SpecSpl} allows us to prove that
		$$\Fc_l^\pm(\zeta_m^\pm(z))=z \quad \text{ and } \quad  \frac{d\zeta_m^\pm}{dz}(z)\frac{d\Fc_l^\pm}{d\kappa}(\zeta_m^\pm(z))=1. $$
		Thus, since $\Lambda_{l,q}^\pm(z) = -\lambda_{l,q}^\pm$, we have that 
		$$ 1= \left(\frac{d\zeta_m^\pm}{dz}(z){\lambda_{l,q}^\pm}\right)^{-1} x_1^\pm(z)$$ 
		and we deduce \eqref{eg:x1}.
	\end{proof}
	
	\subsection{Choice of a precise basis of \texorpdfstring{$E_0^\pm(z)$}{E0+-(z)} for \texorpdfstring{$z$}{z} near \texorpdfstring{$1$}{1}}
	
	Now that we have a better understanding of the spectrum of $M^\pm(z)$, we are going to prove a lemma that is quite similar to the geometric dichotomy. This lemma corresponds to \cite[Lemma 3.1]{Godillon}, itself inspired by \cite[Proposition 3.1]{ZH}.
	
	\begin{lemma}\label{lem_choice_base}
		There exist a radius $\delta_1\in]0,\delta_0]$ and two constants $C,c>0$ such that for all $m\in\lc 1,\ppp, d(p+q) \rc$ and $z\in B(1,\delta_1)$, there exists a sequence $(V_m^\pm(z,j))_{j\in\Z}\in{\C^{(p+q)d}}^\Z$ such that :
		\begin{itemize}
			\item For all $j\in\Z$, the function $V^\pm_m(\cdot,j)$ is holomorphic on $B(1,\delta_1)$.
			\item The functions $z\in B(1,\delta_1)\mapsto (V_m^+(z,j))_{j\in\N}\in \ell^\infty(\N,\C^{d(p+q)})$ and $z\in B(1,\delta_1)\mapsto (V_m^-(z,j))_{j\in-\N}\in \ell^\infty(-\N,\C^{d(p+q)})$ are holomorphic. Furthermore, up to considering a smaller radius $\delta_1$, those functions and their derivatives are bounded on $B(1,\delta_1)$.
			\item For $z\in B(1,\delta_1)$, if we define $W_m^\pm(z,j):= \zeta_m^\pm(z)^jV_m^\pm(z,j)$ for all $j\in \Z$, then $W_m^\pm(z,\cdot)$ is a solution of \eqref{syst_dyn}, i.e.
			$$\forall j\in\Z,\quad  W_m^\pm(z,j+1)=M_j(z)W_m^\pm(z,j).$$
			\item We have 
			$$\forall z\in B(1,\delta_1),  \quad \begin{array}{ccc}
				\forall j\in \N, &\left| V_m^+(z, j) - R_m^+(z)\right| &\leq Ce^{-c|j|},\\
				\forall j\in -\N, &\left| V_m^-(z,j) - R_m^-(z)\right| &\leq C e^{-c|j|}.
			\end{array}$$
		\end{itemize}
	\end{lemma}
	
	The proof of this lemma is quite similar to the construction of $Q^\pm_U(z)$ in Lemma \ref{lem_geo_dich} and is fairly based on the proof of \cite[Proposition 3.1]{ZH}.
	
	\begin{proof}
		We will focus on the construction of $(V_m^+(z,j))_{j\in\Z}$ for an integer $m\in\lc 1,\ppp, d(p+q) \rc$. Because of \eqref{CV_expo_mat}, we have a constant $C>0$ such that 
		$$\forall j\in \N, \forall z\in B(1,\delta_0), \quad |\Ec_j^+(z)|\leq Ce^{-\alpha j}.$$
		We fix $\Delta:=\frac{\alpha}{4}$ and define the sets
		$$\begin{array}{c}
			I^s_m=\lc \nu\in \lc1,\ppp,d(p+q)\rc, \quad |\zeta_{\nu}^+(1)|<|\zeta_m^+(1)|e^{-\Delta}\rc,\\
			I^u_m=\lc \nu\in \lc1,\ppp,d(p+q)\rc, \quad |\zeta_{\nu}^+(1)|\geq|\zeta_m^+(1)|e^{-\Delta}\rc.
		\end{array}$$
		Because the functions $\zeta_{\nu}^+$ depend holomorphicaly on $z$ in $B(1,\delta_0)$, there exists $\delta_1\in]0,\delta_0[$ such that
		\begin{equation}\label{borne_vp}
			\forall z\in B(1,\delta_1), \quad \begin{array}{cc}\forall \nu\in I^s_m,& |\zeta_{\nu}^+(z)|<|\zeta_m^+(z)|e^{-\Delta}, \\ 
				\forall \nu\in I^u_m,& |\zeta_{\nu}^+(z)|>|\zeta_m^+(z)|e^{-\frac{3}{2}\Delta}. \end{array}
		\end{equation}
		
		We define for $z\in B(1,\delta_1)$
		$$\begin{array}{c}
			E^s_m(z) := Span\left(R_{\nu}^+(z), \quad \nu\in I^s_m\right),\\
			E^u_m(z) := Span\left(R_{\nu}^+(z), \quad \nu\in I^u_m\right).
		\end{array}$$
		We have that
		$$\C^{(p+q)d}=E^s_m(z)\oplus E^u_m(z).$$
		We define $P^s_m(z)$ and $P^u_m(z)$ the projectors defined by this decomposition of $\C^{(p+q)d}$. They depend holomorphically on $z$ and commute with $M^+(z)$. Because of \eqref{borne_vp}, there exists a constant $C>0$ such that 
		\begin{equation}\label{borne_expo}
			\forall z \in B(1,\delta_1), \forall j \in \N, \quad  \begin{array}{c}
				\left|\left(\zeta_m^+(z)^{-1}M^+(z)\right)^jP^s_m(z)\right|\leq C\exp\left(-\frac{\Delta}{2}j\right),\\
				\left|\left(\zeta_m^+(z)^{-1}M^+(z)\right)^{-j}P^u_m(z)\right|\leq C\exp\left(2\Delta j\right).
		\end{array}\end{equation}
		
		We consider $J\in\N$ and we will make a more precise choice later. For $z\in B(1,\delta_1)$, we define the linear map $\varphi(z) \in \Lc\left(\ell^\infty\left(\lc j\in\Z,  j\geq J\rc,\C^{(p+q)d}\right)\right)$ such that for $Y\in \ell^\infty\left(\lc j\in\Z,  j\geq J\rc,\C^{(p+q)d}\right)$ and $j\geq J$, we have:
		\begin{align}\label{eg:varphi}
			\begin{split}
				\left(\varphi(z)Y\right)_j:=& \sum_{k=J}^{j-1} \left(\zeta_m^+(z)^{-1}M^+(z)\right)^{j-1-k}P^s_m(z)\zeta_m^+(z)^{-1}\Ec_k^+(z)Y_k \\ &-\sum_{k=j}^{+\infty} \left(\zeta_m^+(z)^{-1}M^+(z)\right)^{j-1-k}P^u_m(z)\zeta_m^+(z)^{-1}\Ec_k^+(z)Y_k .
			\end{split}
		\end{align}
		
		Using the inequalities \eqref{borne_expo}, we have that:
		\begin{align}\label{in:lemJost1}
			\begin{split}
				\sum_{k=J}^{j-1} \left|\left(\zeta_m^+(z)^{-1}M^+(z)\right)^{j-1-k}P^s_m(z)\zeta_m^+(z)^{-1}\Ec_k^+(z)Y_k \right|
				&\lesssim  \left\|Y\right\|_\infty \sum_{k=J}^{j-1}e^{-\frac{\Delta}{2}(j-k)}e^{-\alpha k}\\
				&\lesssim  \left\|Y\right\|_\infty e^{-\frac{\Delta}{2}j} \sum_{k=J}^{+\infty}e^{\left(\frac{\Delta}{2}-\alpha\right) k}\\
				&\lesssim  \left\|Y\right\|_\infty e^{-\frac{\Delta}{2}(j-J)} e^{-\alpha J}
			\end{split}
		\end{align}
		and :
		\begin{align}\label{in:lemJost2}
			\begin{split}
				\sum_{k=j}^{+\infty} \left|\left(\zeta_m^+(z)^{-1}M^+(z)\right)^{j-1-k}P^u_m(z)\zeta_m^+(z)^{-1}\Ec_k^+(z)Y_k \right|
				&\lesssim  \left\|Y\right\|_\infty e^{-\alpha j} \sum_{k=j}^{+\infty}e^{(2\Delta-\alpha)(k-j)}\\
				&\lesssim  \left\|Y\right\|_\infty e^{-\alpha j}.
			\end{split}
		\end{align}
		We have thus proved that the linear map $\varphi(z)$ is well-defined and that there exists a constant $C>0$ independent from $J$ such that
		$$\forall z\in B(1,\delta_1),\quad \left\| \varphi(z)\right\|_{\Lc\left(\ell^\infty\right)}\leq Ce^{-\alpha J}.$$		
		We can then choose $J$ large enough so that there exists a constant $C\in]0,1[$ such that
		$$\forall z\in B(1,\delta_1),\quad \left\|\varphi(z)\right\|_{\Lc\left(\ell^\infty\right)}\leq C<1.$$
		Furthermore, $\varphi$ depends holomorphically on $z$. We can thus define for $z\in B(1,\delta_1)$
		$$V(z):= (Id-\varphi(z))^{-1}\left(R_m^+(z)\right)_{j\geq J}\in \ell^\infty\left(\lc j\in\Z,  j\geq J\rc,\C^{(p+q)d}\right)$$
		which depends holomorphically on $z$. We have that: 
		$$\forall z\in B(1,\delta_1), \quad V(z) = \left(R_m^+(z)\right)_{j\geq J} +\varphi(z) V(z).$$
		Thus, for $ z\in B(1,\delta_1)$ and $j\geq J$, we find that:
		\begin{align*}
			V_{j+1}(z) &= R_m^+(z) +\sum_{k=J}^j  \left(\zeta_m^+(z)^{-1}M^+(z)\right)^{j-k}P^s_m(z)\zeta_m^+(z)^{-1}\Ec_k^+(z)V_k(z) \\&- \sum_{k=j+1}^{+\infty} \left(\zeta_m^+(z)^{-1}M^+(z)\right)^{j-k}P^u_m(z)\zeta_m^+(z)^{-1}\Ec_k^+(z)V_k(z)\\
			&= \zeta_m^+(z)^{-1}M^+(z)\left(R_m^+(z) +\sum_{k=J}^j  \left(\zeta_m^+(z)^{-1}M^+(z)\right)^{j-1-k}P^s_m(z)\zeta_m^+(z)^{-1}\Ec_k^+(z)V_k(z) \right.\\&\left.- \sum_{k=j+1}^{+\infty} \left(\zeta_m^+(z)^{-1}M^+(z)\right)^{j-1-k}P^u_m(z)\zeta_m^+(z)^{-1}\Ec_k^+(z)V_k(z)\right)\\
			&= \zeta_m^+(z)^{-1}M^+(z)\left( R_m^+(z) +(\varphi(z)V(z))_j +\left(\zeta_m^+(z)^{-1}M^+(z)\right)^{-1}\zeta_m^+(z)^{-1} \Ec_j^+(z) V_j(z) \right)\\
			&=  \zeta_m^+(z)^{-1}M^+(z)V_j(z)+ \zeta_m^+(z)^{-1}\Ec_j^+(z)V_j(z) \\
			& = \zeta_m^+(z)^{-1}M_j(z)V_j(z).
		\end{align*}
		Thus, for $j\geq J$, we have
		$$V_j(z)=\zeta_m^+(z)^{J-j}X_j(z)X_{J}(z)^{-1}V_J(z).$$
	We define for $z\in B(1,\delta_1)$ and $j\in \Z$
		$$V_m^+(z,j):= \zeta_m^+(z)^{J-j}X_j(z)X_{J}(z)^{-1}V_J(z) $$
		and
		$$W_m^+(z,j):= \zeta_m^+(z)^jV_m^+(z,j).$$
		The two first points of the statement of Lemma \ref{lem_choice_base} are easily proved from the previous observations. There remains to prove the inequalities in the third point of Lemma \ref{lem_choice_base}. For $z\in B(1,\delta_1)$ and $j\geq J$, we have using \eqref{eg:varphi}-\eqref{in:lemJost2}
		$$\left|V_m^+(z,j)-R_m^+(z)\right| = |(\varphi(z)V(z))_j|\lesssim e^{-\frac{\Delta}{2}j}+e^{-\alpha j}.$$
	\end{proof}
	
	We recall that for $z\in \Oc_\rho\cap B(1,\delta_1)$, we have for $m\in\lc1,\ppp,dp\rc$ that 
	$$|\zeta_m^+(z)|<1.$$
	Therefore, $(W_m^+(z,0))_{m\in \lc1,\ppp,dp\rc}$ is a family of elements of $E_0^+(z)=\Im Q(z)$ for $z\in \Oc_\rho\cap B(1,\delta_1)$. In the same way, we prove that for all $z\in \Oc_\rho\cap B(1,\delta_1)$, $(W_m^-(z,0))_{m\in \lc dp+1,\ppp,d(p+q)\rc}$ is a family of elements of $E_0^-(z)=\ker Q(z)$. We are going to prove the following lemma.
	
	\begin{lemma}\label{lemBasis}
		For all $z\in B(1,\delta_1)$ and $j\in\Z$, $(W_m^+(z,j))_{m\in\lc1\ppp,d(p+q)\rc}$ and $(W_m^-(z,j))_{m\in\lc1\ppp,d(p+q)\rc}$ are bases of $\C^{d(p+q)}$. The same is then also true for the families $(V_m^+(z,j))_{m\in\lc1\ppp,d(p+q)\rc}$ and $(V_m^-(z,j))_{m\in\lc1\ppp,d(p+q)\rc}$.
	\end{lemma}
	
	\begin{proof}
		We will write the proof for the family of vectors $(W_m^+(z,j))_{m\in\lc1\ppp,d(p+q)\rc}$. We consider $z\in B(1,\delta_1)$ and $j\in\Z$ such that the family of vectors $(W_m^+(z,j))_{m\in\lc1,\ppp,d(p+q)\rc}$ is not linearly independent. We can then introduce a family $(c_m)_{m\in\lc1,\ppp,d(p+q)\rc}\in\C^{d(p+q)}\backslash\lc0\rc$ such that
		\begin{equation}\label{lemBasis:eg}
			0=\sum_{m=1}^{d(p+q)} c_mW_m^+(z,j).
		\end{equation}
		Since the sequences $(W_m^+(z,j))_{j\in\Z}$ are solutions of \eqref{syst_dyn}, \eqref{lemBasis:eg} is verified for all $j\in\Z$. We define 
		\begin{align*}
			I_n&:=\lc m\in\lc1,\ppp,d(p+q)\rc, \quad c_m\neq 0\rc\neq \emptyset,\\
			R&:= \max_{m\in I_n} |\zeta_m^+(z)|>0,\\
			I_R&:= \text{argmax}_{m\in I_n} |\zeta_m^+(z)|\neq \emptyset.
		\end{align*} 
		Using \eqref{lemBasis:eg}, we obtain
		\begin{align*}
			0&= \sum_{m\in I_n}c_m \frac{W_m^+(z,j)}{R^j}\\
			&= \sum_{m\in I_n}c_m \left(\frac{\zeta_m^+(z)}{R}\right)^jR_m^+(z)+\sum_{m\in I_n}c_m \frac{W_m^+(z,j)-\zeta_m^+(z)^jR_m^+(z)}{R^j}.
		\end{align*}
		\begin{itemize}
			\item Using Lemma \ref{lem_choice_base}, there exist two positive constants $C,c$ such that we have for $m\in I_n$ and $j\in\N$
			$$\left|\frac{W_m^+(z,j)-\zeta_m^+(z)^jR_m^+(z)}{R^j}\right|\leq Ce^{-cj}\left(\frac{|\zeta_m^+(z)|}{R}\right)^j\leq Ce^{-cj}.$$
			Thus,
			$$\frac{W_m^+(z,j)-\zeta_m^+(z)^jR_m^+(z)}{R^j}\underset{j\rightarrow +\infty}\rightarrow 0.$$
			\item For $m\in I_n\backslash I_R$, we have 
			$$\left(\frac{\zeta_m^+(z)}{R}\right)^jR_m^+(z)\underset{j\rightarrow +\infty}\rightarrow 0.$$
		\end{itemize}
		Thus, we have that
		$$\sum_{m\in I_R}c_m\left(\frac{\zeta_m^+(z)}{R}\right)^jR_m^+(z)\underset{j\rightarrow+\infty}\rightarrow0.$$
		Since $I_R\neq \emptyset$, we fix $m_0\in I_R$. Because of Lemma \ref{lemRm}, the projection of the previous expression on $\mathrm{Span}(R_{m_0}^+(z))$ along $\mathrm{Span}(R_m^+(z), m\neq m_0)$ implies that
		$$c_{m_0}\left(\frac{\zeta_{m_0}^+(z)}{R}\right)^j\underset{j\rightarrow+\infty}\rightarrow0.$$
		But, $m_0$ belongs to $I_R$ so $|\zeta_{m_0}^+(z)|=R$. This implies that $c_{m_0}=0$. However, $m_0\in I_R\subset I_n$ so $c_{m_0}\neq 0$. This is a contradiction.
	\end{proof}
	
	For $z\in \Oc_\rho\cap B(1,\delta_1)$, we recall that 
	$$\dim E_0^+(z) =dp \quad \text{ and } \quad \dim E_0^-(z) =dq.$$
	Thus, Lemma \ref{lemBasis} implies that the family $\left(W_m^+(z,0)\right)_{m\in\lc1,\ppp,dp\rc}$ (resp. $\left(W_m^-(z,0)\right)_{m\in\lc dp+ 1,\ppp,d(p+q)\rc}$) is a basis of $E^+_0(z)$ (resp. $E_0^-(z)$). We can then extend holomorphically the subspaces $E_0^+(z)$ and $E_0^-(z)$ on the whole ball $B(1,\delta_1)$ as
	\begin{equation}\label{eg:E_0^+E_0^-}
		\forall z\in B(1,\delta_1), \quad E_0^+(z):=\mathrm{Span}\left(W_m^+(z,0)\right)_{m\in\lc1,\ppp,dp\rc} \quad \text{ and }\quad E_0^-(z):=\mathrm{Span}\left(W_m^-(z,0)\right)_{m\in\lc dp+1,\ppp,d(p+q)\rc}.
	\end{equation}
	We will also define
	\begin{equation}\label{def:Iss,cs,cu,su}
		\begin{array}{ccc}
			I_{ss}:=I_{ss}^+, & \quad  &I_{cs}:=I_{cs}^+,\\
			I_{cu}:=I_{cu}^-, & \quad  &I_{su}:=I_{su}^-.
		\end{array}
	\end{equation}
	
	\subsection{Definition of the Evans function}
	
	In this section, we are going to define an Evans function and prove that the eigenspace of the operator $\Lcc$ associated with $1$ is of dimension $1$ when it acts on $\ell^2(\Z,\C^d)$. 
	
	An important part of the study of the spatial Green's function far from $1$ was dedicated to introduce the projection $Q$ of the geometric dichotomy. The main ingredient of the introduction of $Q(z)$ has been to understand when $E_0^+(z)$ and $E_0^-(z)$ are supplementary, as it allowed via Lemma \ref{lem_spec_ess} to conclude on which elements of $\Oc$ where eigenvalues of the operator $\Lcc$ and which elements of $\Oc$ are in the resolvent set of the operator $\Lcc$. For $z$ near $1$, because of \eqref{eg:E_0^+E_0^-}, studying whether the vector subspace $E_0^+(z)$ and $E_0^-(z)$ are supplementary comes down to knowing when $\left(W^+_1(z,0),\ppp,W_{dp}^+(z,0),W_{dp+1}^-(z,0),\ppp,W_{d(p+q)}^-(z,0)\right)$ is a basis of $\C^{d(p+q)}$. We define the Evans function as
	\begin{equation}\label{def:Ev}
		\forall z\in B(1,\delta_1),\quad \mathrm{Ev}(z) := \det(W_1^+(z,0),\ppp,W_{dp}^+(z,0),W_{dp+1}^-(z,0),\ppp,W_{d(p+q)}^-(z,0)).
	\end{equation}
	The function $\mathrm{Ev}$ is holomorphic on $B(1,\delta_1)$. Furthermore, for $z\in\Oc\cap B(1,\delta_1)$, \eqref{dimEigSpaceLcc} and the set equality \eqref{eg:E_0^+E_0^-} imply that the function $\mathrm{Ev}$ vanishes when $z$ is an eigenvalue of the operator $\Lcc$. Thus, Hypothesis \ref{H:spec} implies that the function $\mathrm{Ev}$ is not uniformly equal to $0$. We will now prove the following lemma which links the behavior at $z=1$ of the Evans function $\mathrm{Ev}$, the eigenspace associated with the eigenvalue $1$ for the operator $\Lcc$ and the vector subspace
	$$\mathrm{Span}\left(W_m^+(1,0),m\in I_{ss}\right)\cap\mathrm{Span}\left(W_m^-(1,0),m\in I_{su}\right).$$

		\begin{lemma}\label{lemCombiWm}
			\begin{subequations}
				We have that $\mathrm{Ev}(1)=0$. Furthermore, if we assume that $1$ is a simple zero of the Evans function (see Hypothesis \ref{H:Evans}), then we have that:
				\begin{align}
					\dim \mathrm{Span}\left(W_m^+(1,0),m\in I_{ss}\right)\cap\mathrm{Span}\left(W_m^-(1,0),m\in I_{su}\right)&=1,\label{lemCombiWm:egdim}\\
					\dim \ker(Id_{\ell^2}-\Lcc)&=1.\label{lemCombiWm:egdim2}
				\end{align}
				Moreover, if we consider a vector $V_0\in \mathrm{Span}\left(W_m^+(1,0),m\in I_{ss}\right)\cap\mathrm{Span}\left(W_m^-(1,0),m\in I_{su}\right)\backslash\lc0\rc$, then:
				\begin{equation}\label{lemCombiWm:eg}
					\ker(Id_{\ell^2}-\Lcc) =\mathrm{Span}(\Pi(X_j(1)V_0))_{j\in\Z}
				\end{equation}
				where the operator $\Pi$ defined by \eqref{def:Pi} is the linear map which extracts the center values of a vector of size $d(p+q)$.
			\end{subequations}
		\end{lemma}

		\begin{proof}
			The proof is separated in a few steps.
			
			\underline{\textbf{Step 1:}} Let us start by proving that:
			$$\dim \mathrm{Span}\left(W_m^+(1,0),m\in I_{ss}\right)\cap\mathrm{Span}\left(W_m^-(1,0),m\in I_{su}\right)\geq 1.$$
			
			For $m\in\lbrace1,\hdots,d(p+q)\rbrace$, $z\in B(1,\delta_1)$ and $j \in\Z$, we define the vectors ${W_m^\pm}^{(-p)}(z,j),\hdots,{W_m^\pm}^{(q-1)}(z,j)\in\C^d$ such that:
			$$W^\pm_m(z,j)=:\begin{pmatrix}
				{W^\pm_m}^{(q-1)}(z,j)\\ \vdots \\ {W^\pm_m}^{(-p)}(z,j)
			\end{pmatrix}.$$
			We also define:
			\begin{equation}\label{def:wmpm}
				w_m^\pm(z,j):=\Pi(W_m^\pm(z,j))={W_m^\pm}^{(0)}(z,j)\in\C^d.
			\end{equation}
			We recall that the sequence $W_m^\pm(z,\cdot)$ is a solution of the dynamical system \eqref{syst_dyn} and thus:
			$$\forall j \in\Z,\quad W_m^\pm(z,j+1)= M_j(z)W_m^\pm(z,j).$$
			Using the definition \eqref{def:Mj} of the matrix $M_j(z)$, we have:
			\begin{subequations}
				\begin{align}
					\forall j \in\Z,\forall k \in\lbrace-p,\hdots,q-2\rbrace,\quad & {W_m^\pm}^{(k)}(z,j+1)={W_m^\pm}^{(k+1)}(z,j),\label{eg:lemCombiWm1}\\ 
					\forall j\in\Z,\quad & {W_m^\pm}^{(q-1)}(z,j+1) = -\A_{j,q}(z)^{-1}\left(\sum_{k=-p}^{q-1}\A_{j,k}(z){W_m^\pm}^{(k)}(z,j)\right).\label{eg:lemCombiWm2}
				\end{align}
			\end{subequations}
			Using \eqref{eg:lemCombiWm1}, we thus have that:
			\begin{equation}\label{eg:lemCombiWm3}
				\forall j\in\Z,\forall k\in\lbrace-p,\hdots,q-1\rbrace,\quad {W_m^\pm}^{(k)}(z,j)= w_m^\pm(z,j+k).
			\end{equation}
			This implies that:
			\begin{equation}\label{eg:W}
				\forall j \in\Z,\quad W_m^\pm(z,j)= \begin{pmatrix}
					w_m^\pm(z,j+q-1)\\ \vdots \\ w_m^\pm(z,j-p)
				\end{pmatrix}.
			\end{equation}
			Furthermore, combining \eqref{eg:lemCombiWm2} and \eqref{eg:lemCombiWm3}, we obtain:
			$$\forall j\in\Z,\quad \sum_{k=-p}^q\A_{j,k}(z)w_m^\pm(z,j+k)=0 $$
			and thus using the definitions \eqref{def:AAjk} and \eqref{def:Ajk} of the matrices $\A_{j,k}(z)$ and $A_{j,k}$:
			\begin{equation}\label{eg:lemCombiWm4}
				\forall j\in\Z,\quad (z-1)w_m^\pm(z,j) = \sum_{k=-p}^{q-1}B_{j,k}w_m^\pm(z,j+k)-\sum_{k=-p}^{q-1}B_{j+1,k}w_m^\pm(z,j+1+k).
			\end{equation}
			Evaluating \eqref{eg:lemCombiWm4} at $z=1$, we thus obtain that the sequence $\left(\sum_{k=-p}^{q-1}B_{j,k}w_m^\pm(1,j+k)\right)_{j\in\Z}$ is constant.
			
			\begin{subequations}\label{eg:wm}
				\begin{itemize}
					\item If $m\in I_{ss}=I_{ss}^+$, since $\zeta_m^+(1)=\uzet_m^+\in\D$, Lemma \ref{lem_choice_base} implies that:
					$$W_m^+(1,j)={\uzet_m^+}^jV_m^+(1,j)\underset{j\rightarrow +\infty}\rightarrow 0 \quad \text{ and thus } \quad w_m^+(1,j)\underset{j\rightarrow +\infty}\rightarrow 0.$$
					As a consequence, since the sequence $\left(\sum_{k=-p}^{q-1}B_{j,k}w_m^\pm(1,j+k)\right)_{j\in\Z}$ is constant and the family of matrices $(B_{j,k})_{(j,k)\in\Z\times\lbrace-p,\hdots,q-1\rbrace}$ is uniformly bounded, we have that:
					\begin{equation}\label{eg:wmIss}
						\forall j\in\Z,\quad \sum_{k=-p}^{q-1}B_{j,k}w_m^+(1,j+k) = 0.
					\end{equation}
					
					\item If $m\in I_{su}=I_{su}^-$, using a similar proof as for \eqref{eg:wmIss}, we obtain that:
					\begin{equation}\label{eg:wmIsu}
						\forall j\in\Z,\quad \sum_{k=-p}^{q-1}B_{j,k}w_m^-(1,j+k) = 0.
					\end{equation}
					
					\item If $m\in I_{cs}=I_{cs}^+$, we have that $\zeta_m^+(1)=\uzet_m^+=1$. Furthermore, using Lemma \ref{lem_choice_base}, we have that:
					$$W_m^+(1,j) =V_m^+(1,j)\underset{j\rightarrow +\infty}\rightarrow R_m^+(1) = \begin{pmatrix}
					\rg_l^\pm \\ \vdots \\ \rg_l^\pm
					\end{pmatrix} $$
					where $m=l+(k-1)d\in\lc1,\ppp,d(p+q)\rc$ with $k\in\lc1,\ppp,p+q\rc$ and $l\in\lc1,\ppp,d\rc$. Therefore, we have that:
					$$w_m^+(1,j) \underset{j\rightarrow +\infty}\rightarrow\rg_l^+.$$
					As a consequence, since the sequence $\left(\sum_{k=-p}^{q-1}B_{j,k}w_m^\pm(1,j+k)\right)_{j\in\Z}$ is constant, we have that:
					$$\forall j\in\Z,\quad \sum_{k=-p}^{q-1}B_{j,k}w_m^-(1,j+k) = \sum_{k=-p}^{q-1}B_k^+\rg_l^+. $$
					Using the definition \eqref{def:Bk} of the matrices $B_k^+$, the consistency condition \eqref{cond:consistency}, the definition \eqref{eg:FcFin} of $\alpha_l^+$ and the fact that $\rg_l^+$ is a eigenvector of $df(u^+)$ associated with $\lambg_l^+$, we have:
					$$\sum_{k=-p}^{q-1}B_k^+\rg_l^+ = \nug df(u^+) \rg_l^+ =\nug \lambg_l^+\rg_l^+=\alpha_l^+\rg_l^+. $$
					We thus obtain that:					
					\begin{equation}\label{eg:wmIcs}
						\forall j\in\Z,\quad \sum_{k=-p}^{q-1}B_{j,k}w_m^-(1,j+k) = \alpha_l^+\rg_l^+.
					\end{equation}
					
					\item If $m\in I_{cu}=I_{cu}^-$, using a similar proof as for \eqref{eg:wmIcs}, we obtain that:
					\begin{equation}\label{eg:wmIcu}
						\forall j\in\Z,\quad \sum_{k=-p}^{q-1}B_{j,k}w_m^-(1,j+k) = \alpha_l^-\rg_l^-.
					\end{equation}
				\end{itemize}
			\end{subequations}
			The equalities \eqref{eg:wmIcs} and \eqref{eg:wmIcu} will be used later on in the proof of Lemma \ref{lem:LiuMajda}. Using Hypothesis \ref{H:inv} to invert the matrix $B_{j,-p}$ and combining \eqref{eg:W} with \eqref{eg:wmIss} and \eqref{eg:wmIsu}, we have:
			$$\left\{\begin{array}{cc}
				\forall m \in I_{ss}, \quad & W_m^+(1,0)\in \left\{\begin{pmatrix}
					V\\DV
				\end{pmatrix},\quad  V\in \C^{d(p+q-1)}\right\} ,\\
				\forall m \in I_{su}, \quad & W_m^-(1,0)\in \left\{ \begin{pmatrix}
					V\\DV
				\end{pmatrix},\quad  V\in \C^{d(p+q-1)}\right\},
			\end{array}\right. $$
			where
			$$D:=\left(-B_{0,-p}^{-1}B_{0,q-1}\quad \ppp\quad -B_{0,-p}^{-1}B_{0,-p+1}\right)\in \Mc_{d,d(p+q-1)}(\C).$$
			Also, 
			$$\dim \lc \begin{pmatrix}	V\\DV	\end{pmatrix},\quad  V\in \C^{d(p+q-1)}\rc=d(p+q-1),$$
			and Hypothesis \ref{H:Lax} implies that
			$$\#I_{ss}\cup I_{su}=d(p+q-1)+1.$$
			Therefore, since Lemma \ref{lemBasis} implies that the families $\left(W_m^+(1,0)\right)_{m\in I_{ss}}$ and $\left(W_m^-(1,0)\right)_{m\in I_{su}}$ are both linearly independent, we can conclude that
			\begin{equation}\label{lemCombiWm:inegdim}
				\dim \mathrm{Span}\left(W_m^+(1,0),m\in I_{ss}\right)\cap\mathrm{Span}\left(W_m^-(1,0),m\in I_{su}\right)\geq 1.
			\end{equation}
			
			\textbf{\underline{Step 2:}} We observe that the definition \eqref{def:Ev} of the Evans function and the inequality \eqref{lemCombiWm:inegdim} immediately imply that $\mathrm{Ev}(1)=0$. Let us now assume that Hypothesis \ref{H:Evans} is verified, i.e. we assume that:
			$$\frac{\partial \mathrm{Ev}}{\partial z}(1)\neq0.$$
			Therefore, using \eqref{lemCombiWm:inegdim} and differentiating the expression \eqref{def:Ev} of the Evans function $\mathrm{Ev}$, we also easily deduce \eqref{lemCombiWm:egdim}.
			
			\textbf{\underline{Step 3:}} We will now prove \eqref{lemCombiWm:eg} and thus obtain \eqref{lemCombiWm:egdim2}. Let us start by considering a vector $V_0$ such that:
			$$V_0\in\mathrm{Span}\left(W_m^+(1,0),m\in I_{ss}\right)\cap\mathrm{Span}\left(W_m^-(1,0),m\in I_{su}\right)\backslash\lc0\rc.$$
			We then obviously have that:
			$$ \forall j\in\Z,\quad  X_j(1)V_0\in\mathrm{Span}\left(W_m^+(1,j),m\in I_{ss}\right)\cap\mathrm{Span}\left(W_m^-(1,j),m\in I_{su}\right).$$
			Lemma \ref{lem_choice_base} thus implies that $\left(\Pi\left(X_j(1)V_0\right)\right)_{j\in\Z}$ has exponential decay as $j$ tends towards $+\infty$ and $-\infty$ and thus belongs to $\ell^2(\Z,\C^d)$. Furthermore, since $(X_j(1)V_0)_{j\in\Z}$ is a solution of \eqref{syst_dyn} for $z=1$, we have that $\left(\Pi\left(X_j(1)V_0\right)\right)_{j\in\Z}$ belongs to $\ker(Id_{\ell^2}-\Lcc)$. Thus,
			$$\mathrm{Span} \left(\Pi\left(X_j(1)V_0\right)\right)_{j\in\Z} \subset \ker(Id_{\ell^2}-\Lcc).$$		
			
			We now consider $w\in \ker\left(Id_{\ell^2}-\Lcc\right)$ and define 
			$$W_0:=\begin{pmatrix}
				w_{q-1}\\ \vdots \\ w_{-p}
			\end{pmatrix}\in\C^{d(p+q)}.$$
			We then have that
			\begin{equation}\label{lemCombiWm:egWj}
				\forall j\in \Z, \quad W_j:=X_j(1)W_0 = \begin{pmatrix}
				w_{j+q-1}\\ \vdots \\ w_{j-p}
			\end{pmatrix}.
			\end{equation}
			Since $(W_m^+(1,0))_{m\in\lbrace1,\hdots,d(p+q)\rbrace}$ is a basis of $\C^{d(p+q)}$, we introduce a family of complex scalars $(c_m)_{m\in\lc1,\ppp,d(p+q)\rc}$ such that:
			\begin{equation}\label{lemCombiWm:ProofDecomp}
				W_0=\sum_{m=1}^{d(p+q)}c_mW_m^+(1,0).
			\end{equation}
			Then, since $W_m^+$ are solutions of \eqref{syst_dyn}, we have that: 
			$$\forall j\in \Z, \quad W_j=\sum_{m=1}^{d(p+q)}c_mW_m^+(1,j).$$
			We will now use a strategy similar to the proof of Lemma \ref{lemBasis} to prove that: 
			\begin{equation}\label{lemCombiWm:Proof}
				\forall m\in\lc1,\ppp,d(p+q)\rc\backslash I_{ss},\quad  c_m=0 .
			\end{equation}
			We introduce:
			\begin{align*}
				I_n&:=\lc m\in\lc1,\ppp,d(p+q)\rc, \quad c_m\neq 0\rc\neq \emptyset,\\
				R&:= \max_{m\in I_n} |\uzet_m^+|>0,\\
				I_R&:= \text{argmax}_{m\in I_n} |\uzet_m^+|\neq \emptyset.
			\end{align*} 
			Let us assume that the assertion \eqref{lemCombiWm:Proof} is false. Then, $I_R$ is not an empty set and the scalar $R$ is larger or equal to $1$. We have that for all $j\in\Z$:
			\begin{equation}\label{lemCombiWm:Proof2}
				\frac{W_j}{R^j} = \sum_{m\in I_R}c_m \left(\frac{\uzet_m^+}{R}\right)^jR_m^+(1)+\sum_{m\in I_n\backslash I_R}c_m \left(\frac{\uzet_m^+}{R}\right)^jR_m^+(1)+\sum_{m\in I_n}c_m \frac{W_m^+(1,j)-{\uzet_m^+}^jR_m^+(1)}{R^j}.
			\end{equation}
			We observe that \eqref{lemCombiWm:egWj} implies that $W_j$ converges towards $0$ as $j$ tends towards $+\infty$ because the sequence $w$ belongs to $\ell^2(\Z,\C^d)$. Since $R$ is larger or equal than $1$, we have that:
			$$\frac{W_j}{R^j} \underset{j\rightarrow +\infty}{\rightarrow}0.$$
			Furthermore, using similar ideas as in the proof of Lemma \ref{lemBasis}, we have that:
			$$\sum_{m\in I_n\backslash I_R}c_m \left(\frac{\uzet_m^+}{R}\right)^jR_m^+(1)+\sum_{m\in I_n}c_m \frac{W_m^+(1,j)-{\uzet_m^+}^jR_m^+(1)}{R^j}\underset{j\rightarrow +\infty}{\rightarrow}0.$$
			Thus, \eqref{lemCombiWm:Proof2} implies that:
			$$\sum_{m\in I_R}c_m \left(\frac{\uzet_m^+}{R}\right)^jR_m^+(1)\underset{j\rightarrow +\infty}{\rightarrow}0.$$
			Since $I_R$ is not empty, projecting the equality above on $\mathrm{Span}(R_{m_0}^+(1))$ along $\mathrm{Span}(R_m^+(1), m\neq m_0)$ for some $m_0\in I_R$ implies that:
			$$c_{m_0}\left(\frac{\uzet_{m_0}^+}{R}\right)^j\underset{j\rightarrow+\infty}\rightarrow0.$$
			But, $m_0$ belongs to $I_R$ so $|\uzet_{m_0}^+|=R$. This implies that $c_{m_0}=0$. However, $m_0\in I_R\subset I_n$ so $c_{m_0}\neq 0$. This is a contradiction which implies that \eqref{lemCombiWm:Proof} has to be verified. Then, using \eqref{lemCombiWm:ProofDecomp}, we have that:
			$$W_0\in \mathrm{Span}(W_m^+(1,0),m\in I_{ss}).$$	
			Furthermore, using a similar proof with the family $(W_m^-(1,0))_{m\in\lc1,\ppp,d(p+q)\rc}$, we have that $W_0$ belongs to $\mathrm{Span}(W_m^+(1,0),m\in I_{ss})\cap\mathrm{Span}(W_m^-(1,0),m\in I_{su})= \mathrm{Span} V_0$. Therefore, since for all $j\in\Z$ we have: 
			$$w_j=\Pi(W_j)=\Pi(X_j(1)W_0),$$
			we conclude that the sequence $w$ belongs to $\mathrm{Span} \left(\Pi\left(X_j(1)V_0\right)\right)_{j\in\Z}$. We can thus finally verify \eqref{lemCombiWm:eg} and \eqref{lemCombiWm:egdim2}.
		\end{proof}

	First, as a consequence of Lemma \ref{lemCombiWm}, since the Evans function $\mathrm{Ev}$ is holomorphic on $B(1,\delta_1)$ and not uniformly equal to $0$, the equality $\mathrm{Ev}(1)=0$ implies that we can consider $\delta_1$ small enough so that the Evans function $\mathrm{Ev}$ only vanishes at $z=1$ on the ball $B(1,\delta_1)$. 
	
	Our new goal for the rest of this section will be to use Lemma \ref{lemCombiWm} to introduce below two new bases $\left(\Phi_m(z,0)\right)_{m\in\lc1,\ppp,dp\rc}$ and $\left(\Phi_m(z,0)\right)_{m\in\lc dp+1,\ppp,d(p+q)\rc}$ in \eqref{defPhi} for the vector spaces $E_0^+(z)$ and $E_0^-(z)$ more suitable for the study of the spatial Green's function when $z$ is close to $1$. We will also define in \eqref{def:Gc-} below a new Evans function $D^\Phi$ associated with this new choice of bases which will share the same properties as $\mathrm{Ev}$.
	
	The equality \eqref{lemCombiWm:egdim} of Lemma \ref{lemCombiWm} implies that there exist two non zero families of complex numbers $(\theta_{s,m})_{m\in I_{ss}}$ and $(\theta_{u,m})_{m\in I_{su}}$  such that:
	\begin{equation}\label{egWssWsu}
		\sum_{m\in I_{ss}}\theta_{s,m} W_m^+(1,0)=\sum_{m\in I_{su}}\theta_{u,m} W_m^-(1,0).
	\end{equation}
	In the rest of the paper, we fix the choice of families of coefficients $\theta_{s,m}$ and $\theta_{u,m}$. Even if we have to reindex the eigenvalues $\uzet_m^\pm$, we will assume that $\theta_{s,1},\theta_{u,d(p+q)}\neq 0$. Furthermore, we also define: 
	\begin{equation}\label{def:theta_su}
		\theta_{s,m}:= 0 \text{ for } m\in \lc 1, \ppp,d(p+q)\rc\backslash I_{ss}\quad \text{ and }\quad  \theta_{u,m}:= 0 \text{ for } m\in \lc 1, \ppp,d(p+q)\rc\backslash I_{su}.
	\end{equation}
	We define for $m\in\lc1,\ppp,d(p+q)\rc$:
	\begin{equation}\label{defPhi}
		\forall z\in B(1,\delta_1), \forall j\in \Z, \quad \Phi_m(z,j):=\lc\begin{array}{cc}
			\sum_{m\in I_{ss}}\theta_{s,m} W_m^+(z,j), & \text{ if }m=1,\\
			W_m^+(z,j), & \text{ if }m\in\lc2, \ppp, dp\rc,\\
			W_m^-(z,j), & \text{ if }m\in\lc dp+1, \ppp, d(p+q)-1\rc,\\
			\sum_{m\in I_{su}}\theta_{u,m} W_m^-(z,j), & \text{ if }m=d(p+q).\\
		\end{array}\right.
	\end{equation}
	Since $\theta_{s,1},\theta_{u,d(p+q)}\neq 0$, we have that $\left(\Phi_m(z,0)\right)_{m\in\lc1,\ppp,dp\rc}$ and $\left(\Phi_m(z,0)\right)_{m\in\lc dp+ 1,\ppp,d(p+q)\rc}$ are respectively bases of $E^+_0(z)$ and $E_0^-(z)$. Equality \eqref{egWssWsu} implies that $\Phi_1(1,0)=\Phi_{d(p+q)}(1,0)$ and since $\Phi_1(1,\cdot)$ and $\Phi_{d(p+q)}(1,\cdot)$ are solutions of \eqref{syst_dyn} for $z=1$, we have:
	\begin{equation}\label{egPhi}
		\forall j\in \Z, \quad\Phi_1(1,j)=\Phi_{d(p+q)}(1,j).
	\end{equation}
	Furthermore, using the expression of $\Phi_1(z,j)$ and $\Phi_{d(p+q)}(z,j)$ as well as Lemma \ref{lem_choice_base} and inequalities \eqref{inZetaSs} and \eqref{inZetaSu}, we prove that there exists a positive constant $C$ such that:
	\begin{align}\label{in:decroissance_expo_Phi_1_et_d(p+q)}
		\begin{split}
			\forall z\in B(1,\delta_1), \forall j\in \N, \quad & \left|\Phi_1(z,j)\right|\leq Ce^{-2c_*|j|},\\
			\forall z\in B(1,\delta_1), \forall j\in -\N, \quad & \left|\Phi_{d(p+q)}(z,j)\right|\leq Ce^{-2c_*|j|}.
		\end{split}
	\end{align}
	If we define: 
	\begin{equation}\label{def:V}
		\forall j\in\Z, \quad V(j)=\Pi(\Phi_1(1,j))=\Pi(\Phi_{d(p+q)}(1,j)),
	\end{equation}
	then using \eqref{lemCombiWm:eg} because: 
	$$\Phi_1(1,0)=\Phi_{d(p+q)}(1,0)\in\mathrm{Span}\left(W_m^+(1,0),m\in I_{ss}\right)\cap\mathrm{Span}\left(W_m^-(1,0),m\in I_{su}\right)\backslash\lc0\rc,$$
	we have that $V$ is a sequence in $\ell^2(\Z,\C^d)\backslash\lc0\rc$ such that \eqref{egV} and \eqref{decExpoV} are verified. It will correspond to the sequence $V$ in Theorem \ref{th:Green}.
	
	If we summarize, for $z\in B(1,\delta_1)$, we have five families to describe the solutions of the dynamical system \eqref{syst_dyn}:
	\begin{itemize}
		\item The bases $\left(\Phi_1(z,j),W_2^+(z,j),\ppp,W_{d(p+q)}^+(z,j)\right)$ and $\left(W^+_1(z,j),W_2^+(z,j),\ppp,W_{d(p+q)}^+(z,j)\right)$ for which we know the asymptotic behavior when $j$ tends to $+\infty$ thanks to Lemma \ref{lem_choice_base}.
		\item The bases $\left(W_1^-(z,j),\ppp,W_{d(p+q)-1}^-(z,j),\Phi_{d(p+q)}(z,j)\right)$ and $\left(W_1^-(z,j),\ppp,W_{d(p+q)-1}^-(z,j),W^-_{d(p+q)}(z,j)\right)$ for which we know the asymptotic behavior when $j$ tends to $-\infty$ thanks to Lemma \ref{lem_choice_base}.
		\item The family $\left(\Phi_m(z,j)\right)_{m\in\lc1,\ppp,d(p+q)\rc}$ which is linked to the solutions of the dynamical system \eqref{syst_dyn} which tend towards $0$ as $j$ tends towards $+\infty$ or $-\infty$, at least when $z\in \Oc\cap B(1,\delta_1)$. It is a basis of $\C^{d(p+q)}$ if and only if $E_0^+(z)\oplus E_0^-(z) =\C^{d(p+q)}$, i.e. for $z$ different from $1$.
	\end{itemize}
	
	We introduce a few more notations. For $z\in B(1,\delta_1)$ and $j\in \Z$, we define:
	\begin{subequations}\label{def:Mat_et_det}
		\begin{align}
			\Gc^\Phi(z,j)&:=\left(\Phi_1(z,j)|\ppp|\Phi_{d(p+q)}(z,j)\right), & D^\Phi(z) &: = \det(\Gc^\Phi(z,0)),\label{def:GcPhi}\\
			\Gc^+(z,j) &:=\left(\Phi_1(z,j)|W^+_2(z,j)|\ppp|W^+_{d(p+q)}(z,j)\right),  & D^+(z) &: = \det(\Gc^+(z,0)),\label{def:Gc+}\\
			\Gc^-(z,j) &:=\left(W^-_1(z,j)|\ppp|W^-_{d(p+q)-1}(z,j)|\Phi_{d(p+q)}(z,j)\right),  & D^-(z) &: = \det(\Gc^-(z,0)),\label{def:Gc-}\\
			\widetilde{\Gc}^+(z,j) &:=\left(W^+_1(z,j)|W^+_2(z,j)|\ppp|W^+_{d(p+q)}(z,j)\right), &&\label{def:GcTilde+}\\
			\widetilde{\Gc}^-(z,j) &:=\left(W^-_1(z,j)|W^-_2(z,j)|\ppp|W^-_{d(p+q)}(z,j)\right). &&\label{def:GcTilde-}
		\end{align}
	\end{subequations}
	Those functions are holomorphic on $B(1,\delta_1)$. We now make a few observations:
	
	$\bullet$ We observe using the definition \eqref{defPhi} of $\Phi_1$ and $\Phi_{d(p+q)}$ that for $z\in B(1,\delta_1)$ and $j\in\Z$ 
	\begin{equation}\label{eg:GcGcTilde}
		\Gc^+(z,j)= \widetilde{\Gc}^+(z,j)\begin{pmatrix}
			\theta_{s,1}&    & (0)&\\
			\theta_{s,2} & 1 & &\\
			\vdots 			 &    &  \ddots & \\
			\theta _{s, d(p+q)}&(0) & & 1
		\end{pmatrix} \quad \text{ and } \quad \Gc^-(z,j)= \widetilde{\Gc}^-(z,j)\begin{pmatrix}
			1 &    & (0)& \theta_{u,1} \\
			& \ddots & & \vdots \\
			&    &  1&\theta_{u,d(p+q)-1}  \\
			&(0) & & \theta_{u,d(p+q)} 
		\end{pmatrix}.
	\end{equation}
	
	$\bullet$ Using \eqref{defPhi} and \eqref{def:theta_su}, we have for $z\in B(1,\delta_1)$:
	$$\Gc^\Phi(z,0) = \left(W_1^+(z,0)|\ppp|W_{dp}^+(z,0)|W_{dp+1}^-(z,0)|\ppp|W_{d(p+q)}^-(z,0)\right)\begin{pmatrix}
		\theta_{s,1}&&&& \theta_{u,1} \\
		\theta_{s,2}& 1 &&(0)&\theta_{u,2}\\
		\vdots & & \ddots&&\vdots\\
		\theta_{s,d(p+q)-1} & (0)& & 1 &\theta_{u,d(p+q)-1}\\
		\theta_{s,d(p+q)} & & & & \theta_{u,d(p+q)}
	\end{pmatrix}$$
	and thus since $\theta_{s,d(p+q)},\theta_{u,1}=0$, we have that:
	\begin{equation}\label{eg:Ev}
		D^\Phi(z) = \theta_{s,1} \theta_{u,d(p+q)}\mathrm{Ev}(z).
	\end{equation}
	The function $D^\Phi$ thus shares the same properties as the Evans function $\mathrm{Ev}$, i.e. the function $D^\Phi$ is holomorphic on $B(1,\delta_1)$, vanishes only at $z=1$ and $1$ is a simple zero of $D^\Phi$. We will thus also call $D^\Phi$ Evans function.
	
	$\bullet$ For $z\in B(1,\delta_1)\backslash\lc1\rc$, the function $D^\Phi$ does not vanish at $z$ and thus $\left(\Phi_m(z,0)\right)_{m\in\lc 1,\ppp,d(p+q)\rc}$ is a basis. We can define for $m\in \lc 1,\ppp,d(p+q)\rc$ the projector $\Pi_m(z)$ on $\mathrm{Span}(\Phi_m(z,0))$ along $\mathrm{Span}(\Phi_{\nu}(z,0))_{\nu\in\lc1,\ppp,d(p+q)\rc\backslash\lc m\rc}$. We observe that
	\begin{equation}\label{def:Pim}
		\Pi_m(z)=\Gc^\Phi(z,0)P_m\Gc^\Phi(z,0)^{-1}
	\end{equation}
	where $P_m=\left(\delta_{i,m}\delta_{j,m}\right)_{i,j\in \lc1,\ppp,d(p+q)\rc}\in \Mc_{d(p+q)}(\C)$. The function $\Pi_m$ is holomorphic on $B(1,\delta_1)\backslash\lc1\rc$.

		We conclude this section with an additional lemma.
		
		\begin{lemma}\label{lem:LiuMajda}
			We have that:
			$$\det\left(\rg_1^-,\hdots,\rg_{I-1}^-,\sum_{j\in\Z}V(j),\rg_{I+1}^+,\hdots,\rg_d^+\right)\neq0$$
			where the sequence $V$ is defined by \eqref{def:V}.
		\end{lemma}
		
		Lemma \ref{lem:LiuMajda} corresponds to a generalization of the so-called \textit{Liu-Majda condition} (see \cite[(1.24)]{ZH} and \cite[(3.59)]{BenzoniHuotRousset}) which is a usual necessary condition linked with the spectral stability assumption. It will play a central role in the conclusion of the proof of Theorem \ref{th:Green} (see Section \ref{subsec:ConcluTh1}). The proof of Lemma \ref{lem:LiuMajda} is essentially an adaptation in our context of the proofs of \cite[Proposition 2.1]{Godillon} and \cite[(5.8)]{Serre} which were written for three-point schemes.
		
		\begin{proof}
			Let us start by observing that, using \eqref{egPhi} and the definition of the Evans function $D^\Phi$, we have:
			\begin{equation}\label{eg:derEvans}
				\frac{dD^\Phi}{dz}(1) = \det\left(\frac{\partial \Phi_1}{\partial z}(1,0)-\frac{\partial \Phi_{d(p+q)}}{\partial z}(1,0),\Phi_2(1,0),\hdots,\Phi_{d(p+q)}(1,0)\right).
			\end{equation}
			We will need some information on the vectors appearing in the determinant on the right-side of \eqref{eg:derEvans}.
			
			For $m\in\lbrace1,\hdots,d(p+q)\rbrace$, $j\in\Z$ and $z\in B(1,\delta_1)$, we define:
			$$\phi_m(z,j):=\Pi(\Phi_m(z,j))$$
			where the operator $\Pi$ is defined by \eqref{def:Pi}. The functions $\phi_m(\cdot,j)$ are holomorphic on $B(1,\delta_1)$. Furthermore, using the definition \eqref{defPhi} of the function $\Phi_m$, the equations \eqref{eg:W} and \eqref{eg:wm} imply that for $m\in\lbrace1,\hdots,d(p+q)\rbrace$ and $j\in\Z$:
			\begin{equation}\label{lem:LiuMajda:PR1}
				\forall z\in B(1,\delta_1),\quad \Phi_m(z,j) = \begin{pmatrix}
					\phi_m(z,j+q-1)\\
					\vdots\\
					\phi_m(z,j-p)
				\end{pmatrix}\quad \text{ and thus } \quad \frac{\partial\Phi_m}{\partial z}(z,j) = \begin{pmatrix}
				\frac{\partial\phi_m}{\partial z}(z,j+q-1)\\
				\vdots\\
				\frac{\partial\phi_m}{\partial z}(z,j-p)
				\end{pmatrix}
			\end{equation}
			and if we write $m=l+(k-1)d\in\lc1,\ppp,d(p+q)\rc$ with $k\in\lc1,\ppp,p+q\rc$ and $l\in\lc1,\ppp,d\rc$:
			\begin{equation}\label{lem:LiuMajda:PR2}
				\forall j\in\Z,\quad \left\{\begin{array}{cc}
					\displaystyle\sum_{k=-p}^{q-1}B_{j,k}\phi_m(1,j+k) =0 & \quad \text{ if } m\in I_{ss}\cup I_{su},\\
					\displaystyle\sum_{k=-p}^{q-1}B_{j,k}\phi_m(1,j+k) =\alpha_l^+\rg_l^+ & \quad \text{ if } m\in I_{cs},\\
					\displaystyle\sum_{k=-p}^{q-1}B_{j,k}\phi_m(1,j+k) =\alpha_l^-\rg_l^- & \quad \text{ if } m\in I_{cu}.\\
				\end{array}\right.
			\end{equation}
			
			We are now going to prove that:
			\begin{equation}\label{lem:LiuMajda:PR3}
				\sum_{k=-p}^{q-1} B_{0,k}\left(\frac{\partial\phi_1}{\partial z}(1,k) - \frac{\partial\phi_{d(p+q)}}{\partial z}(1,k) \right) = \sum_{j\in\Z} V(j)
			\end{equation}
			where the sequence $V$ is defined by \eqref{def:V}. First, we notice that:
			\begin{equation}\label{lem:LiuMajda:PR4}
				\forall j\in\Z,\quad V(j) = \phi_1(1,j)=\phi_{d(p+q)}(1,j).
			\end{equation}
			For $m\in\lbrace1,\hdots,d(p+q)\rbrace$, if we differentiate \eqref{eg:lemCombiWm4} and evaluate it at $z=1$, we have:
			\begin{equation}\label{lem:LiuMajda:PR5}
				\forall j \in\Z,\quad w_m^\pm(1,j)=\sum_{k=-p}^{q-1}B_{j,k}\frac{\partial w_m^\pm}{\partial z}(1,j+k)-\sum_{k=-p}^{q-1}B_{j+1,k}\frac{\partial w_m^\pm}{\partial z}(1,j+1+k).
			\end{equation}
			\begin{subequations}\label{lem:LiuMajda:PR6}
				\begin{itemize}
					\item Lemma \ref{lem_choice_base} implies that the sequences $V_m^+(1,\cdot)$ and $\frac{\partial V_m^+}{\partial z}(1,\cdot)$ belong to $\ell^\infty(\N,\C^{d(p+q)})$. For $m\in I_{ss}$, we have that $\zeta_m^+(1)=\uzet_m^+\in\D$ and thus for $j\in\Z$:
					$$\frac{\partial W_m^+}{\partial z}(1,j) = {\uzet_m^+}^{j-1}\left(j\frac{d\zeta_m^+}{dz}(1)V_m^+(1,j)+ \uzet_m^+\frac{\partial V_m^+}{\partial z}(1,j)\right)\underset{j\rightarrow+\infty}{\rightarrow}0.$$
					Since the family of matrices $(B_{j,k})_{(j,k)\in\Z\times\lbrace-p,\hdots,q-1\rbrace}$ is uniformly bounded, we have:
					$$\sum_{k=-p}^{q-1}B_{j,k}\frac{\partial w_m^+}{\partial z}(1,j+k)\underset{j\rightarrow+\infty}\rightarrow0.$$
					Thus, by summing \eqref{lem:LiuMajda:PR5} for $j\geq0$, we have for $m\in I_{ss}$:
					\begin{equation}
						\sum_{k=-p}^{q-1}B_{0,k}\frac{\partial w_m^+}{\partial z}(1,k) = \sum_{j=0}^{+\infty}w_m^+(1,j).
					\end{equation}
					\item Using a similar proof, by summing \eqref{lem:LiuMajda:PR5} for $j\leq-1$, we have for $m\in I_{su}$:
					\begin{equation}
						-\sum_{k=-p}^{q-1}B_{0,k}\frac{\partial w_m^-}{\partial z}(1,k) = \sum_{j=-\infty}^{-1}w_m^-(1,j).
					\end{equation}
				\end{itemize}
			\end{subequations}
			Thus, using \eqref{lem:LiuMajda:PR4}, \eqref{lem:LiuMajda:PR6} and the definition \eqref{defPhi} of $\Phi_1$ and $\Phi_{d(p+q)}$, we obtain \eqref{lem:LiuMajda:PR3}.
			
			Combining \eqref{eg:derEvans}, \eqref{lem:LiuMajda:PR2} and \eqref{lem:LiuMajda:PR3}, we have that:
			\begin{align}
				\begin{split}
					&\det(B_{0,-p})\frac{dD^\Phi}{dz}(1)\\
					=& \det\left(\begin{pmatrix}
						I_d & & & &\\
						&I_d  & & &\\
						(0)& & \ddots & &(0)\\
						& & & I_d & \\
						B_{0,q-1} & \hdots & \hdots & B_{0,-p+1} & B_{0,-p}
					\end{pmatrix}\left(\begin{array}{c|c|c|c}
					\frac{\partial \Phi_1}{\partial z}(1,0)-\frac{\partial \Phi_{d(p+q)}}{\partial z}(1,0) &\Phi_2(1,0) &\hdots &\Phi_{d(p+q)}(1,0)
					\end{array}\right)\right)\\
					=& \det\left(\begin{array}{c|c}
						\begin{array}{c}
							\displaystyle\frac{\partial \phi_1}{\partial z}(1,q-1)-\frac{\partial \phi_{d(p+q)}}{\partial z}(1,q-1)\\
							\hline \vdots\\
							\hline \displaystyle\frac{\partial \phi_1}{\partial z}(1,-p+1)-\frac{\partial \phi_{d(p+q)}}{\partial z}(1,-p+1)
						\end{array} & {\hspace{1cm} A\hspace{1cm}} \\
						\hline
						\displaystyle\sum_{j\in\Z}V(j) & \hspace{1cm}B\hspace{1cm}
					\end{array}\right)
				\end{split}\label{lem:LiuMajda:PR7}
			\end{align}
			where the matrices $A$ and $B$ are defined by:
			\begin{align*}
				A&= \left(\begin{array}{c|c|c}
					\phi_2(1,q-1)& \hdots & \phi_{d(p+q)}(1,q-1)\\
					\hline \vdots & \vdots & \vdots \\
					\hline \phi_2(1,-p+1)& \hdots & \phi_{d(p+q)}(1,-p+1)
				\end{array}\right)\in \Mc_{d(p+q-1),d(p+q)-1}(\C),\\
				B & = \biggl(\, \underset{d(p-1)+I-1 \text{ zeroes}}{\underbrace{0\,\bigg|\,\hdots\,\bigg|\,0}}\,\bigg|\, \alpha_{I+1}^+\rg_{I+1}^+\,\bigg|\, \hdots \,\bigg|\, \alpha_d^+\rg_d^+ \,\bigg|\, \alpha_1^-\rg_1^- \,\bigg|\, \hdots \,\bigg|\, \alpha_{I-1}^-\rg_{I-1}^- \,\bigg|\, \underset{dq-I+1 \text{ zeroes}}{\underbrace{0 \,\bigg|\, \hdots \,\bigg|\,0}}\,\biggr) \in \Mc_{d,d(p+q)-1}(\C).
			\end{align*}
			After permutations of the columns of the matrix on the right-hand side of \eqref{lem:LiuMajda:PR7} to rearrange it as a block triangular matrix, we can obtain:
			\begin{equation}\label{lem:LiuMajda:PR8}
				\left|\det(B_{0,-p})\frac{dD^\Phi}{dz}(1)\right| = \left|\det(C)\det\left(\alpha_{I+1}^+\rg_{I+1}^+,\hdots,\alpha_d^+\rg_d^+,\alpha_1^-\rg_1^-,\hdots,\alpha_{I-1}^-\rg_{I-1}^-,\sum_{j\in\Z}V(j)\right)\right|
			\end{equation}
			where the matrix $C$ is defined by:
			\begin{multline*}
				C= \\ \left(\begin{array}{c|c|c|c|c|c}
				\phi_2(1,q-1) & \hdots & \phi_{d(p-1)+I}(1,q-1) &\phi_{dp+I}(1,q-1) &\hdots& \phi_{d(p+q)}(1,q-1)\\\hline
				\vdots & \vdots & \vdots & \vdots &\vdots& \vdots\\\hline
				\phi_2(1,-p+1) & \hdots & \phi_{d(p-1)+I}(1,-p+1) &\phi_{dp+I}(1,-p+1) &\hdots& \phi_{d(p+q)}(1,-p+1)
			\end{array}\right)\in\Mc_{d(p+q-1)}(\C)
			\end{multline*}
			Finally, Hypotheses \ref{H:Evans} and \ref{H:inv} imply that the left-hand side of \eqref{lem:LiuMajda:PR8} is different from $0$. We then easily obtain that:
			$$\det\left(\rg_1^-,\hdots,\rg_{I-1}^-,\sum_{j\in\Z}V(j),\rg_{I+1}^+,\hdots,\rg_{d}^+\right)\neq0.$$
		\end{proof}

	\subsection{Behavior of the spatial Green's function near \texorpdfstring{$1$}{1}} \label{secGreenSpClose}
	
	Now that we have introduced all the tools necessary, our goal is to prove the following proposition which implies that the spatial Green's function can be meromorphically extended near $1$ and which gives expressions of the spatial Green's function using the solutions $W_m^\pm$ and $\Phi_m$ of the dynamical system \eqref{syst_dyn} defined previously respectively in Lemma \ref{lem_choice_base} and \eqref{defPhi}.
	
	\begin{prop}\label{prop:GS_near_1}
		We consider $j_0,j\in\Z$ and $\textbf{e}\in\C^d$. The function $W(\cdot,j_0,j,\textbf{e})$ can be meromorphically extended on the whole ball $B(1,\delta_1)$ with a pole of order $1$ at $1$. Furthermore, using the functions $\widetilde{\Ccc}_{m^\prime}^+$ and $\widetilde{g}_{m^\prime,m}^+$ defined below respectively by \eqref{defCtilde} and \eqref{def:gtilde}, we have that for $j_0\in\N$, $j\in\Z$, $\textbf{e}\in \C^d$ and $z\in B(1,\delta_1)\backslash\lc1\rc$:
		\begin{subequations}\label{ExpGs}
			\begin{itemize}
				\item For $j\geq j_0+1$:
				\begin{multline}\label{ExpGs: l geq 0 : j geq l+1}
					W(z,j_0,j,\textbf{e})= -\sum_{m=1}^{dp}  \widetilde{\Ccc}^+_{m}(z,j_0,\textbf{e})W^+_m(z,j) -\sum_{m=2}^{dp} \sum_{m^\prime=dp+1}^{d(p+q)}  \widetilde{g}_{m^\prime,m}^+(z)\widetilde{\Ccc}^+_{m^\prime}(z,j_0,\textbf{e})W^+_m(z,j)\\- \sum_{m^\prime=dp+1}^{d(p+q)}  \widetilde{g}_{m^\prime,1}^+(z)\widetilde{\Ccc}^+_{m^\prime}(z,j_0,\textbf{e})\Phi_1(z,j).
				\end{multline}
				\item For $j\in\lc0,\ppp,j_0\rc$:
				\begin{multline}\label{ExpGs: l geq 0 : 0 leq j leq l}
					W(z,j_0,j,\textbf{e})= \sum_{m=dp+1}^{d(p+q)}\widetilde{\Ccc}_{m}^+(z,j_0,\textbf{e}) W_{m}^+(z,j) -\sum_{m=2}^{dp} \sum_{m^\prime=dp+1}^{d(p+q)}  \widetilde{g}_{m^\prime,m}^+(z)\widetilde{\Ccc}^+_{m^\prime}(z,j_0,\textbf{e})W^+_m(z,j)\\- \sum_{m^\prime=dp+1}^{d(p+q)}  \widetilde{g}_{m^\prime,1}^+(z)\widetilde{\Ccc}^+_{m^\prime}(z,j_0,\textbf{e})\Phi_1(z,j).
				\end{multline}
				
				\item For $j<0$:
				\begin{multline}\label{ExpGs: l geq 0 : j < 0}
					W(z,j_0,j,\textbf{e})= \sum_{m=dp+1}^{d(p+q)-1} \sum_{m^\prime=dp+1}^{d(p+q)}  \widetilde{g}_{m^\prime,m}^+(z)\widetilde{\Ccc}^+_{m^\prime}(z,j_0,\textbf{e})W^-_m(z,j)\\+\sum_{m^\prime=dp+1}^{d(p+q)}  \widetilde{g}_{m^\prime,d(p+q)}^+(z)\widetilde{\Ccc}^+_{m^\prime}(z,j_0,\textbf{e})\Phi_{d(p+q)}(z,j).
				\end{multline}
			\end{itemize}
			Similar expressions can be deduced when $j_0\in-\N\backslash\lc0\rc$ using the functions $\widetilde{\Ccc}_{m^\prime}^-$ and $\widetilde{g}_{m^\prime,m}^-$.
		\end{subequations}
	\end{prop}

	To enter into more details, in Section \ref{subsec:Near1:PremièreExpression}, we will use the families of solutions of the dynamical system \eqref{syst_dyn} previously introduced to decompose the expressions \eqref{Wr} and \eqref{Wl} proved on the spatial Green's function in Section \ref{sec:GSloin} and obtain the expressions \eqref{expW3} below of the spatial Green's function on $B(1,\delta_1)\cap\Oc_\rho$. Since the sequence $W_m^\pm$ are holomorphic on $B(1,\delta_1)$, this will allow us to extend the spatial Green's function meromorphically on the whole ball $B(1,\delta_1)$ with a pole of order $1$ at $1$. The computations performed in this section will be fairly inspired by \cite[Section 3]{Godillon} and the expression \eqref{expW3} corresponds to the result of \cite[Proposition 3.1]{Godillon}. However, our goal is to improve the expression \eqref{expW3} to obtain the result \eqref{ExpGs}. Those calculations will be performed in Sections \ref{subsubsec:Dm} and \ref{subsubsec:ReexpGreenSpatial}.
	
	\subsubsection{Meromorphic extension of the spatial Green's function near \texorpdfstring{$1$}{1}}\label{subsec:Near1:PremièreExpression}

	We observe that for $z\in \Oc_\rho\cap B(1,\delta_1)$, we have that $Q(z)$ is the projector on $E_0^+(z)$ along $E_0^-(z)$ and thus
	$$Q(z)=\sum_{m=1}^{dp}\Pi_m(z) \quad \text{ and }\quad Id-Q(z)=\sum_{m=dp+1}^{(p+q)d}\Pi_m(z)$$
	where the projectors $\Pi_m(z)$ are defined by \eqref{def:Pim}. Since the equalities \eqref{Wr} and \eqref{Wl} are still verified for $z\in\Oc_\rho\cap B(1,\delta_1)$, if we define for $m\in \lc 1, \ppp,d(p+q)\rc$ the functions $\nu_m$:
	$$\forall z\in B(1,\delta_1)\backslash\lc1\rc, \forall j_0,j\in \Z, \forall \textbf{e}\in \C^d, \quad \nu_m(z,j_0,j,\textbf{e}):= X_j(z)\Pi_m(z)X_{j_0+1}(z)^{-1}\begin{pmatrix}
		A_{j_0,q}^{-1}\textbf{e}\\0\\\vdots\\0
	\end{pmatrix}$$
	then, we have that for $z\in \Oc_\rho\cap B(1,\delta_1)$
	\begin{subequations}\label{expW2}
		\begin{align}
			\forall j \geq j_0+1, \quad& W(z,j_0,j,\textbf{e})= -\sum_{m=1}^{dp}\nu_m(z,j_0,j,\textbf{e}), \label{Wr2}\\
			\forall j  \leq j_0, \quad& W(z,j_0,j,\textbf{e})=\sum_{m=dp+1}^{(p+q)d} \nu_m(z,j_0,j,\textbf{e}).\label{Wl2}
		\end{align}
	\end{subequations}
	Since the right hand terms of \eqref{expW2} are holomorphic on $B(1,\delta_1)\backslash\lc 1\rc$, we can extend holomorphically the function $W(\cdot,j_0,j,\textbf{e})$ on $B(1,\delta_1)\backslash\lc 1\rc$.
	
	We observe that \eqref{def:Pim} implies for $z\in B(1,\delta_1)\backslash\lc1\rc$ and $m\in \lc1, \ppp,(p+q)d\rc$,
	$$\nu_m(z,j_0,j,\textbf{e})= \frac{\Gc^\Phi(z,j)}{D^\Phi(z)}P_m \text{com}(\Gc^\Phi(z,0))^TX_{j_0+1}(z)^{-1}\begin{pmatrix} A_{j_0,q}^{-1}\textbf{e}\\0\\\vdots\\0 \end{pmatrix}.$$
	If we define
	\begin{equation}\label{defDChap}
		\forall z\in B(1,\delta_1), \forall j_0\in \Z, \forall \textbf{e}\in \C^d, \quad\begin{pmatrix} 
			\widehat{D}_1(z,j_0,\textbf{e}) \\
			\vdots \\
			\widehat{D}_{d(p+q)}(z,j_0,\textbf{e})  \end{pmatrix}:=\text{com}(\Gc^\Phi(z,0))^TX_{j_0+1}(z)^{-1}\begin{pmatrix} A_{j_0,q}^{-1}\textbf{e}\\0\\\vdots\\0 \end{pmatrix},
	\end{equation}
	then the functions $\widehat{D}_m(\cdot,j_0,\textbf{e})$ is holomorphic on $B(1,\delta_1)$ and we have:
	$$\forall z\in B(1,\delta_1)\backslash\lc1\rc, \forall j_0,j\in \Z, \forall \textbf{e}\in \C^d, \quad\nu_m(z,j_0,j,\textbf{e}) = \frac{\widehat{D}_m(z,j_0,\textbf{e})}{D^\Phi(z)} \Phi_m(z,j).$$
	Therefore, \eqref{Wr2} and \eqref{Wl2} can be rewritten for $z\in B(1,\delta_1)\backslash\lc 1\rc$ as
	\begin{subequations}\label{expW3}
		\begin{align}
			\forall j \geq j_0+1, \quad& W(z,j_0,j,\textbf{e})= -\sum_{m=1}^{dp}\frac{\widehat{D}_m(z,j_0,\textbf{e})}{D^\Phi(z)} \Phi_m(z,j), \label{Wr3}\\
			\forall j  \leq j_0, \quad& W(z,j_0,j,\textbf{e})=\sum_{m=dp+1}^{(p+q)d} \frac{\widehat{D}_m(z,j_0,\textbf{e})}{D^\Phi(z)} \Phi_m(z,j).\label{Wl3}
		\end{align}
	\end{subequations}
	
	Thus, recalling that $1$ is a simple zero of the Evans function $D^\Phi$, the expressions \eqref{Wr3} and \eqref{Wl3} allow us to conclude that the spatial Green's function has been meromorphically extended on $B(1,\delta_1)\backslash\lc1\rc$ with a pole of order $1$ at $1$. 
	
	\subsubsection{Decomposing the function \texorpdfstring{$\widehat{D}_m$}{Dm}}\label{subsubsec:Dm}
	
	Now that we have found the meromorphic extension of the spatial Green's function near $1$ using the family $\left(\Phi_m(z,j)\right)_{m\in\lc1,\ppp,d(p+q)\rc}$, we will use the other families $\left(W_1^\pm(z,j),\ppp,W_{d(p+q)}^\pm(z,j)\right)$ for which we know precisely the behavior as $j$ tends towards $\pm\infty$ to find the new improved expressions \eqref{ExpGs} of the spatial Green's function of Proposition \ref{prop:GS_near_1}. This will rely on finding precise expressions for $\widehat{D}_m$ defined by \eqref{defDChap}.
	
	We begin this section by introducing the vectors
	\begin{subequations}\label{def:C_Ctilde}
		\begin{align}
			\forall z\in B(1,\delta_1),\forall j_0\in\Z, \forall \textbf{e}\in \C^d, \quad	&\Ccc^\pm(z,j_0,\textbf{e})=\begin{pmatrix}\Ccc^\pm_1(z,j_0,\textbf{e}) \\ \vdots \\ \Ccc^\pm_{d(p+q)}(z,j_0,\textbf{e})\end{pmatrix} := \Gc^\pm(z,j_0+1)^{-1}\begin{pmatrix}	A_{j_0,q}^{-1}\textbf{e} \\0 \\ \vdots \\ 0 \end{pmatrix}, \label{defC}\\
			\forall z\in B(1,\delta_1),\forall j_0\in\Z, \forall \textbf{e}\in \C^d, \quad	&\widetilde{\Ccc}^\pm(z,j_0,\textbf{e})=\begin{pmatrix}\widetilde{\Ccc}^\pm_1(z,j_0,\textbf{e}) \\ \vdots \\ \widetilde{\Ccc}^\pm_{d(p+q)}(z,j_0,\textbf{e})\end{pmatrix} := \widetilde{\Gc}^\pm(z,j_0+1)^{-1}\begin{pmatrix}	A_{j_0,q}^{-1}\textbf{e} \\0 \\ \vdots \\ 0 \end{pmatrix}.\label{defCtilde}
		\end{align}
	\end{subequations}
	
	Using \eqref{eg:GcGcTilde} and \eqref{def:theta_su}, we obtain that for $z\in B(1,\delta_1)$, $j_0\in\Z$ and $ \textbf{e}\in \C^d$
	\begin{subequations}\label{eg:CccCccTilde}
		\begin{align}
			\widetilde{\Ccc}^+_1(z,j_0,\textbf{e})	&=\theta_{s,1}\Ccc^+_1(z,j_0,\textbf{e}), \label{eg:CccCccTilde1}\\
			\forall m \in I_{ss}\backslash\lc 1\rc, \quad \widetilde{\Ccc}^+_m(z,j_0,\textbf{e})	&=\Ccc^+_m(z,j_0,\textbf{e})	+\theta_{s,m}\Ccc^+_1(z,j_0,\textbf{e}), \label{eg:CccCccTilde2}\\
			\forall m \in \lc 1,\ppp,d(p+q)\rc\backslash I_{ss}, \quad \widetilde{\Ccc}^+_m(z,j_0,\textbf{e})	&=\Ccc^+_m(z,j_0,\textbf{e})	, \label{eg:CccCccTilde3}\\
			\forall m \in \lc 1,\ppp,d(p+q)\rc\backslash I_{su}, \quad \widetilde{\Ccc}^-_m(z,j_0,\textbf{e})	&=\Ccc^-_m(z,j_0,\textbf{e})	, \label{eg:CccCccTilde4}\\
			\forall m \in I_{su}\backslash\lc d(p+q)\rc, \quad \widetilde{\Ccc}^-_m(z,j_0,\textbf{e})	&=\Ccc^-_m(z,j_0,\textbf{e})	+\theta_{u,m}\Ccc^-_{d(p+q)}(z,j_0,\textbf{e}),\label{eg:CccCccTilde5} \\
			\widetilde{\Ccc}^-_{d(p+q)}(z,j_0,\textbf{e})	&=\theta_{u,d(p+q)}\Ccc^-_{d(p+q)}(z,j_0,\textbf{e}).\label{eg:CccCccTilde6}
		\end{align}
	\end{subequations}
	
	In Section \ref{subsubsec:Borne} and more precisely in Lemma \ref{lem:Delta} below, we will prove estimates to bound $\widetilde{\Ccc}^\pm_m$. However, the functions $\widetilde{\Ccc}^\pm_m$ will be put on the side for now and will naturally reappear later on in Section \ref{subsubsec:ReexpGreenSpatial} using \eqref{eg:CccCccTilde}. For now, we will mainly focus on properties linked to the functions $\Ccc^\pm_m$. 
	
	We introduce the matrices
	\begin{equation}\label{defM}
		\forall z\in B(1,\delta_1),\quad \Mc^\pm(z):= \Gc^\pm(z,0)^{-1}\Gc^\Phi(z,0) \quad \text{ and } \quad \left(g_{m^\prime,m}^\pm(z)\right)_{(m^\prime,m)\in \lc1,\ppp,d(p+q)\rc}:=\text{com}\left(\Mc^\pm(z)\right).
	\end{equation}
	Using the definition \eqref{defDChap} of $\widehat{D}_m$, we have that:
	$$\forall z\in B(1,\delta_1),\forall j_0\in\Z,\forall \textbf{e}\in\C^d, \quad \begin{pmatrix}
	\widehat{D}_1(z,j_0,\textbf{e})\\ \vdots \\ \widehat{D}_{d(p+q)}(z,j_0,\textbf{e})
	\end{pmatrix} = D^\pm(z) \mathrm{com}(\Mc^\pm(z))^T\Ccc^\pm(z,j_0,\textbf{e})$$
	and thus:
	\begin{equation}\label{eg:Dm}
		\forall m\in \lc 1, \ppp, d(p+q)\rc,\forall z\in B(1, \delta_1), \forall j_0\in \Z, \forall \textbf{e}\in \C^d, \quad \widehat{D}_m(z,j_0,\textbf{e})= D^\pm(z)\sum_{m^\prime=1}^{d(p+q)}g_{m^\prime,m}^\pm(z)\Ccc^\pm_{m^\prime}(z,j_0,\textbf{e}).
	\end{equation}
	Looking at \eqref{expW3}, we are interested in studying the quotient of $\widehat{D}_m(\cdot, j_0, \textbf{e})$ and $D^\Phi$. In order to have lighter expressions later on, we also introduce the functions $\widetilde{g}_{m^\prime,m}^\pm$ defined by:
	\begin{equation}\label{def:gtilde}
		\forall m,m^\prime\in\lc1, \ppp,d(p+q)\rc, \forall z\in B(1, \delta_1)\backslash\lc1\rc, \quad \widetilde{g}_{m^\prime,m}^\pm(z):=D^\pm(z)\frac{g_{m^\prime,m}^\pm(z)}{D^\Phi(z)}.
	\end{equation}
	The function $\widetilde{g}_{m^\prime,m}^\pm$ is meromorphic on $B(1,\delta_1)\backslash\lc1\rc$ with a pole of order at most $1$ at $1$ since $1$ is a simple zero of the Evans function $D^\Phi$. If $g_{m^\prime,m}^\pm(1)=0$, it can thus be extended holomorphically on the whole ball $B(1,\delta_1)$.
	
	This decomposition of the functions $\widehat{D}_m$ will be used in Section \ref{subsubsec:ReexpGreenSpatial} with \eqref{expW3} to obtain a better expression of the spatial Green's function. We end this section by proving the following lemma using \eqref{egPhi}. 
	
	\begin{lemma}\label{lemCofac}
		\begin{enumerate}[label=(\roman*)]
			\item For $m^\prime\in \lc 1, \ppp, dp\rc$ and $m\in\lc1,\ppp,d(p+q)\rc$, we have
			$$\forall z\in B(1,\delta_1),\quad g_{m^\prime,m}^+(z)=\delta_{m^\prime,m}\frac{D^\Phi(z)}{D^+(z)}.$$
			
			\item For $m\notin\lc 1, d(p+q)\rc$ and $m^\prime\in\lc1,\ppp,d(p+q)\rc$, we have
			$$g_{m^\prime,m}^+(1)=0.$$
			
			\item For $m^\prime\in \lc dp+1, \ppp, d(p+q)\rc$ and $m\in\lc1,\ppp,d(p+q)\rc$, we have
			$$\forall z\in B(1,\delta_1),\quad g_{m^\prime,m}^-(z)=\delta_{m^\prime,m}\frac{D^\Phi(z)}{D^-(z)}.$$
			
			\item For $m\notin\lc 1, d(p+q)\rc$ and $m^\prime\in\lc1,\ppp,d(p+q)\rc$, we have
			$$g_{m^\prime,m}^-(1)=0.$$
			
			\item We have
			$$g^+_{1,1}(1)=g^+_{1,d(p+q)}(1)=0\quad \text{ and } \quad g^-_{d(p+q),d(p+q)}(1)=g^-_{d(p+q),1}(1)=0.$$
			
			\item For $m^\prime \in\lc1, \ppp,d(p+q)\rc$, we have
			\begin{equation}\label{eg:Cofac}
				g_{m^\prime,1}^+(1) = -  g_{m^\prime,d(p+q)}^+(1 )\quad \text{ and } \quad  g_{m^\prime,1}^-(1) = -  g_{m^\prime,d(p+q)}^-(1 ).
			\end{equation}
		\end{enumerate}
	\end{lemma}
	
	Lemma \ref{lemCofac} determines the couple of indexes $m,m^\prime\in\lc1,\ppp,d(p+q)\rc$ such that $g_{m^\prime,m}^\pm$ is equal to 0. This allows us to remove many terms in \eqref{eg:Dm}. It also determines the couple of indexes $m,m^\prime\in\lc1,\ppp,d(p+q)\rc$ such that $g_{m^\prime,m}^\pm$ vanishes at $z=1$ which implies that the function $\widetilde{g}_{m^\prime,m}^\pm$ defined by \eqref{def:gtilde} can be holomorphically extended on the whole ball $B(1,\delta_1)$.
	
	\begin{proof}
		We will focus on proving the statements involving $g^+_{m^\prime,m}$ since every statement involving $g^-_{m^\prime,m}$ will have similar proofs. We observe that the definition \eqref{defPhi} of $\Phi_m(z,0)$ implies that if we define for $z\in B(1,\delta_1)$ and $m\in\lc1,\ppp,d(p+q)\rc$
		$$C_m^+(z):=\Gc^+(z,0)^{-1}\Phi_m(z,0)$$
		then, for $m\in \lc 1,\ppp,dp\rc$, we have
		$$C_m^+(z)=e_m$$
		where $(e_j)_{j\in\lc1,\ppp,d(p+q)\rc}$ is the canonical basis of $\C^{d(p+q)}$. Thus,
		\begin{equation}\label{calcMc}
			\Mc^+(z)=\left(\begin{array}{c|c|c|c}
				\begin{array}{c}
					I_{dp}\\\hline
					0
				\end{array} & C_{dp+1}^+(z) & \ppp & C_{d(p+q)}^+(z) 
			\end{array}\right).
		\end{equation}
		For $m,m^\prime\in \lc 1, \ppp, d(p+q)\rc$, we recall that $g_{m^\prime,m}^+(z)$ is the $(m^\prime,m)$-cofactor of the matrix above. 
		
		We observe that
		\begin{equation}\label{lemCofac:eg1}
			\mathrm{com}(\Mc^+(z))^T\Mc^+(z)=\frac{D^\Phi(z)}{D^+(z)}I_{d(p+q)}.
		\end{equation}
		Looking at the first $dp$ columns of $\Mc^+(z)$ in \eqref{calcMc} and using \eqref{lemCofac:eg1}, we then conclude that the first $dp$ lines of $\mathrm{com}(\Mc^+(z))$ are equal to
		$$\left(\begin{array}{c|c}
			\frac{D^\Phi(z)}{D^+(z)}I_{dp} & 0
		\end{array}\right)$$
		which implies Point (i).
		
		The equality \eqref{egPhi} also implies that:
		\begin{equation*}
			C_{d(p+q)}^+(1) = e_1 =C_1^+(1).
		\end{equation*}
		and thus:
		$$\Mc^+(1)=\left(\begin{array}{c|c|c|c|c}
			\begin{array}{c}
				I_{dp}\\\hline
				0
			\end{array} & C_{dp+1}^+(1) & \ppp & C_{d(p+q)-1}^+(1) &  \begin{array}{c}
				1\\
				0\\
				\vdots\\
				0
			\end{array}
		\end{array}\right).$$
		Points (ii) and (v) are then easily deduced by equality between the first and last columns of the matrix above. 
		
		There just remains to prove Point (vi). We observe that 
		$$\Mc^+(1) \mathrm{com}(\Mc^+(1))^T=\frac{D^\Phi(1)}{D^+(1)}I_{d(p+q)}=0.$$
		Looking at the coefficient at the first line and $m^\prime$-th column, we have that
		$$g_{m^\prime,1}^+(1) + \sum_{m=dp+1}^{d(p+q)-1} (C_m^+(1))_1 g_{m^\prime,m}^+(1) + g_{m^\prime,d(p+q)}^+(1)=0. $$
		Using Point (ii), we easily conclude the proof of Point (vi).
	\end{proof}

	\subsubsection{Final expression of the spatial Green's function}\label{subsubsec:ReexpGreenSpatial}
	
	We will now prove the expressions \eqref{ExpGs} on the spatial Green's function when $j_0$ is larger than $0$ (i.e. the location of the initial perturbation $\delta_{j_0}$ is on the right of the shock). The case where $j_0<0$ would be handled similarly and give another expression of the spatial Green's function on $B(1,\delta_1)$ that would be necessary to prove a decomposition of the temporal Green's function similar to \eqref{decompoGreen} when $j_0<0$.
	
	\textbf{\underline{Case where $j_0\geq 0$ and $j\geq j_0+1$:}}
	
	Using \eqref{Wr3} and \eqref{eg:Dm}, we have 
	$$ W(z,j_0,j,\textbf{e})= -\sum_{m=1}^{dp} \sum_{m^\prime=1}^{d(p+q)}  \widetilde{g}_{m^\prime,m}^+(z)\Ccc^+_{m^\prime}(z,j_0,\textbf{e})\Phi_m(z,j).$$
	Point (i) of Lemma \ref{lemCofac} then implies that 
	$$W(z,j_0,j,\textbf{e})= -\sum_{m=1}^{dp}\Ccc^+_{m}(z,j_0,\textbf{e})\Phi_m(z,j) -\sum_{m=1}^{dp} \sum_{m^\prime=dp+1}^{d(p+q)}  \widetilde{g}_{m^\prime,m}^+(z) \Ccc^+_{m^\prime}(z,j_0,\textbf{e})\Phi_m(z,j).$$
	Using the definition \eqref{defPhi} of $\Phi_1$, we then have that
	\begin{multline*}
		W(z,j_0,j,\textbf{e})= -\theta_{s,1} \Ccc^+_{1}(z,j_0,\textbf{e})W^+_1(z,j) -\sum_{m\in I_{ss}\backslash\lc1\rc} (\Ccc^+_{m}(z,j_0,\textbf{e})+\theta_{s,m}\Ccc^+_{1}(z,j_0,\textbf{e}))W^+_m(z,j)\\ -\sum_{m\in I_{cs}} \Ccc^+_{m}(z,j_0,\textbf{e})W^+_m(z,j)-\sum_{m=1}^{dp} \sum_{m^\prime=dp+1}^{d(p+q)}  \widetilde{g}_{m^\prime,m}^+(z)\Ccc^+_{m^\prime}(z,j_0,\textbf{e})\Phi_m(z,j).
	\end{multline*}
	Using \eqref{eg:CccCccTilde1}-\eqref{eg:CccCccTilde3} which link the functions $\Ccc^\pm_m$ and $\widetilde{\Ccc}^\pm_m$ and the definition \eqref{defPhi} of $\Phi_m$ for $m\in\lc2,\ppp,dp\rc$, we obtain that \eqref{ExpGs: l geq 0 : j geq l+1} is verified, i.e. :
	\begin{multline*}
		W(z,j_0,j,\textbf{e})= -\sum_{m=1}^{dp}  \widetilde{\Ccc}^+_{m}(z,j_0,\textbf{e})W^+_m(z,j) -\sum_{m=2}^{dp} \sum_{m^\prime=dp+1}^{d(p+q)}  \widetilde{g}_{m^\prime,m}^+(z)\widetilde{\Ccc}^+_{m^\prime}(z,j_0,\textbf{e})W^+_m(z,j)\\- \sum_{m^\prime=dp+1}^{d(p+q)}  \widetilde{g}_{m^\prime,1}^+(z)\widetilde{\Ccc}^+_{m^\prime}(z,j_0,\textbf{e})\Phi_1(z,j).
	\end{multline*}
	
	\begin{remark}
		If $m\in\lc2,\ppp,dp\rc$ and $m^\prime\in \lc dp+1,\ppp,d(p+q)\rc$, we have that $g^+_{m^\prime,m}(1)=0$ because of Lemma \ref{lemCofac}.  Thus, the function 
		$$z\in B(1,\delta_1)\backslash\lc1\rc\mapsto  \widetilde{g}_{m^\prime,m}^+(z)\widetilde{\Ccc}^+_{m^\prime}(z,j_0,\textbf{e})\Phi_m(z,j)$$ can be holomorphically extended on the whole ball $B(1,\delta_1)$.
	\end{remark}
	
	\textbf{\underline{Case where $j_0\geq 0$ and $j\in \lc0,\ppp,j_0\rc$:}}
	
	Let us consider $m\in\lc dp+1,\ppp,d(p+q)\rc$. Using the respective definitions \eqref{def:GcPhi}, \eqref{def:Gc+} and \eqref{defM} of the matrices $\Gc^\Phi$, $\Gc^+$ and $\Mc^+$, we have the following expression of the vector $\Phi_m(z,j)$ depending on the family $(W_k^+(z,j))_{k\in\lc1,\ppp,d(p+q)\rc}$:
	$$ \Phi_m(z,j) = \Gc^+(z,j)(\Mc^+(z)_{k,m})_{k\in\lc1,\ppp,d(p+q)\rc} =\Mc^+(z)_{1,m}\Phi_1(z,j)+\sum_{k=2}^{d(p+q)} \Mc^+(z)_{k,m}W_k^+(z,j).$$
	Thus, since $m\in\lc dp+1,\ppp,d(p+q)\rc$, using \eqref{eg:Dm} and the fact that Lemma \ref{lemCofac} implies $g^+_{m^\prime,m}=0$ for $m^\prime\in\lc1,\ppp,dp\rc$  we have
	\begin{multline*}
		\frac{\widehat{D}_m(z,j_0,\textbf{e})}{D^\Phi(z)} \Phi_m(z,j) \\ = \frac{D^+(z)}{D^\Phi(z)}\sum_{m^\prime=dp+1}^{d(p+q)}g^+_{m^\prime,m}(z)\Ccc^+_{m^\prime}(z,j_0,\textbf{e}) \left(\Mc^+(z)_{1,m}\Phi_1(z,j)+\sum_{k=2}^{d(p+q)} \Mc^+(z)_{k,m}W_k^+(z,j)\right).
	\end{multline*}
	Using \eqref{Wl3}, we then have that:
	\begin{multline}\label{eg:temp1}
		W(z,j_0,j,\textbf{e})\\= \sum_{m^\prime=dp+1}^{d(p+q)} \left[\left(\sum_{m=dp+1}^{d(p+q)}\Mc^+(z)_{1,m}g^+_{m^\prime,m}(z)\right)\Phi_1(z,j)+\sum_{k=2}^{d(p+q)}\left(\sum_{m=dp+1}^{d(p+q)}\Mc^+(z)_{k,m}g^+_{m^\prime,m}(z)\right)W^+_k(z,j)\right]\\ \frac{D^+(z)}{D^\Phi(z)}\Ccc^+_{m^\prime}(z,j_0,\textbf{e}) .
	\end{multline}	
	
	Let us find expressions for the sums $\sum_{m=dp+1}^{d(p+q)}\Mc^+(z)_{k,m}g^+_{m^\prime,m}(z)$ when $m^\prime\in\lc dp+1,\ppp,d(p+q)\rc$ and $k\in\lc1,\ppp,d(p+q)\rc$. We recall that $g_{m^\prime,m}^+(z)$ is the $(m^\prime,m)$-cofactor of the matrix $\Mc^+(z)$. Furthermore, by definition \eqref{defM} of the matrix $\Mc^+$, we have that
	\begin{equation}\label{egM+}
		\Mc^+(z)\mathrm{com}(\Mc^+(z))^T=\frac{D^\Phi(z)}{D^+(z)}Id.
	\end{equation}
	Thus, by observing \eqref{calcMc} implies that for $k\in\lc dp+1,\ppp,d(p+q)\rc$ and $m\in\lc1,\ppp,dp\rc$ we have $\Mc^+(z)_{k,m}=0$, we conclude looking at the $k$-th line and $m^\prime$-th column of \eqref{egM+} that
	\begin{equation}\label{eg:temp2}
		\forall k\in\lc dp+1,\ppp,d(p+q)\rc,\forall m^\prime\in\lc dp+1,\ppp,d(p+q)\rc,\quad \sum_{m=dp+1}^{d(p+q)}\Mc^+(z)_{k,m}g^+_{m^\prime,m}(z)=\frac{D^\Phi(z)}{D^+(z)}\delta_{k,m^\prime}.
	\end{equation}
	Furthermore, \eqref{calcMc} implies that
	$$\forall k\in\lc1,\ppp,dp\rc,\forall m \in\lc1,\ppp,dp\rc,\quad \Mc^+(z)_{k,m}=\delta_{k,m}.$$
	Thus, looking once again at the $k$-th line and $m^\prime$-th column of \eqref{egM+}, we have
	\begin{equation}\label{eg:temp3}
		\forall k\in\lc 1,\ppp,dp\rc,\forall m^\prime\in\lc dp+1,\ppp,d(p+q)\rc,\quad \sum_{m=dp+1}^{d(p+q)}\Mc^+(z)_{k,m}g^+_{m^\prime,m}(z)=-g_{m^\prime,k}^+(z).
	\end{equation}
	We finally conclude using \eqref{eg:CccCccTilde1}-\eqref{eg:CccCccTilde3} which links the functions $\Ccc^\pm_m$ and $\widetilde{\Ccc}^\pm_m$ and combining \eqref{eg:temp1}, \eqref{eg:temp2} and \eqref{eg:temp3} that \eqref{ExpGs: l geq 0 : 0 leq j leq l} is verified, i.e.:
	\begin{multline*}
		W(z,j_0,j,\textbf{e})= \sum_{m=dp+1}^{d(p+q)}\widetilde{\Ccc}_{m}^+(z,j_0,\textbf{e}) W_{m}^+(z,j) -\sum_{m=2}^{dp} \sum_{m^\prime=dp+1}^{d(p+q)}  \widetilde{g}_{m^\prime,m}^+(z)\widetilde{\Ccc}^+_{m^\prime}(z,j_0,\textbf{e})W^+_m(z,j)\\- \sum_{m^\prime=dp+1}^{d(p+q)}  \widetilde{g}_{m^\prime,1}^+(z)\widetilde{\Ccc}^+_{m^\prime}(z,j_0,\textbf{e})\Phi_1(z,j).
	\end{multline*}
	
	\begin{remark}
		We observe that some terms in \eqref{ExpGs: l geq 0 : 0 leq j leq l} are equal to terms of \eqref{ExpGs: l geq 0 : j geq l+1}. We will see later on that they contribute to the reflected waves in the decomposition of Theorem \ref{th:Green}.
	\end{remark}
	
	\textbf{\underline{Case where $j_0\geq 0$ and $j<0$:}}
	
	Using \eqref{Wl3} and \eqref{eg:Dm}, 
	$$W(z,j_0,j,\textbf{e})= \sum_{m=dp+1}^{d(p+q)} \sum_{m^\prime=1}^{d(p+q)}  \widetilde{g}_{m^\prime,m}^+(z)\Ccc^+_{m^\prime}(z,j_0,\textbf{e})\Phi_m(z,j).$$
	Lemma \ref{lemCofac} implies that
	$$\forall m\in\lc dp+1,\ppp,d(p+q)\rc,\forall m^\prime\in\lc1,\ppp,dp\rc,\quad g^+_{m^\prime,m}=0.$$
	Thus, using \eqref{eg:CccCccTilde1}-\eqref{eg:CccCccTilde3} which links the functions $\Ccc^\pm_m$ and $\widetilde{\Ccc}^\pm_m$, we have that \eqref{ExpGs: l geq 0 : j < 0} is verified, i.e.:
	\begin{multline*}
		W(z,j_0,j,\textbf{e})= \sum_{m=dp+1}^{d(p+q)-1} \sum_{m^\prime=dp+1}^{d(p+q)}  \widetilde{g}_{m^\prime,m}^+(z)\widetilde{\Ccc}^+_{m^\prime}(z,j_0,\textbf{e})W^-_m(z,j)\\+\sum_{m^\prime=dp+1}^{d(p+q)}  \widetilde{g}_{m^\prime,d(p+q)}^+(z)\widetilde{\Ccc}^+_{m^\prime}(z,j_0,\textbf{e})\Phi_{d(p+q)}(z,j).
	\end{multline*}
	
	\begin{remark}
		If $m\in\lc dp+1,\ppp,d(p+q)-1\rc$ and $m^\prime\in \lc dp+1,\ppp,d(p+q)\rc$, we have that $g^+_{m^\prime,m}(1)=0$.  Thus, the function 
		$$z\in B(1,\delta_1)\backslash\lc1\rc\mapsto  \widetilde{g}_{m^\prime,m}^+(z)\widetilde{\Ccc}^+_{m^\prime}(z,j_0,\textbf{e})\Phi_m(z,j)$$ can be holomorphically extended on the whole ball $B(1,\delta_1)$.
	\end{remark}
	
	\subsubsection{Useful estimates}\label{subsubsec:Borne}
	
	In this section, we will introduce the necessary observations to properly bound the terms appearing in the decomposition of the spatial Green's function of Section \ref{subsubsec:ReexpGreenSpatial}. We will in particular introduce a new expression of the functions $\widetilde{\Ccc}^\pm_m(z,j_0,\textbf{e})$, prove that they roughly act like $\zeta_m^\pm(z)^{-j_0}$ and determine their behavior as $j_0$ tends towards $\pm\infty$.
	
	For $z\in B(1,\delta_1)$ and $j\in \Z$, we recall that Lemma \ref{lemBasis} implies that $(V^\pm_m(z,j))_{m\in\lc1, \ppp,d(p+q)\rc}$ is a basis of $\C^{d(p+q)}$. Thus, we can define for $z\in B(1,\delta_1)$, $j_0\in\Z$ and $\textbf{e}\in\C^d$
	\begin{equation}
		N^\pm(z,j_0):=\left(\begin{array}{c|c|c}
			V_1^\pm(z,j_0+1) & \hdots  &V_{d(p+q)}^\pm(z,j_0+1) 
		\end{array}\right)\quad \text{and} \quad \begin{pmatrix}
			\Delta_1^\pm(z,j_0,\textbf{e})\\ \vdots \\ \Delta_{d(p+q)}^\pm(z,j_0,\textbf{e})
		\end{pmatrix}:=  N^\pm(z,j_0)^{-1}\begin{pmatrix}	A_{j_0,q}^{-1}\textbf{e} \\0 \\ \vdots \\ 0 \end{pmatrix}.\label{def:NDelta}
	\end{equation}
	
	We observe that \eqref{defCtilde} implies that for all $z\in B(1,\delta_1)$, $j_0\in \Z$ and $\textbf{e}\in\C^d$, we have
	$$\sum_{m=1}^{d(p+q)} \zeta_m^\pm(z)^{j_0+1} \widetilde{\Ccc}^\pm_m(z,j_0,\textbf{e}) V_m^\pm(z,j_0+1)= \begin{pmatrix}	A_{j_0,q}^{-1}\textbf{e} \\0 \\ \vdots \\ 0 \end{pmatrix}.$$
	Thus, we have that for $m\in\lc1, \ppp,d(p+q)\rc$, $z\in B(1,\delta_1)$, $j_0\in\Z$ and $\textbf{e}\in\C^d$
	\begin{equation}\label{eg:CccTilde}
		\widetilde{\Ccc}^\pm_m(z,j_0,\textbf{e}) =\zeta_m^\pm(z)^{-j_0-1} \Delta_m^\pm(z,j_0,\textbf{e}).
	\end{equation}
	We now prove the following lemma which gives us the asymptotic behavior of $\Delta_m^\pm(z,j_0,\textbf{e})$.
	
	\begin{lemma}\label{lem:Delta}
		There exist a radius $\delta_2\in]0,\delta_1]$ and two constants $C,c>0$ such that for all $z\in B(1,\delta_2)$, $m=l+(k-1)d\in\lc1,\ppp,d(p+q)\rc$ with $k\in\lc1,\ppp,p+q\rc$ and $l\in\lc1,\ppp,d\rc$ and $\textbf{e}\in\C^d$, we have
		\begin{subequations}
			\begin{align}
				\forall j\in \N,  \quad &  \left|V^+_m(z,j)-V^+_m(1,j)\right|\leq C|z-1|,\label{lem:BornesVDelta:inV+}\\
				\forall j\in -\N,  \quad & \left|V^-_m(z,j)-V^-_m(1,j)\right|\leq C|z-1|,\label{lem:BornesVDelta:inV-} \\
				\forall j\in \N, \quad & \left|\Phi_1(z,j)-\Phi_1(1,j)\right|\leq C|z-1|e^{-\frac{3c_*}{2}|j|},\label{lem:BornesVDelta:inPhi1}\\
				\forall j\in -\N, \quad & \left|\Phi_{d(p+q)}(z,j)-\Phi_{d(p+q)}(1,j)\right|\leq C|z-1|e^{-\frac{3c_*}{2}|j|},\label{lem:BornesVDelta:inPhid(p+q)}\\
				\forall j_0\in\N,\quad & \left|\Delta_m^+(z,j_0,\textbf{e}) - \frac{d\zeta_{m}^+}{dz}(z) {\lg_l^+}^T\textbf{e} \right|\leq C|\textbf{e}|e^{-c|j_0|},\label{in:DeltaInf+}\\
				\forall j_0\in-\N,\quad &\left|\Delta_m^-(z,j_0,\textbf{e}) - \frac{d\zeta_{m}^-}{dz}(z) {\lg_l^-}^T\textbf{e} \right|\leq C|\textbf{e}|e^{-c|j_0|},\label{in:DeltaInf-}\\
				\forall j_0\in\N,\quad & \left|\Delta_m^+(z,j_0,\textbf{e}) \right|\leq C|\textbf{e}|,\label{in:Delta+}\\
				\forall j_0\in-\N,\quad &\left|\Delta_m^-(z,j_0,\textbf{e}) \right|\leq C|\textbf{e}|.\label{in:Delta-}\\
				\forall j_0\in \N,  \quad & \left|\Delta^+_m(z,j_0,\textbf{e})-\Delta^+_m(1,j_0,\textbf{e})\right|\leq C|z-1||\textbf{e}|, \label{lem:BornesVDelta:inDelta+}\\
				\forall j_0\in -\N,  \quad & \left|\Delta^-_m(z,j_0,\textbf{e})-\Delta^-_m(1,j_0,\textbf{e})\right|\leq C|z-1||\textbf{e}|,\label{lem:BornesVDelta:inDelta-}
			\end{align}
		\end{subequations}
		where $c_*$ is the positive constant in \eqref{inZeta}.
	\end{lemma}
	
	\begin{proof}
		The first two inequalities are direct consequences from the fact that $z\in B(1,\delta_1)\mapsto \frac{\partial V_m^\pm}{\partial z}(z,\cdot)\in \ell^\infty(\pm\N, \C^{d(p+q)})$ is bounded (see Lemma \ref{lem_choice_base}).
		
		We will prove \eqref{lem:BornesVDelta:inPhi1}. The proof of \eqref{lem:BornesVDelta:inPhid(p+q)} would be similar. Using the definition \eqref{defPhi} of $\Phi_1$, we conclude that we only have to prove that for all $m\in I_{ss}$ that there exists a constant $C>0$ and a radius $\delta_2\in]0,\delta_1]$ such that
		$$\forall z\in B(1,\delta_2), \forall j\in \N, \quad \left|W_m^+(z,j)-W_m^+(1,j)\right|\leq C|z-1|e^{-\frac{3c_*}{2}|j|}.$$
		We observe that for $z\in B(1,\delta_1)$ and $j\in \N$, we have 
		$$\left|W_m^+(z,j)-W_m^+(1,j)\right| \leq \left|\zeta_m^+(z)\right|^j \left|V_m^+(z,j)-V_m^+(1,j)\right|+\left|V_m^+(1,j)\right| \left|\zeta_m^+(z)^j-\zeta_m^+(1)^j\right|.$$
		Using \eqref{lem:BornesVDelta:inV+}, \eqref{inZetaSs} and the fact that $\frac{d\zeta_m^+}{dz}$ is bounded on $B(1,\delta_2)$ for $\delta_2\in]0,\delta_2[$, we easily conclude the proof of \eqref{lem:BornesVDelta:inPhi1}.
		
		We will focus on \eqref{in:DeltaInf+} as \eqref{in:DeltaInf-} would be proved in a similar way. We observe that Lemma \ref{lem:dualBasis}, \eqref{def:lambda_k} and \eqref{def:l} imply that:
		$$L_m^+(z)^T\begin{pmatrix}	{A^+_q}^{-1}\textbf{e} \\0 \\ \vdots \\ 0 \end{pmatrix} = \lambda_{l,q}^+ \frac{d\zeta_{m}^+}{dz}(z) {\lg_l^+}^T{A_q^+}^{-1}\textbf{e}=\frac{d\zeta_{m}^+}{dz}(z) {\lg_l^+}^T\textbf{e}.$$
		Thus, $\Delta_m^+(z,j_0,\textbf{e})-\frac{d\zeta_{m}^+}{dz}(z) {\lg_l^+}^T\textbf{e}$ is the $m$-th coefficient of the vector 
		$$ N^+(z,j_0)^{-1}\begin{pmatrix}	A_{j_0,q}^{-1}\textbf{e} \\0 \\ \vdots \\ 0 \end{pmatrix}-N^{+,\infty}(z)^{-1}\begin{pmatrix}	{A_{q}^+}^{-1}\textbf{e} \\0 \\ \vdots \\ 0 \end{pmatrix}$$
		where the matrix $N^{+,\infty}$ is defined by \eqref{def:Ninf}. We then just have to find bounds for this difference of vectors. We have
		\begin{align}\label{lem:Delta:eg1}
			\begin{split}
				&N^+(z,j_0)^{-1}\begin{pmatrix}	A_{j_0,q}^{-1}\textbf{e} \\0 \\ \vdots \\ 0 \end{pmatrix}-N^{+,\infty}(z)^{-1}\begin{pmatrix}	{A_{q}^\pm}^{-1}\textbf{e} \\0 \\ \vdots \\ 0 \end{pmatrix} \\
				=&  N^+(z,j_0)^{-1}\left(N^{+,\infty}(z)-N^+(z,j_0)\right)N^{+,\infty}(z)^{-1}\begin{pmatrix}	A_{j_0,q}^{-1}\textbf{e} \\0 \\ \vdots \\ 0 \end{pmatrix} +N^{+,\infty}(z)^{-1}\begin{pmatrix}	(A_{j_0,q}^{-1}-{A^+_q}^{-1})\textbf{e} \\0 \\ \vdots \\ 0 \end{pmatrix}.
			\end{split}
		\end{align}
		
		We wish to bound each term in the right-hand side of the equality \eqref{lem:Delta:eg1}. Let us start by looking at the first term. Lemma \ref{lem_choice_base} implies that the functions $N^+(\cdot,j_0)^{-1}$ are bounded on $B(1,\delta_1)$ and that the bound can considered to be uniform for $j_0\in\N$. The function $N^{+,\infty}(\cdot)^{-1}$ is also bounded on $B(1,\delta_1)$. Since $A_{j_0,q}^{-1}$ converges towards ${A_q^+}^{-1}$ as $j_0$ converges towards $+\infty$, we also have that the family of matrices $(A_{j_0,q}^{-1})_{j_0\in\N}$ is bounded. Finally, using Lemma \ref{lem_choice_base}, we have that there exist two constants $C,c>0$ such that
		$$\forall z\in B(1,\delta_1),\forall j_0\in\N, \quad \left|N^{+,\infty}(z) - N^+(z,j_0)\right|\leq Ce^{-cj_0}.$$
		Thus, there exists another constant $C>0$ such that 
		$$\forall z\in B(1,\delta_1),\forall \textbf{e}\in\C^d,\forall j_0\in\N, \quad \left|N^+(z,j_0)^{-1}\left(N^{+,\infty}(z)-N^+(z,j_0)\right)N^{+,\infty}(z)^{-1}\begin{pmatrix}	A_{j_0,q}^{-1}\textbf{e} \\0 \\ \vdots \\ 0 \end{pmatrix}\right|\leq C|\textbf{e}|e^{-cj_0}.$$
		
		We now focus on the second term. The function $N^{\pm,\infty}(\cdot)^{-1}$ is bounded on $B(1,\delta_1)$. Furthermore, Hypothesis \ref{H:CVexpo} allows us to determine that there exist two constants $C,c>0$ such that
		$$\forall j_0\in \N, \quad \left|A_{j_0,q}^{-1}-{A_q^+ }^{-1}\right|\leq Ce^{-cj_0}. $$
		Therefore, there exists a new constant $C>0$ such that 
		$$\forall z\in B(1,\delta_1),\forall \textbf{e}\in\C^d,\forall j_0\in\N,\quad \left| N^{+,\infty}(z)^{-1}\begin{pmatrix}	(A_{j_0,q}^{-1}-{A^+_q}^{-1})\textbf{e} \\0 \\ \vdots \\ 0 \end{pmatrix}\right|\leq C|\textbf{e}|e^{-cj_0}.$$
		We can then conclude the proof of \eqref{in:DeltaInf+}.
		
		We observe that \eqref{in:Delta+} and \eqref{in:Delta-} are direct consequences of \eqref{in:DeltaInf+} and \eqref{in:DeltaInf-}. 
		
		There remains to prove \eqref{lem:BornesVDelta:inDelta+} as \eqref{lem:BornesVDelta:inDelta-} would be proved similarly. We observe that for $z\in B(1,\delta_2)$ and $j_0\in\N$, we have
		\begin{equation*}
			N^+(z,j_0)^{-1}-N^+(1,j_0)^{-1}=N^+(z,j_0)^{-1}\left(N^+(1,j_0)-N^+(z,j_0)\right)N^+(1,j_0)^{-1}.
		\end{equation*}
		Using \eqref{lem:BornesVDelta:inV+} and the observations above which claimed that $N^+(z,j_0)^{-1}$ is bounded uniformly for $z\in B(1,\delta_2)$ and $j_0\in\N$, we have that there exists a positive constant $C$ such that
		$$\forall z\in B(1,\delta_2),\forall j_0\in\N, \quad |N^+(z,j_0)^{-1}-N^+(1,j_0)^{-1}|\leq C|z-1|.$$
		The definition \eqref{def:NDelta} and the fact that the family of matrices $(A_{j_0,q}^{-1})_{j_0\in\N}$ is bounded imply that \eqref{lem:BornesVDelta:inDelta+} is verified for some constant $C>0$.
	\end{proof}
	
	\section{Temporal Green's function and proof of Theorem \ref{th:Green}}\label{sec:GT}
	
	The previous Sections \ref{sec:GSloin} and \ref{sec:GSnear1} served respectively to describe the spatial Green's function far from $1$ and near $1$. Our objective is now to focus on the core of the article: the study of temporal Green's function and the proof of Theorem \ref{th:Green}. In the present section, we will express the temporal Green's function with the spatial Green's function using functional analysis. We will then use the different results of the previous sections (mainly Propositions \ref{GreenSpatialFar} and \ref{prop:GS_near_1}) to obtain the result of Theorem \ref{th:Green}. Just as when we proved Proposition \ref{prop:GS_near_1} on the spatial Green's function near $1$, the proof of Theorem \ref{th:Green} will be done whilst assuming that $j_0$ is larger than $0$ to obtain \eqref{decompoGreen} and \eqref{decompoDerGreen}. The case where $j_0<0$ would be handled similarly and would necessitate to prove expressions of the spatial Green's function on $B(1,\delta_1)$ similar to \eqref{ExpGs: l geq 0 : j geq l+1}-\eqref{ExpGs: l geq 0 : j < 0} when $j_0<0$.
	
	\subsection{Link between the spatial and temporal Green's function}
	
	First, we recall that in Sections \ref{sec:GSloin} and \ref{sec:GSnear1}, we studied the vectors $W(z,j_0,j,\textbf{e})$ defined in Section \ref{subsec:BoundsGsLoin} which are composed of several components of the spatial Green's function. The inverse Laplace transform implies that if we introduce a path $\widetilde{\Gamma}$ that surrounds the spectrum $\sigma(\Lcc)$, for instance $\widetilde{\Gamma}_r:= r\S^1$ where $r>1$, then we have
	$$\forall n\in\N,\forall j_0,j\in\Z,\forall \textbf{e}\in\C^d,\quad \Gcc(n,j_0,j)\textbf{e} = \frac{1}{2i\pi} \int_{\widetilde{\Gamma}_r} z^nG(z,j_0,j)\textbf{e}dz= \frac{1}{2i\pi} \int_{\widetilde{\Gamma}_r} z^n\Pi(W(z,j_0,j,\textbf{e}))dz$$
	where $\Pi$ is the linear application defined by \eqref{def:Pi} which extracts the center values of a large vector. We consider the change of variables $z=\exp(\tau)$. If we define $\Gamma_r=\lc r+it, t\in [-\pi,\pi]\rc$, then we have
	\begin{equation}\label{egGSpaTemp1}
		\forall n\in\N,\forall j_0,j\in\Z,\forall \textbf{e}\in\C^d,\quad \Gcc(n,j_0,j)\textbf{e} = \frac{1}{2i\pi} \int_{\Gamma_r} e^{n\tau} e^\tau \Pi(W(e^\tau,j_0,j,\textbf{e}))d\tau.
	\end{equation} 
	The goal will now be to use Cauchy's formula and/or the residue theorem to modify our choice of path $\Gamma$ and to use at best the properties we proved on the spatial Green's function in Propositions \ref{GreenSpatialFar} and \ref{prop:GS_near_1}. 
	
	In Proposition \ref{prop:GS_near_1} of Section \ref{sec:GSnear1}, we proved that we can meromorphically extend the spatial Green's function on a ball $B(1,\delta_1)$ with a pole of order $1$ at $1$ and have found the decompositions \eqref{ExpGs} of the spatial Green's function. We also introduced an even smaller ball $B(1,\delta_2)$ on which we have more precise bounds (Lemma \ref{lem:Delta}) that will help us later on in the proof. We consider a radius $\varepsilon_0^\star\in]0,\pi[$ such that
	\begin{equation}\label{def:varepsilon_0^*}
		\forall \tau \in B(0,\varepsilon_0^\star),\quad e^\tau\in B(1,\delta_2).
	\end{equation}
	
	\begin{lemma}\label{lem:defOmegaEta}
		For all radii $\varepsilon\in ]0,\varepsilon_0^\star[$, there exists a width $\eta_\varepsilon>0$ such that if we define 
		$$\Omega_\varepsilon:= \lc\tau\in \C, \quad \Re \tau \in]-\eta_\varepsilon,\pi], \Im\tau\in[-\pi,\pi]\rc\cup B(0,\varepsilon),$$
		then for all $j,j_0\in \Z$ and $\textbf{e}\in\C^d$, the function $\tau\mapsto W(e^\tau,j_0,j,\textbf{e})$ is meromorphically defined on $\Omega_\varepsilon\backslash\lc0\rc$ with a pole of order $1$ at $0$ and there exist two positive constants $C_\varepsilon,c_\varepsilon$ such that
		\begin{equation}\label{in:GSfar}
			\forall \tau\in\Omega_\varepsilon\backslash B(0,\varepsilon),\forall j,j_0\in\Z, \forall \textbf{e}\in\C^d, \quad |W(e^\tau,j_0,j,\textbf{e})|\leq C_\varepsilon|\textbf{e}|e^{-c_\varepsilon|j-j_0|}.
		\end{equation}
	\end{lemma}
	
	Defining this width $\eta_\varepsilon$ is important for the following calculations since we have defined a set $\Omega_\varepsilon$ on which we can change the path of integration of \eqref{egGSpaTemp1} using the residue theorem. Furthermore, for $\tau\in \Omega_\varepsilon$, either $\tau\in B(0,\varepsilon)$ which implies that $e^\tau\in B(1,\delta_2)$ and that we can thus use the decomposition \eqref{ExpGs} of the spatial Green's function we obtained in Proposition \ref{prop:GS_near_1}, or $\tau\notin B(0,\varepsilon)$ and we can use \eqref{in:GSfar} to obtain exponential bounds on the spatial Green's function.
	
	\begin{proof}
		The proof is identical as \cite[Lemma 5.2]{CoeuIBVP} and will thus not be detailled. It just relies on observing that for any radius $\varepsilon\in]0,\varepsilon_0^\star[$, the results of Section \ref{sec:GSnear1} imply that for all $j_0,j\in\Z$ and $\textbf{e}\in\C^d$, the function $\tau\mapsto W(e^\tau,j_0,j,\textbf{e})$ is meromorphically extended on $B(0,\varepsilon)\backslash\lc0\rc$ with a pole of order $1$ at $0$. We then use Proposition \ref{GreenSpatialFar} on a neighborhood of each point of the set
		$$U_\varepsilon:= \lc\tau\in \C, \quad \Re \tau\in [0,\pi], \Im\tau\in[-\pi,\pi]\rc\backslash B(0,\varepsilon)$$ 
		and conclude via a compactness argument on the existence of a width $\eta_\varepsilon\in]0,\varepsilon[$ and of two positive constants $C_\varepsilon,c_\varepsilon$ such that \eqref{in:GSfar} is verified.
	\end{proof}
	
	Let us observe that for all $m\in I_{cs}^\pm\cup I_{cu}^\pm$, we have that the function $\zeta_m^\pm$ (which we recall are defined in Section \ref{subsec:EtudeSpectreMPrèsde1} and are eigenvalues of the matrices $M^\pm(z)$) is holomorphic and  $\zeta_m^\pm(1)=1$. Therefore, there exists a radius $\varepsilon_1^\star\in]0,\varepsilon_0^\star[$ so that for all $m\in I_{cs}^\pm\cup I_{cu}^\pm$ that we write as $m=l+(k-1)d$ with $k\in\lc1,\ppp,p+q\rc$ and $l\in\lc1,\ppp,d\rc$, there exists an holomorphic function $\varpi_l^\pm: B(0,\varepsilon_1^\star) \rightarrow \C$ such that  $\varpi_l^\pm(0)=0$ and
	\begin{equation}\label{def:varpi}
		\forall \tau \in B(0,\varepsilon_1^\star),\quad \zeta_m^\pm(e^\tau) = \exp(\varpi_l^\pm(\tau)).
	\end{equation}
	Since $\zeta_m^\pm(e^\tau)$ is an eigenvalue of $M_l^\pm(e^\tau)$, Lemma \ref{lem:SpecSpl} implies that 
	$$\forall \tau\in B(0,\varepsilon_1^\star), \quad \Fc_l^\pm(e^{\varpi_l^\pm(\tau)}) = e^\tau.$$
	If we define the holomorphic function
	\begin{equation}\label{def:varphi}
		\begin{array}{cccc}
			\varphi_l^\pm : & \C & \rightarrow & \C \\ & \tau & \mapsto & -\frac{\tau}{\alpha_l^\pm} +(-1)^{\mu+1} \frac{\beta_l^\pm}{{\alpha_l^\pm}^{2\mu+1}}\tau^{2\mu}
		\end{array},
	\end{equation}
	then, up to considering $\varepsilon_1^\star$ to be slightly smaller, the asymptotic expansion \eqref{F} implies that there exists a bounded holomorphic function $\xi_l^\pm:B(0,\varepsilon_1^\star)\rightarrow \C$ such that
	\begin{equation}\label{eg:LienVarpiVarphi}
		\forall \tau\in B(0,\varepsilon_1^\star),\quad \varpi_l^\pm(\tau) = \varphi_l^\pm(\tau) + \tau^{2\mu+1} \xi_l^\pm(\tau).
	\end{equation}
	
	The following lemma introduces central bounds for the calculations performed in the rest of the section.
	
	\begin{lemma}\label{lem:BorneVarphiVarpi}
		There exists a radius $\varepsilon_2^\star\in]0,\varepsilon_1^\star[$ and two positive constants $A_R,A_I$ such that for all $l\in \lc 1, \ppp,d\rc$
		\begin{align}
			\forall \tau\in \C, \quad&& \alpha_l^\pm\Re(\varphi_l^\pm(\tau)) & \leq -\Re(\tau) +A_R\Re(\tau)^{2\mu}-A_I\Im(\tau)^{2\mu},\label{in:varphi}\\
			\forall \tau\in B(0,\varepsilon_2^\star), \quad&& \alpha_l^\pm\Re(\varpi_l^\pm(\tau)) +|\alpha_l^\pm\xi_l^\pm(\tau)\tau^{2\mu+1}| & \leq -\Re(\tau) +A_R\Re(\tau)^{2\mu}-A_I\Im(\tau)^{2\mu}.\label{in:varpi}
		\end{align}
	\end{lemma}
	
	The proof is identical as \cite[Lemma 5.3]{CoeuIBVP} and will thus not be detailed here.
	
	\vspace{0.2cm} 	
	\textbf{Choice of the radius $\varepsilon$ and of the width $\eta$}
	\vspace{0.2cm} 	
	
	We will now fix choices for a radius $\varepsilon>0$ and a width $\eta>0$ which will satisfy a list of conditions. Those conditions will be centralized here in order to fix the notations and are especially important to prove some technical lemmas in Section \ref{subsec:LemmesGauss}. We will try to indicate at best where those conditions are used.
	
	First, we fix a choice of radius $\varepsilon\in\left]0,\min\left(\varepsilon_2^\star,\left(\frac{1}{2\mu A_R}\right)^\frac{1}{2\mu-1}\right)\right[$ where the radius $\varepsilon_2^\star$ is defined in Lemma \ref{lem:BorneVarphiVarpi}. This choice for $\varepsilon$ will allow us to use the results of Lemmas \ref{lem:defOmegaEta} and \ref{lem:BorneVarphiVarpi}. Furthermore, if we introduce the function: 
	\begin{equation}\label{defPsi}
		\begin{array}{cccc}
			\Psi: & \R & \rightarrow & \R \\ & \tau_p & \mapsto & \tau_p-A_R {\tau_p}^{2\mu}
		\end{array}
	\end{equation}
	which we will use to define a family of parameterized curve later on in Lemma \ref{lem:BornesGaussiennes}, then the function $\Psi$ is continuous and strictly increasing on $\left]-\infty,\varepsilon\right]$. This conclusion on the function $\Psi$ will be essential in the proof of Lemma \ref{lem:BornesGaussiennes} to construct the path $\Gamma$ appearing in \eqref{lem:BornesGaussiennesRes3}.
	
	We now introduce the function:
	\begin{equation}\label{defreps}
		\begin{array}{cccc}
			r_\varepsilon: & ]0,\varepsilon[& \rightarrow & \R \\
			& \eta & \mapsto & \sqrt{\varepsilon^2-\eta^2}
		\end{array}
	\end{equation}
	which serves to define the imaginary parts of the extremities of the curve $-\eta+i\R\cap B(0,\varepsilon)$. We recall that we defined a width $\eta_\varepsilon$ in Lemma \ref{lem:defOmegaEta}. We claim that there exists a width $\eta\in]0,\min(\varepsilon,\eta_\varepsilon)[$ that we fix for the rest of the paper such that the three following properties are verified.
	\begin{itemize}
		\item The following inequality is satisfied:
		\begin{equation}\label{condEta}
			\frac{\eta}{2}>A_R\eta^{2\mu}.
		\end{equation}
		It is used for instance in the proof of Lemma \ref{lem:BornesGaussiennes}.
		\item We have:
		\begin{equation}\label{condEta2}
			\eta+A_R\eta^{2\mu}- \frac{A_I}{2}r_\varepsilon(\eta)^{2\mu}<0.
		\end{equation}
		It is quite clear that we can choose $\eta$ small enough to satisfy this condition since, when $\eta$ tends towards $0$, the first two terms on the left hand side converge towards $0$  and the third converges towards $-\frac{A_I}{2}\varepsilon$. The condition \eqref{condEta2} is used in Lemma \ref{lem:BornesGaussiennes} to prove \eqref{lem:BornesGaussiennesRes2}. A consequence of \eqref{condEta2} is that 
		\begin{equation}\label{condEta2Reformulee}
			\forall n\in\N, \forall x\in \left[\frac{n}{2},2n\right], \forall t\in\left[-\eta,\eta\right], \quad (n-x)t+xA_Rt^{2\mu}-x\frac{A_Ir_\varepsilon(\eta)^{2\mu}}{2}\leq 0.
		\end{equation}
		Indeed, using the convexity with regards to $t$ of the left hand side of \eqref{condEta2Reformulee}, we have that
		$$(n-x)t+xA_Rt^{2\mu}-x\frac{A_Ir_\varepsilon(\eta)^{2\mu}}{2}\leq |n-x|\eta +xA_R\eta^{2\mu} -x\frac{A_Ir_\varepsilon(\eta)^{2\mu}}{2}.$$
		We observe that $n\in\left[\frac{x}{2},2x\right]$ and thus, using \eqref{condEta2}, we have:
		\begin{equation*}
			(n-x)t+xA_Rt^{2\mu}-x\frac{A_Ir_\varepsilon(\eta)^{2\mu}}{2}\leq x\left(\eta+A_R\eta^{2\mu}-\frac{A_Ir_\varepsilon(\eta)^{2\mu}}{2}\right)\leq 0.
		\end{equation*}
		The consequence \eqref{condEta2Reformulee} of \eqref{condEta2} will be used in the proof of Lemma \ref{lem:ComportementPrincGaussienne}.
		
		\item There exists a radius $\varepsilon_\#\in]0,\varepsilon[$ such that if we define 
		\begin{equation}\label{defIextr}
			l_{extr}:= \left(\frac{\Psi(\varepsilon_\#)-\Psi(-\eta)}{A_I}\right)^\frac{1}{2\mu},
		\end{equation}
		then $-\eta+il_{extr}\in B(0,\varepsilon)$. It is used in the proof of Lemma \ref{lem:BornesGaussiennes}.
	\end{itemize}
	
	We introduce the paths $\Gamma_{out}(\eta)$, $\Gamma_{in}^\pm(\eta)$, $\Gamma_{in}^0(\eta)$, $\Gamma_{in}(\eta)$, $\Gamma(\eta)$, $\Gamma_{d}(\eta)$ represented on Figure \ref{FigChem} and defined as:
	\begin{align}\label{def:Paths}
		\begin{split}
			\Gamma_{out}(\eta)& := [-\eta-i\pi,-\eta-ir_\varepsilon(\eta)]\cup [-\eta+ir_\varepsilon(\eta),-\eta+i\pi], \\
			\Gamma_{in}^\pm(\eta)& := \left[-\eta\pm ir_\varepsilon(\eta),\eta\pm ir_\varepsilon(\eta)\right],\\
			\Gamma_{in}^0(\eta)& := \left[\eta- ir_\varepsilon(\eta),\eta+ ir_\varepsilon(\eta)\right],\\
			\Gamma_{in}(\eta)& := \Gamma_{in}^-(\eta)\cup\Gamma_{in}^0(\eta) \cup\Gamma_{in}^+(\eta)\\
			\Gamma(\eta)& := \Gamma_{in}(\eta)\cup \Gamma_{out}(\eta),\\
			\Gamma_{d}(\eta)& := \left[-\eta-ir_\varepsilon(\eta),-\eta+ir_\varepsilon(\eta)\right].
		\end{split}
	\end{align}
	
	\begin{figure}
		\begin{center}
			\begin{minipage}{0.3\textwidth}
				\hfill
				\begin{tikzpicture}
					\draw (0,3) node[left] {$\Gamma_{out}(\eta)$:};
					\draw[red,thick] (0,3) -- (1,3) node[midway,scale=1.5] {$\circ$};
					
					\draw (0,2) node[left] {$\Gamma_{in}^\pm(\eta)$:};
					\draw[blue,thick] (0,2) -- (1,2) node[midway,scale=1.5] {$\vartriangleright$};
					
					\draw (0,1) node[left] {$\Gamma_{in}^0(\eta)$:};
					\draw[purple,thick] (0,1) -- (1,1) node[midway,scale=1.5] {$\star$};
					
					\draw (0,0) node[left] {$\Gamma_{d}(\eta)$:};
					\draw[dartmouthgreen,thick] (0,0) -- (1,0) node[midway,scale=1.5] {$\parallel$};
				\end{tikzpicture}
			\end{minipage}
			\hfill
			\begin{minipage}{0.6\textwidth}
				\begin{tikzpicture}[scale=1]
					\draw[->] (-2,0) -- (3,0) node[right] {$\Re(\tau)$};
					\draw[->] (0,-3.4) -- (0,3.4) node[above] {$\Im(\tau)$};
					\draw[dashed] (-2,pi) -- (3,pi);
					\draw[dashed] (-2,-pi) -- (3,-pi);
					
					\draw (0,0) node {$\times$} circle (1.5);
					\draw (1.8,0.7)  node {$B_\varepsilon(0)$};
					
					\draw (0,pi) node {$\bullet$} node[below right] {$i\pi$};
					\draw (0,-pi) node {$\bullet$} node[above right] {$-i\pi$};
					\draw (-1,0) node {$\bullet$} node[below right] {$-\eta$};	
					\draw (1,0) node {$\bullet$} node[below right] {$\eta$};					
					
					\draw[red,thick] (-1,-pi) -- (-1,{-sqrt(1.5^2-1^2)}) node[midway] {$\circ$};
					\draw[red,thick] (-1,{sqrt(1.5^2-1^2)}) -- (-1,pi) node[midway, left] {$\Gamma_{out}(\eta)$} node[midway] {$\circ$};
					
					\draw[dartmouthgreen,thick] (-1,{-sqrt(1.5^2-1^2)}) -- (-1,{sqrt(1.5^2-1^2)}) node[near start, sloped] {$\parallel$} node[near end, sloped] {$\parallel$} ;
					\draw[dartmouthgreen] (-1.5,0.5) node[left] {$\Gamma_d(\eta)$};
					
					\draw[thick,blue] (-1,{-sqrt(1.5^2-1^2)}) -- (1,{-sqrt(1.5^2-1^2)}) node[midway,sloped] {$\vartriangleright$};
					\draw[blue] (0.6,-1.7) node {$\Gamma_{in}^-(\eta)$};
					
					\draw[thick,blue] (1,{sqrt(1.5^2-1^2)})  -- (-1,{sqrt(1.5^2-1^2)})  node[midway,sloped] {$\vartriangleleft$};
					\draw[blue] (0.6,1.7) node {$\Gamma_{in}^+(\eta)$};
					
					\draw[purple,thick] (1,{-sqrt(1.5^2-1^2)}) -- (1,{sqrt(1.5^2-1^2)}) node[near start,scale=1.3] {$\star$} node[near end, scale=1.3] {$\star$} ;
					\draw[purple] (2,-0.5) node {$\Gamma_{in}^0(\eta)$};
				\end{tikzpicture}
			\end{minipage}
			\caption{A representation of the path described in \eqref{def:Paths}: $\Gamma_{out}(\eta)$, $\Gamma_{in}^\pm(\eta)$, $\Gamma_{in}^0(\eta)$ and $\Gamma_{d}(\eta)$}
			\label{FigChem}
		\end{center}
	\end{figure}
	
	We observe that those paths lie in $\Omega_\varepsilon$. Using the Cauchy formula and acknowledging the "$2i\pi$"-periodicity of $\tau\mapsto W(e^\tau,j_0,j,\textbf{e})$, we can prove via the equality \eqref{egGSpaTemp1} that:
	\begin{align}\label{egGSpaTemp2}
		\begin{split}
			\Gcc(n,j_0,j)\textbf{e} &= \frac{1}{2i\pi}\int_{\Gamma(\eta)} e^{n\tau}e^\tau \Pi(W(e^\tau,j_0,j,\textbf{e}))d\tau\\
			& = \frac{1}{2i\pi}\int_{\Gamma_{out}(\eta)} e^{n\tau}e^\tau \Pi(W(e^\tau,j_0,j,\textbf{e}))d\tau+ \frac{1}{2i\pi}\int_{\Gamma_{in}(\eta)} e^{n\tau}e^\tau \Pi(W(e^\tau,j_0,j,\textbf{e}))d\tau .
		\end{split}
	\end{align}

	\begin{lemma}\label{estGreenExt}
		There exist two positive constants $C,c$ such that for all $n\in \N$, $j_0,j\in \Z$ and $\textbf{e}\in\C^d$ we have that
		$$\left|\frac{1}{2i\pi}\int_{\Gamma_{out}(\eta)} e^{n\tau}e^\tau \Pi(W(e^\tau,j_0,j,\textbf{e}))d\tau\right| \leq C|\textbf{e}|e^{-n\eta}.$$
	\end{lemma}
	
	\begin{proof}
		The conclusion of the lemma directly follows from \eqref{in:GSfar} and the definition of $\Gamma_{out}(\eta)$ which implies that
		$$\forall \tau\in \Gamma_{out}(\eta), \quad |e^{n\tau}|= e^{-n\eta}.$$
	\end{proof}
	
	The equality \eqref{egGSpaTemp2} and the sharp exponential bounds on the term:
	$$\frac{1}{2i\pi}\int_{\Gamma_{out}(\eta)} e^{n\tau}e^\tau \Pi(W(e^\tau,j_0,j,\textbf{e}))d\tau,$$
	that we just proved imply that there just remains to handle the term 
	\begin{equation}\label{terme}
		\frac{1}{2i\pi}\int_{\Gamma_{in}(\eta)} e^{n\tau}e^\tau \Pi(W(e^\tau,j_0,j,\textbf{e}))d\tau
	\end{equation}
	to have the description \eqref{decompoGreen} of the temporal Green's function expected in Theorem \ref{th:Green}. We recall that $\Gamma_{in}(\eta)$ is a path that lies inside the set $B(0,\varepsilon)$ by construction and that we chose the radius $\varepsilon\in]0,\varepsilon_2^*[$ to be small enough so that 
	$$\forall \tau\in B(0,\varepsilon),\quad e^\tau\in B(1,\delta_2).$$
	Thus, recalling that we consider $j_0\geq 0$, we can use the expressions \eqref{ExpGs: l geq 0 : j geq l+1}, \eqref{ExpGs: l geq 0 : 0 leq j leq l} and \eqref{ExpGs: l geq 0 : j < 0} to decompose the integral \eqref{terme} into different terms depending on the position of $j$ with respect to $0$ and $j_0$. Our new goal can now be separated in three parts:
	
		\begin{itemize}
			\item In Section \ref{subsec:Decompo}, we will prove Lemmas \ref{lem:OndesPropag}-\ref{lem:OndesExcitedCentral} that will allow to bound the several terms that can appear when decomposing the integral \eqref{terme} using \eqref{ExpGs}.
			\item In Section \ref{subsec:ConcluTh1}, we conclude the proof of the decomposition \eqref{decompoGreen} of the temporal Green's function using the previously proved Lemma \ref{estGreenExt} and the following Lemmas  \ref{lem:OndesPropag}-\ref{lem:OndesExcitedCentral}.
			\item In Section \ref{subsec:decompoDer},  we explain how the analysis performed in the two previous steps can be adapted to also deduce the decomposition \eqref{decompoDerGreen} of the discrete derivative of the temporal Green's function.
		\end{itemize}

	\subsection{Decomposition of the integral within \texorpdfstring{$B(0,\varepsilon)$}{B(0,epsilon)}}\label{subsec:Decompo}
	
	This section will be mainly devoted to the proof of Lemmas \ref{lem:OndesPropag}-\ref{lem:OndesExcitedCentral} that will allow to study each term that can appear in the decomposition of the integral \eqref{terme} using the expressions \eqref{ExpGs: l geq 0 : j geq l+1}, \eqref{ExpGs: l geq 0 : 0 leq j leq l} and \eqref{ExpGs: l geq 0 : j < 0}. However, we are first going to need to introduce a few more technical lemmas that will be used relentlessly throughout the rest of the paper. 
	
	\subsubsection{Gaussian estimates}\label{subsec:LemmesGauss}
	
	First and foremost, we define the set $X$ of the paths going from $-\eta-ir_\varepsilon(\eta)$ to $-\eta +ir_\varepsilon(\eta)$ whilst remaining in $B(0,\varepsilon)$. We observe in particular that $\Gamma_d(\eta),\Gamma_{in}(\eta)\in X$.
	
	\begin{lemma}\label{lem:BornesGaussiennes}
		We consider an integer $k\in \N$. 
		
		\begin{subequations}
			\begin{itemize}
				\item There exist two positive constants $C,c$ such that for all $n\in \N\backslash\lc 0\rc$ and $x\in \left[0,\frac{n}{2}\right]$
				\begin{equation}\label{lem:BornesGaussiennesRes1}
					\int_{\Gamma_{d}(\eta)}|\tau|^k\exp\left(n\Re(\tau)+x\left(-\Re(\tau)+A_R\Re(\tau)^{2\mu}-A_I\Im(\tau)^{2\mu}\right)\right)|d\tau| \leq Ce^{-cn}.
				\end{equation}
				
				\item There exist two positive constants $C,c$ such that for all $n\in \N\backslash\lc 0\rc$ and $x\in [2n,+\infty[$
				\begin{equation}\label{lem:BornesGaussiennesRes2}
					\int_{\Gamma_{in}(\eta)}|\tau|^k\exp\left(n\Re(\tau)+x\left(-\Re(\tau)+A_R\Re(\tau)^{2\mu}-A_I\Im(\tau)^{2\mu}\right)\right)|d\tau| \leq Ce^{-cn}.
				\end{equation}
				
				\item There exist two positive constants $C,c$ such that for all $n\in \N\backslash\lc 0\rc$ and $x\in \left[\frac{n}{2},2n\right]$, there exists a path $\Gamma\in X$ such that
				\begin{multline}\label{lem:BornesGaussiennesRes3}
					\int_\Gamma |\tau|^k\exp\left(n\Re(\tau)+x\left(-\Re(\tau)+A_R\Re(\tau)^{2\mu}-A_I\Im(\tau)^{2\mu}\right)\right)|d\tau| \\ \leq \displaystyle\frac{C}{n^\frac{k+1}{2\mu}}\exp\left(-c\left(\frac{|n-x|}{n^\frac{1}{2\mu}}\right)^\frac{2\mu}{2\mu-1}\right).
				\end{multline}
			\end{itemize}
		\end{subequations}
	\end{lemma}
	
	Lemma \ref{lem:BornesGaussiennes} will allow us to obtain generalized Gaussian bounds for several terms throughout the proof of Theorem \ref{th:Green}. The inequalities \eqref{lem:BornesGaussiennesRes1}-\eqref{lem:BornesGaussiennesRes3} separate different cases depending on $x$. An important point to observe is that the path $\Gamma$ appearing in \eqref{lem:BornesGaussiennesRes3} depends on $n$ and $x$ whereas the constants $C,c$ are uniform.
	
	 The way Lemma \ref{lem:BornesGaussiennes} will be used is to first observe that the integral of some holomorphic function over some path of $X$ is equal by Cauchy's formula to the integral of the same function over any path of $X$. We then prove that the integrand can be well bounded and use the result of Lemma \ref{lem:BornesGaussiennes}. The proof of Lemma \ref{lem:BornesGaussiennes} can be adapted from \cite{C-F,C-FIBVP,CoeuLLT,CoeuIBVP} and will be done in the Appendix (Section \ref{sec:Appendix}). 	
	
	\begin{lemma}\label{lem:PassageVarpi-Varphi}
		There exists a constant $C>0$ such that for all $l\in \lc 1, \ppp,d\rc$, $n\in\N\backslash\lc0\rc$ and $x\in\left[0,2n\right]$
		$$\left|e^{x\alpha_l^\pm\varpi_l^\pm(\tau)}-e^{x\alpha_l^\pm\varphi_l^\pm(\tau)}\right| \leq C n|\tau|^{2\mu+1} \exp\left(x\left(-\Re(\tau)+A_R\Re(\tau)^{2\mu}-A_I\Im(\tau)^{2\mu}\right)\right).$$
	\end{lemma}
	
	The proof of Lemma \ref{lem:PassageVarpi-Varphi} is similar to the proof of \cite[Lemma 16]{CoeuLLT} and will not be detailed here. We recall that the functions $\varpi_l^\pm$ and $\varphi_l^\pm$ respectively defined by \eqref{def:varpi} and \eqref{def:varphi} are linked by the equality \eqref{eg:LienVarpiVarphi}. Lemma \ref{lem:PassageVarpi-Varphi} will allow us to "extract" the principal part $\varphi_l^\pm$ of the function $\varpi_l^\pm$. This principal part will then appear in terms that can be studied using the following lemma.
	
	\begin{lemma}\label{lem:ComportementPrincGaussienne}
		We consider $l,l^\prime\in\lc1,\ppp,d\rc$ and $s,s^\prime\in \lc -,+\rc$. There exist two constants $C,c>0$ such that for all $n\in\N\backslash\lc0\rc$, we have:
		\begin{itemize}
			\item For $x,y\in [0,+\infty[$ such that $x+y\in\left[\frac{n}{2},2n\right]$ and $\Gamma\in X$
			\begin{subequations}
				\begin{multline}\label{lem:ComportementPrincGaussienneRes1}
					\left|\frac{1}{2i\pi}\int_\Gamma \exp\left(n\tau+x\alpha_l^s\varphi_l^s(\tau)+y\alpha_{l^\prime}^{s^\prime}\varphi_{l^\prime}^{s^\prime}(\tau)\right)d\tau-\frac{|\alpha_l^s|}{n^\frac{1}{2\mu}}H_{2\mu}\left(\frac{x}{n}\beta_l^s+\frac{y}{n}\beta_{l^\prime}^{s^\prime}\left(\frac{\alpha_l^s}{\alpha_{l^\prime}^{s^\prime}}\right)^{2\mu}; \frac{\alpha_l^s(n-(x+y))}{n^\frac{1}{2\mu}}\right)\right| \\ \leq Ce^{-cn}.
				\end{multline}
				\item For $x\in\left[\frac{n}{2},2n\right]$ and $\Gamma\in X$
				\begin{equation}\label{lem:ComportementPrincGaussienneRes2}
					\left|\frac{1}{2i\pi}\int_\Gamma \exp\left(n\tau+x\alpha_l^s\varphi_l^s(\tau)\right)d\tau-\frac{|\alpha_l^s|}{n^\frac{1}{2\mu}}H_{2\mu}\left(\beta_l^s; \frac{\alpha_l^s(n-x)}{n^\frac{1}{2\mu}}\right)\right|\leq \frac{C}{n^\frac{1}{\mu}}\exp\left(-c\left(\frac{|n-x|}{n^\frac{1}{2\mu}}\right)^\frac{2\mu}{2\mu-1}\right).
				\end{equation}
				\item For $x\in\left[\frac{n}{2},2n\right]$,
				\begin{equation}\label{lem:ComportementPrincGaussienneRes3}
					\left|\frac{1}{2i\pi}\int_{\Gamma_{in}(\eta)} \frac{\exp\left(n\tau+x\alpha_l^s\varphi_l^s(\tau)\right)}{\tau}d\tau-E_{2\mu}\left(\beta_l^s; \frac{-|\alpha_l^s|(n-x)}{n^\frac{1}{2\mu}}\right)\right|\leq \frac{C}{n^\frac{1}{2\mu}}\exp\left(-c\left(\frac{|n-x|}{n^\frac{1}{2\mu}}\right)^\frac{2\mu}{2\mu-1}\right).
				\end{equation}
			\end{subequations}
		\end{itemize}
	\end{lemma}
	
	The proof of Lemma \ref{lem:ComportementPrincGaussienne} is a summary of calculations performed in \cite{CoeuLLT,CoeuIBVP} and will be done in the Appendix (Section \ref{sec:Appendix}). Let us observe that there is no condition on the paths of integration in \eqref{lem:ComportementPrincGaussienneRes1} and \eqref{lem:ComportementPrincGaussienneRes2}. However, since the integrand is only meromorphic in \eqref{lem:ComportementPrincGaussienneRes1}, we only consider the path $\Gamma_{in}(\eta)$.
	
	\subsubsection{Outgoing and incoming waves}
	
	We will start by looking at the outgoing and incoming waves by proving the following lemma. 
	
	\begin{lemma}\label{lem:OndesPropag}
		We consider $m\in\lc1,\ppp,d(p+q)\rc$ and write it as $m=l+(k-1)d$ with $k\in\lc1,\ppp,p+q\rc$ and $l\in\lc1,\ppp,d\rc$. There exists a constant $c>0$ such that for all $n\in \N\backslash\lc0\rc$, $j_0,j\in\N$ such that $j-j_0\in\lc-nq,\ppp,np\rc$ and $\textbf{e}\in\C^d$ we have:
		
		\begin{subequations}
			$\bullet$ If $m\in I_{cs}^+\cup I_{cu}^+$ and $\frac{j-j_0}{\alpha_l^+}\in \left[\frac{n}{2},2n\right]$, we have that:
			\begin{multline}\label{lem:OndesPropagRes1}
				-\frac{\mathrm{sgn}(\alpha_l^+)}{2i\pi}\int_{\Gamma_{in}(\eta)} e^{n\tau} e^\tau\widetilde{\Ccc}_m^+(e^\tau,j_0,\textbf{e})\Pi(W_m^+(e^\tau,j))d\tau -S^+_l(n,j_0,j)\textbf{e} \\=  \exp\left(-c\left(\frac{\left|n-\left(\frac{j-j_0}{\alpha_l^+}\right)\right|}{n^\frac{1}{2\mu}}\right)^\frac{2\mu}{2\mu-1}\right)\left(O\left(\frac{|\textbf{e}|e^{-c|j|}}{n^\frac{1}{2\mu}}\right)+O_\C\left(\frac{|\textbf{e}|e^{-c|j_0|}}{n^\frac{1}{2\mu}}\right)\rg_l^+ + O_\C\left(\frac{1}{n^\frac{1}{\mu}}\right){\lg_l^+}^T\textbf{e}\rg_l^+\right).
			\end{multline}
			
			$\bullet$ If $m\in I_{cs}^+\cup  I_{cu}^+$, $\frac{j-j_0}{\alpha_l^+}\notin \left[\frac{n}{2},2n\right]$ and $\frac{j-j_0}{\alpha_l^+}\geq 0$, we have that:
			\begin{equation}\label{lem:OndesPropagRes2}
				-\frac{\mathrm{sgn}(\alpha_l^+)}{2i\pi}\int_{\Gamma_{in}(\eta)} e^{n\tau} e^\tau\widetilde{\Ccc}_m^+(e^\tau,j_0,\textbf{e})\Pi(W_m^+(e^\tau,j))d\tau -S^+_l(n,j_0,j)\textbf{e} = O\left(|\textbf{e}|e^{-cn}\right).
			\end{equation}
			
			$\bullet$ If $m\in I_{ss}^+$ and $j\geq j_0+1$ or if $m\in I_{su}^+$ and $j\in\lc0,\ppp,j_0\rc$, we have that:
			\begin{equation}\label{lem:OndesPropagRes3}
				\frac{1}{2i\pi}\int_{\Gamma_{in}(\eta)} e^{n\tau} e^\tau\widetilde{\Ccc}_m^+(e^\tau,j_0,\textbf{e})\Pi(W_m^+(e^\tau,j))d\tau =O\left(|\textbf{e}|e^{-cn}\right).
			\end{equation}
		\end{subequations}
	\end{lemma}
	
	\begin{proof}
		$\bullet$ We start by proving \eqref{lem:OndesPropagRes1}. The proofs of \eqref{lem:OndesPropagRes2} and \eqref{lem:OndesPropagRes3} will be done afterwards as they are fairly less complicated. We consider $m\in I_{cs}^+\cup I_{cu}^+$ such that $\frac{j-j_0}{\alpha_l^+}\in \left[\frac{n}{2},2n\right]$.
		
		Using the expressions of $W_m^+$ and $\widetilde{\Ccc}_m^+$ given respectively by Lemma \ref{lem_choice_base} and \eqref{eg:CccTilde}, we have using Cauchy's formula that for any $\Gamma \in X$
		\begin{equation}\label{lem:OndesPropag1}
			\int_{\Gamma_{in}(\eta)} e^{n\tau} e^\tau\widetilde{\Ccc}_m^+(e^\tau,j_0,\textbf{e})\Pi(W_m^+(e^\tau,j))d\tau=\int_{\Gamma} e^{n\tau} \zeta_m^+(e^\tau)^{j-j_0-1} e^\tau\Delta_m^+(e^\tau,j_0,\textbf{e})\Pi(V_m^+(e^\tau,j))d\tau.
		\end{equation}
		Using Cauchy's formula once again, \eqref{def:varpi} since the eigenvalue we consider is central and the definition \eqref{def:S+} of the function $S_l^+$, we then have that
		\begin{equation}\label{lem:OndesPropag2}
			-\frac{\mathrm{sgn}(\alpha_l^+)}{2i\pi}\int_{\Gamma_{in}(\eta)} e^{n\tau} e^\tau\widetilde{\Ccc}_m^+(e^\tau,j_0,\textbf{e})\Pi(W_m^+(e^\tau,j))d\tau -S^+_l(n,j_0,j)\textbf{e}\\ =  E_1 +E_2 \rg_l^+ + \left(E_3  +E_4+E_5  \right){\lg_l^+}^T\textbf{e} \rg_l^+
		\end{equation}
		where $E_1$ is a vector and $E_2,\ppp,E_5$ are complex scalars defined by
		\begin{align*}
			E_1 &= -\frac{\mathrm{sgn}(\alpha_l^+)}{2i\pi}\int_{\Gamma_1} e^{n\tau} e^{(j-j_0-1)\varpi_l^+(\tau)} e^\tau\Delta_m^+(e^\tau,j_0,\textbf{e})\Pi(V_m^+(e^\tau,j)-R_m^+(e^\tau))d\tau, \\
			E_2 &= -\frac{\mathrm{sgn}(\alpha_l^+)}{2i\pi}\int_{\Gamma_2} e^{n\tau} e^{(j-j_0-1)\varpi_l^+(\tau)} e^\tau\left(\Delta_m^+(e^\tau,j_0,\textbf{e})- \frac{d\zeta_m^+}{dz}(e^\tau) {\lg_l^+}^T \textbf{e}\right)d\tau, \\
			E_3 &= -\frac{\mathrm{sgn}(\alpha_l^+)}{2i\pi}\int_{\Gamma_3} e^{n\tau} e^{(j-j_0)\varpi_l^+(\tau)} \left(e^\tau\zeta_m^+(e^\tau)^{-1}\frac{d\zeta_m^+}{dz}(e^\tau) +\frac{1}{\alpha_l^+}\right)d\tau, \\
			E_4 &= \frac{1}{2i\pi|\alpha_l^+|}\int_{\Gamma_4} e^{n\tau} \left(e^{(j-j_0)\varpi_l^+(\tau)}- e^{(j-j_0)\varphi_l^+(\tau)}\right)d\tau,\\
			E_5 &=\frac{1}{2i\pi|\alpha_l^+|}\int_{\Gamma_5} e^{n\tau} e^{(j-j_0)\varphi_l^+(\tau)}d\tau - \frac{1}{n^\frac{1}{2\mu}} H_{2\mu}\left(\beta_l^+;\frac{n\alpha_l^++j_0-j}{n^\frac{1}{2\mu}}\right),
		\end{align*}
		and $\Gamma_1,\ppp,\Gamma_5$ are paths belonging to the set $X$ defined at the beginning of Section \ref{subsec:LemmesGauss}. We just have to prove correct bounds on the terms $E_1,\ppp,E_5$ appearing in \eqref{lem:OndesPropag2} to obtain \eqref{lem:OndesPropagRes1}. In particular, we will use good choices of paths $\Gamma_1,\ppp,\Gamma_5$ to optimize the bounds using Lemma \ref{lem:BornesGaussiennes}.
		
		$\blacktriangleright$ Using \eqref{in:Delta+}, \eqref{in:varpi} and Lemma \ref{lem_choice_base} which claims that the vectors $V_m^+(z,j)$ converge exponentially fast towards $R_m^+(e^\tau)$, we have that there exist two positive constants $C,c$ independent from $n$, $j_0$, $j$ and $\textbf{e}$ such that 
		$$|E_1| \leq Ce^{-c|j|}|\textbf{e}| \int_{\Gamma_1}\exp\left(n\Re(\tau)+\left(\frac{j-j_0}{\alpha_l^+}\right)\left(-\Re(\tau)+A_R\Re(\tau)^{2\mu}-A_I\Im(\tau)^{2\mu}\right)\right)|d\tau|.$$
		
		$\blacktriangleright$ Using \eqref{in:DeltaInf+} and \eqref{in:varpi}, we have that there exist two positive constants $C,c$ independent from $n$, $j_0$, $j$ and $\textbf{e}$ such that 
		$$|E_2| \leq Ce^{-c|j_0|}|\textbf{e}| \int_{\Gamma_2}\exp\left(n\Re(\tau)+\left(\frac{j-j_0}{\alpha_l^+}\right)\left(-\Re(\tau)+A_R\Re(\tau)^{2\mu}-A_I\Im(\tau)^{2\mu}\right)\right)|d\tau|.$$
		
		$\blacktriangleright$ We notice that $\zeta_m^+(1)=1$ and $\frac{d\zeta_m^+}{dz}(1)=-\frac{1}{\alpha_l^+}$. Using a Taylor expansion and \eqref{in:varpi}, we have that there exists a positive constant $C$ independent from $n$, $j_0$ and $j$ such that 
		$$|E_3| \leq C \int_{\Gamma_3}|\tau|\exp\left(n\Re(\tau)+\left(\frac{j-j_0}{\alpha_l^+}\right)\left(-\Re(\tau)+A_R\Re(\tau)^{2\mu}-A_I\Im(\tau)^{2\mu}\right)\right)|d\tau|.$$
		
		$\blacktriangleright$ Since $\frac{j-j_0}{\alpha_l^+}\in\left[\frac{n}{2},2n\right]$, we can use Lemma \ref{lem:PassageVarpi-Varphi} and prove that there exists a positive constant $C$ independent from $n$, $j_0$ and $j$ such that
		$$|E_4| \leq C n \int_{\Gamma_4}|\tau|^{2\mu+1}\exp\left(n\Re(\tau)+\left(\frac{j-j_0}{\alpha_l^+}\right)\left(-\Re(\tau)+A_R\Re(\tau)^{2\mu}-A_I\Im(\tau)^{2\mu}\right)\right)|d\tau|.$$
		
		Using Lemma \ref{lem:BornesGaussiennes} which gives a good choices of path $\Gamma_1,\ppp,\Gamma_4\in X$ depending on $n$, $j_0$ and $j$ to handle the integrals in the terms above as well as Lemma \ref{lem:ComportementPrincGaussienne} to take care of the term $E_5$, there exist new constants $C,c>0$ independent from $n$, $j_0$, $j$ and $\textbf{e}$ such that 
		\begin{align*}
			|E_1|& \leq \frac{Ce^{-c|j|}|\textbf{e}|}{n^\frac{1}{2\mu}} \exp\left(-c\left(\frac{\left|n-\left(\frac{j-j_0}{\alpha_l^+}\right)\right|}{n^\frac{1}{2\mu}}\right)^\frac{2\mu}{2\mu-1}\right), & |E_2|& \leq \frac{Ce^{-c|j_0|}|\textbf{e}|}{n^\frac{1}{2\mu}} \exp\left(-c\left(\frac{\left|n-\left(\frac{j-j_0}{\alpha_l^+}\right)\right|}{n^\frac{1}{2\mu}}\right)^\frac{2\mu}{2\mu-1}\right) ,\\
			|E_3| &\leq \frac{C}{n^\frac{1}{\mu}} \exp\left(-c\left(\frac{\left|n-\left(\frac{j-j_0}{\alpha_l^+}\right)\right|}{n^\frac{1}{2\mu}}\right)^\frac{2\mu}{2\mu-1}\right),& |E_4| &\leq \frac{C}{n^\frac{1}{\mu}} \exp\left(-c\left(\frac{\left|n-\left(\frac{j-j_0}{\alpha_l^+}\right)\right|}{n^\frac{1}{2\mu}}\right)^\frac{2\mu}{2\mu-1}\right), \\
			|E_5| &\leq \frac{C}{n^\frac{1}{\mu}} \exp\left(-c\left(\frac{\left|n-\left(\frac{j-j_0}{\alpha_l^+}\right)\right|}{n^\frac{1}{2\mu}}\right)^\frac{2\mu}{2\mu-1}\right). &&
		\end{align*}
		We have thus obtained \eqref{lem:OndesPropagRes1}.
		
		$\bullet$ We now focus on \eqref{lem:OndesPropagRes2}. We thus consider $m\in I_{cs}^+$ and $j\geq j_0+1$ or $m\in I_{cu}^+$ and $j\in\lc0,\ppp, j_0\rc$ such that $\frac{j-j_0}{\alpha_l^+}\notin \left[\frac{n}{2},2n\right]$. We observe that in particular, $S_l^+(n,j_0,j)=0$. Using \eqref{lem:OndesPropag1}, \eqref{def:varpi} since the eigenvalue we consider is central, Lemma \ref{lem_choice_base} which claims that the vectors $V_m^+(z,j)$ are uniformly bounded for $z\in B(1,\delta_1)$ and $j\in\N$, \eqref{in:Delta+} and \eqref{in:varpi}, there exists a positive constant $C$ independent from $n$, $j_0$, $j$ and $\textbf{e}$ such that for all $\Gamma\in X$
		\begin{multline*}
			\left|-\frac{\mathrm{sgn}(\alpha_l^+)}{2i\pi}\int_{\Gamma_{in}(\eta)} e^{n\tau} e^\tau\widetilde{\Ccc}_m^+(e^\tau,j_0,\textbf{e})\Pi(W_m^+(e^\tau,j))d\tau -S^+_l(n,j_0,j)\textbf{e}\right| \\ \leq C|\textbf{e}| \int_{\Gamma}\exp\left(n\Re(\tau)+\left(\frac{j-j_0}{\alpha_l^+}\right)\left(-\Re(\tau)+A_R\Re(\tau)^{2\mu}-A_I\Im(\tau)^{2\mu}\right)\right)|d\tau|.
		\end{multline*}
		We then use Lemma \ref{lem:BornesGaussiennes} to prove that there exist two new positive constants $C,c$ independent from $n$, $j_0$, $j$ and $\textbf{e}$ such that
		$$ \left|-\frac{\mathrm{sgn}(\alpha_l^+)}{2i\pi}\int_{\Gamma_{in}(\eta)} e^{n\tau} e^\tau\widetilde{\Ccc}_m^+(e^\tau,j_0,\textbf{e})\Pi(W_m^+(e^\tau,j))d\tau -S^+_l(n,j_0,j)\textbf{e}\right|\leq C|\textbf{e}|e^{-cn}.$$
		
		$\bullet$ We now focus on \eqref{lem:OndesPropagRes3}. We will consider the case where the integer $m$ belongs to $I_{ss}^+$ and $j\geq j_0+1$. The second case considered in \eqref{lem:OndesPropagRes3} would be handled similarly. Using \eqref{lem:OndesPropag1}, Lemma \ref{lem_choice_base} which claims that the vectors $V_m^+(z,j)$ are uniformly bounded for $z\in B(1,\delta_1)$ and $j\in\N$, \eqref{in:Delta+} and \eqref{inZetaSs} whilst noticing that $j\geq j_0+1$, there exists a positive constant $C$ independent from $n$, $j_0$, $j$ and $\textbf{e}$ such that
		\begin{align*}
			\left|\frac{1}{2i\pi}\int_{\Gamma_{in}(\eta)} e^{n\tau}e^\tau \widetilde{\Ccc}_m^+(e^\tau,j_0,\textbf{e}) \Pi(W_m^+(e^\tau,j))d\tau\right| &\leq C|\textbf{e}|e^{-2c_*|j-j_0|} \int_{\Gamma_d(\eta)}e^{n\Re(\tau)}|d\tau|\\
			& \leq 2r_\varepsilon(\eta)C|\textbf{e}|e^{-n\eta-2c_*|j-j_0|} .
		\end{align*}
	\end{proof}
	
	\subsubsection{Reflected waves}
	
	We now look at the reflected waves.
	
	\begin{lemma}\label{lem:OndesReflechies}
		We consider $m\in \lc2, \ppp,dp\rc$ and $m^\prime\in \lc dp+1,\ppp,d(p+q)\rc$ and write them as $m=l+(k-1)d$ and $m^\prime=l^\prime+(k^\prime-1)d$ with $k,k^\prime\in\lc1,\ppp,p+q\rc$ and $l,l^\prime\in\lc1,\ppp,d\rc$. There exists a positive constant $c$ such that for all $n\in \N\backslash\lc0\rc$, $j_0,j\in \N$ such that $j-j_0\in\lc-nq,\ppp,np\rc$ and $\textbf{e}\in\C^d$, we have:
		
		\begin{subequations}
			
			$\bullet$ If $m\in I_{cs}^+$, $m^\prime\in I_{cu}^+$ and $\frac{j}{\alpha_l^+}-\frac{j_0}{\alpha_{l^\prime}^+}\in\left[\frac{n}{2},2n\right]$, we have that:
			\begin{multline}\label{lem:OndesReflechies:CsCu1}
				-\frac{1}{2i\pi}\int_{\Gamma_{in}(\eta)} e^{n\tau} e^\tau\widetilde{g}_{m^\prime,m}^+(e^\tau)\widetilde{\Ccc}_m^+(e^\tau,j_0,\textbf{e})\Pi(W_m^+(e^\tau,j))d\tau  - \frac{\alpha_l^+}{\alpha_{l^\prime}^+}\widetilde{g}_{m^\prime,m}^+(1)R_{l^\prime,l}^+(n,j_0,j)\textbf{e}\\ =  \exp\left(-c\left(\frac{\left|n-\left(\frac{j}{\alpha_l^+}-\frac{j_0}{\alpha_{l^\prime}^+}\right)\right|}{n^\frac{1}{2\mu}}\right)^\frac{2\mu}{2\mu-1}\right)\left(O\left(\frac{|\textbf{e}|e^{-c|j|}}{n^\frac{1}{2\mu}}\right)+O_\C\left(\frac{|\textbf{e}|e^{-c|j_0|}}{n^\frac{1}{2\mu}}\right)\rg_l^+ + O_\C\left(\frac{1}{n^\frac{1}{\mu}}\right){\lg_{l^\prime}^+}^T\textbf{e}\rg_l^+\right).
			\end{multline}
			
			$\bullet$ If $m\in I_{cs}^+$, $m^\prime\in I_{cu}^+$ and $\frac{j}{\alpha_l^+}-\frac{j_0}{\alpha_{l^\prime}^+}\notin\left[\frac{n}{2},2n\right]$, we have that:
			\begin{equation}\label{lem:OndesReflechies:CsCu2}
				-\frac{1}{2i\pi}\int_{\Gamma_{in}(\eta)} e^{n\tau} e^\tau\widetilde{g}_{m^\prime,m}^+(e^\tau)\widetilde{\Ccc}_m^+(e^\tau,j_0,\textbf{e})\Pi(W_m^+(e^\tau,j))d\tau  - \frac{\alpha_l^+}{\alpha_{l^\prime}^+}\widetilde{g}_{m^\prime,m}^+(1)R_{l^\prime,l}^+(n,j_0,j)\textbf{e}=  O(|\textbf{e}|e^{-cn}).
			\end{equation}
			
			$\bullet$ If $m\in I_{ss}^+$, $m^\prime\in I_{cu}^+$ and $-\frac{j_0}{\alpha_{l^\prime}^+}\in\left[\frac{n}{2},2n\right]$, we have that:
			\begin{equation}\label{lem:OndesReflechies:SsCu1}
				-\frac{1}{2i\pi}\int_{\Gamma_{in}(\eta)} e^{n\tau} e^\tau\widetilde{g}^+_{m^\prime,m}(e^\tau)\widetilde{\Ccc}_m^+(e^\tau,j_0,\textbf{e})\Pi(W_m^+(e^\tau,j))d\tau =  O\left(\frac{|\textbf{e}|e^{-c|j|}}{n^\frac{1}{2\mu}}\exp\left(-c\left(\frac{\left|n+\frac{j_0}{\alpha_{l^\prime}^+}\right|}{n^\frac{1}{2\mu}}\right)^\frac{2\mu}{2\mu-1}\right)\right).
			\end{equation}
			
			$\bullet$ If $m\in I_{ss}^+$, $m^\prime\in I_{cu}^+$ and $-\frac{j_0}{\alpha_{l^\prime}^+}\notin\left[\frac{n}{2},2n\right]$, we have that:
			\begin{equation}\label{lem:OndesReflechies:SsCu2}
				-\frac{1}{2i\pi}\int_{\Gamma_{in}(\eta)} e^{n\tau} e^\tau\widetilde{g}^+_{m^\prime,m}(e^\tau)\widetilde{\Ccc}_m^+(e^\tau,j_0,\textbf{e})\Pi(W_m^+(e^\tau,j))d\tau =  O\left(|\textbf{e}|e^{-cn}\right).
			\end{equation}
			
			$\bullet$ If $m\in I_{cs}^+$, $m^\prime\in I_{su}^+$ and $\frac{j}{\alpha_l^+}\in\left[\frac{n}{2},2n\right]$, we have that:
			\begin{multline}\label{lem:OndesReflechies:CsSu1}
				-\frac{1}{2i\pi}\int_{\Gamma_{in}(\eta)} e^{n\tau} e^\tau\widetilde{g}^+_{m^\prime,m}(e^\tau)\widetilde{\Ccc}_m^+(e^\tau,j_0,\textbf{e})\Pi(W_m^+(e^\tau,j))d\tau \\ = O(|\textbf{e}|e^{-cn})+O_\C\left(\frac{|\textbf{e}|e^{-c|j_0|}}{n^\frac{1}{2\mu}}\exp\left(-c\left(\frac{\left|n-\frac{j}{\alpha_l^+}\right|}{n^\frac{1}{2\mu}}\right)^\frac{2\mu}{2\mu-1}\right)\right)\rg_l^+.
			\end{multline}
			
			$\bullet$ If $m\in I_{cs}^+$, $m^\prime\in I_{su}^+$ and $\frac{j}{\alpha_l^+}\notin\left[\frac{n}{2},2n\right]$, we have that:
			\begin{equation}\label{lem:OndesReflechies:CsSu2}
				-\frac{1}{2i\pi}\int_{\Gamma_{in}(\eta)} e^{n\tau} e^\tau\widetilde{g}^+_{m^\prime,m}(e^\tau)\widetilde{\Ccc}_m^+(e^\tau,j_0,\textbf{e})\Pi(W_m^+(e^\tau,j))d\tau =  O\left(|\textbf{e}|e^{-cn}\right).
			\end{equation}
			
			$\bullet$ If $m\in I_{ss}^+$, $m^\prime\in I_{su}^+$, we have that:
			\begin{equation}\label{lem:OndesReflechies:SsSu}
				-\frac{1}{2i\pi}\int_{\Gamma_{in}(\eta)} e^{n\tau} e^\tau\widetilde{g}^+_{m^\prime,m}(e^\tau)\widetilde{\Ccc}_m^+(e^\tau,j_0,\textbf{e})\Pi(W_m^+(e^\tau,j))d\tau =  O\left(|\textbf{e}|e^{-cn}\right).
			\end{equation}
			
		\end{subequations}
	\end{lemma}
	
	We observe that since we consider $m\in\lc2,\ppp,dp\rc$ and $m^\prime\in\lc dp+1,\ppp,d(p+q)\rc$, Lemma \ref{lemCofac} implies that $\widetilde{g}_{m^\prime,m}^+$ can be holomorphically extended on the whole ball $B(1,\delta_1)$ and thus the term $\widetilde{g}_{m^\prime,m}^+(1)$ is well defined.
	
	\begin{proof}
		$\bullet$ We start by proving \eqref{lem:OndesReflechies:CsCu1}. We consider that $m\in I_{cs}^+$, $m^\prime \in I_{cu}^+$ and $\frac{j}{\alpha_l^+}-\frac{j_0}{\alpha_{l^\prime}^+}\in\left[\frac{n}{2},2n\right]$.
		
		Since the function $\widetilde{g}_{m^\prime,m}^+$ can be holomorphically extended on the whole ball $B(1,\varepsilon)$, using the expressions of $W_m^+$ and $\widetilde{\Ccc}_{m^\prime}^+$ given respectively by Lemma \ref{lem_choice_base} and \eqref{eg:CccTilde}, we have using Cauchy's formula that for any $\Gamma \in X$
		\begin{multline}\label{lem:OndesReflechies1}
			\int_{\Gamma_{in}(\eta)} e^\tau\widetilde{g}^+_{m^\prime,m}(e^\tau)\widetilde{\Ccc}_m^+(e^\tau,j_0,\textbf{e})\Pi(W_m^+(e^\tau,j))d\tau \\=\int_{\Gamma} e^{n\tau} \widetilde{g}^+_{m^\prime,m}(e^\tau)\zeta_m^+(e^\tau)^{j}\zeta_{m^\prime}^+(e^\tau)^{-j_0-1} e^\tau\Delta_{m^\prime}^+(e^\tau,j_0,\textbf{e})\Pi(V_m^+(e^\tau,j))d\tau.
		\end{multline}
		Using Cauchy's formula once again, \eqref{def:varpi} since the eigenvalues we consider are central and \eqref{def:R+}, we then have that
		\begin{multline}\label{lem:OndesReflechies2}
			-\frac{1}{2i\pi}\int_{\Gamma_{in}(\eta)} e^{n\tau} e^\tau\widetilde{g}_{m^\prime,m}^+(e^\tau)\widetilde{\Ccc}_m^+(e^\tau,j_0,\textbf{e})\Pi(W_m^+(e^\tau,j))d\tau  - \frac{\alpha_l^+}{\alpha_{l^\prime}^+}\widetilde{g}_{m^\prime,m}^+(1)R_{l^\prime,l}^+(n,j_0,j)\textbf{e}\\ =  E_1 +E_2 \rg_l^+ + \left(E_3  +E_4+E_5 +E_6 \right){\lg_{l^\prime}^+}^T\textbf{e} \rg_l^+
		\end{multline}
		where $E_1$ is a vector and $E_2,\ppp,E_6$ are complex scalars defined by
		\begin{align*}
			E_1 &= -\frac{1}{2i\pi}\int_{\Gamma_1} e^{n\tau} \widetilde{g}^+_{m^\prime,m}(e^\tau)e^{j\varpi_l^+(\tau)}e^{-(j_0+1)\varpi_{l^\prime}^+(\tau)} e^\tau\Delta_{m^\prime}^+(e^\tau,j_0,\textbf{e})\Pi(V_m^+(e^\tau,j)-R_m^+(e^\tau))d\tau,  \\
			E_2 &= -\frac{1}{2i\pi}\int_{\Gamma_2} e^{n\tau} \widetilde{g}^+_{m^\prime,m}(e^\tau)e^{j\varpi_l^+(\tau)}e^{-(j_0+1)\varpi_{l^\prime}^+(\tau)} e^\tau\left(\Delta_{m^\prime}^+(e^\tau,j_0,\textbf{e})-\frac{d\zeta_{m^\prime}^+}{dz}(e^\tau){\lg_{l^\prime}^+}^T\textbf{e}\right)d\tau,\\
			E_3 &=  -\frac{1}{2i\pi}\int_{\Gamma_3} e^{n\tau} e^{j\varpi_l^+(\tau)}e^{-j_0\varpi_{l^\prime}^+(\tau)} \left(\widetilde{g}_{m^\prime,m}^+(e^\tau)e^\tau\zeta_{m^\prime}^+(e^\tau)^{-1}\frac{d\zeta_{m^\prime}^+}{dz}(e^\tau)+\frac{\widetilde{g}_{m^\prime,m}^+(1)}{\alpha_{l^\prime}^+}\right) d\tau, \\
			E_4 &= \frac{\widetilde{g}_{m^\prime,m}^+(1)}{2i\pi\alpha_{l^\prime}^+}\int_{\Gamma_4} e^{n\tau} \left(e^{j\varpi_l^+(\tau)}-e^{j\varphi_l^+(\tau)}\right)e^{-j_0\varpi_{l^\prime}^+(\tau)}d\tau,\\
			E_5 &= \frac{\widetilde{g}_{m^\prime,m}^+(1)}{2i\pi\alpha_{l^\prime}^+}\int_{\Gamma_5} e^{n\tau}e^{j\varphi_l^+(\tau)} \left(e^{-j_0\varpi_{l^\prime}^+(\tau)}-e^{-j_0\varphi_{l^\prime}^+(\tau)}\right)d\tau,
		\end{align*}
		\begin{multline*}
			E_6 = \frac{\widetilde{g}_{m^\prime,m}^+(1)}{2i\pi\alpha_{l^\prime}^+}\int_{\Gamma_5} e^{n\tau}e^{j\varphi_l^+(\tau)} e^{-j_0\varphi_{l^\prime}^+(\tau)}d\tau \\ - \frac{\alpha_l^+}{\alpha_{l^\prime}^+}\widetilde{g}_{m^\prime,m}^+(1)\frac{1}{n^\frac{1}{2\mu}}H_{2\mu}\left(\frac{j}{n\alpha_l^+}\beta_l^+ -\frac{j_0}{n\alpha_{l^\prime}^+} \beta_{l^\prime}^+ \left(\frac{\alpha_l^+}{\alpha_{l^\prime}^+}\right)^{2\mu},\frac{n\alpha_l^++j_0\frac{\alpha_l^+}{\alpha_{l^\prime}^+}-j}{n^\frac{1}{2\mu}}\right),
		\end{multline*}
		and $\Gamma_1,\ppp,\Gamma_6$ are paths belonging to the set $X$. We just have to prove bounds on the terms $E_1,\ppp,E_6$. In particular, we will use good choices of paths $\Gamma_1,\ppp,\Gamma_6$ to optimize the bounds using Lemma \ref{lem:BornesGaussiennes}.
		
		$\blacktriangleright$ Using \eqref{in:Delta+}, \eqref{in:varpi} and Lemma \ref{lem_choice_base} which claims that the vectors $V_m^+(z,j)$ converge exponentially fast towards $R_m^+(e^\tau)$, we have that there exist two positive constants $C,c$ independent from $n$, $j_0$, $j$ and $\textbf{e}$ such that 
		$$|E_1| \leq Ce^{-c|j|}|\textbf{e}| \int_{\Gamma_1}\exp\left(n\Re(\tau)+\left(\frac{j}{\alpha_l^+}-\frac{j_0}{\alpha_{l^\prime}^+}\right)\left(-\Re(\tau)+A_R\Re(\tau)^{2\mu}-A_I\Im(\tau)^{2\mu}\right)\right)|d\tau|.$$
		
		$\blacktriangleright$ Using \eqref{in:DeltaInf+} and \eqref{in:varpi}, we have that there exist two positive constants $C,c$ independent from $n$, $j_0$, $j$ and $\textbf{e}$ such that 
		$$|E_2| \leq Ce^{-c|j_0|}|\textbf{e}| \int_{\Gamma_2}\exp\left(n\Re(\tau)+\left(\frac{j}{\alpha_l^+}-\frac{j_0}{\alpha_{l^\prime}^+}\right)\left(-\Re(\tau)+A_R\Re(\tau)^{2\mu}-A_I\Im(\tau)^{2\mu}\right)\right)|d\tau|.$$
		
		$\blacktriangleright$ We notice that $\zeta_{m^\prime}^+(1)=1$ and $\frac{d\zeta_{m^\prime}^+}{dz}(1)=-\frac{1}{\alpha_{\prime}^+}$. Using a Taylor expansion and \eqref{in:varpi}, we have that there exists a positive constant $C$ independent from $n$, $j_0$ and $j$ such that 
		$$|E_3| \leq C \int_{\Gamma_3}|\tau|\exp\left(n\Re(\tau)+\left(\frac{j}{\alpha_l^+}-\frac{j_0}{\alpha_{l^\prime}^+}\right)\left(-\Re(\tau)+A_R\Re(\tau)^{2\mu}-A_I\Im(\tau)^{2\mu}\right)\right)|d\tau|.$$
		
		$\blacktriangleright$ We observe that $\frac{j}{\alpha_l^+}$ and $-\frac{j_0}{\alpha_{l^\prime}^+}$ are positive and $\frac{j}{\alpha_l^+}-\frac{j_0}{\alpha_{l^\prime}^+}\in\left[\frac{n}{2},2n\right]$. Thus, we have that $\frac{j}{\alpha_l^+}\in[0,2n]$. We can then use Lemma \ref{lem:PassageVarpi-Varphi} and \eqref{in:varpi} to prove that there exists a positive constant $C$ independent from $n$, $j_0$ and $j$ such that
		$$|E_4| \leq C n \int_{\Gamma_4}|\tau|^{2\mu+1}\exp\left(n\Re(\tau)+\left(\frac{j}{\alpha_l^+}-\frac{j_0}{\alpha_{l^\prime}^+}\right)\left(-\Re(\tau)+A_R\Re(\tau)^{2\mu}-A_I\Im(\tau)^{2\mu}\right)\right)|d\tau|.$$
		Furthermore, we also have that and $-\frac{j_0}{\alpha_{l^\prime}^+}\in[0,2n]$. We can also use Lemma \ref{lem:PassageVarpi-Varphi} and \eqref{in:varphi} to prove that there exists a positive constant $C$ independent from $n$, $j_0$ and $j$ such that
		$$|E_5| \leq C n \int_{\Gamma_5}|\tau|^{2\mu+1}\exp\left(n\Re(\tau)+\left(\frac{j}{\alpha_l^+}-\frac{j_0}{\alpha_{l^\prime}^+}\right)\left(-\Re(\tau)+A_R\Re(\tau)^{2\mu}-A_I\Im(\tau)^{2\mu}\right)\right)|d\tau|.$$
		
		Using Lemma \ref{lem:BornesGaussiennes} which gives a good choices of path $\Gamma_1,\ppp,\Gamma_5\in X$ depending on $n$, $j_0$ and $j$ to handle the integrals in the terms above as well as Lemma \ref{lem:ComportementPrincGaussienne} to take care of the term $E_6$, there exist new constants $C,c>0$ independent from $n$, $j_0$, $j$ and $\textbf{e}$ such that 
		\begin{align*}
			|E_1|& \leq \frac{Ce^{-c|j|}|\textbf{e}|}{n^\frac{1}{2\mu}} \exp\left(-c\left(\frac{\left|n-\left(\frac{j}{\alpha_l^+}-\frac{j_0}{\alpha_{l^\prime}^+}\right)\right|}{n^\frac{1}{2\mu}}\right)^\frac{2\mu}{2\mu-1}\right), & |E_2|& \leq \frac{Ce^{-c|j_0|}|\textbf{e}|}{n^\frac{1}{2\mu}} \exp\left(-c\left(\frac{\left|n-\left(\frac{j}{\alpha_l^+}-\frac{j_0}{\alpha_{l^\prime}^+}\right)\right|}{n^\frac{1}{2\mu}}\right)^\frac{2\mu}{2\mu-1}\right), \\
			|E_3| &\leq \frac{C}{n^\frac{1}{\mu}} \exp\left(-c\left(\frac{\left|n-\left(\frac{j}{\alpha_l^+}-\frac{j_0}{\alpha_{l^\prime}^+}\right)\right|}{n^\frac{1}{2\mu}}\right)^\frac{2\mu}{2\mu-1}\right),& |E_4| &\leq \frac{C}{n^\frac{1}{\mu}} \exp\left(-c\left(\frac{\left|n-\left(\frac{j}{\alpha_l^+}-\frac{j_0}{\alpha_{l^\prime}^+}\right)\right|}{n^\frac{1}{2\mu}}\right)^\frac{2\mu}{2\mu-1}\right), \\
			|E_5| &\leq \frac{C}{n^\frac{1}{\mu}} \exp\left(-c\left(\frac{\left|n-\left(\frac{j}{\alpha_l^+}-\frac{j_0}{\alpha_{l^\prime}^+}\right)\right|}{n^\frac{1}{2\mu}}\right)^\frac{2\mu}{2\mu-1}\right), &|E_6| &\leq \frac{C}{n^\frac{1}{\mu}} \exp\left(-c\left(\frac{\left|n-\left(\frac{j}{\alpha_l^+}-\frac{j_0}{\alpha_{l^\prime}^+}\right)\right|}{n^\frac{1}{2\mu}}\right)^\frac{2\mu}{2\mu-1}\right).
		\end{align*}
		We have thus obtained \eqref{lem:OndesReflechies:CsCu1}.
		
		$\bullet$ We now focus on \eqref{lem:OndesReflechies:CsCu2}. We consider that $m\in I_{cs}^+$, $m^\prime \in I_{cu}^+$ and $\frac{j}{\alpha_l^+}-\frac{j_0}{\alpha_{l^\prime}^+}\notin\left[\frac{n}{2},2n\right]$. We observe that in particular, $R_{l^\prime,l}^+(n,j_0,j)=0$. Using \eqref{lem:OndesReflechies1}, \eqref{def:varpi} since the eigenvalue we consider is central, Lemma \ref{lem_choice_base} which claims that the vectors $V_m^+(z,j)$ are uniformly bounded for $z\in B(1,\delta_1)$ and $j\in\N$, \eqref{in:Delta+} and \eqref{in:varpi}, there exists a positive constant $C$ independent from $n$, $j_0$, $j$ and $\textbf{e}$ such that for all $\Gamma\in X$
		\begin{multline*}
			\left|-\frac{1}{2i\pi}\int_{\Gamma_{in}(\eta)} e^{n\tau} e^\tau\widetilde{g}_{m^\prime,m}^+(e^\tau)\widetilde{\Ccc}_m^+(e^\tau,j_0,\textbf{e})\Pi(W_m^+(e^\tau,j))d\tau  - \frac{\alpha_l^+}{\alpha_{l^\prime}^+}\widetilde{g}_{m^\prime,m}^+(1)R_{l^\prime,l}^+(n,j_0,j)\textbf{e}\right| \\ \leq C|\textbf{e}| \int_{\Gamma}\exp\left(n\Re(\tau)+\left(\frac{j}{\alpha_l^+}-\frac{j_0}{\alpha_{l^\prime}^+}\right)\left(-\Re(\tau)+A_R\Re(\tau)^{2\mu}-A_I\Im(\tau)^{2\mu}\right)\right)|d\tau|.
		\end{multline*}
		We then use Lemma \ref{lem:BornesGaussiennes} to prove that there exist two new positive constants $C,c$ independent from $n$, $j_0$, $j$ and $\textbf{e}$ such that
		$$ \left|-\frac{1}{2i\pi}\int_{\Gamma_{in}(\eta)} e^{n\tau} e^\tau\widetilde{g}_{m^\prime,m}^+(e^\tau)\widetilde{\Ccc}_m^+(e^\tau,j_0,\textbf{e})\Pi(W_m^+(e^\tau,j))d\tau  - \frac{\alpha_l^+}{\alpha_{l^\prime}^+}\widetilde{g}_{m^\prime,m}^+(1)R_{l^\prime,l}^+(n,j_0,j)\textbf{e}\right|\leq C|\textbf{e}|e^{-cn}.$$
		
		$\bullet$ We now focus on \eqref{lem:OndesReflechies:SsCu1} and \eqref{lem:OndesReflechies:SsCu2}. We consider that $m\in I_{ss}^+$, $m^\prime \in I_{cu}^+$. Using \eqref{lem:OndesReflechies1}, \eqref{inZetaSs} to bound $\zeta_m^+$, \eqref{def:varpi} since the index $m^\prime$ belongs to $I_{cu}^+$, Lemma \ref{lem_choice_base} which claims that the vectors $V_m^+(z,j)$ are uniformly bounded for $z\in B(1,\delta_1)$ and $j\in\N$, \eqref{in:Delta+} and \eqref{in:varpi}, there exists a positive constant $C$ independent from $n$, $j_0$, $j$ and $\textbf{e}$ such that for all $\Gamma\in X$
		\begin{multline*}
			\left|-\frac{1}{2i\pi}\int_{\Gamma_{in}(\eta)} e^{n\tau} e^\tau\widetilde{g}_{m^\prime,m}^+(e^\tau)\widetilde{\Ccc}_m^+(e^\tau,j_0,\textbf{e})\Pi(W_m^+(e^\tau,j))d\tau  - \frac{\alpha_l^+}{\alpha_{l^\prime}^+}\widetilde{g}_{m^\prime,m}^+(1)R_{l^\prime,l}^+(n,j_0,j)\textbf{e}\right| \\ \leq Ce^{-2c_\star|j|}|\textbf{e}| \int_{\Gamma}\exp\left(n\Re(\tau)+\left(-\frac{j_0}{\alpha_{l^\prime}^+}\right)\left(-\Re(\tau)+A_R\Re(\tau)^{2\mu}-A_I\Im(\tau)^{2\mu}\right)\right)|d\tau|.
		\end{multline*}
		We observe that $-\frac{j_0}{\alpha_{l^\prime}^+}$ is positive since $m^\prime$ belongs to $I_{cu}^+$. Using \eqref{lem:BornesGaussiennesRes3} when $-\frac{j_0}{\alpha_{l^\prime}^+}\in\left[\frac{n}{2},2n\right]$ and \eqref{lem:BornesGaussiennesRes1} and \eqref{lem:BornesGaussiennesRes2} else, we end up proving \eqref{lem:OndesReflechies:SsCu1} and \eqref{lem:OndesReflechies:SsCu2}.
		
		$\bullet$ We now focus on \eqref{lem:OndesReflechies:CsSu1} and \eqref{lem:OndesReflechies:CsSu2}. We consider that $m\in I_{cs}^+$, $m^\prime \in I_{su}^+$. Using \eqref{lem:OndesReflechies1}, \eqref{def:varpi} since the index $m$ belongs to $I_{cs}^+$ and Cauchy's formula, we have that 
		$$-\frac{1}{2i\pi}\int_{\Gamma_{in}(\eta)} e^{n\tau} e^\tau\widetilde{g}^+_{m^\prime,m}(e^\tau)\widetilde{\Ccc}_m^+(e^\tau,j_0,\textbf{e})\Pi(W_m^+(e^\tau,j))d\tau =E_1+E_2\rg_l^+$$
		where the vector $E_1$ and the complex scalar $E_2$ are defined by
		\begin{align*}
			E_1:= & \int_{\Gamma_1}e^{n\tau} e^\tau \widetilde{g}_{m^\prime,m}^+(e^\tau) e^{j\varpi_l^+(\tau)}\zeta_{m^\prime}^+(e^\tau)^{-j_0-1} \Delta^+_{m^\prime}(e^\tau,j_0,\textbf{e})\Pi(V_m^+(e^\tau,j)-R_m^+(e^\tau))d\tau\\
			E_2:= &  \int_{\Gamma_2}e^{n\tau} e^\tau \widetilde{g}_{m^\prime,m}^+(e^\tau) e^{j\varpi_l^+(\tau)}\zeta_{m^\prime}^+(e^\tau)^{-j_0-1} \Delta^+_{m^\prime}(e^\tau,j_0,\textbf{e})d\tau
		\end{align*}
		where $\Gamma_1,\Gamma_2$ are paths belonging to the set $X$. Using \eqref{inZetaSu} to bound $\zeta_{m^\prime}^+$, Lemma \ref{lem_choice_base} which claims that the vectors $V_m^+(z,j)$ converge exponentially fast towards $R_m^+(e^\tau)$, \eqref{in:Delta+} and \eqref{in:varpi}, we prove that there exist two positive constants $C,c$ independent from $n$, $j_0$, $j$, $\textbf{e}$, $\Gamma_1$ and $\Gamma_2$ such that
		\begin{align*}
			|E_1|\leq & C|\textbf{e}|e^{-c|j|}e^{-c|j_0|}\int_{\Gamma_1}\exp\left(n\Re(\tau)+\frac{j}{\alpha_l^+}\left(-\Re(\tau)+A_R\Re(\tau)^{2\mu}-A_I\Im(\tau)^{2\mu}\right)\right)|d\tau|,\\
			|E_2|\leq & C|\textbf{e}|e^{-c|j_0|} \int_{\Gamma_2}\exp\left(n\Re(\tau)+\frac{j}{\alpha_l^+}\left(-\Re(\tau)+A_R\Re(\tau)^{2\mu}-A_I\Im(\tau)^{2\mu}\right)\right)|d\tau|.
		\end{align*}
		Whilst observing that $\frac{j}{\alpha_l^+}$ is positive since $m$ belongs to $I_{cs}^+$, when $\frac{j}{\alpha_l^+}\notin\left[\frac{n}{2},2n\right]$, Lemma \ref{lem:BornesGaussiennes} allows us to prove exponential bounds with regard to $n$ on the terms $E_1$ and $E_2$ and to thus immediately conclude the proof of \eqref{lem:OndesReflechies:CsSu2}. When $\frac{j}{\alpha_l^+}\in\left[\frac{n}{2},2n\right]$, Lemma \ref{lem:BornesGaussiennes} allows us to choose $\Gamma_1$ and $\Gamma_2$ depending on $n$, $j_0$, $j$ and $\textbf{e}$ so that there exist new constants $C,c>0$ independent from $n$, $j_0$, $j$ and $\textbf{e}$ such that 
		\begin{align*}
			|E_1|& \leq \frac{C|\textbf{e}|e^{-c|j|}e^{-c|j_0|}}{n^\frac{1}{2\mu}} \exp\left(-c\left(\frac{\left|n-\frac{j}{\alpha_l^+}\right|}{n^\frac{1}{2\mu}}\right)^\frac{2\mu}{2\mu-1}\right) & |E_2|& \leq \frac{C|\textbf{e}|e^{-c|j_0|}}{n^\frac{1}{2\mu}} \exp\left(-c\left(\frac{\left|n-\frac{j}{\alpha_l^+}\right|}{n^\frac{1}{2\mu}}\right)^\frac{2\mu}{2\mu-1}\right).
		\end{align*}
		Since $j\in\left[\frac{n}{2},2n\right]$, we have that there exist two other constants $C,c>0$ independent from $n$, $j_0$, $j$ and $\textbf{e}$ such that
		$$|E_1|\leq C|\textbf{e}|e^{-cn}.$$
		This allows us to conclude \eqref{lem:OndesReflechies:CsSu1}.
		
		$\bullet$ There only remains to prove \eqref{lem:OndesReflechies:SsSu}. We observe using \eqref{lem:OndesReflechies1} with $\Gamma=\Gamma_d(\eta)\in X$, \eqref{inZetaSs} and \eqref{inZetaSu} to bound $\zeta_m^+$ and $\zeta_{m^\prime}^+$, \eqref{in:Delta+} and Lemma \ref{lem_choice_base} which claims that the vectors $V_m^+(z,j)$ are uniformly bounded for $z\in B(1,\delta_1)$ and $j\in\N$, we have that there exists a positive constant $C$ independent from $n$, $j_0$, $j$ and $\textbf{e}$ such that 
		$$ \left|-\frac{1}{2i\pi}\int_{\Gamma_{in}(\eta)} e^{n\tau} e^\tau\widetilde{g}^+_{m^\prime,m}(e^\tau)\widetilde{\Ccc}_m^+(e^\tau,j_0,\textbf{e})\Pi(W_m^+(e^\tau,j))d\tau\right|\leq C|\textbf{e}|e^{-2c_*(|j|+|j_0|)-n\eta}.$$
	\end{proof}

	\subsubsection{Transmitted waves}
	
	We now look at the transmitted waves. 
	
	\begin{lemma}\label{lem:OndesTransmises}
		We consider $m\in \lc dp+1, \ppp,d(p+q)-1\rc$ and $m^\prime\in \lc dp+1,\ppp,d(p+q)\rc$ and write them as $m=l+(k-1)d$ and $m^\prime=l^\prime+(k^\prime-1)d$ with $k,k^\prime\in\lc1,\ppp,p+q\rc$ and $l,l^\prime\in\lc1,\ppp,d\rc$. There exists a positive constant $c$ such that for all $n\in \N\backslash\lc0\rc$, $j_0\in \N$, $j\in-\N$ such that $j-j_0\in\lc-nq,\ppp,np\rc$ and $\textbf{e}\in\C^d$, we have:
		
		\begin{subequations}
			$\bullet$ If $m\in I_{cu}^-$, $m^\prime\in I_{cu}^+$ and $\frac{j}{\alpha_l^-}-\frac{j_0}{\alpha_{l^\prime}^+}\in\left[\frac{n}{2},2n\right]$, we have that:
			\begin{multline}\label{lem:OndesTransmises:CuCu1}
				\frac{1}{2i\pi}\int_{\Gamma_{in}(\eta)} e^{n\tau} e^\tau\widetilde{g}_{m^\prime,m}^+(e^\tau)\widetilde{\Ccc}_m^+(e^\tau,j_0,\textbf{e})\Pi(W_m^-(e^\tau,j))d\tau  - \frac{\alpha_l^-}{\alpha_{l^\prime}^+}\widetilde{g}_{m^\prime,m}^+(1)T_{l^\prime,l}^+(n,j_0,j)\textbf{e}\\ =  \exp\left(-c\left(\frac{\left|n-\left(\frac{j}{\alpha_l^-}-\frac{j_0}{\alpha_{l^\prime}^+}\right)\right|}{n^\frac{1}{2\mu}}\right)^\frac{2\mu}{2\mu-1}\right)\left(O\left(\frac{|\textbf{e}|e^{-c|j|}}{n^\frac{1}{2\mu}}\right)+O_\C\left(\frac{|\textbf{e}|e^{-c|j_0|}}{n^\frac{1}{2\mu}}\right)\rg_l^- + O_\C\left(\frac{1}{n^\frac{1}{\mu}}\right){\lg_{l^\prime}^+}^T\textbf{e}\rg_l^-\right).
			\end{multline}

			$\bullet$ If $m\in I_{cu}^-$, $m^\prime\in I_{cu}^+$ and $\frac{j}{\alpha_l^-}-\frac{j_0}{\alpha_{l^\prime}^+}\notin\left[\frac{n}{2},2n\right]$, we have that:
			\begin{equation}\label{lem:OndesTransmises:CuCu2}
				\frac{1}{2i\pi}\int_{\Gamma_{in}(\eta)} e^{n\tau} e^\tau\widetilde{g}_{m^\prime,m}^+(e^\tau)\widetilde{\Ccc}_m^+(e^\tau,j_0,\textbf{e})\Pi(W_m^-(e^\tau,j))d\tau  - \frac{\alpha_l^-}{\alpha_{l^\prime}^+}\widetilde{g}_{m^\prime,m}^+(1)T_{l^\prime,l}^+(n,j_0,j)\textbf{e} =  O\left(|\textbf{e}|e^{-cn}\right).
			\end{equation}
			
			$\bullet$ If $m\in I_{su}^-$, $m^\prime\in I_{cu}^+$ and $-\frac{j_0}{\alpha_{l^\prime}^+}\in\left[\frac{n}{2},2n\right]$, we have that:
			\begin{equation}\label{lem:OndesTransmises:SuCu1}
				\frac{1}{2i\pi}\int_{\Gamma_{in}(\eta)} e^{n\tau} e^\tau\widetilde{g}^+_{m^\prime,m}(e^\tau)\widetilde{\Ccc}_m^+(e^\tau,j_0,\textbf{e})\Pi(W_m^-(e^\tau,j))d\tau =  O\left(\frac{e^{-c|j|}|\textbf{e}|}{n^\frac{1}{2\mu}}\exp\left(-c\left(\frac{\left|n+\frac{j_0}{\alpha_{l^\prime}^+}\right|}{n^\frac{1}{2\mu}}\right)^\frac{2\mu}{2\mu-1}\right)\right).
			\end{equation}
			
			$\bullet$ If $m\in I_{su}^-$, $m^\prime\in I_{cu}^+$ and $-\frac{j_0}{\alpha_{l^\prime}^+}\notin\left[\frac{n}{2},2n\right]$, we have that:
			\begin{equation}\label{lem:OndesTransmises:SuCu2}
				\frac{1}{2i\pi}\int_{\Gamma_{in}(\eta)} e^{n\tau} e^\tau\widetilde{g}^+_{m^\prime,m}(e^\tau)\widetilde{\Ccc}_m^+(e^\tau,j_0,\textbf{e})\Pi(W_m^-(e^\tau,j))d\tau =  O\left(|\textbf{e}|e^{-cn}\right).
			\end{equation}
			
			$\bullet$ If $m\in I_{cu}^-$, $m^\prime\in I_{su}^+$ and $\frac{j}{\alpha_l^-}\in\left[\frac{n}{2},2n\right]$, we have that:
			\begin{multline}\label{lem:OndesTransmises:CuSu1}
				\frac{1}{2i\pi}\int_{\Gamma_{in}(\eta)} e^{n\tau} e^\tau\widetilde{g}^+_{m^\prime,m}(e^\tau)\widetilde{\Ccc}_m^+(e^\tau,j_0,\textbf{e})\Pi(W_m^-(e^\tau,j))d\tau \\= O(|\textbf{e}|e^{-cn})+O_\C\left(\frac{|\textbf{e}|e^{-c|j_0|}}{n^\frac{1}{2\mu}}\exp\left(-c\left(\frac{\left|n-\frac{j}{\alpha_l^-}\right|}{n^\frac{1}{2\mu}}\right)^\frac{2\mu}{2\mu-1}\right)\right)\rg_l^-.
			\end{multline}
			
			$\bullet$ If $m\in I_{cu}^-$, $m^\prime\in I_{su}^+$ and $\frac{j}{\alpha_l^-}\notin\left[\frac{n}{2},2n\right]$, we have that:
			\begin{equation}\label{lem:OndesTransmises:CuSu2}
				\frac{1}{2i\pi}\int_{\Gamma_{in}(\eta)} e^{n\tau} e^\tau\widetilde{g}^+_{m^\prime,m}(e^\tau)\widetilde{\Ccc}_m^+(e^\tau,j_0,\textbf{e})\Pi(W_m^-(e^\tau,j))d\tau =  O\left(|\textbf{e}|e^{-cn}\right).
			\end{equation}
			
			$\bullet$ If $m\in I_{su}^-$, $m^\prime\in I_{su}^+$, we have that:
			\begin{equation}\label{lem:OndesTransmises:SuSu}
				\frac{1}{2i\pi}\int_{\Gamma_{in}(\eta)} e^{n\tau} e^\tau\widetilde{g}^+_{m^\prime,m}(e^\tau)\widetilde{\Ccc}_m^+(e^\tau,j_0,\textbf{e})\Pi(W_m^-(e^\tau,j))d\tau =  O\left(|\textbf{e}|e^{-cn}\right).
			\end{equation}
		\end{subequations}
	\end{lemma}
	
	Just like in the case of the reflected waves, since we consider $m\in\lc dp+1,\ppp,d(p+q)-1\rc$ and $m^\prime\in\lc dp+1,\ppp,d(p+q)\rc$, Lemma \ref{lemCofac} implies that $\widetilde{g}_{m^\prime,m}^+$ can be holomorphically extended on the whole ball $B(1,\delta_1)$ and thus the term $\widetilde{g}_{m^\prime,m}^+(1)$ is well defined.
	
	\begin{proof}
		The proof of Lemma \ref{lem:OndesTransmises} is sensibly the same one as for Lemma \ref{lem:OndesReflechies} so the proof is left to the reader. Let us just point out that in order to prove \eqref{lem:OndesTransmises:CuCu1}, we have using Cauchy's formula that
		\begin{multline*}
			\frac{1}{2i\pi}\int_{\Gamma_{in}(\eta)} e^{n\tau} e^\tau\widetilde{g}_{m^\prime,m}^+(e^\tau)\widetilde{\Ccc}_m^+(e^\tau,j_0,\textbf{e})\Pi(W_m^-(e^\tau,j))d\tau  - \frac{\alpha_l^-}{\alpha_{l^\prime}^+}\widetilde{g}_{m^\prime,m}^+(1)T_{l^\prime,l}^+(n,j_0,j)\textbf{e}\\ =  E_1 +E_2 \rg_l^- + \left(E_3  +E_4+E_5 +E_6 \right){\lg_{l^\prime}^+}^T\textbf{e} \rg_l^-
		\end{multline*}
		where $E_1$ is a vector and $E_2,\ppp,E_6$ are complex scalars defined by
		\begin{align*}
			E_1 &= \frac{1}{2i\pi}\int_{\Gamma_1} e^{n\tau} \widetilde{g}^+_{m^\prime,m}(e^\tau)e^{j\varpi_l^-(\tau)}e^{-(j_0+1)\varpi_{l^\prime}^+(\tau)} e^\tau\Delta_{m^\prime}^+(e^\tau,j_0,\textbf{e})\Pi(V_m^-(e^\tau,j)-R_m^-(e^\tau))d\tau,  \\
			E_2 &= \frac{1}{2i\pi}\int_{\Gamma_2} e^{n\tau} \widetilde{g}^+_{m^\prime,m}(e^\tau)e^{j\varpi_l^-(\tau)}e^{-(j_0+1)\varpi_{l^\prime}^+(\tau)} e^\tau\left(\Delta_{m^\prime}^+(e^\tau,j_0,\textbf{e})-\frac{d\zeta_{m^\prime}^+}{dz}(e^\tau){\lg_{l^\prime}^+}^T\textbf{e}\right)d\tau,\\
			E_3 &= \frac{1}{2i\pi}\int_{\Gamma_3} e^{n\tau} e^{j\varpi_l^-(\tau)}e^{-j_0\varpi_{l^\prime}^+(\tau)} \left(\widetilde{g}_{m^\prime,m}^+(e^\tau)e^\tau\zeta_{m^\prime}^+(e^\tau)^{-1}\frac{d\zeta_{m^\prime}^+}{dz}(e^\tau)+\frac{\widetilde{g}_{m^\prime,m}^+(1)}{\alpha_{l^\prime}^+}\right) d\tau, \\
			E_4 &= -\frac{\widetilde{g}_{m^\prime,m}^+(1)}{2i\pi\alpha_{l^\prime}^+}\int_{\Gamma_4} e^{n\tau} \left(e^{j\varpi_l^-(\tau)}-e^{j\varphi_l^-(\tau)}\right)e^{-j_0\varpi_{l^\prime}^+(\tau)}d\tau\\
			E_5 &= -\frac{\widetilde{g}_{m^\prime,m}^+(1)}{2i\pi\alpha_{l^\prime}^+}\int_{\Gamma_5} e^{n\tau}e^{j\varphi_l^-(\tau)} \left(e^{-j_0\varpi_{l^\prime}^+(\tau)}-e^{-j_0\varphi_{l^\prime}^+(\tau)}\right)d\tau,\\
			E_6 &= -\frac{\widetilde{g}_{m^\prime,m}^+(1)}{2i\pi\alpha_{l^\prime}^+}\int_{\Gamma_5} e^{n\tau}e^{j\varphi_l^-(\tau)} e^{-j_0\varphi_{l^\prime}^+(\tau)}d\tau\\ 
			&\hspace{5cm}- \frac{\alpha_l^-}{\alpha_{l^\prime}^+}\widetilde{g}_{m^\prime,m}^+(1)\frac{1}{n^\frac{1}{2\mu}}H_{2\mu}\left(\frac{j}{n\alpha_l^-}\beta_l^- -\frac{j_0}{n\alpha_{l^\prime}^+} \beta_{l^\prime}^+ \left(\frac{\alpha_l^-}{\alpha_{l^\prime}^+}\right)^{2\mu},\frac{n\alpha_l^-+j_0\frac{\alpha_l^-}{\alpha_{l^\prime}^+}-j}{n^\frac{1}{2\mu}}\right),
		\end{align*}
		and $\Gamma_1,\ppp,\Gamma_6$ are paths belonging to the set $X$. We just have to prove bounds on the terms $E_1,\ppp,E_6$ just like in the proof of \eqref{lem:OndesReflechies:CsCu1}.
	\end{proof}
	
	\subsubsection{Unstable excited mode}
	
	\begin{lemma}\label{lem:OndesExcitedUnstable}
		There exist two positive constants $C,c$ such that for all $m^\prime \in I_{su}^+$, $n\in \N\backslash\lc0\rc$, $j_0\in \N$, $j\in \Z$ such that $j-j_0\in\lc-nq,\ppp,np\rc$ and $\textbf{e}\in\C^d$, we have
		
		\begin{subequations}
			$\bullet$ For $j\in\N$, we have that:
			\begin{multline}\label{lem:OndesExcitedUnstable:Res1}
				\left|-\frac{1}{2i\pi}\int_{\Gamma_{in}(\eta)}e^{n\tau}e^\tau\widetilde{g}_{m^\prime, 1}^+(e^\tau)\widetilde{\Ccc}^+_{m^\prime}(e^\tau,j_0,\textbf{e})\Pi(\Phi_1(e^\tau,j))d\tau +\zeta_{m^\prime}^+(1)^{-j_0-1}\Delta_{m^\prime}^+(1,j_0,\textbf{e})\mathrm{Res}(\widetilde{g}_{m^\prime}^+,1)V(j)\right|\\ \leq C|\textbf{e}|e^{-cn}.
			\end{multline}	
			
			$\bullet$ For $j\in-\N$, we have that:
			\begin{multline}\label{lem:OndesExcitedUnstable:Res2}
				\left|\frac{1}{2i\pi}\int_{\Gamma_{in}(\eta)}e^{n\tau}e^\tau\widetilde{g}_{m^\prime, d(p+q)}^+(e^\tau)\widetilde{\Ccc}^+_{m^\prime}(e^\tau,j_0,\textbf{e})\Pi(\Phi_{d(p+q)}(e^\tau,j))d\tau +\zeta_{m^\prime}^+(1)^{-j_0-1}\Delta_{m^\prime}^+(1,j_0,\textbf{e})\mathrm{Res}(\widetilde{g}_{m^\prime}^+,1)V(j)\right|\\ \leq C|\textbf{e}|e^{-cn}.
			\end{multline}
		\end{subequations}
	\end{lemma}
	
	We recall that the sequence $V$ is defined by \eqref{def:V}.
	
	\begin{proof}
		We are going to prove \eqref{lem:OndesExcitedUnstable:Res1}. We consider $m^\prime\in I_{su}^+$, $n\in\N\backslash\lc0\rc$, $j_0\in\N$. For $j\in\N$, using the residue theorem and the equality \eqref{eg:CccTilde}, we have
		\begin{multline*}
			-\frac{1}{2i\pi}\int_{\Gamma_{in}(\eta)} e^{n\tau} e^\tau \widetilde{g}^+_{m^\prime,1}(e^\tau) \widetilde{\Ccc}^+_{m^\prime}(e^\tau,j_0,\textbf{e})\Pi(\Phi_1(e^\tau,j))d \tau +\zeta^+_{m^\prime}(1)^{-j_0-1} \Delta_{m^\prime}^+(1,j_0,\textbf{e})\mathrm{Res}(\widetilde{g}_{m^\prime,1}^+,1)V(j) \\ = -\frac{1}{2i\pi}\int_{\Gamma_{d}(\eta)} e^{n\tau} \zeta^+_{m^\prime}(e^\tau)^{-j_0-1}e^\tau \widetilde{g}^+_{m^\prime,1}(e^\tau) \Delta^+_{m^\prime}(e^\tau,j_0,\textbf{e})\Pi(\Phi_1(e^\tau,j))d \tau   .
		\end{multline*} 
		Using \eqref{inZetaSu} to bound $\zeta_{m^\prime}^+$, \eqref{in:Delta+} to handle the term $\Delta_{m^\prime}^+$ and \eqref{in:decroissance_expo_Phi_1_et_d(p+q)} to bound $\Phi_1$, there exists another positive constant $C$ independent from $n$, $j_0$, $j$ and $\textbf{e}$ such that
		\begin{multline*}
			\left|-\frac{1}{2i\pi}\int_{\Gamma_{d}(\eta)} e^{n\tau} \zeta^+_{m^\prime}(e^\tau)^{-j_0-1}e^\tau \widetilde{g}^+_{m^\prime,1}(e^\tau) \Delta^+_{m^\prime}(e^\tau,j_0,\textbf{e})\Pi(\Phi_1(e^\tau,j))d \tau\right|\\ \leq C |\textbf{e}|\int_{\Gamma_d(\eta)}e^{n\Re(\tau)}|d\tau|\leq 2r_\varepsilon(\eta)C |\textbf{e}|e^{-n\eta}.
		\end{multline*} 
		We thus obtain \eqref{lem:OndesExcitedUnstable:Res1}. The proof of \eqref{lem:OndesExcitedUnstable:Res2} is fairly similar and is left to the reader.
	\end{proof}
	
	\subsubsection{Central excited mode}

		\begin{lemma}\label{lem:OndesExcitedCentral}
			We consider $m^\prime\in I_{cu}=I_{cu}^+$ and write it as $m^\prime=l^\prime+(k^\prime-1)d$ with $k^\prime\in\lc1,\ppp,p+q\rc$ and $l^\prime\in\lc1,\ppp,d\rc$. There exists a positive constant $c$ such that for all $n\in \N\backslash\lc0\rc$, $j_0\in \N$, $j\in \Z$ such that $j-j_0\in\lc-nq,\ppp,np\rc$ and $\textbf{e}\in\C^d$, we have:
			
			\begin{subequations}
				$\bullet$ For $-\frac{j_0}{\alpha_{l^\prime}^+}\in\left[\frac{n}{2},2n\right]$ and $j\geq0$, we have that:
				\begin{multline}\label{lem:OndesExcitedCentral:Res+1}
					-\frac{1}{2i\pi}\int_{\Gamma_{in}(\eta)}e^{n\tau}e^\tau\widetilde{g}_{m^\prime, 1}^+(e^\tau)\widetilde{\Ccc}^+_{m^\prime}(e^\tau,j_0,\textbf{e})\Pi(\Phi_1(e^\tau,j))d\tau \\
					-\left(\frac{1}{\alpha_{l^\prime}^+}E_{l^\prime}^+(n,j_0)\textbf{e} -\left(\Delta_{m^\prime}^+(1,j_0,\textbf{e})+\frac{{\lg_{l^\prime}^+}^T\textbf{e}}{\alpha_{l^\prime}^+}\right)\right)\mathrm{Res}(\widetilde{g}_{m^\prime,1}^+,1)V(j)= \\ 
					O\left(\frac{|\textbf{e}|e^{-c|j|}}{n^\frac{1}{2\mu}}\exp\left(-c\left(\frac{\left|n+\frac{j_0}{\alpha_l^+}\right|}{n^\frac{1}{2\mu}}\right)^\frac{2\mu}{2\mu-1}\right)\right)+O\left(|\textbf{e}|e^{-cn}\right).
				\end{multline}	
				
				$\bullet$ For $-\frac{j_0}{\alpha_{l^\prime}^+}\notin\left[\frac{n}{2},2n\right]$ and $j\geq0$, we have that:
				\begin{multline}\label{lem:OndesExcitedCentral:Res+2}
					-\frac{1}{2i\pi}\int_{\Gamma_{in}(\eta)}e^{n\tau}e^\tau\widetilde{g}_{m^\prime, 1}^+(e^\tau)\widetilde{\Ccc}^+_{m^\prime}(e^\tau,j_0,\textbf{e})\Pi(\Phi_1(e^\tau,j))d\tau \\
					-\left(\frac{1}{\alpha_{l^\prime}^+}E_{l^\prime}^+(n,j_0)\textbf{e} -\left(\Delta_{m^\prime}^+(1,j_0,\textbf{e})+\frac{{\lg_{l^\prime}^+}^T\textbf{e}}{\alpha_{l^\prime}^+}\right)\right)\mathrm{Res}(\widetilde{g}_{m^\prime,1}^+,1)V(j)=O(|\textbf{e}|e^{-cn}).
				\end{multline}	
				
				$\bullet$ For $-\frac{j_0}{\alpha_{l^\prime}^+}\in\left[\frac{n}{2},2n\right]$ and $j<0$, we have that:
				\begin{multline}\label{lem:OndesExcitedCentral:Res-1}
					\frac{1}{2i\pi}\int_{\Gamma_{in}(\eta)}e^{n\tau}e^\tau\widetilde{g}_{m^\prime, d(p+q)}^+(e^\tau)\widetilde{\Ccc}^+_{m^\prime}(e^\tau,j_0,\textbf{e})\Pi(\Phi_{d(p+q)}(e^\tau,j))d\tau\\
					-\left(\frac{1}{\alpha_{l^\prime}^+}E_{l^\prime}^+(n,j_0)\textbf{e} -\left(\Delta_{m^\prime}^+(1,j_0,\textbf{e})+\frac{{\lg_{l^\prime}^+}^T\textbf{e}}{\alpha_{l^\prime}^+}\right)\right)\mathrm{Res}(\widetilde{g}_{m^\prime,1}^+,1)V(j)= \\ 
					O\left(\frac{|\textbf{e}|e^{-c|j|}}{n^\frac{1}{2\mu}}\exp\left(-c\left(\frac{\left|n+\frac{j_0}{\alpha_l^+}\right|}{n^\frac{1}{2\mu}}\right)^\frac{2\mu}{2\mu-1}\right)\right)+O\left(|\textbf{e}|e^{-cn}\right).
				\end{multline}	
				
				$\bullet$ For $-\frac{j_0}{\alpha_{l^\prime}^+}\notin\left[\frac{n}{2},2n\right]$ and $j<0$, we have that:
				\begin{multline}\label{lem:OndesExcitedCentral:Res-2}
					\frac{1}{2i\pi}\int_{\Gamma_{in}(\eta)}e^{n\tau}e^\tau\widetilde{g}_{m^\prime, d(p+q)}^+(e^\tau)\widetilde{\Ccc}^+_{m^\prime}(e^\tau,j_0,\textbf{e})\Pi(\Phi_{d(p+q)}(e^\tau,j))d\tau\\
					-\left(\frac{1}{\alpha_{l^\prime}^+}E_{l^\prime}^+(n,j_0)\textbf{e} -\left(\Delta_{m^\prime}^+(1,j_0,\textbf{e})+\frac{{\lg_{l^\prime}^+}^T\textbf{e}}{\alpha_{l^\prime}^+}\right)\right)\mathrm{Res}(\widetilde{g}_{m^\prime,1}^+,1)V(j)= O(|\textbf{e}|e^{-cn}).
				\end{multline}	
			\end{subequations}
		\end{lemma}
		
		We recall once again that the sequence $V$ is defined by \eqref{def:V}.
		
		\begin{proof}
			We will focus on proving \eqref{lem:OndesExcitedCentral:Res+1} and \eqref{lem:OndesExcitedCentral:Res+2} as the proof of  \eqref{lem:OndesExcitedCentral:Res-1} and \eqref{lem:OndesExcitedCentral:Res-2} would be similar whilst observing that the equality \eqref{egPhi} is verified and that the definition \eqref{def:gtilde} of $\widetilde{g}^\pm_{m^\prime,m}$ and \eqref{eg:Cofac} imply that:
			$$\mathrm{Res}(\widetilde{g}_{m^\prime,d(p+q)}^+,1) = -\mathrm{Res}(\widetilde{g}_{m^\prime,1}^+,1).$$
			
			$\bullet$ \textbf{Proof of \eqref{lem:OndesExcitedCentral:Res+1}:}
			
			Using Cauchy's formula, \eqref{eg:CccTilde} and \eqref{def:varpi} since the eigenvalue $\zeta^+_{m^\prime}$ is central, we have
			\begin{multline}\label{lem:OndesExcitedCentral:1}
				-\frac{1}{2i\pi}\int_{\Gamma_{in}(\eta)}e^{n\tau}e^\tau\widetilde{g}_{m^\prime, 1}^+(e^\tau)\widetilde{\Ccc}^+_{m^\prime}(e^\tau,j_0,\textbf{e})\Pi(\Phi_1(e^\tau,j))d\tau \\
				-\left(\frac{1}{\alpha_{l^\prime}^+}E_{l^\prime}^+(n,j_0)\textbf{e} -\left(\Delta_{m^\prime}^+(1,j_0,\textbf{e})+\frac{{\lg_{l^\prime}^+}^T\textbf{e}}{\alpha_{l^\prime}^+}\right)\right)\mathrm{Res}(\widetilde{g}_{m^\prime,1}^+,1)V(j)\\  = E_1+E_2V(j)+(E_3+E_4+E_5)\Delta^+_{m^\prime}(1,j_0,\textbf{e}) V(j)+ E_6V(j)
			\end{multline}
			where
			\begin{align*}
				E_1 := & -\frac{1}{2i\pi} \int_{\Gamma_1} e^{n\tau}e^\tau\widetilde{g}^+_{m^\prime,1}(e^\tau)e^{-(j_0+1)\varpi^+_{l^\prime}(\tau)}\Delta^+_{m^\prime}(e^\tau,j_0,\textbf{e}) \Pi(\Phi_1(e^\tau,j)-\Phi_1(1,j)) d\tau,\\
				E_2 := & -\frac{1}{2i\pi} \int_{\Gamma_2} e^{n\tau}e^\tau\widetilde{g}^+_{m^\prime,1}(e^\tau)e^{-(j_0+1)\varpi^+_{l^\prime}(\tau)}\left(\Delta^+_{m^\prime}(e^\tau,j_0,\textbf{e}) -\Delta^+_{m^\prime}(1,j_0,\textbf{e}) \right)d\tau,\\
				E_3 := & -\frac{1}{2i\pi} \int_{\Gamma_3} e^{n\tau}e^{-j_0\varpi^+_{l^\prime}(\tau)}\left(e^\tau\widetilde{g}^+_{m^\prime,1}(e^\tau)\zeta^+_{m^\prime}(e^\tau)^{-1}-\frac{\mathrm{Res}(\widetilde{g}^+_{m^\prime,1},1)}{\tau} \right)d\tau,\\
				E_4 := & -\frac{\mathrm{Res}(\widetilde{g}^+_{m^\prime,1},1)}{2i\pi} \int_{\Gamma_4} e^{n\tau}\frac{e^{-j_0\varpi^+_{l^\prime}(\tau)}-e^{-j_0\varphi^+_{l^\prime}(\tau)}}{\tau}d\tau,\\
				E_5 := & -\mathrm{Res}(\widetilde{g}^+_{m^\prime,1},1)\left(\frac{1}{2i\pi} \int_{\Gamma_{in}(\eta)} e^{n\tau}\frac{e^{-j_0\varphi^+_{l^\prime}(\tau)}}{\tau}d\tau-E_{2\mu}\left(\beta^+_{l^\prime};\frac{n\alpha^+_{l^\prime}+j_0}{n^\frac{1}{2\mu}}\right)\right),\\
				E_6 := & \mathrm{Res}(\widetilde{g}^+_{m^\prime,1},1)\left(1-E_{2\mu}\left(\beta^+_{l^\prime},\frac{n\alpha^+_{l^\prime}+j_0}{n^\frac{1}{2\mu}}\right)\right)\left( \Delta^+_{m^\prime}(1,j_0,\textbf{e})+\frac{{\lg^+_{l^\prime}}^T\textbf{e}}{\alpha^+_{l^\prime}}\right),
			\end{align*}
			and $\Gamma_1,\ppp,\Gamma_4\in X$. Let us observe that, since the function $\widetilde{g}^+_{m^\prime,1}$ has a simple pole of order $1$ at $1$, we have the right to use the Cauchy's formula for the first 4 terms as the functions inside the integrals can be holomorphically extended on the whole ball $B(1,\varepsilon)$.
			
			$\blacktriangleright$ Using \eqref{in:Delta+} to bound $\Delta^+_{m^\prime}(e^\tau,j_0,\textbf{e})$, \eqref{lem:BornesVDelta:inPhi1} to bound $\widetilde{g}^+_{m^\prime,1}(e^\tau)\Pi\left(\Phi_1(e^\tau,j)-\Phi_1(1,j)\right)$ and \eqref{in:varpi}, there exists a constant $C>0$ independent from $n$, $j_0$, $j$, $\textbf{e}$ and $\Gamma_1$ such that 
			$$|E_1|\leq C|\textbf{e}|e^{-\frac{3c_*}{2}|j|}\int_{\Gamma_1}\exp\left(n\Re(\tau)-\frac{j_0}{\alpha^+_{l^\prime}}\left(-\Re(\tau)+A_R\Re(\tau)^{2\mu}-A_I\Im(\tau)^{2\mu}\right)\right) |d\tau|.$$ 
			
			$\blacktriangleright$ Using \eqref{lem:BornesVDelta:inDelta+} to bound $\widetilde{g}^+_{m^\prime,1}(e^\tau)\left(\Delta^+_{m^\prime}(e^\tau,j_0,\textbf{e})-\Delta^+_{m^\prime}(1,j_0,\textbf{e})\right)$ and \eqref{in:varpi}, there exists a constant $C>0$ independent from $n$, $j_0$, $j$, $\textbf{e}$ and $\Gamma_2$ such that 
			$$|E_2|\leq C|\textbf{e}|\int_{\Gamma_2}\exp\left(n\Re(\tau)-\frac{j_0}{\alpha^+_{l^\prime}}\left(-\Re(\tau)+A_R\Re(\tau)^{2\mu}-A_I\Im(\tau)^{2\mu}\right)\right) |d\tau|.$$ 
			
			$\blacktriangleright$ Using \eqref{in:varpi}, there exists a constant $C>0$ independent from $n$, $j_0$, $j$, $\textbf{e}$ and $\Gamma_3$ such that 
			$$|E_3|\leq C\int_{\Gamma_3}\exp\left(n\Re(\tau)-\frac{j_0}{\alpha^+_{l^\prime}}\left(-\Re(\tau)+A_R\Re(\tau)^{2\mu}-A_I\Im(\tau)^{2\mu}\right)\right) |d\tau|.$$ 
			
			$\blacktriangleright$ Using Lemma \ref{lem:BorneVarphiVarpi}, there exists a constant $C>0$ independent from $n$, $j_0$, $j$, $\textbf{e}$ and $\Gamma_4$ such that 
			$$|E_4|\leq Cn\int_{\Gamma_4}|\tau|^{2\mu}\exp\left(n\Re(\tau)-\frac{j_0}{\alpha^+_{l^\prime}}\left(-\Re(\tau)+A_R\Re(\tau)^{2\mu}-A_I\Im(\tau)^{2\mu}\right)\right) |d\tau|.$$ 
			
			Using Lemma \ref{lem:BornesGaussiennes} which gives a good choices of path $\Gamma_1,\ppp,\Gamma_4\in X$ depending on $n$, $j_0$ and $j$ to handle the integrals in the terms above, there exist new constants $C,c>0$ independent from $n$, $j_0$, $j$ and $\textbf{e}$ such that 
			\begin{align*}
				|E_1|& \leq \frac{Ce^{-c|j|}|\textbf{e}|}{n^\frac{1}{2\mu}} \exp\left(-c\left(\frac{\left|n+\frac{j_0}{\alpha_{l^\prime}^+}\right|}{n^\frac{1}{2\mu}}\right)^\frac{2\mu}{2\mu-1}\right), & |E_2|& \leq \frac{C|\textbf{e}|}{n^\frac{1}{2\mu}} \exp\left(-c\left(\frac{\left|n+\frac{j_0}{\alpha_{l^\prime}^+}\right|}{n^\frac{1}{2\mu}}\right)^\frac{2\mu}{2\mu-1}\right), \\
				|E_3| &\leq \frac{C}{n^\frac{1}{2\mu}} \exp\left(-c\left(\frac{\left|n+\frac{j_0}{\alpha_{l^\prime}^+}\right|}{n^\frac{1}{2\mu}}\right)^\frac{2\mu}{2\mu-1}\right),& |E_4| &\leq \frac{C}{n^\frac{1}{2\mu}} \exp\left(-c\left(\frac{\left|n+\frac{j_0}{\alpha_{l^\prime}^+}\right|}{n^\frac{1}{2\mu}}\right)^\frac{2\mu}{2\mu-1}\right).
			\end{align*}
			There remains to bound the terms $E_5$ and $E_6$. We use \eqref{lem:ComportementPrincGaussienneRes3} of Lemma \ref{lem:ComportementPrincGaussienne} to bound the term $E_5$ and obtain the existence of two positive constants $C,c$ such that:
			$$ 	|E_5| \leq \frac{C}{n^\frac{1}{2\mu}} \exp\left(-c\left(\frac{\left|n+\frac{j_0}{\alpha_{l^\prime}^+}\right|}{n^\frac{1}{2\mu}}\right)^\frac{2\mu}{2\mu-1}\right).$$
			We observe that \eqref{def:varpi} and the asymptotic expansion \eqref{eg:LienVarpiVarphi} imply that:
			$$ \frac{d\zeta_{m^\prime}^+}{dz}(1) =-\frac{1}{\alpha_{l^\prime}^+}.$$
			Then, using \eqref{in:DeltaInf+} and the fact that the function $E_{2\mu}(\beta^+_{l^\prime};\cdot)$ is bounded to handle $E_6$, we can also consider that we chose the constants $C,c$ so that:
			$$ |E_6| \leq C|\textbf{e}|e^{-c|j_0|}.$$
			Since $-\frac{j_0}{\alpha_{l^\prime}^+}\in\left[\frac{n}{2},2n\right]$, we can change the constants $C,c$ such that:
			$$|E_6| \leq C|\textbf{e}|e^{-cn}.$$

			Using \eqref{in:decroissance_expo_Phi_1_et_d(p+q)} and \eqref{in:Delta+} to bound $V(j)=\Pi(\Phi_1(1,j))$ and $\Delta^+_{m^\prime}$ in \eqref{lem:OndesExcitedCentral:1} and the several bounds on the terms $E_1,\hdots,E_6$ that we just proved, we can conclude the proof of \eqref{lem:OndesExcitedCentral:Res+1}.
			
			$\bullet$ \textbf{Proof of \eqref{lem:OndesExcitedCentral:Res+2}:}
			
			We will separate this proof into two parts.
			
			$\blacktriangleright$ Let us assume that $-\frac{j_0}{\alpha_{^\prime}^+}\in\left[0,\frac{n}{2}\right]$. We recall that $\widetilde{g}_{m^\prime,1}^+$ is a meromorphic function with a pole of order $1$ at $1$. Using the Residue Theorem, we have:
			\begin{multline*}
				-\frac{1}{2i\pi}\int_{\Gamma_{in}(\eta)}e^{n\tau}e^\tau\widetilde{g}_{m^\prime, 1}^+(e^\tau)\widetilde{\Ccc}^+_{m^\prime}(e^\tau,j_0,\textbf{e})\Pi(\Phi_1(e^\tau,j))d\tau\\
				-\left(\frac{1}{\alpha_{l^\prime}^+}E_{l^\prime}^+(n,j_0)\textbf{e} -\left(\Delta_{m^\prime}^+(1,j_0,\textbf{e})+\frac{{\lg_{l^\prime}^+}^T\textbf{e}}{\alpha_{l^\prime}^+}\right)\right)\mathrm{Res}(\widetilde{g}_{m^\prime,1}^+,1)V(j) \\ 
				=-\frac{1}{2i\pi}\int_{\Gamma_{d}(\eta)}e^{n\tau}e^\tau\widetilde{g}_{m^\prime, 1}^+(e^\tau)e^{-(j_0+1)\varpi^+_{l^\prime}(\tau)}\Delta^+_{m^\prime}(e^\tau,j_0,\textbf{e})\Pi(\Phi_1(e^\tau,j))d\tau\\
				-\frac{\mathrm{Res}(\widetilde{g}_{m^\prime,1}^+,1)}{\alpha_{l^\prime}^+}\left(1-E_{2\mu}\left(\beta^+_{l^\prime};\frac{n\alpha^+_{l^\prime}+j_0}{n^\frac{1}{2\mu}}\right)\right){\lg^+_{l^\prime}}^T\textbf{e}V(j).
			\end{multline*}
			We need to obtain exponential bounds on both terms on the right hand side of the equality above. First, we observe that since $-\frac{j_0}{\alpha^+_{l^\prime}}$ belongs to $\left[0,\frac{n}{2}\right]$, we have 
			$$\frac{n\alpha^+_{l^\prime}+j_0}{n^\frac{1}{2\mu}}\leq \frac{\alpha^+_{l^\prime}}{2}n^\frac{2\mu-1}{2\mu}. $$
			We have that $\alpha^+_{l^\prime}<0$ since $m^\prime$ belongs to $I_{cu}^+$ and thus, using \eqref{inE-} and \eqref{in:decroissance_expo_Phi_1_et_d(p+q)}, we have that there exist two positive constants $C,c$ such that for all $n\in\N\backslash\lc0\rc$, $j_0,j\in\N$, $\textbf{e}\in\C^d$ such that $-\frac{j_0}{\alpha_{^\prime}^+}\in\left[0,\frac{n}{2}\right]$:
			$$\left|\frac{\mathrm{Res}(\widetilde{g}_{m^\prime,1}^+,1)}{\alpha_{l^\prime}^+}\left(1-E_{2\mu}\left(\beta^+_{l^\prime};\frac{n\alpha^+_{l^\prime}+j_0}{n^\frac{1}{2\mu}}\right)\right){\lg^+_{l^\prime}}^T\textbf{e}V(j)\right|\leq C|\textbf{e}| e^{-c|j|}e^{-cn}.$$
			
			We now observe that using \eqref{in:Delta+}, \eqref{in:decroissance_expo_Phi_1_et_d(p+q)} and \eqref{in:varpi}, we can prove that there exist two positive constants $C,c$ such that such that for all $n\in\N\backslash\lc0\rc$, $j_0,j\in\N$, $\textbf{e}\in\C^d$ 
			\begin{multline*}
				\left| -\frac{1}{2i\pi}\int_{\Gamma_{d}(\eta)}e^{n\tau}e^\tau\widetilde{g}_{m^\prime, 1}^+(e^\tau)e^{-(j_0+1)\varpi^+_{l^\prime}(\tau)}\Delta^+_{m^\prime}(e^\tau,j_0,\textbf{e})\Pi(\Phi_1(e^\tau,j))d\tau \right| \\
				\leq C|\textbf{e}| e^{-c|j|} \int_{\Gamma_{d}(\eta)}\exp\left(n\Re(\tau)-\frac{j_0}{\alpha^+_{l^\prime}}\left(-\Re(\tau)+A_R\Re(\tau)^{2\mu}-A_I\Im(\tau)^{2\mu}\right)\right)|d\tau|.
			\end{multline*}
			Using \eqref{lem:BornesGaussiennesRes1}, we can find exponential bounds for the integral in the right hand term above. This allows us to conclude the proof of \eqref{lem:OndesExcitedCentral:Res+2} when $-\frac{j_0}{\alpha^+_{l^\prime}}$ belongs to $\left[0,\frac{n}{2}\right]$.
			
			$\blacktriangleright$ Let us assume that $-\frac{j_0}{\alpha_{^\prime}^+}\in\left[2n,+\infty\right[$. Since $-\frac{j_0}{\alpha^+_{l^\prime}}$ belongs to $\left[2n,+\infty\right[$, we have 
			$$\frac{n\alpha^+_{l^\prime}+j_0}{n^\frac{1}{2\mu}}\geq -\alpha^+_{l^\prime}n^\frac{2\mu-1}{2\mu}. $$
			We have that $\alpha^+_{l^\prime}<0$ since $m^\prime$ belongs to $I_{cu}^+$ and thus, using \eqref{inE+} and \eqref{in:decroissance_expo_Phi_1_et_d(p+q)}, we have that there exist two positive constants $C,c$ such that for all $n\in\N\backslash\lc0\rc$, $j_0,j\in\N$, $\textbf{e}\in\C^d$ with $-\frac{j_0}{\alpha_{^\prime}^+}\in\left[2n,+\infty\right[$:
			$$\left|\frac{\mathrm{Res}(\widetilde{g}_{m^\prime,1}^+,1)}{\alpha_{l^\prime}^+}E_{l^\prime}^+(n,j_0)\textbf{e} V(j)\right|=\left|\frac{\mathrm{Res}(\widetilde{g}_{m^\prime,1}^+,1)}{\alpha_{l^\prime}^+}E_{2\mu}\left(\beta^+_{l^\prime};\frac{n\alpha^+_{l^\prime}+j_0}{n^\frac{1}{2\mu}}\right){\lg^+_{l^\prime}}^T\textbf{e}V(j)\right|\leq C|\textbf{e}| e^{-c|j|}e^{-cn}.$$
			
			Furthermore, using \eqref{in:decroissance_expo_Phi_1_et_d(p+q)} to bound $V(j)$ as well as \eqref{in:DeltaInf+}, there exist two positive constants $C,c$ such that such that for all $n\in\N\backslash\lc0\rc$, $j_0,j\in\N$, $\textbf{e}\in\C^d$ with $-\frac{j_0}{\alpha_{l^\prime}^+}\in\left[2n,+\infty\right[$:
			$$\left|\left(\Delta_{m^\prime}^+(1,j_0,\textbf{e})+\frac{{\lg_{l^\prime}^+}^T\textbf{e}}{\alpha_{l^\prime}^+}\right)\mathrm{Res}(\widetilde{g}_{m^\prime,1}^+,1)V(j)\right|\leq C|\textbf{e}|e^{-c|j|}e^{-cn}. $$
			
			Finally, using \eqref{in:Delta+}, \eqref{in:decroissance_expo_Phi_1_et_d(p+q)} and \eqref{in:varpi}, we can prove that there exist two positive constants $C,c$ such that such that for all $n\in\N\backslash\lc0\rc$, $j_0,j\in\N$, $\textbf{e}\in\C^d$ 
			\begin{multline*}
				\left| -\frac{1}{2i\pi}\int_{\Gamma_{in}(\eta)}e^{n\tau}e^\tau\widetilde{g}_{m^\prime, 1}^+(e^\tau)\widetilde{\Ccc}^+_{m^\prime}(e^\tau,j_0,\textbf{e})\Pi(\Phi_1(e^\tau,j))d\tau \right| \\
				\leq C|\textbf{e}| e^{-c|j|} \int_{\Gamma_{in}(\eta)}\exp\left(n\Re(\tau)-\frac{j_0}{\alpha^+_{l^\prime}}\left(-\Re(\tau)+A_R\Re(\tau)^{2\mu}-A_I\Im(\tau)^{2\mu}\right)\right)|d\tau|.
			\end{multline*}
			Using \eqref{lem:BornesGaussiennesRes2}, we can find exponential bounds for the integral in the right hand term above. This allows us to conclude the proof of \eqref{lem:OndesExcitedCentral:Res+2} when $-\frac{j_0}{\alpha^+_{l^\prime}}$ belongs to $\left[2n,+\infty\right[$.
		\end{proof}

	\subsection{Proof of the decomposition \eqref{decompoGreen} of the temporal Green's function}\label{subsec:ConcluTh1}

		To conclude the proof of the decomposition \eqref{decompoGreen} of the temporal Green's function, we will start by using the expression \eqref{ExpGs} of the spatial Green's function and the calculations of Section \ref{subsec:Decompo} to obtain a first "preliminary" decomposition \eqref{decompoGreenTempo} of the Green's function. We will then have to perform what could be described as a \textit{mass analysis} of the temporal Green's function to finally obtain \eqref{decompoGreen}.
		
		\subsubsection{A preliminary decomposition}
		
		First, let us define the coefficients $C_{l^\prime,l}^{R,+}$, $C_{l^\prime,l}^{T,+}$ and $C_{l^\prime}^{E,+}$ and a sequence of line vectors $P_U$ that will appear in the calculations below. The constants defined here will also correspond to those in Theorem \ref{th:Green}.
		\begin{subequations}\label{coeff}
			\begin{itemize}
				\item The coefficients $C_{l^\prime,l}^{R,+}$ are given by:
				\begin{equation}
					\forall l^\prime\in\lc1, \ppp,I\rc, \forall l\in\lc I+1,\ppp,d\rc,\quad C_{l^\prime,l}^{R,+}:= \frac{\alpha_l^+}{\alpha_{l^\prime}^+}\widetilde{g}_{m^\prime, m}^+(1),
				\end{equation}
				where $m^\prime=l^\prime+dp\in I_{cu}^+$ and $m=l+d(p-1)\in I_{cs}^+$.
				\item The coefficients $C_{l^\prime,l}^{T,+}$ are given by:
				\begin{equation}
					\forall l^\prime\in\lc1, \ppp,I\rc, \forall l\in\lc 1,\ppp,I-1\rc,\quad C_{l^\prime,l}^{T,+}:= \frac{\alpha_l^-}{\alpha_{l^\prime}^+}\widetilde{g}_{m^\prime, m}^+(1),
				\end{equation}
				where $m^\prime=l^\prime+dp\in I_{cu}^+$ and $m=l+dp\in I_{cu}^-$.
				\item The coefficients $C_{l^\prime}^{E,+}$ are given by:
				\begin{equation}
					\forall l^\prime\in \lc 1,\ppp,I\rc,\quad  C_{l^\prime}^{E,+} := \frac{\mathrm{Res}(\widetilde{g}^+_{m^\prime,1},1)}{\alpha_{l^\prime}^+},
				\end{equation}
				where $m^\prime=l^\prime+dp\in I_{cu}^+$,
				\item Because of the linearity of $\Delta_m^\pm(z,j_0,\cdot)$ (see \eqref{def:NDelta}), we define the line vectors $P_U(j_0)\in\Mc_{1,d}(\C)$ by: 
				\begin{multline}\label{def:PU}
					\forall j_0\in \Z, \forall \textbf{e}\in\C^d,\quad  P_U(j_0)\textbf{e} := -\sum_{m^\prime\in I_{cu}^+} \left(\Delta_{m^\prime}^+(1,j_0,\textbf{e})+\frac{{\lg_{l^\prime}^+}^T\textbf{e}}{\alpha^+_{l^\prime}}\right)\mathrm{Res}(\widetilde{g}^+_{m^\prime,1},1)\\-\sum_{m^\prime\in I_{su}^+} \zeta_{m^\prime}^+(1)^{-j_0-1}\Delta_{m^\prime}^+(1,j_0,\textbf{e}) \mathrm{Res}(\widetilde{g}^+_{m^\prime,1},1).
				\end{multline}
			\end{itemize}
		\end{subequations}
		
		Using \eqref{in:DeltaInf+} and noticing that for $m^\prime\in I^+_{su}$ the eigenvalues $\zeta^+_{m^\prime}(1)$ belongs to the set $\D$, we observe that there exist two constants $C,c>0$ such that:
		\begin{equation}\label{in:PU}
			\forall j_0\in\N,\quad \left|P_U(j_0)\right|\leq Ce^{-c|j_0|}.
		\end{equation}
		
		Let us start by proving a decomposition of the temporal Green's function when $j\geq j_0+1$. The cases when $j\in\lc0,\ppp,j_0\rc$ and $j\leq -1$ would be handled similarly using \eqref{ExpGs: l geq 0 : 0 leq j leq l} and \eqref{ExpGs: l geq 0 : j < 0}. Using \eqref{egGSpaTemp2}, Lemma \ref{estGreenExt} and \eqref{ExpGs: l geq 0 : j geq l+1}, there exists a constant $c>0$ such that for $n\in\N\backslash\lc0\rc$, $j_0\in \N$, $j\geq j_0+1$ and $\textbf{e}\in\C^d$ which verify $j-j_0\in\lc-nq,\ppp,np\rc$:
		\begin{align*}
			\Gcc(n,j_0,j)\textbf{e}	=& \frac{1}{2i\pi}\int_{\Gamma_{in}(\eta)} e^{n\tau}e^\tau \Pi(W(e^\tau,j_0,j,\textbf{e}))d\tau+ \frac{1}{2i\pi}\int_{\Gamma_{out}(\eta)} e^{n\tau}e^\tau \Pi(W(e^\tau,j_0,j,\textbf{e}))d\tau \\
			=&  \sum_{m\in I_{ss}^+\cup I_{cs}^+} \left(-\frac{1}{2i\pi} \int_{\Gamma_{in}(\eta)} e^{n\tau}e^\tau \widetilde{\Ccc}^+_{m}(e^\tau,j_0,\textbf{e})\Pi\left(W^+_m(e^\tau,j)\right)d\tau\right) \\ &+\sum_{m\in I_{ss}^+\cup I_{cs}^+\backslash\lc1\rc} \sum_{m^\prime\in I_{cu}^+\cup I_{su}^+} \left(-\frac{1}{2i\pi} \int_{\Gamma_{in}(\eta)} e^{n\tau}e^\tau\widetilde{g}_{m^\prime,m}^+(e^\tau) \widetilde{\Ccc}^+_{m^\prime}(e^\tau,j_0,\textbf{e})\Pi\left(W^+_m(e^\tau,j)\right)d\tau\right)\\
			& + \sum_{m^\prime\in I_{cu}^+\cup I_{su}^+}  \left(-\frac{1}{2i\pi} \int_{\Gamma_{in}(\eta)} e^{n\tau}e^\tau\widetilde{g}_{m^\prime,1}^+(e^\tau) \widetilde{\Ccc}^+_{m^\prime}(e^\tau,j_0,\textbf{e})\Pi\left(\Phi_1(e^\tau,j)\right)d\tau\right) +O(|\textbf{e}|e^{-cn}).
		\end{align*}
		
		Thus, using Lemmas \ref{lem:OndesPropag}-\ref{lem:OndesExcitedCentral}, we obtain that for all $n\in\N\backslash\lc0\rc$, $j_0\in \N$, $j\geq j_0+1$ and $\textbf{e}\in\C^d$ which verify $j-j_0\in\lc-nq,\ppp,np\rc$:	
		\begin{align*}
			&\Gcc(n,j_0,j)\textbf{e}	\\
			=& \sum_{l=I+1}^d S_l^+(n,j_0,j)\textbf{e} +\sum_{l^\prime=1}^I\sum_{l=I+1}^d C_{l^\prime,l}^{R,+}R_{l^\prime,l}^+(n,j_0,j)\textbf{e} + \left(\sum_{l^\prime=1}^IC_{l^\prime}^{E,+}E_{l^\prime}^+(n,j_0)\textbf{e}+P_U^+(j_0)\textbf{e}\right)V(j)\\ &+ \sum_{l=I+1}^d\exp\left(-c\left(\frac{\left|n-\frac{j-j_0}{\alpha_l^+}\right|}{n^\frac{1}{2\mu}}\right)^\frac{2\mu}{2\mu-1}\right)\left(O\left(\frac{|\textbf{e}|e^{-c|j|}}{n^\frac{1}{2\mu}}\right)+O_\C\left(\frac{|\textbf{e}|e^{-c|j_0|}}{n^\frac{1}{2\mu}}\right)\rg_l^+ + O_\C\left(\frac{1}{n^\frac{1}{\mu}}\right){\lg_l^+}^T\textbf{e}\rg_l^+\right)\\
			&+\sum_{l^\prime=1}^I\sum_{l=I+1}^d \exp\left(-c\left(\frac{\left|n-\left(\frac{j}{\alpha_l^+}-\frac{j_0}{\alpha_{l^\prime}^+}\right)\right|}{n^\frac{1}{2\mu}}\right)^\frac{2\mu}{2\mu-1}\right)\left(O\left(\frac{|\textbf{e}|e^{-c|j|}}{n^\frac{1}{2\mu}}\right)+O_\C\left(\frac{|\textbf{e}|e^{-c|j_0|}}{n^\frac{1}{2\mu}}\right)\rg_l^+ + O_\C\left(\frac{1}{n^\frac{1}{\mu}}\right){\lg_{l^\prime}^+}^T\textbf{e}\rg_l^+\right)\\
			&+\sum_{l^\prime=1}^I O\left(\frac{|\textbf{e}|e^{-c|j|}}{n^\frac{1}{2\mu}}\exp\left(-c\left(\frac{\left|n+\frac{j_0}{\alpha_{l^\prime}^+}\right|}{n^\frac{1}{2\mu}}\right)^\frac{2\mu}{2\mu-1}\right)\right) +\sum_{l=I+1}^d  O_\C\left(\frac{|\textbf{e}|e^{-c|j_0|}}{n^\frac{1}{2\mu}}\exp\left(-c\left(\frac{\left|n-\frac{j}{\alpha_l^+}\right|}{n^\frac{1}{2\mu}}\right)^\frac{2\mu}{2\mu-1}\right)\right)\rg_l^+\\
			&+O(|\textbf{e}|e^{-cn}).
		\end{align*}
		
		Thus, observing that $S_l^+(n,j_0,j)$ for $l\in \lc1,\ppp,I\rc$ and $T_{l^\prime,l}^+(n,j_0,j)$ for $l^\prime\in\lc1,\ppp,I\rc$ and $l\in \lc1,\ppp,I-1\rc$ are equal to $0$ for $j$ larger than $j_0+1$, we thus obtain for $n\in\N\backslash\lc0\rc$, $j_0\in \N$ and $j\geq j_0+1$ which verify $j-j_0\in\lc-nq,\ppp,np\rc$:
		\begin{multline}\label{decompoGreenTempo}
			\Gcc(n,j_0,j) = \sum_{l=1}^d S_l^+(n,j_0,j) +\sum_{l^\prime=1}^I\left(\sum_{l=I+1}^d C^{R,+}_{l^\prime,l}R_{l^\prime,l}^+(n,j_0,j) + \sum_{l=1}^{I-1} C^{T,+}_{l^\prime,l}T_{l^\prime,l}^+(n,j_0,j)\right)\\
			+V(j)\left(\sum_{l^\prime=1}^IC^{E,+}_{l^\prime}E_{l^\prime}^+(n,j_0) +P_U(j_0)\right)+ \Rc(n,j_0,j)
		\end{multline}
		where the remainder $\Rc(n,j_0,j)$ has the form \eqref{def:ResTh1} described in Theorem \ref{th:Green}. Similar proofs using \eqref{ExpGs: l geq 0 : 0 leq j leq l} and \eqref{ExpGs: l geq 0 : j < 0} allow us to extend the description of the Green's function \eqref{decompoGreenTempo} respectively for $j\in\lbrace0,\hdots,j_0\rbrace$ and for $j\leq -1$.
		
		\subsubsection{Analysis of the mass of the temporal Green's function}
		
		To pass from the decomposition \eqref{decompoGreenTempo} above of the temporal Green's function to the decomposition \eqref{decompoGreen} announced in Theorem \ref{th:Green}, we just have to prove that the line vectors $P_U(j_0)$ are equal to $0$ for all $j_0\in\N$. We will prove this fact using a \textit{mass analysis} of the temporal Green's function $\Gcc(n,j_0,\cdot)$. 
		
		Firstly, the definition \eqref{def:linearizedScheme} of the operator $\Lcc$ and the definition \eqref{def:Ajk} of the matrices $A_{j,k}$ imply the following conservation of mass property:
		\begin{multline}\label{cons_mass}
			\forall h\in\ell^1(\Z,\C^d),\\ \sum_{j\in\Z}\left(\Lcc h\right)_j =\sum_{j\in\Z}\left(\sum_{k=-p}^qA_{j,k}h_{j+k}\right)=\sum_{j\in\Z}\left( h_j +\sum_{k=-p}^{q-1}B_{j,k}h_{j+k}-\sum_{k=-p}^{q-1}B_{j+1,k}h_{j+1+k}\right)  =\sum_{j\in\Z} h_j.
		\end{multline}
		As a consequence, one can use the conservation of mass \eqref{cons_mass} with the definition \eqref{defGreenTempo} of the temporal Green's function and Lemma \ref{lem:GreenTempoOutside} to obtain that the mass of the temporal Green's function is conserved:
		\begin{equation}\label{mass_equality_Green}
			\forall j_0\in\Z,\forall n\in\N,\quad Id =  \sum_{j=j_0-nq}^{j_0+np} \Gcc(n,j_0,j).
		\end{equation}
		We observe that \eqref{mass_equality_Green} and the decomposition \eqref{decompoGreenTempo} imply that for all $\textbf{e}\in\C^d$, $j_0\in\N$ and $n\in\N\backslash\lbrace0\rbrace$:
		\begin{align}
			 \begin{split}
			 	\textbf{e}& = \sum_{l=1}^d\left(\sum_{j=j_0-nq}^{j_0+np}S_l^+(n,j_0,j)\textbf{e}\right)
				+\sum_{l^\prime=1}^I\sum_{l=I+1}^d C^{R,+}_{l^\prime,l}\left(\sum_{j=j_0-nq}^{j_0+np}R_{l^\prime,l}^+(n,j_0,j)\textbf{e}\right) \\
				&+ \sum_{l^\prime=1}^I\sum_{l=1}^{I-1} C^{T,+}_{l^\prime,l}\left(\sum_{j=j_0-nq}^{j_0+np}T_{l^\prime,l}^+(n,j_0,j)\textbf{e}\right) +\left(\sum_{l^\prime=1}^IC^{E,+}_{l^\prime}E_{l^\prime}^+(n,j_0)\textbf{e} +P_U(j_0)\textbf{e}\right)\left(\sum_{j=j_0-nq}^{j_0+np}V(j)\right)\\
				&+ \sum_{j=j_0-nq}^{j_0+np}\Rc(n,j_0,j)\textbf{e}.
			 \end{split}\label{mass_equality_2}
		\end{align}
		The equality \eqref{mass_equality_2} essentially expresses the mass of the temporal Green's function as the sum of the masses of the waves in the decomposition \eqref{decompoGreenTempo}. We now recall the so-called Liu-Majda condition proved in Lemma \ref{lem:LiuMajda}:
		\begin{equation}\label{LiuMajda}
			\det\left(\rg_1^-,\hdots,\rg_{I-1}^-,\sum_{j\in\Z}V(j),\rg_{I+1}^+,\hdots,\rg_d^+\right)\neq0.
		\end{equation}
		We can thus define the projector $\widehat{\Pi}$ on the vector space $\mathrm{Span}\left(\sum_{j\in\Z}V(j)\right)$ along the vector space spanned by the vectors $\rg_1^-,\hdots,\rg_{I-1}^-,\rg_{I+1}^+,\hdots,\rg_{d}^+$. 
		
		We observe that the terms $R^+_{l^\prime,l}(n,j_0,j)\textbf{e}$ and $T^+_{l^\prime,l}(n,j_0,j)\textbf{e}$ in \eqref{mass_equality_2} are vectors in the vector space spanned by the vectors $\rg_1^-,\hdots,\rg_{I-1}^-,\rg_{I+1}^+,\hdots,\rg_d^+$ by the definitions \eqref{def:R+}-\eqref{def:T+} of $R_{l^\prime,l}^+$ and $T_{l^\prime,l}^+$. This is also the case for the terms $S^+_{l}(n,j_0,j)\textbf{e}$ when $l$ belongs to $\lbrace I+1,\hdots,d\rbrace$ by the definition \eqref{def:S+} of $S_{l}^+$. Thus, by applying the projector $\widehat{\Pi}$ to the equality \eqref{mass_equality_2}, we have that for all $\textbf{e}\in\C^d$, $j_0\in\N$ and $n\in\N\backslash\lbrace0\rbrace$:
		\begin{multline}\label{mass_equality_3}
			\widehat{\Pi}\textbf{e} =  \sum_{l=1}^I\sum_{j=j_0-nq}^{j_0+np}\widehat{\Pi}\left(S_l^+(n,j_0,j)\textbf{e}\right) +\left(\sum_{l^\prime=1}^IC^{E,+}_{l^\prime}E_{l^\prime}^+(n,j_0)\textbf{e} +P_U(j_0)\textbf{e}\right)\widehat{\Pi}\left(\sum_{j=j_0-nq}^{j_0+np}V(j)\right)\\ + \sum_{j=j_0-nq}^{j_0+np}\widehat{\Pi}\left(\Rc(n,j_0,j)\textbf{e}\right).
		\end{multline}
		We claim that:
		\begin{subequations}\label{limites_termes_mass}
			\begin{align}
				\forall l^\prime\in\lbrace1,\hdots,I\rbrace,\forall j_0\in\N,\forall \textbf{e} \in\C^d,\quad&& C^{E,+}_{l^\prime} E^+_{l^\prime}(n,j_0)\textbf{e}\widehat{\Pi}\left(\sum^{j_0+np}_{j=j_0-nq}V(j)\right) &\underset{n\rightarrow+\infty}\rightarrow C^{E,+}_{l^\prime}{\lg_{l^\prime}^+}^T\textbf{e}\sum_{j\in\Z}V(j),\label{limites_termes_mass1}\\
				\forall j_0\in\N,\forall \textbf{e} \in\C^d,\quad&& P_U(j_0)\textbf{e} \widehat{\Pi}\left(\sum^{j_0+np}_{j=j_0-nq}V(j)\right)&\underset{n\rightarrow+\infty}\rightarrow P_U(j_0)\textbf{e}\sum_{j\in\Z}V(j),\label{limites_termes_mass2}\\
				\forall l\in\lbrace1,\hdots,I\rbrace,\forall j_0\in\N,\forall \textbf{e} \in\C^d,\quad&& \sum_{j=j_0-nq}^{j_0+np} \widehat{\Pi}\left(S^+_{l}(n,j_0,j)\textbf{e}\right) &\underset{n\rightarrow+\infty}\rightarrow 0, \label{limites_termes_mass3}\\
				\forall j_0\in\N,\forall \textbf{e} \in\C^d,\quad&&\sum_{j=j_0-nq}^{j_0+np} \widehat{\Pi}\left(\Rc(n,j_0,j)\textbf{e}\right)&\underset{n\rightarrow+\infty}\rightarrow0.\label{limites_termes_mass4}
			\end{align}
		\end{subequations}
		The proof of those limits will be detailed at the end of the Appendix (Section \ref{sec:Appendix}). By letting $n$ grow towards $+\infty$ in \eqref{mass_equality_3}, the limits \eqref{limites_termes_mass} imply that for all $\textbf{e}\in\C^d$ and $j_0\in\N$:
		\begin{equation}\label{mass_equality_4}
			\widehat{\Pi}\textbf{e} = \left(\sum_{l^\prime=1}^IC^{E,+}_{l^\prime}{\lg_{l^\prime}^+}^T +P_U(j_0)\right)\textbf{e}\sum_{j\in\Z}V(j).
		\end{equation}
		We notice that \eqref{mass_equality_4} implies that, for all $\textbf{e}\in\C^d$, the sequence: 
		$$\left(P_U(j_0)\textbf{e}\left(\sum_{j\in\Z}V(j)\right)\right)_{j_0\in\N}$$
		is constant. Because of the exponential decay \eqref{in:PU} of the sequence of line vectors $\left(P_U(j_0)\right)_{j_0\in\N}$, we thus obtain that:
		\begin{equation}\label{Mass_eq3}
			\forall \textbf{e}\in\C^d,\forall j_0\in\N, \quad  P_U(j_0)\textbf{e}\left(\sum_{j\in\Z}V(j)\right)=0.
		\end{equation}
		Noticing that the Liu-Majda condition \eqref{LiuMajda} gives us $\sum_{j\in\Z}V(j)\neq0$, the equality \eqref{Mass_eq3} implies that:
		\begin{equation}\label{eg:PU}
			\forall j_0\in\N,\quad P_U(j_0)=0.
		\end{equation}
		This allows us to conclude the proof of the decomposition \eqref{decompoGreen} of the temporal Green's function using \eqref{decompoGreenTempo}, at least when $j_0\geq0$. A similar proof for $j_0<0$ can be performed once a temporary decomposition of the Green's function \eqref{decompoGreenTempo} in this case has been deduced.

		\subsection{Proof of the decomposition \texorpdfstring{\eqref{decompoDerGreen}}{} of discrete derivative of the temporal Green's function}\label{subsec:decompoDer}
		
		In order to obtain the decomposition \eqref{decompoDerGreen} of discrete derivative of the temporal Green's function, we must follow a similar path as for the proof of the decomposition \eqref{decompoGreen} of the temporal Green's function. This requires the introduction of some additional lemmas (Lemmas \ref{lem:OndesPropagDer}-\ref{lem:OndesExcitedCentralDer}) which are "discrete derivative versions" of Lemmas \ref{lem:OndesPropag}-\ref{lem:OndesExcitedCentral}. We will start by stating those different lemmas and detail their proofs. We will finally clarify how to conclude the proof of \eqref{decompoDerGreen} of the discrete derivative of the temporal Green's function.
		
		\subsubsection{Discrete derivative version of lemmas on outgoing, incoming, reflected and transmitted waves}
		
		In this subsection, we start by stating the three Lemmas \ref{lem:OndesPropagDer}, \ref{lem:OndesReflechiesDer} and \ref{lem:OndesTransmisesDer} which are similar to the Lemmas \ref{lem:OndesPropag}, \ref{lem:OndesReflechies} and \ref{lem:OndesTransmises}.
		
		\begin{lemma}[Discrete derivative version of Lemma \ref{lem:OndesPropag} on outgoing and incoming waves]\label{lem:OndesPropagDer}
			\hfill
			
			We consider $m\in\lc1,\ppp,d(p+q)\rc$ and write it as $m=l+(k-1)d$ with $k\in\lc1,\ppp,p+q\rc$ and $l\in\lc1,\ppp,d\rc$. There exists a constant $c>0$ such that for all $n\in \N\backslash\lc0\rc$, $j_0\in\N\backslash\lc0\rc$, $j\in\N$ such that $j-j_0\in\lc-nq-1,\ppp,np\rc$ and $\textbf{e}\in\C^d$ we have:
			
			\begin{subequations}\label{lem:OndesPropagResDer}
				$\bullet$ If $m\in I_{cs}^+\cup I_{cu}^+$ and $\frac{j-j_0}{\alpha_l^+}\in \left[\frac{n}{2},2n\right]$, we have that:
				\begin{multline}\label{lem:OndesPropagResDer1}
					\int_{\Gamma_{in}(\eta)} e^{n\tau} e^\tau\left(\widetilde{\Ccc}_m^+(e^\tau,j_0,\textbf{e})-\widetilde{\Ccc}_m^+(e^\tau,j_0-1,\textbf{e})\right)\Pi(W_m^+(e^\tau,j))d\tau  \\=  \exp\left(-c\left(\frac{\left|n-\left(\frac{j-j_0}{\alpha_l^+}\right)\right|}{n^\frac{1}{2\mu}}\right)^\frac{2\mu}{2\mu-1}\right)\left(O\left(\frac{|\textbf{e}|e^{-c|j|}}{n^\frac{1}{2\mu}}\right) +O_{\C}\left(\frac{|\textbf{e}|e^{-c|j_0|}}{n^\frac{1}{2\mu}}\right)\rg_l^+ + O_\C\left(\frac{1}{n^\frac{1}{\mu}}\right){\lg_l^+}^T\textbf{e}\rg_l^+\right).
				\end{multline}
				
				$\bullet$ If $m\in I_{cs}^+\cup  I_{cu}^+$, $\frac{j-j_0}{\alpha_l^+}\notin \left[\frac{n}{2},2n\right]$ and $\frac{j-j_0}{\alpha_l^+}\geq 0$, we have that:
				\begin{equation}\label{lem:OndesPropagResDer2}
					\int_{\Gamma_{in}(\eta)} e^{n\tau} e^\tau\left(\widetilde{\Ccc}_m^+(e^\tau,j_0,\textbf{e})-\widetilde{\Ccc}_m^+(e^\tau,j_0-1,\textbf{e})\right)\Pi(W_m^+(e^\tau,j))d\tau = O\left(|\textbf{e}|e^{-cn}\right).
				\end{equation}
			\end{subequations}
		\end{lemma}
		
		\begin{lemma}[Discrete derivative version of Lemma \ref{lem:OndesReflechies} on reflected waves]\label{lem:OndesReflechiesDer}
			\hfill 
			
			We consider $m\in \lc2, \ppp,dp\rc$ and $m^\prime\in \lc dp+1,\ppp,d(p+q)\rc$ and write them as $m=l+(k-1)d$ and $m^\prime=l^\prime+(k^\prime-1)d$ with $k,k^\prime\in\lc1,\ppp,p+q\rc$ and $l,l^\prime\in\lc1,\ppp,d\rc$. There exists a positive constant $c$ such that for all $n\in \N\backslash\lc0\rc$, $j_0\in\N\backslash\lc0\rc$, $j\in \N$ such that $j-j_0\in\lc-nq-1,\ppp,np\rc$ and $\textbf{e}\in\C^d$, we have:
			
			\begin{subequations}
				
				$\bullet$ If $m\in I_{cs}^+$, $m^\prime\in I_{cu}^+$ and $\frac{j}{\alpha_l^+}-\frac{j_0}{\alpha_{l^\prime}^+}\in\left[\frac{n}{2},2n\right]$, we have that:
				\begin{multline}\label{lem:OndesReflechiesDer:CsCu1}
					\int_{\Gamma_{in}(\eta)} e^{n\tau} e^\tau\widetilde{g}_{m^\prime,m}^+(e^\tau)\left(\widetilde{\Ccc}_m^+(e^\tau,j_0,\textbf{e})-\widetilde{\Ccc}_m^+(e^\tau,j_0-1,\textbf{e})\right)\Pi(W_m^+(e^\tau,j))d\tau \\ =  \exp\left(-c\left(\frac{\left|n-\left(\frac{j}{\alpha_l^+}-\frac{j_0}{\alpha_{l^\prime}^+}\right)\right|}{n^\frac{1}{2\mu}}\right)^\frac{2\mu}{2\mu-1}\right)\left(O\left(\frac{|\textbf{e}|e^{-c|j|}}{n^\frac{1}{2\mu}}\right)+O_\C\left(\frac{|\textbf{e}|e^{-c|j_0|}}{n^\frac{1}{2\mu}}\right)\rg_l^+ + O_\C\left(\frac{1}{n^\frac{1}{\mu}}\right){\lg_{l^\prime}^+}^T\textbf{e}\rg_l^+\right).
				\end{multline}
				
				$\bullet$ If $m\in I_{cs}^+$, $m^\prime\in I_{cu}^+$ and $\frac{j}{\alpha_l^+}-\frac{j_0}{\alpha_{l^\prime}^+}\notin\left[\frac{n}{2},2n\right]$, we have that:
				\begin{equation}\label{lem:OndesReflechiesDer:CsCu2}
					\int_{\Gamma_{in}(\eta)} e^{n\tau} e^\tau\widetilde{g}_{m^\prime,m}^+(e^\tau)\left(\widetilde{\Ccc}_m^+(e^\tau,j_0,\textbf{e})-\widetilde{\Ccc}_m^+(e^\tau,j_0-1,\textbf{e})\right)\Pi(W_m^+(e^\tau,j))d\tau=  O(|\textbf{e}|e^{-cn}).
				\end{equation}
			\end{subequations}
		\end{lemma}
		
		\begin{lemma}[Discrete derivative version of Lemma \ref{lem:OndesTransmises} on transmitted waves]\label{lem:OndesTransmisesDer}
			\hfill 
			
			We consider $m\in \lc dp+1, \ppp,d(p+q)-1\rc$ and $m^\prime\in \lc dp+1,\ppp,d(p+q)\rc$ and write them as $m=l+(k-1)d$ and $m^\prime=l^\prime+(k^\prime-1)d$ with $k,k^\prime\in\lc1,\ppp,p+q\rc$ and $l,l^\prime\in\lc1,\ppp,d\rc$. There exists a positive constant $c$ such that for all $n\in \N\backslash\lc0\rc$, $j_0\in \N\backslash\lc0\rc$, $j\in-\N$ such that $j-j_0\in\lc-nq-1,\ppp,np\rc$ and $\textbf{e}\in\C^d$, we have:
			
			\begin{subequations}
				$\bullet$ If $m\in I_{cu}^-$, $m^\prime\in I_{cu}^+$ and $\frac{j}{\alpha_l^-}-\frac{j_0}{\alpha_{l^\prime}^+}\in\left[\frac{n}{2},2n\right]$, we have that:
				\begin{multline}\label{lem:OndesTransmisesDer:CuCu1}
					\int_{\Gamma_{in}(\eta)} e^{n\tau} e^\tau\widetilde{g}_{m^\prime,m}^+(e^\tau)\left(\widetilde{\Ccc}_m^+(e^\tau,j_0,\textbf{e})-\widetilde{\Ccc}_m^+(e^\tau,j_0-1,\textbf{e})\right)\Pi(W_m^-(e^\tau,j))d\tau\\ =  \exp\left(-c\left(\frac{\left|n-\left(\frac{j}{\alpha_l^-}-\frac{j_0}{\alpha_{l^\prime}^+}\right)\right|}{n^\frac{1}{2\mu}}\right)^\frac{2\mu}{2\mu-1}\right)\left(O\left(\frac{|\textbf{e}|e^{-c|j|}}{n^\frac{1}{2\mu}}\right)+O_\C\left(\frac{|\textbf{e}|e^{-c|j_0|}}{n^\frac{1}{2\mu}}\right)\rg_l^- + O_\C\left(\frac{1}{n^\frac{1}{\mu}}\right){\lg_{l^\prime}^+}^T\textbf{e}\rg_l^-\right).
				\end{multline}

				$\bullet$ If $m\in I_{cu}^-$, $m^\prime\in I_{cu}^+$ and $\frac{j}{\alpha_l^-}-\frac{j_0}{\alpha_{l^\prime}^+}\notin\left[\frac{n}{2},2n\right]$, we have that:
				\begin{equation}\label{lem:OndesTransmisesDer:CuCu2}
					\int_{\Gamma_{in}(\eta)} e^{n\tau} e^\tau\widetilde{g}_{m^\prime,m}^+(e^\tau)\left(\widetilde{\Ccc}_m^+(e^\tau,j_0,\textbf{e})-\widetilde{\Ccc}_m^+(e^\tau,j_0-1,\textbf{e})\right)\Pi(W_m^-(e^\tau,j))d\tau =  O\left(|\textbf{e}|e^{-cn}\right).
				\end{equation}
			\end{subequations}
		\end{lemma}
		
		\begin{proof}\textbf{of Lemmas \ref{lem:OndesPropagDer}, \ref{lem:OndesReflechiesDer} and \ref{lem:OndesTransmisesDer}}
			
			The proofs of Lemmas \ref{lem:OndesPropagDer}, \ref{lem:OndesReflechiesDer} and \ref{lem:OndesTransmisesDer} are fairly similar to those of Lemmas \ref{lem:OndesPropag}-\ref{lem:OndesTransmises} with some slight adjustments. We will prove in details the result of Lemma \ref{lem:OndesPropagDer} and the other Lemmas \ref{lem:OndesReflechiesDer} and \ref{lem:OndesTransmisesDer} will be left to the reader.
			
			We consider $m\in I_{cs}^+\cup I_{cu}^+$. We recall that $\varpi_l^+$ is defined using \eqref{def:varpi}. Using Cauchy's formula and the expressions of $W_m^+$ and $\widetilde{\Ccc}_m^+$ given respectively by Lemma \ref{lem_choice_base} and \eqref{eg:CccTilde}, we have that:
			\begin{equation*}
				\int_{\Gamma_{in}(\eta)} e^{n\tau} e^\tau\left(\widetilde{\Ccc}_m^+(e^\tau,j_0,\textbf{e})-\widetilde{\Ccc}_m^+(e^\tau,j_0-1,\textbf{e})\right)\Pi(W_m^+(e^\tau,j))d\tau=E_1+E_2\rg_l^++E_3{\lg_l^+}^T\textbf{e}\rg_l^+
			\end{equation*}
			where the vector $E_1$ and the scalars $E_2$ and $E_3$ are defined as:
			$$E_1 := \int_{\Gamma_1} e^{n\tau} e^{(j-j_0)\varpi_l^+(\tau)} e^\tau\left(e^{-\varpi_l^+(\tau)}\Delta_m^+(e^\tau,j_0,\textbf{e})- \Delta_m^+(e^\tau,j_0-1,\textbf{e})\right)\Pi\left(V_m^+(e^\tau,j)-R_m^+(e^\tau)\right)d\tau,$$
			\begin{multline*}
				E_2 := \int_{\Gamma_2} e^{n\tau} e^{(j-j_0)\varpi_l^+(\tau)} e^\tau\left(e^{-\varpi_l^+(\tau)}\left(\Delta_m^+(e^\tau,j_0,\textbf{e})-\frac{d\zeta_m^+}{dz}(e^\tau){\lg_l^+}^T\textbf{e}\right)\right.\\
				\left.-\left(\Delta_m^+(e^\tau,j_0-1,\textbf{e})-\frac{d\zeta_m^+}{dz}(e^\tau){\lg_l^+}^T\textbf{e}\right)\right)d\tau,
			\end{multline*}
			$$E_3:= \int_{\Gamma_3} e^{n\tau} e^{(j-j_0)\varpi_l^+(\tau)} e^\tau\left(e^{-\varpi_l^+(\tau)}-1\right)\frac{d\zeta_m^+}{dz}(e^\tau)d\tau,$$
			and $\Gamma_1$, $\Gamma_2$ and $\Gamma_3$ are paths belonging to the set $X$. 
			
			$\blacktriangleright$ Using \eqref{in:varpi} to bound the function $\varpi_l^+$ and Lemma \ref{lem_choice_base} to bound $V_m^+(e^\tau,j)-R_m^+(e^\tau)$, we prove that there exist two positive constants $C,c$ independent from $n$, $j_0$, $j$ and $\textbf{e}$ such that: 
			$$|E_1| \leq Ce^{-c|j|}|\textbf{e}| \int_{\Gamma_1}\exp\left(n\Re(\tau)+\left(\frac{j-j_0}{\alpha_l^+}\right)\left(-\Re(\tau)+A_R\Re(\tau)^{2\mu}-A_I\Im(\tau)^{2\mu}\right)\right)|d\tau|.$$
			
			$\blacktriangleright$ Using \eqref{in:varpi} to bound the function $\varpi_l^+$ and  \eqref{in:DeltaInf+}, we prove that there exist two positive constants $C,c$ independent from $n$, $j_0$, $j$ and $\textbf{e}$ such that: 
			$$|E_2| \leq Ce^{-c|j_0|}|\textbf{e}| \int_{\Gamma_2}\exp\left(n\Re(\tau)+\left(\frac{j-j_0}{\alpha_l^+}\right)\left(-\Re(\tau)+A_R\Re(\tau)^{2\mu}-A_I\Im(\tau)^{2\mu}\right)\right)|d\tau|.$$
			
			$\blacktriangleright$ Using \eqref{in:varpi} to bound the function $\varpi_l^+$ and a Taylor expansion, we have that there exist two positive constants $C,c$ independent from $n$, $j_0$, $j$ and $\textbf{e}$ such that 
			$$|E_3| \leq C \int_{\Gamma_3}|\tau|\exp\left(n\Re(\tau)+\left(\frac{j-j_0}{\alpha_l^+}\right)\left(-\Re(\tau)+A_R\Re(\tau)^{2\mu}-A_I\Im(\tau)^{2\mu}\right)\right)|d\tau|.$$
			
			Using Lemma \ref{lem:BornesGaussiennes} which gives a good choices of path $\Gamma_1,\Gamma_2,\Gamma_3\in X$ depending on $n$, $j_0$ and $j$ to handle the integrals in the terms, there exist new constants $C,c>0$ independent from $n$, $j_0$, $j$ and $\textbf{e}$ such that if $\frac{j-j_0}{\alpha_l^+}\in \left[\frac{n}{2},2n\right]$:
			\begin{align*}
				|E_1|\leq \frac{Ce^{-c|j|}|\textbf{e}|}{n^\frac{1}{2\mu}}\exp\left(-c\left(\frac{\left|n-\frac{j-j_0}{\alpha_l^+}\right|}{n^\frac{1}{2\mu}}\right)^\frac{2\mu}{2\mu-1}\right) \quad &\text{ and } \quad |E_2|\leq \frac{Ce^{-c|j_0|}|\textbf{e}|}{n^\frac{1}{2\mu}}\exp\left(-c\left(\frac{\left|n-\frac{j-j_0}{\alpha_l^+}\right|}{n^\frac{1}{2\mu}}\right)^\frac{2\mu}{2\mu-1}\right)\\
				|E_3|\leq \frac{C}{n^\frac{1}{\mu}}\exp\left(-c\left(\frac{\left|n-\frac{j-j_0}{\alpha_l^+}\right|}{n^\frac{1}{2\mu}}\right)^\frac{2\mu}{2\mu-1}\right)&
			\end{align*}
			and if $\frac{j-j_0}{\alpha_l^+}\notin \left[\frac{n}{2},2n\right]$:
			\begin{align*}
				|E_1|\leq C|\textbf{e}|e^{-cn},\quad |E_2|\leq C|\textbf{e}|e^{-cn}\quad \text{ and } \quad |E_3|\leq Ce^{-cn}.
			\end{align*}
			We have thus obtained the results \eqref{lem:OndesPropagResDer}.
		\end{proof}
		
		\subsubsection{Discrete derivative version of the lemma on the central excited mode}
		
		Just like in the previous section, we state Lemma \ref{lem:OndesExcitedCentralDer} which is a similar discrete derivative version of Lemma \ref{lem:OndesExcitedCentral}.
		
		\begin{lemma}[Discrete derivative version of Lemma \ref{lem:OndesExcitedCentral} on the central excited mode]\label{lem:OndesExcitedCentralDer}
			\hfill 
			
			We consider $m^\prime\in I_{cu}^+$ and write it as $m^\prime=l^\prime+(k^\prime-1)d$ with $k^\prime\in\lc1,\ppp,p+q\rc$ and $l^\prime\in\lc1,\ppp,d\rc$. There exists a positive constant $c$ such that for all $n\in \N\backslash\lc0\rc$, $j_0\in \N\backslash\lc0\rc$, $j\in \Z$ such that $j-j_0\in\lc-nq-1,\ppp,np\rc$ and $\textbf{e}\in\C^d$, we have:
			
			\begin{subequations}
				$\bullet$ For $-\frac{j_0}{\alpha_{l^\prime}^+}\in\left[\frac{n}{2},2n\right]$ and $j\geq0$, we have that:
				\begin{multline}\label{lem:OndesExcitedCentralDer:Res+1}
					\int_{\Gamma_{in}(\eta)}e^{n\tau}e^\tau\widetilde{g}_{m^\prime, 1}^+(e^\tau)\left(\widetilde{\Ccc}^+_{m^\prime}(e^\tau,j_0,\textbf{e})-\widetilde{\Ccc}^+_{m^\prime}(e^\tau,j_0-1,\textbf{e})\right)\Pi(\Phi_1(e^\tau,j))d\tau = \\ 
					O\left(\frac{|\textbf{e}|e^{-c|j|}}{n^\frac{1}{2\mu}}\exp\left(-c\left(\frac{\left|n+\frac{j_0}{\alpha_l^+}\right|}{n^\frac{1}{2\mu}}\right)^\frac{2\mu}{2\mu-1}\right)\right)+O\left(|\textbf{e}|e^{-cn}\right).
				\end{multline}	
				
				$\bullet$ For $-\frac{j_0}{\alpha_{l^\prime}^+}\notin\left[\frac{n}{2},2n\right]$ and $j\geq0$, we have that:
				\begin{equation}\label{lem:OndesExcitedCentralDer:Res+2}
					\int_{\Gamma_{in}(\eta)}e^{n\tau}e^\tau\widetilde{g}_{m^\prime, 1}^+(e^\tau)\left(\widetilde{\Ccc}^+_{m^\prime}(e^\tau,j_0,\textbf{e})-\widetilde{\Ccc}^+_{m^\prime}(e^\tau,j_0-1,\textbf{e})\right)\Pi(\Phi_1(e^\tau,j))d\tau = O(|\textbf{e}|e^{-cn}).
				\end{equation}	
				
				$\bullet$ For $-\frac{j_0}{\alpha_{l^\prime}^+}\in\left[\frac{n}{2},2n\right]$ and $j<0$, we have that:
				\begin{multline}\label{lem:OndesExcitedCentralDer:Res-1}
					\int_{\Gamma_{in}(\eta)}e^{n\tau}e^\tau\widetilde{g}_{m^\prime, d(p+q)}^+(e^\tau)\left(\widetilde{\Ccc}^+_{m^\prime}(e^\tau,j_0,\textbf{e})-\widetilde{\Ccc}^+_{m^\prime}(e^\tau,j_0-1,\textbf{e})\right)\Pi(\Phi_{d(p+q)}(e^\tau,j))d\tau  = \\ 
					O\left(\frac{|\textbf{e}|e^{-c|j|}}{n^\frac{1}{2\mu}}\exp\left(-c\left(\frac{\left|n+\frac{j_0}{\alpha_l^+}\right|}{n^\frac{1}{2\mu}}\right)^\frac{2\mu}{2\mu-1}\right)\right)+O\left(|\textbf{e}|e^{-cn}\right).
				\end{multline}	
				
				$\bullet$ For $-\frac{j_0}{\alpha_{l^\prime}^+}\notin\left[\frac{n}{2},2n\right]$ and $j<0$, we have that:
				\begin{equation}\label{lem:OndesExcitedCentralDer:Res-2}
					\int_{\Gamma_{in}(\eta)}e^{n\tau}e^\tau\widetilde{g}_{m^\prime, d(p+q)}^+(e^\tau)\left(\widetilde{\Ccc}^+_{m^\prime}(e^\tau,j_0,\textbf{e})-\widetilde{\Ccc}^+_{m^\prime}(e^\tau,j_0-1,\textbf{e})\right)\Pi(\Phi_{d(p+q)}(e^\tau,j))d\tau  =	O(|\textbf{e}|e^{-cn}).
				\end{equation}	
			\end{subequations}
		\end{lemma}
		
		\begin{proof}\textbf{of Lemma \ref{lem:OndesExcitedCentralDer}}
			
			We consider $m^\prime\in I_{cu}^+$ and write it as $m^\prime=l^\prime+(k^\prime-1)d$ with $k^\prime\in\lc1,\ppp,p+q\rc$ and $l^\prime\in\lc1,\ppp,d\rc$. We start by proving that there exist two positive constants $C,c$ such that:
			\begin{multline}\label{in:E2mu}
				\forall n\in\N\backslash\lc0\rc,\forall j_0\in\Z,\quad \left|E_{2\mu}\left(\beta_{l^\prime}^+;\frac{n\alpha^+_{l^\prime}+j_0-1}{n^\frac{1}{2\mu}}\right)-E_{2\mu}\left(\beta_{l^\prime}^+;\frac{n\alpha^+_{l^\prime}+j_0}{n^\frac{1}{2\mu}}\right)\right| \\ \leq \frac{C}{n^\frac{1}{2\mu}}\exp\left(-c\left(\frac{\left|n+\frac{j_0}{\alpha^+_{l^\prime}}\right|}{n^\frac{1}{2\mu}}\right)^\frac{2\mu}{2\mu-1}\right).
			\end{multline}
			Since $-H_{2\mu}(\beta^+_{l^\prime};\cdot)$ is the derivative of the function $E_{2\mu}(\beta^+_{l^\prime};\cdot)$, using \eqref{inH}, there exist two positive constants $C,c$ such that:
			\begin{multline*}
				\forall n\in\N\backslash\lc0\rc,\forall j_0\in\Z,\quad \left|E_{2\mu}\left(\beta_{l^\prime}^+;\frac{n\alpha^+_{l^\prime}+j_0-1}{n^\frac{1}{2\mu}}\right)-E_{2\mu}\left(\beta_{l^\prime}^+;\frac{n\alpha^+_{l^\prime}+j_0}{n^\frac{1}{2\mu}}\right)\right| \\ \leq \frac{C}{n^\frac{1}{2\mu}}\sup_{t\in\left[\frac{n\alpha^+_{l^\prime}+j_0-1}{n^\frac{1}{2\mu}},\frac{n\alpha^+_{l^\prime}+j_0}{n^\frac{1}{2\mu}}\right]}\exp\left(-c|t|^\frac{2\mu}{2\mu-1}\right).
			\end{multline*}
			Then, there exist two new positive constants $C,c$ such that \eqref{in:E2mu} is verified. Using the definition \eqref{def:E+} of $E^+_{l^\prime}$ and the bound \eqref{decExpoV} of $V(j)$, the inequality \eqref{in:E2mu} implies that there exist two positive constants $C,c$ such that for all $n\in\N\backslash\lc0\rc$, $j_0\in\N\backslash\lc0\rc$, $j\in\Z$ such that $j-j_0\in\lc-nq,\ppp,np\rc$ and $\textbf{e}\in\C^d$, we have that:
			
			\begin{subequations}\label{lem:OndesCentrales:in}
				$\bullet$ For $-\frac{j_0}{\alpha^+_{l^\prime}}\in\left[\frac{n}{2},2n\right]$:
				\begin{equation}
					\left|\left(E^+_{l^\prime}(n,j_0)-E^+_{l^\prime}(n,j_0-1)\right)\textbf{e}V(j)\right|\leq \frac{C|\textbf{e}|e^{-c|j|}}{n^\frac{1}{2\mu}}\exp\left(-c\left(\frac{\left|n+\frac{j_0}{\alpha^+_{l^\prime}}\right|}{n^\frac{1}{2\mu}}\right)^\frac{2\mu}{2\mu-1}\right).
				\end{equation}
				
				$\bullet$ For $-\frac{j_0}{\alpha^+_{l^\prime}}\notin\left[\frac{n}{2},2n\right]$:
				\begin{equation}
					\left|\left(E^+_{l^\prime}(n,j_0)-E^+_{l^\prime}(n,j_0-1)\right)\textbf{e}V(j)\right|\leq C|\textbf{e}|e^{-cn}.
				\end{equation}
			\end{subequations}
			
			Thus, using \eqref{lem:OndesCentrales:in} and the result of Lemma \ref{lem:OndesExcitedCentral} for $j_0$ and $j_0-1$, we conclude the proof of Lemma \ref{lem:OndesExcitedCentralDer}.
		\end{proof}
		
		\subsubsection{Concluding the proof of the decomposition \eqref{decompoDerGreen}}
		
		There just remains to conclude the proof of the decomposition \eqref{decompoDerGreen} of the discrete derivative of the temporal Green's function. We will give details on the case when $j\geq j_0+1$. The cases when $j\in\lc0,\ppp,j_0\rc$ and $j<0$ would be handled similarly using \eqref{ExpGs: l geq 0 : 0 leq j leq l} and \eqref{ExpGs: l geq 0 : j < 0} rather than \eqref{ExpGs: l geq 0 : j geq l+1}
		
		Using \eqref{egGSpaTemp2}, Lemma \ref{estGreenExt} and \eqref{ExpGs: l geq 0 : j geq l+1}, there exists a constant $c>0$ such that for $n\in\N\backslash\lc0\rc$, $j_0\in \N\backslash\lc0\rc$, $j\geq j_0+1$ and $\textbf{e}\in\C^d$ which verify $j-j_0\in\lc-nq-1,\ppp,np\rc$:
		\begin{align*}
			&\left(\Gcc(n,j_0,j)-\Gcc(n,j_0-1,j)\right)\textbf{e}\\
			=& \frac{1}{2i\pi}\int_{\Gamma_{in}(\eta)} e^{n\tau}e^\tau \Pi(W(e^\tau,j_0,j,\textbf{e}))d\tau-\frac{1}{2i\pi}\int_{\Gamma_{in}(\eta)} e^{n\tau}e^\tau \Pi(W(e^\tau,j_0-1,j,\textbf{e}))d\tau\\
			& + \frac{1}{2i\pi}\int_{\Gamma_{out}(\eta)} e^{n\tau}e^\tau \Pi(W(e^\tau,j_0,j,\textbf{e}))d\tau- \frac{1}{2i\pi}\int_{\Gamma_{out}(\eta)} e^{n\tau}e^\tau \Pi(W(e^\tau,j_0-1,j,\textbf{e}))d\tau \\
			=& \textcolor{dartmouthgreen}{\sum_{m\in I_{cs}^+} \left(-\frac{1}{2i\pi} \int_{\Gamma_{in}(\eta)} e^{n\tau}e^\tau \left(\widetilde{\Ccc}^+_{m}(e^\tau,j_0,\textbf{e})-\widetilde{\Ccc}^+_{m}(e^\tau,j_0-1,\textbf{e})\right)\Pi\left(W^+_m(e^\tau,j)\right)d\tau\right)} \\
			& + \textcolor{dartmouthgreen}{\sum_{m\in I_{cs}^+} \sum_{m^\prime\in I_{cu}^+} \left(-\frac{1}{2i\pi} \int_{\Gamma_{in}(\eta)} e^{n\tau}e^\tau\widetilde{g}_{m^\prime,m}^+(e^\tau) \left(\widetilde{\Ccc}^+_{m^\prime}(e^\tau,j_0,\textbf{e})-\widetilde{\Ccc}^+_{m^\prime}(e^\tau,j_0-1,\textbf{e})\right)\Pi\left(W^+_m(e^\tau,j)\right)d\tau\right)}\\
			& + \textcolor{dartmouthgreen}{\sum_{m^\prime\in I_{cu}^+}  \left(-\frac{1}{2i\pi} \int_{\Gamma_{in}(\eta)} e^{n\tau}e^\tau\widetilde{g}_{m^\prime,1}^+(e^\tau) \left(\widetilde{\Ccc}^+_{m^\prime}(e^\tau,j_0,\textbf{e})-\widetilde{\Ccc}^+_{m^\prime}(e^\tau,j_0-1,\textbf{e})\right)\Pi\left(\Phi_1(e^\tau,j)\right)d\tau\right)}\\
			&+\textcolor{blue}{\sum_{m\in I_{ss}^+} \left(-\frac{1}{2i\pi} \int_{\Gamma_{in}(\eta)} e^{n\tau}e^\tau \left(\widetilde{\Ccc}^+_{m}(e^\tau,j_0,\textbf{e})-\widetilde{\Ccc}^+_{m}(e^\tau,j_0-1,\textbf{e})\right)\Pi\left(W^+_m(e^\tau,j)\right)d\tau\right)} \\
			& + \textcolor{blue}{\sum_{m\in I_{ss}^+\backslash\lc1\rc} \sum_{m^\prime\in I_{cu}^+} \left(-\frac{1}{2i\pi} \int_{\Gamma_{in}(\eta)} e^{n\tau}e^\tau\widetilde{g}_{m^\prime,m}^+(e^\tau) \left(\widetilde{\Ccc}^+_{m^\prime}(e^\tau,j_0,\textbf{e})-\widetilde{\Ccc}^+_{m^\prime}(e^\tau,j_0-1,\textbf{e})\right)\Pi\left(W^+_m(e^\tau,j)\right)d\tau\right)}\\
			& + \textcolor{blue}{\sum_{m\in I_{cs}^+} \sum_{m^\prime\in I_{su}^+} \left(-\frac{1}{2i\pi} \int_{\Gamma_{in}(\eta)} e^{n\tau}e^\tau\widetilde{g}_{m^\prime,m}^+(e^\tau) \left(\widetilde{\Ccc}^+_{m^\prime}(e^\tau,j_0,\textbf{e})-\widetilde{\Ccc}^+_{m^\prime}(e^\tau,j_0-1,\textbf{e})\right)\Pi\left(W^+_m(e^\tau,j)\right)d\tau\right)}\\
			& + \textcolor{blue}{\sum_{m\in I_{ss}^+\backslash\lc1\rc} \sum_{m^\prime\in I_{su}^+} \left(-\frac{1}{2i\pi} \int_{\Gamma_{in}(\eta)} e^{n\tau}e^\tau\widetilde{g}_{m^\prime,m}^+(e^\tau) \left(\widetilde{\Ccc}^+_{m^\prime}(e^\tau,j_0,\textbf{e})-\widetilde{\Ccc}^+_{m^\prime}(e^\tau,j_0-1,\textbf{e})\right)\Pi\left(W^+_m(e^\tau,j)\right)d\tau\right)}\\
			&+ \textcolor{blue}{\sum_{m^\prime\in I_{su}^+}  \left(-\frac{1}{2i\pi} \int_{\Gamma_{in}(\eta)} e^{n\tau}e^\tau\widetilde{g}_{m^\prime,1}^+(e^\tau) \left(\widetilde{\Ccc}^+_{m^\prime}(e^\tau,j_0,\textbf{e})-\widetilde{\Ccc}^+_{m^\prime}(e^\tau,j_0-1,\textbf{e})\right)\Pi\left(\Phi_1(e^\tau,j)\right)d\tau\right)}+O(|\textbf{e}|e^{-cn}).
		\end{align*}
		
		Using Lemmas \ref{lem:OndesPropagDer}-\ref{lem:OndesExcitedCentralDer} to bound the first three green terms and Lemmas \ref{lem:OndesPropag}-\ref{lem:OndesExcitedCentral} to bound the remaining blue terms, we obtain the decomposition \eqref{decompoDerGreen} with an additional term 
		$$\left(P_U(j_0)-P_U(j_0-1)\right)V(j) $$
		where the line vectors $P_U(j_0)$ are defined by \eqref{def:PU}. However, we have proved in \eqref{eg:PU} that those line vectors are equal to $0$. We have thus finally obtained the decomposition \eqref{decompoDerGreen} of the discrete temporal Green's function when $j\geq j_0+1$. The decomposition \eqref{decompoDerGreen} for $j\in\lc0,\ppp, j_0\rc$ and $j<0$ would be handled similarly.

	\section{Numerical example}\label{sec:Num}
	
	In the present section, we will prove that several hypotheses of the article are verified under some CFL condition for the modified Lax-Friedrichs scheme. We will then apply the modified Lax-Friedrichs scheme to Burgers equation and numerically observe the result of Theorem \ref{th:Green} in the scalar case (i.e. $d=1$). Let us observe that there exists an example of numerical simulation for the gas dynamics system ($d=3$) and for the modified Lax-Friedrichs scheme in \cite[Section 5]{Godillon}. 
	
	\subsection*{Some hypotheses verified for the modified Lax-Friedrichs scheme}
	
	Let us start by considering a one-dimensional system of conservation laws \eqref{def:EDP} with $\Uc=\R^d$ and two states $u^-$ and $u^+$ in $\R^d$ which satisfy the Rankine-Hugoniot condition \eqref{eq:RankineHugioniot} and Hypothesis \ref{H:Lax} (i.e. it is a Lax shock). We now consider as an example the modified Lax-Friedrichs scheme for which the numerical flux is defined by:
	$$\forall \nu\in]0,+\infty[,\forall u_{-1},u_0\in\R^d, \quad F(\nu;u_{-1},u_0):=\frac{f(u_{-1})+f(u_0)}{2}+D(u_{-1}-u_0)$$
	where $D$ is a positive constant. The integer $p$ and $q$ are then both equal to $1$. We immediately observe that the consistency condition \eqref{cond:consistency} is verified. We consider a constant $\nug>0$ which corresponds to the ratio between the space and time steps and which satisfies the CFL condition \eqref{cond:CFL}. The discrete evolution operator $\Nc$ is then defined for $u\in(\R^d)^\Z$ by:
	\begin{equation}\label{num:MLF}
		\forall j\in\Z,\quad (\Nc(u))_j := u_j - \nug\left(\frac{f(u_{j+1})-f(u_{j-1})}{2}+D\left(-u_{j+1}+2u_j-u_{j-1}\right)\right).
	\end{equation}
	We have that the matrices $B_k^\pm$ defined by \eqref{def:Bk} are in the present case:
	$$B_{-1}^\pm=\nug\left(\frac{1}{2}df(u^\pm)+DId\right) \quad \text{ and } \quad B_{0}^\pm=\nug\left(\frac{1}{2}df(u^\pm)-DId\right).$$ 
	Hypothesis \ref{H:VPAk} is immediately verified. Furthermore, if we assume that:
	$$\forall l\in\lc1,\ppp,d\rc,\quad |\lambg_l^\pm|<2D $$ 
	then a part of Hypothesis \ref{H:inv} is verified. 
	
	We will now prove that Hypothesis \ref{H:F} on the spectrum of the operator $\Lcc^\pm$ is verified under conditions on $\nug$ and $D$. For $l\in \lc1,\ppp,d\rc$, we have that the functions $\Fc^\pm_l$ defined by \eqref{def:Fcl} verify:
	$$\forall \kappa\in\C\backslash\lc0\rc,\quad \Fc_l^\pm(\kappa)= (1-2D\nug) +D\nug\left(\kappa+\kappa^{-1}\right)-\frac{\lambg_l^\pm\nug}{2}\left(\kappa-\kappa^{-1}\right)$$
	and thus:
	$$\forall \xi\in\R,\quad \Fc_l^\pm\left(e^{i\xi}\right)= (1-2D\nug)+2D\nug\cos(\xi)-i\lambg_l^\pm\nug\sin(\xi).$$
	Therefore, under the assumption that:
	$$\forall l\in\lc1,\ppp,d\rc,\quad |\lambg_l^\pm|\nug <2D\nug<1,$$
	then, we have that:
	$$\forall \xi\in\R,\quad \left|\Fc_l^\pm\left(e^{i\xi}\right)\right|^2= (1-2D\nug)^2+4D\nug(1-2D\nug)\cos(\xi)+(\lambg_l^\pm\nug)^2 +\left((2D\nug)^2-(\lambg_l^\pm\nug)^2\right)\cos^2(\xi)\leq1$$
	with an equality if and only if $\xi\in2\pi\Z$. Furthermore, the asymptotic expansion \eqref{F} is verified for $\mu=1$ and we have that:
	$$\forall l\in\lc1,\ppp,d\rc,\quad \beta_l^\pm:=\frac{2D\nug-(\lambg_l^\pm\nug)^2}{2}\in]0,+\infty[.$$
	Thus, Hypothesis \ref{H:F} is verified.

	Let us conclude this section by proving that Hypothesis \ref{H:Mpm1} is verified. We consider $l\in\lc1,\ppp,d\rc$. For $\kappa \in \C\backslash\lc0\rc$, we have that:
	$$\Fc_l^\pm(\kappa)= 1\Leftrightarrow \left(D\nug-\frac{\lambg_l^\pm\nug}{2}\right)\kappa^2 - 2D\nug\kappa +\left(D\nug+\frac{\lambg_l^\pm\nug}{2}\right)=0.$$
	The equation on the right-hand side above has two solutions: 
	$$1\quad \text{ and } \quad \frac{2D\nug +\lambg_l^\pm\nug}{2D\nug -\lambg_l^\pm\nug}\neq 1.$$
	Thus, Hypothesis \ref{H:Mpm1} is verified. 
	
	The last hypotheses remaining to prove depend on the existence of a SDSP associated with the shock we consider and are therefore fairly situational.
		
	\begin{figure}
		\centering
		\includegraphics[width=0.7\textwidth]{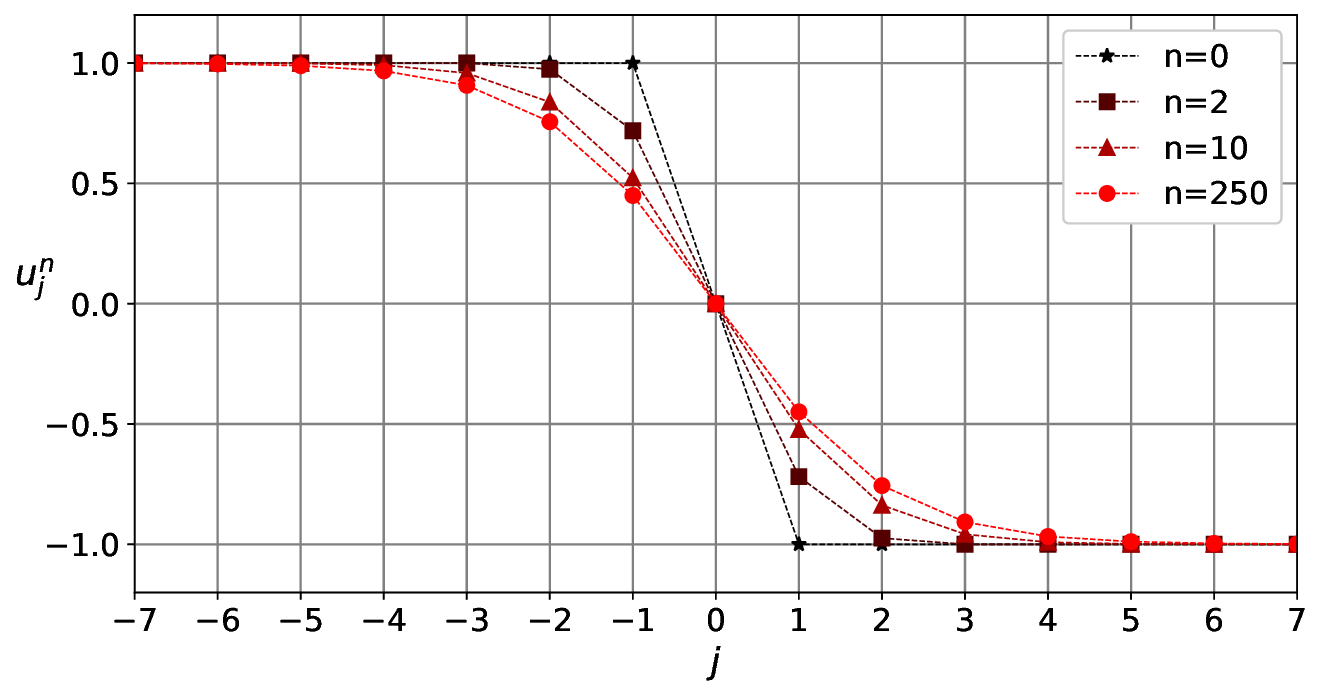}
		\caption{Solution $u^n$ of the numerical scheme \eqref{num:MLF} for the initial condition \eqref{condini}. The solution $u^n$ seems to converge in time (quite fast) towards a stationary discrete shock profile $\dsp$.}
		\label{fig:SDSP}
	\end{figure}	
		
	\subsection*{Application to Burgers equation}

	In the present section, we consider the scalar ($d=1$) conservation law \eqref{def:EDP} where the flux $f$ is defined by:
	$$\forall u\in\R,\quad f(u):= \frac{u^2}{2}.$$
	This corresponds to the so-called Burgers equation. For the shock, we consider the states
	$$u^-=1\quad \text{ and } \quad u^+=-1.$$
	The Rankine-Hugoniot condition \eqref{eq:RankineHugioniot} and Hypothesis \ref{H:Lax} are verified. Thus, the shock associated with the states $u^-$ and $u^+$ is a stationary Lax shock.

	\begin{figure}
		\centering
		\begin{subfigure}{0.45\textwidth}
			\includegraphics[width=\textwidth]{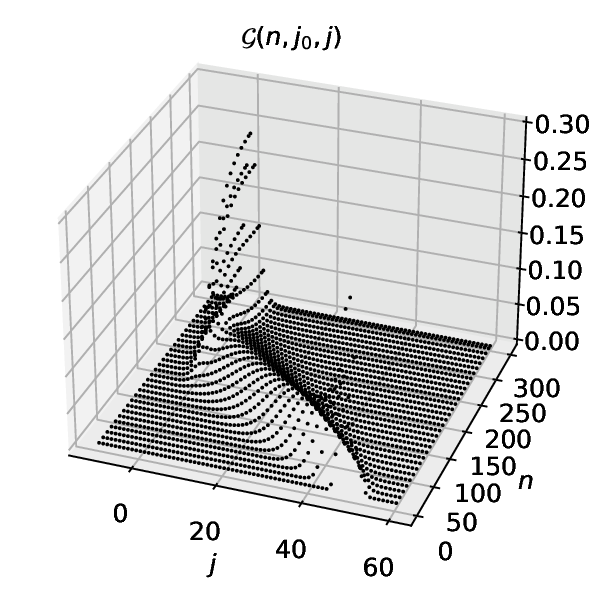}
		\end{subfigure}
		\hfill
		\centering
		\begin{subfigure}{0.37\textwidth}
			\includegraphics[width=\textwidth]{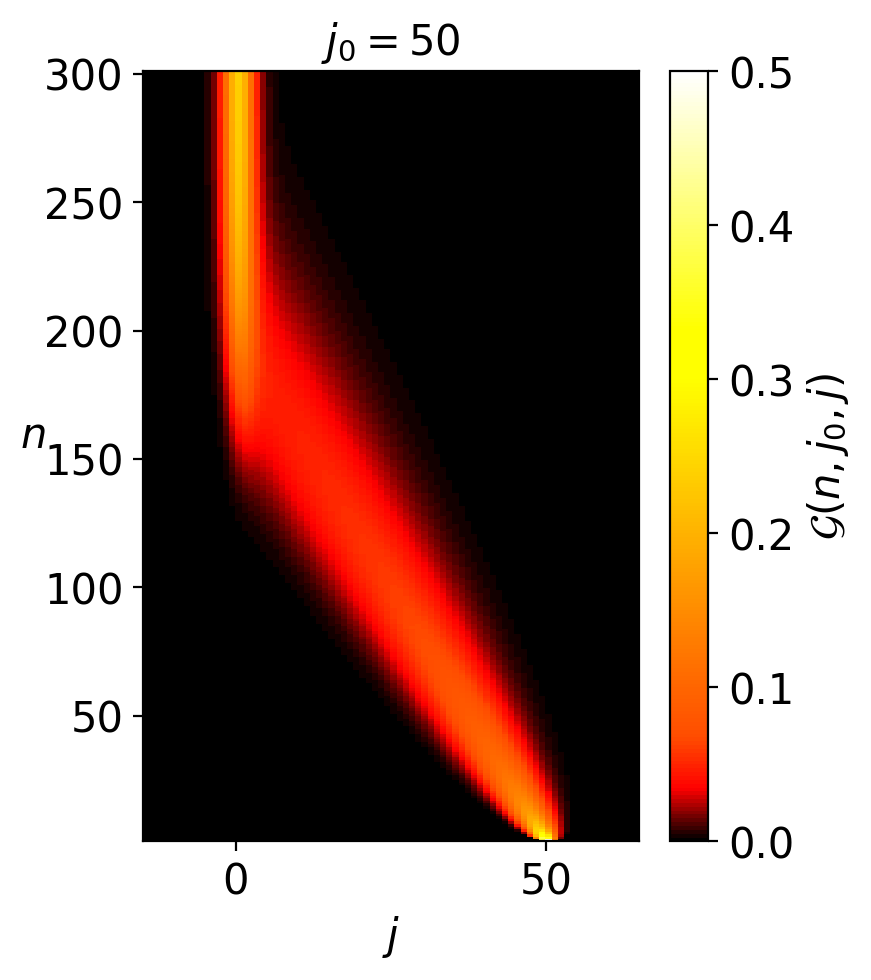}
		\end{subfigure}
		\caption{Representation of the Green's function $\Gcc(n,j_0,j)$ for $j_0=50$. The initial time $n=0$ has been removed for readability since the Green's function is equal to the sequence $\delta_{j_0}$ at this time step.}
		\label{fig:Green_50}
	\end{figure}
	
	For the numerical scheme, we consider the modified Lax-Friedrichs scheme defined above by \eqref{num:MLF} with the positive constants $D$ and $\nug$ satisfying:
	$$ \nug<2D\nug<1.$$
	The section above implies that several of the hypotheses we introduced in the article are verified. From now on, we choose $\nug=0.25$ and $D=1$.

	We want to know whether Hypothesis \ref{H:SDSP} is verified, i.e. if there exists a stationary discrete shock profile associated with the shock we are considering. In Figure \ref{fig:SDSP}, we display the solution $u^n$ of the numerical scheme \eqref{def:SchemeNum} for the initial condition
	\begin{equation}\label{condini}
		\forall j \in\Z,\quad u^0_j:=\lc\begin{array}{cc}1 & \text{ if }j\leq-1,\\0 & \text{ if }j=0,\\-1 & \text{ if }j\geq1.\end{array}\right.
	\end{equation}
	We observe that the solution $u^n$ seems to converge in time (quite fast) towards a fixed point $\dsp$ of the discrete evolution operator $\Nc$. Furthermore, it seems as if there exist two positive constants $C,c$ such that:
	$$\forall j\in\N,\quad \begin{array}{c}
		|\dsp_j-u^+| = |\dsp_j+1| \leq Ce^{-c|j|} ,\\
		|\dsp_{-j}-u^-| = |\dsp_{-j}-1| \leq Ce^{-c|j|} .
	\end{array}$$

	Hypotheses \ref{H:SDSP} and \ref{H:CVexpo} would then be verified. We can then compute the Green's function $\Gcc(n,j_0,j)$ defined by \eqref{defGreenTempo} associated with the operator $\Lcc$. In Figures \ref{fig:Green_50} and \ref{fig:Green_-30}, we display the Green's function $\Gcc(n,j_0,j)$ for $j_0=50$ and $-30$. The behavior displayed by the Green's function fits the result of Theorem \ref{th:Green}. For small times, the Green's function resembles a Gaussian wave traveling along the characteristics of the conservation law \eqref{def:EDP}. Since we are considering Lax shocks, the waves are traveling towards the shock location at $j=0$. When they reach the shock location, there only remains a fixed solution of the operator $\Lcc$ which corresponds to the progressive construction of the component of the Green's function along the vector subspace $\ker(Id_{\ell^2}-\Lcc)$.
	
	\begin{figure}
		\centering
		\begin{subfigure}{0.45\textwidth}
			\includegraphics[width=\textwidth]{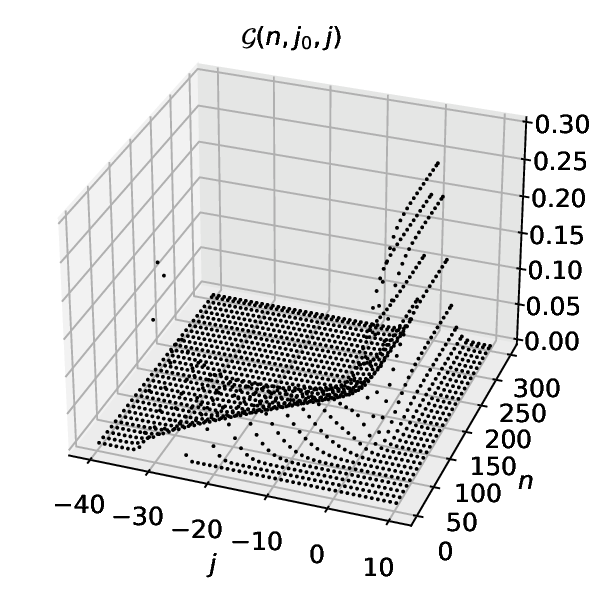}
		\end{subfigure}
		\hfill
		\centering
		\begin{subfigure}{0.37\textwidth}
			\includegraphics[width=\textwidth]{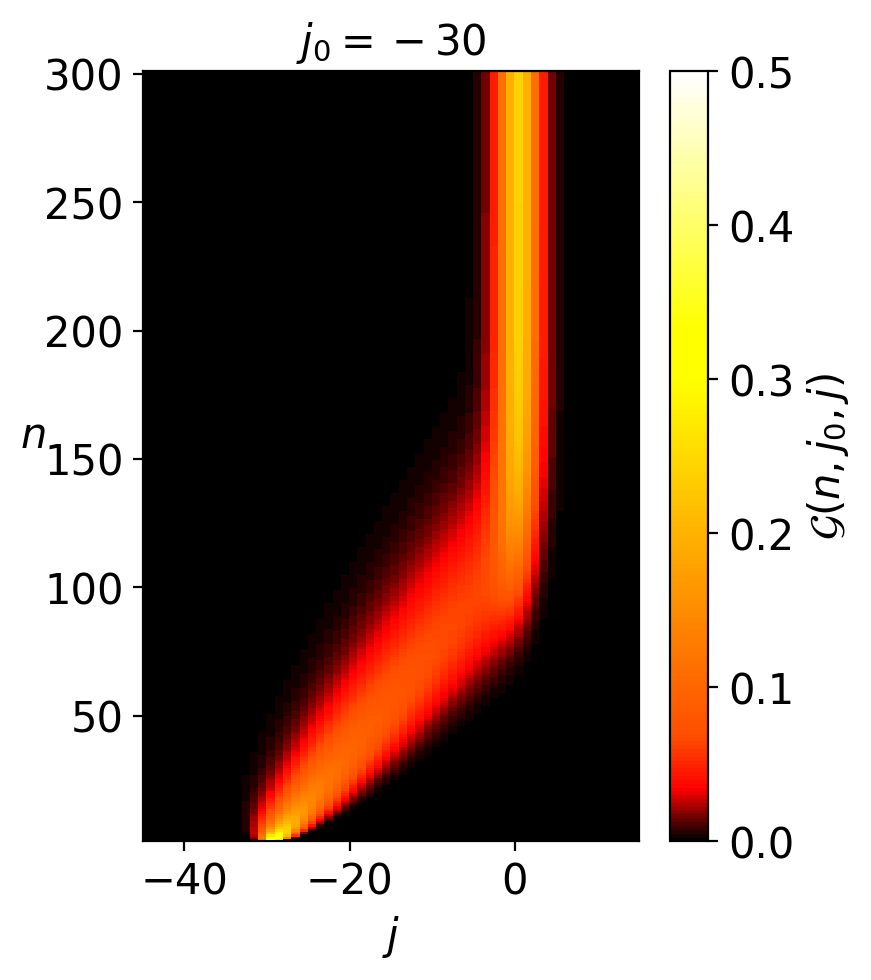}
		\end{subfigure}
		\caption{Representation of the Green's function $\Gcc(n,j_0,j)$ for $j_0=-30$. The initial time $n=0$ has been removed for readability purposes since the Green's function is equal to the sequence $\delta_{j_0}$ at this time step.}
		\label{fig:Green_-30}
	\end{figure}

	\section{Appendix}\label{sec:Appendix}
	
	\subsection*{Proof of Lemma \ref{lemDecExp}}
	
	We define recursively 
	$$\forall j\in \N, \quad z^0_j:=y_j$$
	and
	\begin{equation}\label{egZn}
		\forall n\in \N, \forall j\in \N, \quad z^{n+1}_j:=C_He^{-c_H j}+ \Theta\sum_{k=0}^{+\infty}e^{-c_H|j-1-k|}z^n_k.
	\end{equation}
	We then prove recursively that for all $n\in \N$, the sequence $z^n$ is bounded, has non negative coefficients and
	\begin{equation}\label{lemDecExp:Recu}
		\left\|z^n\right\|_{\infty}\leq \frac{C_H}{1-\theta} \quad \text{ and } \quad \forall j \in \N, \quad z^n_j\leq z^{n+1}_j.
	\end{equation}
	
		Let us point out that proving recursively that the sequence $z^n$ is bounded actually implies that the sequence $z^{n+1}$ is also well-defined by \eqref{egZn}.
		
		Let us initialize our recursive proof. We already assumed that the sequence $z^0=y$ is bounded and has non negative coefficients. We thus want to prove \eqref{lemDecExp:Recu} for $n=0$. We start by noticing that the inequality \eqref{inY} immediately implies the second inequality of \eqref{lemDecExp:Recu}. Using the definition \eqref{def:theta2} of the constant $\theta$ and that the sequence $z^0$ is bounded and has non negative coefficients, we notice that :
		$$\forall j\in\Z,\quad\Theta\sum_{k=0}^{+\infty}e^{-c_H|j-1-k|}z^0_k \leq \Theta\frac{1+e^{-c_H}}{1-e^{-c_H}}\left\|z^0\right\|_{\infty}=\theta\left\|z^0\right\|_{\infty}. $$
		Thus, using the inequality \eqref{inY}, we have that:
		\begin{equation}\label{lemDecExp:Ini}
			\forall j \in\Z, \quad z^0_j\leq C_H + \theta \left\|z^0\right\|_\infty.
		\end{equation}
		Since $z^0$ has non negative coefficients, we obtain:
		$$ \left\|z^0\right\|_\infty\leq C_H + \theta \left\|z^0\right\|_\infty$$
		which allows us to conclude the proof of \eqref{lemDecExp:Recu} for $n=0$.
		
		We now consider $n\in \N$ for which the recursive property is verified we will prove that the recursive property for $n+1$ is satisfied. First, using the equality \eqref{egZn} and the fact that the sequence $z^n$ has non negative coefficients, we immediately have that the sequence $z^{n+1}$ has non negative coefficients. Furthermore, using a similar proof as for \eqref{lemDecExp:Ini}, we have for all $j\in\N$: 
		$$z_j^{n+1}\leq C_H +\theta \left\| z^n\right\|_\infty$$
		and thus using \eqref{lemDecExp:Recu}:
		$$ z_j^{n+1}\leq \frac{C_H}{1-\theta}.$$
		Thus, we have that $z^{n+1}\in\ell^\infty(\N)$ and: 
		$$\left\|z^{n+1}\right\|_\infty\leq \frac{C_H}{1-\theta}.$$
		
	Finally, we observe that using the equality \eqref{egZn} for $n$ and $n+1$, we have:
	$$z^{n+2}_j-z^{n+1}_j=\Theta\sum_{k=0}^{+\infty}e^{-c_H|j-1-k|}\underset{\geq 0}{\underbrace{(z^{n+1}_k-z^n_k)}}\geq 0.$$
	This concludes the recurrence.
	
	We now observe that, for $p,q \in \N\backslash\lc0\rc$, using the equality \eqref{egZn}, we have
	$$z^p_j-z^q_j=\Theta\sum_{k=0}^{+\infty}e^{-c_H|j-1-k|}(z^{p-1}_k-z^{q-1}_k).$$
	This implies that
	$$\left\|z^p-z^q\right\|_\infty\leq \theta \left\|z^{p-1}-z^{q-1}\right\|_\infty.$$
	Thus, we have that
	$$\forall p\geq q\geq 0, \quad \left\| z^p-z^q\right\|_\infty\leq \theta^q\left\|z^{p-q}-y\right\|_\infty\leq \theta^q\frac{2C_H}{1-\theta}.$$
	Since $\theta<1$, the sequence $(z^n)_{n\in\N}$ is a Cauchy sequence of $\ell^\infty(\N)$, thus it converges towards a sequence $z^\infty\in\ell^\infty(\N)$. Since we have $y_j \leq z^n_j$ for all $n,j\in \N$, we obviously have 
	$$\forall j \in \N, \quad y_j\leq z^\infty_j.$$
	Also, the equality \eqref{egZn} implies that
	\begin{equation}\label{egZinf}
		\forall j \in \N, \quad z^\infty_j=C_He^{-c_H j}+\Theta\sum_{k=0}^{+\infty}e^{-c_H|j-1-k|}z^\infty_k.
	\end{equation}
	Thus, there just remain to prove that there exists only one bounded sequence that satisfies \eqref{egZinf} and that it has the form $z^\infty=(\rho r^j)_{j\in\N}$ where $\rho$ and $r$ satisfies the properties we expected. 
	
	We write \eqref{egZinf} for $j$, $j+1$ and $j+2$ and reassemble the terms $z^\infty_j$, $z^\infty_{j+1}$ and $z^\infty_{j+2}$ on the left side. We then have
	\begin{subequations}\label{eg:lem}
		\begin{multline}
			\left(1-e^{-c_H}\Theta\right)z^\infty_j -e^{-2c_H}\Theta z^\infty_{j+1} -e^{-3c_H}\Theta z^\infty_{j+2}  \\= C_He^{-c_Hj}+\Theta\left(\sum_{k=0}^{j-1}e^{-c_H|j-1-k|}z^\infty_k +\sum_{k=j+3}^{+\infty}e^{-c_H|j-1-k|}z^\infty_k\right)
		\end{multline}
		\begin{multline}
			-\Theta z^\infty_j +\left(1-e^{-c_H}\Theta\right)z^\infty_{j+1} -e^{-2c_H}\Theta z^\infty_{j+2}  \\= e^{-c_H}C_He^{-c_Hj}+\Theta\left(e^{-c_H}\sum_{k=0}^{j-1}e^{-c_H|j-1-k|}z^\infty_k +e^{c_H}\sum_{k=j+3}^{+\infty}e^{-c_H|j-1-k|}z^\infty_k\right)
		\end{multline}
		\begin{multline}
			-e^{-c_H}\Theta z^\infty_j -\Theta z^\infty_{j+1} +\left(1-e^{-c_H}\Theta\right)z^\infty_{j+2} \\ = e^{-2c_H}C_He^{-c_Hj}+\Theta\left(e^{-2c_H}\sum_{k=0}^{j-1}e^{-c_H|j-1-k|}z^\infty_k +e^{2c_H}\sum_{k=j+3}^{+\infty}e^{-c_H|j-1-k|}z^\infty_k\right).
		\end{multline}
	\end{subequations}
	
	We consider three scalars $\alpha, \beta,\gamma\in \R$ such that
	$$\lc\begin{array}{cc}\alpha+\beta e^{-c_H}+\gamma e^{-2c_H}&=0\\
		\alpha+\beta e^{c_H}+\gamma e^{2c_H}&=0.\end{array}\right.$$
	A solution is $\alpha=\gamma=1$ and $\beta=-\frac{e^{2c_H}-e^{-2c_H}}{e^{c_H}-e^{-c_H}}=-\frac{\sinh(2c_H)}{\sinh(c_H)}=-2\cosh(c_H)$. We then have that 
	\begin{align*}
		\alpha-\Theta\left(\alpha e^{-c_H}+\beta+\gamma e^{-c_H}\right) &=1-\theta\frac{e^{\frac{c_H}{2}}-e^{-\frac{c_H}{2}}}{e^{\frac{c_H}{2}}+e^{-\frac{c_H}{2}}}\left(2e^{-c_H}-2\cosh(c_H)\right)\\
		&=1+2\theta\sinh(c_H)\frac{\sinh\left(\frac{c_H}{2}\right)}{\cosh\left(\frac{c_H}{2}\right)} \\
		& =1+4\theta\sinh^2\left(\frac{c_H}{2}\right).
	\end{align*}
	Multiplying the equalities \eqref{eg:lem} respectively by $\alpha,\beta,\gamma$ and summing them, we obtain that
	$$\forall j \in \N, \quad z^\infty_{j+2}-2\cosh(c_H)z^\infty_{j+1}+\left(1+4\theta\sinh^2\left(\frac{c_H}{2}\right)\right)z^\infty_j=0.$$
	We are thus led to study the polynomial 
	\begin{equation}\label{def:Poly}
		P:=X^2-2\cosh(c_H)X+\left(1+4\theta\sinh^2\left(\frac{c_H}{2}\right)\right).
	\end{equation}
	Its discriminant is
	$$\Delta = 4\sinh^2(c_H)-16\theta\sinh^2\left(\frac{c_H}{2}\right) = 16\sinh^2\left(\frac{c_H}{2}\right)\left(\cosh^2\left(\frac{c_H}{2}\right)-\theta\right)>0.$$
	Its roots are 
	$$r_\pm:=\cosh(c_H)\pm2\sinh\left(\frac{c_H}{2}\right)\sqrt{\cosh^2\left(\frac{c_H}{2}\right)-\theta}.$$
	We observe that evaluating the polynomial $P$ at $0$ and $1$ gives us
	$$P(0)=1+4\theta\sinh^2\left(\frac{c_H}{2}\right)>0$$
	and
	$$P(1)=2-2\cosh(c_H)+4\theta\sinh^2\left(\frac{c_H}{2}\right) = -4(1-\theta)\sinh^2\left(\frac{c_H}{2}\right)<0.$$
	Thus, $r_+>1$ and $r:=r_-\in]0,1[$. We have that $\rho,\tilde{\rho}\in\R$ such that
	$$\forall j \in \N, \quad z^\infty_j=\rho r^j +  \tilde{\rho}r_+^j.$$
	Since the sequence $z^\infty$ is bounded, we have that $\tilde{\rho}=0$, i.e.
	$$\forall j \in \N, \quad z^\infty_j=\rho r^j .$$
	Let us now compute the value of $\rho$. Using the equality \eqref{egZinf}, we obtain for $j\in\Z$:
	\begin{align*}
		\rho r^j &= C_He^{-c_Hj} + \Theta\rho\sum_{k=0}^{+\infty}e^{-c_H|j-1-k|}r^k\\
		&=C_He^{-c_Hj} + \Theta\rho\left(e^{-c_H(j-1)}\frac{(re^{c_H})^j-1}{re^{c_H}-1}+e^{c_H(j-1)}\frac{(re^{-c_H})^j}{1-re^{-c_H}}\right)\\
		& = e^{-c_Hj}\left(C_H-\Theta\rho\frac{e^{c_H}}{re^{c_H}-1}\right) + r^j \Theta\rho\left(\frac{e^{c_H}}{re^{c_H}-1}+\frac{e^{-c_H}}{1-re^{-c_H}}\right).
	\end{align*}
	Using the definition \eqref{def:theta2} of $\theta$ and the fact that $r$ is a root of the polynomial $P$ defined by \eqref{def:Poly}, we have that:
	\begin{align*}
		\Theta\left(\frac{e^{c_H}}{re^{c_H}-1}+\frac{e^{-c_H}}{1-re^{-c_H}}\right) & = \theta \frac{\sinh\left(\frac{c_H}{2}\right)}{\cosh\left(\frac{c_H}{2}\right)}\frac{2\sinh(c_H)}{-r^2+2r\cosh(c_H)-1}\\
		& = \theta \frac{\sinh\left(\frac{c_H}{2}\right)}{\cosh\left(\frac{c_H}{2}\right)}\frac{2\sinh(c_H)}{4\theta\sinh^2\left(\frac{c_H}{2}\right)}\\
		&=1.
	\end{align*} 
	Thus, we obtain:
	$$\rho r^j = e^{-c_Hj}\left(C_H-\Theta\rho\frac{e^{c_H}}{re^{c_H}-1}\right) +\rho r^j$$
	i.e.
	$$C_H=\Theta\rho\frac{e^{c_H}}{re^{c_H}-1}.$$
	Therefore, the scalar $\rho$ verifies:
	$$\rho=\frac{C_H}{\Theta}(r-e^{-c_H}).$$
	To conclude, we observe that $r>e^{-c_H}$. Indeed, we have that:
	\begin{align*}
		r&=\cosh(c_H)-2\sinh\left(\frac{c_H}{2}\right)\sqrt{\cosh^2\left(\frac{c_H}{2}\right)-\theta}\\
		& > \cosh(c_H)-2\sinh\left(\frac{c_H}{2}\right)\cosh\left(\frac{c_H}{2}\right)\\
		&= \cosh(c_H)-\sinh(c_H)\\
		&=e^{-c_H}.
	\end{align*}
	Therefore, $r$ belongs to the interval $]e^{-c_H},1[$ and $\rho$ is positive.
	
	\subsection*{Proof of Lemma \ref{lem:BornesGaussiennes}}
	
	$\bullet$ \textbf{Proof of \eqref{lem:BornesGaussiennesRes1}}
	
	We consider $n\in\N\backslash\lc0\rc$ and $x\in\left[0,\frac{n}{2}\right]$. Noticing that $\Gamma_d(\eta)\subset B(0,\varepsilon)$ and using \eqref{condEta}, we have
	\begin{align*}
		&\int_{\Gamma_d(\eta)} |\tau|^k\exp\left(n\Re(\tau ) +x (-\Re(\tau) +A_R\Re(\tau)^{2\mu} - A_I \Im(\tau)^{2\mu})\right) |d\tau| \\
		\leq& \varepsilon^k  \int_{-r_\varepsilon(\eta)}^{r_\varepsilon(\eta)} \exp\left(-(n-x)\eta +xA_R \eta^{2\mu}-xA_It^{2\mu}\right)dt\\
		\leq&  2r_\varepsilon(\eta) \varepsilon^k \exp\left(-\frac{n}{2}\left(\eta-A_R\eta^{2\mu}\right)\right)\\
		\leq&  2r_\varepsilon(\eta) \varepsilon^k \exp\left(-\frac{n\eta}{4}\right).
	\end{align*}
	
	$\bullet$ \textbf{Proof of \eqref{lem:BornesGaussiennesRes2}:}
	
	We consider $n\in\N\backslash\lc0\rc$ and $x\in\left[2n,+\infty\right[$. We will separate the integral on the path $\Gamma_{in}(\eta)$ using the paths $\Gamma_{in}^0(\eta)$ and $\Gamma_{in}^\pm(\eta)$ introduced in \eqref{def:Paths}.
	
	$\blacktriangleright$ Noticing that $\Gamma_d(\eta)\subset B(0,\varepsilon)$ and that $n-x\leq-\frac{x}{2}$, we have using condition \eqref{condEta} 
	\begin{align*}
		&\int_{\Gamma_{in}^0(\eta)}|\tau|^k\exp\left(n\Re(\tau)+x(-\Re(\tau)+A_R \Re(\tau)^{2\mu}-A_I\Im(\tau)^{2\mu})\right)|d\tau|\\
		\leq& \varepsilon^k\int_{-r_\varepsilon(\eta)}^{r_\varepsilon(\eta)}\exp\left((n-x)\eta+xA_R\eta^{2\mu}-xA_It^{2\mu}\right)dt\\
		\leq& 2 r_\varepsilon(\eta)\varepsilon^k \exp\left(-x\left(\frac{\eta}{2}-A_R\eta^{2\mu}\right)\right)\\
		\leq& 2 r_\varepsilon(\eta)\varepsilon^k \exp\left(-2n\left(\frac{\eta}{2}-A_R\eta^{2\mu}\right)\right).
	\end{align*}
	We have proved exponential bounds on this first term.
	
	$\blacktriangleright$ We have that using \eqref{condEta2} and that $x\geq 2n$
	\begin{align*}
		&\int_{\Gamma_{in}^\pm(\eta)}|\tau|^k\exp\left(n\Re(\tau)+x(-\Re(\tau)+A_R \Re(\tau)^{2\mu}-A_I\Im(\tau)^{2\mu})\right)|d\tau|\\
		\leq&\varepsilon^k \int_{-\eta}^{\eta}\exp\left((n-x)t+xA_Rt^{2\mu}-xA_Ir_\varepsilon(\eta)^{2\mu}\right)dt\\
		\leq& 2\eta \varepsilon^k \exp\left((x-n)\eta +xA_R\eta^{2\mu}-xA_Ir_\varepsilon(\eta)^{2\mu} \right)\\
		\leq&  2\eta \varepsilon^k \exp\left(x\left(\eta+A_R\eta^{2\mu}-A_Ir_\varepsilon(\eta)^{2\mu}\right)\right)\\
		\leq &2\eta\varepsilon^k \exp\left(-A_Ir_\varepsilon(\eta)^{2\mu}n\right).
	\end{align*}
	We have proved exponential bounds on this second term.
	
	We can then easily conclude the proof of \eqref{lem:BornesGaussiennesRes2}
	
	$\bullet$ \textbf{Proof of \eqref{lem:BornesGaussiennesRes3}:}
	
	We consider $n\in\N\backslash\lc0\rc$ and $x\in\left[\frac{n}{2},2n\right]$. We start by observing that
	\begin{equation}\label{Expo-Gaussian}
		\forall x\in\left[\frac{n}{2},2n\right],\quad \left(\frac{|n-x|}{n^\frac{1}{2\mu}}\right)^\frac{2\mu}{2\mu-1} \leq n.
	\end{equation}
	Thus, obtaining exponential bounds on certain terms when $x\in[\frac{n}{2},2n]$ would also allow to conclude on the proof of \eqref{lem:BornesGaussiennesRes3}.
	
	We will now follow a strategy developed in \cite{ZH} in a continuous setting, which has also been used in \cite{Godillon,C-F,C-FIBVP,CoeuLLT,CoeuIBVP} in the discrete case, and introduce a family of parameterized curves.
	
	We recall that we introduced in \eqref{defPsi} the function $\Psi$ defined by 
	$$\forall \tau_p\in\R, \quad \Psi(\tau_p):= \tau_p-A_R{\tau_p}^{2\mu}.$$
	and that we chose $\varepsilon$ small enough so that the function $\Psi$ is continuous and strictly increasing on $]-\infty,\varepsilon]$. We can therefore introduce for $\tau_p\in[-\eta,\varepsilon]$ the curve $\Gamma_p$ defined by
	$$\Gamma_{p} := \lc \tau\in \C, -\eta\leq \Re(\tau)\leq \tau_p, \quad \Re(\tau) - A_R \Re(\tau)^{2\mu} +  A_I \Im(\tau)^{2\mu}= \Psi(\tau_p)\rc.$$
	It is a symmetric curve with respect to the axis $\R$ which intersects this axis on the point $\tau_p$. If we introduce $\ell_{p}= \left(\frac{\Psi(\tau_p)-\Psi(-\eta)}{A_I}\right)^\frac{1}{2\mu}$, then $-\eta +i\ell_{p}$ and $-\eta -i\ell_{p}$ are the end points of $\Gamma_{p}$. We can also introduce a parametrization of this curve by defining $\gamma_{p}:[-\ell_{p}, \ell_{p}]\rightarrow \C$ such that 
	\begin{equation}\label{param}
		\forall \tau_p\in\left[-\eta,\varepsilon\right], \forall t\in[-\ell_{p},\ell_{p}],\quad \Im(\gamma_{p}(t))=t, \quad \Re(\gamma_{p}(t))=h_{p}(t):=\Psi^{-1}\left(\Psi(\tau_p)-A_It^{2\mu}\right).
	\end{equation}
	
	The above parametrization immediately yields that there exists a constant $C>0$ such that 
	\begin{equation}\label{hp}
		\forall \tau_p \in[-\eta,\varepsilon], \forall t \in[-\ell_{p},\ell_{p}], \quad |h_{p}^\prime(t)|\leq C.
	\end{equation}
	Also, there exists a constant $c_p>0$ such that 
	\begin{equation}
		\forall \tau_p\in[-\eta,\varepsilon], 	\forall \tau \in \Gamma_{p}, \quad \Re(\tau)-\tau_p\leq -c_p \Im(\tau)^{2\mu}. \label{ine_Re}
	\end{equation}
	
	For $\tau\in \Gamma_{p}$, it follows from \eqref{ine_Re} that
	\begin{equation}\label{estClas}
		n\Re(\tau)+ x \left(-\Re(\tau)+A_R\Re(\tau)^{2\mu}-A_I\Im(\tau)^{2\mu}\right) \leq -nc_p \Im(\tau)^{2\mu}+ \left(n-x\right)\tau_p +xA_R\tau_p^{2\mu}. 
	\end{equation}	
	There remains to make an appropriate choice of $\tau_p$ depending on $n$ and $x$ that minimizes the right-hand side of the inequality \eqref{estClas} whilst the paths $\Gamma_{p}$ have to remain within the ball $B(0,\varepsilon)$. We recall that when we fixed our choice of width $\eta$, we defined a radius $\varepsilon_\#\in]0,\varepsilon[$ such that $-\eta+il_{extr}\in B(0,\varepsilon)$ where the real number $l_{extr}$ is defined by \eqref{defIextr}. This implies that the curve $\Gamma_{p}$ associated with $\tau_p=\varepsilon_\#$ intersects the axis $-\eta+i\R$ within $B(0,\varepsilon)$. We let 
	$$\zeta=\frac{x-n}{2\mu n}, \quad \gamma=\frac{xA_R}{n}, \quad \rho\left(\frac{\zeta}{\gamma}\right)=\mathrm{sgn}(\zeta)\left(\frac{|\zeta|}{\gamma}\right)^\frac{1}{2\mu-1}.$$
	We observe that the condition $x\geq\frac{n}{2}$ implies 
	\begin{equation}
		\gamma\geq \frac{A_R}{2}.\label{ineg_gamma}
	\end{equation}
	Coming back to inequality \eqref{estClas}, we now have
	\begin{equation}\label{estClas2}
		n\Re(\tau)+ x \left(-\Re(\tau)+A_R\Re(\tau)^{2\mu}-A_I\Im(\tau)^{2\mu}\right) \leq -nc_p \Im(\tau)^{2\mu}+n(\gamma \tau_p^{2\mu}-2\mu\zeta\tau_p).
	\end{equation}
	Our limiting estimates will come from the case where $\zeta$ is close to $0$. 
	
	Then, we take 
	$$\tau_p:=\lc\begin{array}{ccc}
		\rho\left(\frac{\zeta}{\gamma}\right), & \text{ if }\rho\left(\frac{\zeta}{\gamma}\right)\in[-\frac{\eta}{2},\varepsilon_\#],& \text{(Case \textbf{A})}\\
		\varepsilon_\#, & \text{ if }\rho\left(\frac{\zeta}{\gamma}\right)>\varepsilon_\#, &\text{(Case \textbf{B})}\\
		-\frac{\eta}{2}, & \text{ if }\rho\left(\frac{\zeta}{\gamma}\right)<-\frac{\eta}{2}.&\text{(Case \textbf{C})}\\
	\end{array}\right.$$
	The case \textbf{A} corresponds to the choice to minimize the right-hand side of \eqref{estClas2} since $\rho\left(\frac{\zeta}{\gamma}\right)$ is the unique real root of the polynomial
	$$\gamma X^{2\mu-1}=\zeta.$$
	The cases \textbf{B} and \textbf{C} allow the path $\Gamma_{p}$ to stay within $B(0,\varepsilon)$.
	
	We now define the paths:
	\begin{align*}
		\Gamma_{p,res}:=&\lc-\eta +it,\quad t\in[-r_{\varepsilon}(\eta),-\ell_{p}]\cup[\ell_{p},r_{\varepsilon}(\eta)]\rc,\\
		\Gamma_{p,in}:=&\Gamma_{p}\cup\Gamma_{p,res},
	\end{align*}
	where the function $r_\varepsilon$ is defined by \eqref{defreps}. We observe that $\Gamma_{p,in}$ belongs to the set of paths $X$. We will decompose the integral 
	$$\int_{\Gamma_{p,in}} |\tau|^k\exp(n\Re(\tau)+x\left(-\Re(\tau)+A_R\Re(\tau)^{2\mu}-A_I\Im(\tau)^{2\mu}\right)) |d\tau|$$
	using the paths $\Gamma_{p}$ and $\Gamma_{p,res}$ and we will then bound each term.
	
	$\blacktriangleright$ Let us assume that $x$ and $n$ are such that we are in Case \textbf{A}. Since $\tau_p=\rho\left(\frac{\zeta}{\gamma}\right)$ is the unique root of $\gamma X^{2\mu-1}-\zeta$, we have:
	\begin{equation}\label{egInterm:A}
		\gamma\tau_p^{2\mu}-2\mu\zeta\tau_p = -(2\mu-1)\gamma\tau_p^{2\mu} \leq0.
	\end{equation} 		
	Thus, the inequality \eqref{estClas2} becomes for $\tau\in \Gamma_p$ 
	\begin{equation*}
		n\Re(\tau)+ x \left(-\Re(\tau)+A_R\Re(\tau)^{2\mu}-A_I\Im(\tau)^{2\mu}\right) \leq -nc_p \Im(\tau)^{2\mu}-(2\mu-1)\gamma n\tau_p^{2\mu}.
	\end{equation*}
	Therefore, we have
	\begin{multline*}
		\int_{\Gamma_p} |\tau|^k\exp\left(n\Re(\tau)+ x \left(-\Re(\tau)+A_R\Re(\tau)^{2\mu}-A_I\Im(\tau)^{2\mu}\right)\right)|d\tau| \\ \leq \int_{\Gamma_p} |\tau|^k\exp\left(-nc_p\Im(\tau)^{2\mu}\right)|d\tau|\exp\left(-(2\mu-1)\gamma n\tau_p^{2\mu}\right).
	\end{multline*}
	Using the parametrization \eqref{param} and the inequality \eqref{hp}, we have that
	$$\int_{\Gamma_p} |\tau|^k\exp\left(-nc_p\Im(\tau)^{2\mu}\right)|d\tau|\lesssim \int_{-\ell_p}^{-\ell_p} (|\tau_p|^k+t^k)e^{-nc_p t^{2\mu}}dt.$$
	The change of variables $u=n^\frac{1}{2\mu}t$ and the fact that the functions $y\geq0\mapsto y^k\exp\left(-\frac{2\mu-1}{2}\gamma y^{2\mu}\right)$ are uniformly bounded with respect to $\gamma\geq \frac{A_R}{2}$ imply
	\begin{align*}
		\int_{-\ell_p}^{-\ell_p} |t|^ke^{-nc_p t^{2\mu}}dt& \lesssim \frac{1}{n^\frac{k+1}{2\mu}}\\
		\int_{-\ell_p}^{-\ell_p} |\tau_p|^ke^{-nc_p t^{2\mu}}dt& \lesssim \frac{1}{n^\frac{k+1}{2\mu}}\exp\left(\frac{2\mu-1}{2}\gamma n \tau_p^{2\mu}\right).
	\end{align*}
	Thus,
	$$ \int_{\Gamma_p} |\tau|^k\exp\left(n\Re(\tau)+ x \left(-\Re(\tau)+A_R\Re(\tau)^{2\mu}-A_I\Im(\tau)^{2\mu}\right)\right)|d\tau| \lesssim\frac{1}{n^\frac{k+1}{2\mu}}\exp\left(-\frac{2\mu-1}{2}\gamma n \tau_p^{2\mu}\right) .$$
	Furthermore, since we are in the Case \textbf{A}
	$$-\frac{2\mu-1}{2}\gamma n \tau_p^{2\mu}= -\frac{2\mu-1}{2(2\mu A_R)^\frac{2\mu}{2\mu-1}}A_R\left(\frac{|n-x|}{x^\frac{1}{2\mu}}\right)^\frac{2\mu}{2\mu-1}.$$
	Therefore, there exist two positive constants $C,c$ independent from $n$ and $x$ such that if we are in Case \textbf{A}, 
	$$\int_{\Gamma_p} |\tau|^k\exp\left(n\Re(\tau)+x\left(-\Re(\tau)+A_R\Re(\tau)^{2\mu}-A_I\Im(\tau)^{2\mu}\right)\right) |d\tau| \leq \frac{C}{n^\frac{k+1}{2\mu}} \exp\left(-c\left(\frac{|n-x|}{x^\frac{1}{2\mu}}\right)^\frac{2\mu}{2\mu-1}\right).$$
	Since $x\in[\frac{n}{2},2n]$, this gives us two new constants $C,c$ independent from $n$ and $x$ such that if we are in Case \textbf{A} and $x\in[\frac{n}{2},2n]$, then as expected
	$$\int_{\Gamma_p} |\tau|^k\exp\left(n\Re(\tau)+x\left(-\Re(\tau)+A_R\Re(\tau)^{2\mu}-A_I\Im(\tau)^{2\mu}\right)\right) |d\tau| \leq \frac{C}{n^\frac{k+1}{2\mu}} \exp\left(-c\left(\frac{|n-x|}{n^\frac{1}{2\mu}}\right)^\frac{2\mu}{2\mu-1}\right).$$
	
	$\blacktriangleright$ Let us assume that $x$ and $n$ are such that we are in Case \textbf{B}. Since $\tau_p=\varepsilon_\#<\rho\left(\frac{\zeta}{\gamma}\right)$, we have 
	$$-\zeta\leq -\gamma \varepsilon_\#^{2\mu-1}$$
	and thus using \eqref{ineg_gamma}
	\begin{equation}\label{egInterm:B}
		\gamma\tau_p^{2\mu}-2\mu\zeta\tau_p \leq - (2\mu-1) \gamma\varepsilon_\#^{2\mu}\leq - \frac{2\mu-1}{2}A_R\varepsilon_\#^{2\mu}.
	\end{equation} 		
	Therefore, the inequality \eqref{estClas2} becomes for $\tau\in \Gamma_p$
	\begin{align*}
		n\Re(\tau)+ x \left(-\Re(\tau)+A_R\Re(\tau)^{2\mu}-A_I\Im(\tau)^{2\mu}\right)& \leq -nc_p \Im(\tau)^{2\mu}-\frac{2\mu-1}{2}A_Rn  \varepsilon_\#^{2\mu}\\
		&\leq -\frac{2\mu-1}{2}\varepsilon_\#^{2\mu}A_R n.
	\end{align*}
	We conclude that there exist two positive constants $C,c$ independent from $n$ and $x$ such that if we are in Case \textbf{B}, 
	$$\int_{\Gamma_p} |\tau|^k\exp\left(n\Re(\tau)+x\left(-\Re(\tau)+A_R\Re(\tau)^{2\mu}-A_I\Im(\tau)^{2\mu}\right)\right) |d\tau| \leq C e^{-cn}.$$
	Using \eqref{Expo-Gaussian} if necessary, we obtain the bound expected in the statement of the lemma.
	
	$\blacktriangleright$ Let us assume that $x$ and $n$ are such that we are in Case \textbf{C}. Since $\tau_p=-\frac{\eta}{2}>\rho\left(\frac{\zeta}{\gamma}\right)$, we have 
	$$\zeta\leq -\gamma\left(\frac{\eta}{2}\right)^{2\mu-1} $$
	and thus using \eqref{ineg_gamma}
	\begin{equation}\label{egInterm:C}
		\gamma\tau_p^{2\mu}-2\mu\zeta\tau_p=\gamma\left(\frac{\eta}{2}\right)^{2\mu}+2\mu\zeta\frac{\eta}{2}\leq - (2\mu-1) \gamma\left(\frac{\eta}{2}\right)^{2\mu}\leq -\frac{2\mu-1}{2}A_R\left(\frac{\eta}{2}\right)^{2\mu}.
	\end{equation} 		
	Therefore, the inequality \eqref{estClas2} becomes for $\tau\in \Gamma_p$
	\begin{align*}
		n\Re(\tau)+ x \left(-\Re(\tau)+A_R\Re(\tau)^{2\mu}-A_I\Im(\tau)^{2\mu}\right) & \leq -nc_p \Im(\tau)^{2\mu}-\frac{2\mu-1}{2}A_Rn \left(\frac{\eta}{2}\right)^{2\mu}\\
		& \leq-\frac{2\mu-1}{2} \left(\frac{\eta}{2}\right)^{2\mu}A_R n .
	\end{align*}
	We then conclude that there exist two positive constants $C,c$ independent from $n$ and $x$ such that if we are in Case \textbf{C}, 
	$$\int_{\Gamma_p} |\tau|^k\exp\left(n\Re(\tau)+x\left(-\Re(\tau)+A_R\Re(\tau)^{2\mu}-A_I\Im(\tau)^{2\mu}\right)\right) |d\tau| \leq C e^{-cn}.$$
	Using \eqref{Expo-Gaussian} if necessary, we obtain the bound expected in the statement of the lemma.
	
	$\blacktriangleright$ We recall that $-\eta \pm \ell_p$ belongs to $\Gamma_p$. For $\tau\in \Gamma_{p,res}$, we have that
	$$\Re(\tau)=-\eta \quad \text{ and } \quad |\Im(\tau)|\geq\ell_p.$$
	Thus,
	\begin{align*}
		n \Re(\tau) + x(-\Re(\tau) +A_R\Re(\tau)^{2\mu} - A_I\Im(\tau)^{2\mu}) & \leq -n\eta +x(\eta +A_R \eta^{2\mu} -A_I \ell_p^{2\mu})\\
		& \leq -n\eta +x (-\tau_p +A_R \tau_p^{2\mu})\\
		& \leq -n(\eta+\tau_p) + n(\gamma\tau_p^{2\mu} -2\mu\zeta\tau_p).
	\end{align*}
	In each cases \textbf{A}, \textbf{B} and \textbf{C}, we have that 
	$$\eta+\tau_p\geq \frac{\eta}{2} \quad \text{ and } \quad \gamma\tau_p^{2\mu} -2\mu\zeta\tau_p\leq 0.$$
	Therefore, for all $\tau\in\Gamma_{p,res}$,
	$$n \Re(\tau) + x(-\Re(\tau) +A_R\Re(\tau)^{2\mu} - A_I\Im(\tau)^{2\mu}) \leq -n\frac{\eta}{2}.$$
	We then conclude that
	$$\int_{\Gamma_{p,res}} |\tau|^k\exp\left(n\Re(\tau)+x\left(-\Re(\tau)+A_R\Re(\tau)^{2\mu}-A_I\Im(\tau)^{2\mu}\right)\right) |d\tau| \leq 2\pi\varepsilon^k e^{-n\frac{\eta}{2}}.$$
	Using \eqref{Expo-Gaussian} if necessary, we obtain the bound expected in the statement of the lemma.
	
	Combining all the results we encountered, we easily conclude the proof of \eqref{lem:BornesGaussiennesRes3}.
	
	\subsection*{Proof of Lemma \ref{lem:ComportementPrincGaussienne}}
	
	We will prove the statement of Lemma \ref{lem:ComportementPrincGaussienne} with $s=s^\prime=+$ in order to alleviate the notations.
	
	$\bullet$ \textbf{Proof of \eqref{lem:ComportementPrincGaussienneRes1}:} 
	
	$\blacktriangleright$ We start by defining the paths
	$$\Gamma_0:=\lc it, t\in[-r_\varepsilon(\eta),r_\varepsilon(\eta)]\rc \quad \text{ and } \quad \Gamma_{comp}^\pm:= \lc t \pm i r_\varepsilon(\eta), t\in[-\eta,0]\rc.$$
	For all $\Gamma\in X$, Cauchy's formula implies that
	\begin{align*}
		&\left|\int_\Gamma \exp\left(n\tau+x\alpha_l^+\varphi_l^+(\tau)+y\alpha_{l^\prime}^+\varphi_{l^\prime}^+(\tau)\right)d\tau - \int_{\Gamma_0} \exp\left(n\tau+x\alpha_l^+\varphi_l^+(\tau)+y\alpha_{l^\prime}^+\varphi_{l^\prime}^+(\tau)\right)d\tau\right|\\
		\leq& \left|\int_{\Gamma_{comp}^+} \exp\left(n\tau+x\alpha_l^+\varphi_l^+(\tau)+y\alpha_{l^\prime}^+\varphi_{l^\prime}^+(\tau)\right)d\tau\right|+\left|\int_{\Gamma_{comp}^-} \exp\left(n\tau+x\alpha_l^+\varphi_l^+(\tau)+y\alpha_{l^\prime}^+\varphi_{l^\prime}^+(\tau)\right)d\tau\right|.
	\end{align*}
	We observe that \eqref{in:varphi} implies that
	\begin{multline*}
		\left|\int_{\Gamma_{comp}^\pm} \exp\left(n\tau+x\alpha_l^+\varphi_l^+(\tau)+y\alpha_{l^\prime}^+\varphi_{l^\prime}^+(\tau)\right)d\tau\right| \\ \leq \int_{-\eta}^0 \exp\left((n-(x+y))t + (x+y) A_Rt^{2\mu} - (x+y) A_Ir_\varepsilon(\eta)^{2\mu}\right)dt.
	\end{multline*}
	Using \eqref{condEta2Reformulee} since $t\in[-\eta,0]$ and $x+y\in\left[\frac{n}{2},2n\right]$, we have that
	\begin{equation*}
		(n-(x+y))t + (x+y) A_Rt^{2\mu} - (x+y) A_Ir_\varepsilon(\eta)^{2\mu}\leq -\frac{A_Ir_\varepsilon(\eta)^{2\mu}}{4}n.
	\end{equation*}
	Combining the observations above, we have thus proved that for all path $\Gamma\in X$, $n\in\N\backslash\lc0\rc$, $x,y\in[0,+\infty[$ such that $x+y\in\left[\frac{n}{2},2n\right]$
	\begin{multline}\label{lem:ComportementPrincGaussienne1}
		\left|\int_\Gamma \exp\left(n\tau+x\alpha_l^+\varphi_l^+(\tau)+y\alpha_{l^\prime}^+\varphi_{l^\prime}^+(\tau)\right)d\tau - \int_{\Gamma_0} \exp\left(n\tau+x\alpha_l^+\varphi_l^+(\tau)+y\alpha_{l^\prime}^+\varphi_{l^\prime}^+(\tau)\right)d\tau\right| \\ \leq 2\eta \exp\left(-\frac{A_Ir_\varepsilon(\eta)^{2\mu}}{4}n\right).
	\end{multline}
	
	$\blacktriangleright$ Since $x+y\geq \frac{n}{2}\geq \frac{1}{2}$, we observe that 
	\begin{equation*}
		\int_{r_\varepsilon(\eta)}^{+\infty} \exp\left(-(x+y) A_I t^{2\mu}\right)dt \\ \leq \exp\left(- \frac{A_Ir_\varepsilon(\eta)^{2\mu}}{4} n\right)\int_{r_\varepsilon(\eta)}^{+\infty} \exp\left(- \frac{A_I t^{2\mu}}{4}\right)dt  .
	\end{equation*}
	Therefore, if we introduce the path 
	$$\Gamma_0^\infty:=\lc it, t\in\R\rc$$
	then, using \eqref{in:varphi}, the integral 
	$$\int_{\Gamma_0^\infty} \exp\left(n\tau+x\alpha_l^+\varphi_l^+(\tau)+y\alpha_{l^\prime}^+\varphi_{l^\prime}^+(\tau)\right)d\tau$$
	is defined and we have that
	\begin{multline}\label{lem:ComportementPrincGaussienne2}
		\left|\int_{\Gamma_0} \exp\left(n\tau+x\alpha_l^+\varphi_l^+(\tau)+y\alpha_{l^\prime}^+\varphi_{l^\prime}^+(\tau)\right)d\tau-\int_{\Gamma_0^\infty} \exp\left(n\tau+x\alpha_l^+\varphi_l^+(\tau)+y\alpha_{l^\prime}^+\varphi_{l^\prime}^+(\tau)\right)d\tau\right| \\ \leq 2\int_{r_\varepsilon(\eta)}^{+\infty} \exp\left(- \frac{A_I t^{2\mu}}{4}\right)dt \exp\left(- \frac{A_Ir_\varepsilon(\eta)^{2\mu}}{4} n\right).
	\end{multline}
	
	$\blacktriangleright$ Using the change of variables $u=\frac{n^\frac{1}{2\mu}}{\alpha_l^+}t$, we obtain
	\begin{align*}
		\frac{1}{2i\pi} \int_{\Gamma_0^\infty} \exp\left(n\tau+x\alpha_l^+\varphi_l^+(\tau)+y\alpha_{l^\prime}^+\varphi_{l^\prime}^+(\tau)\right)d\tau &= \frac{1}{2\pi}\int_{-\infty}^{+\infty}\exp\left(it(n-(x+y))-\left(\frac{x\beta_l^+}{{\alpha_l^+}^{2\mu}}+\frac{y\beta_{l^\prime}^+}{{\alpha_{l^\prime}^+}^{2\mu}}\right)t^{2\mu}\right)dt\\
		& = \frac{|\alpha_l^+|}{n^\frac{1}{2\mu}} H_{2\mu}\left(\frac{x}{n}\beta_l^++\frac{y}{n}\beta_{l^\prime}^+\left(\frac{\alpha_l^+}{\alpha_{l^\prime}^+}\right)^{2\mu};\frac{\alpha_l^+(n-(x+y))}{n^\frac{1}{2\mu}}\right).
	\end{align*}
	Combining \eqref{lem:ComportementPrincGaussienne1}, \eqref{lem:ComportementPrincGaussienne2} and the observation above, we obtain the inequality \eqref{lem:ComportementPrincGaussienneRes1}.
	
	$\bullet$ \textbf{Proof of \eqref{lem:ComportementPrincGaussienneRes2}:}
	
	We observe using the change of variables $u=\frac{x^\frac{1}{2\mu}}{\alpha_l^+}t$ that 
	\begin{align*}
		\frac{1}{2i\pi} \int_{\Gamma_0^\infty} \exp\left(n\tau+x\alpha_l^+\varphi_l^+(\tau)\right)d\tau &= \frac{1}{2\pi}\int_{-\infty}^{+\infty}\exp\left(it(n-x)-\frac{x\beta_l^+}{{\alpha_l^+}^{2\mu}}t^{2\mu}\right)dt\\
		& = \frac{|\alpha_l^+|}{x^\frac{1}{2\mu}} H_{2\mu}\left(\beta_l^+;\frac{\alpha_l^+(n-x)}{x^\frac{1}{2\mu}}\right).
	\end{align*}
	Combining \eqref{lem:ComportementPrincGaussienne1}, \eqref{lem:ComportementPrincGaussienne2} and the observation above, we have proved that there exist two constants $C,c>0$ such that 
	\begin{equation}\label{lem:ComportementPrincGaussienne3}
		\forall x\in \left[\frac{n}{2},2n\right] ,\quad \left|\frac{1}{2i\pi}\int_\Gamma \exp\left(n\tau+x\alpha_l^+\varphi_l^+(\tau)\right)d\tau-\frac{|\alpha_l^+|}{x^\frac{1}{2\mu}} H_{2\mu}\left(\beta_l^+;\frac{\alpha_l^+(n-x)}{x^\frac{1}{2\mu}}\right)\right|\leq Ce^{-cn}.
	\end{equation}
	Using \eqref{Expo-Gaussian}, we can obtain the same generalized Gaussian bound as the one expected in \eqref{lem:ComportementPrincGaussienneRes2}.
	
	$\blacktriangleright$ We observe that
	\begin{multline}\label{lem:ComportementPrincGaussienne4}
		\frac{1}{x^\frac{1}{2\mu}} H_{2\mu}\left(\beta_l^+;\frac{\alpha_l^+(n-x)}{x^\frac{1}{2\mu}}\right) -\frac{1}{n^\frac{1}{2\mu}} H_{2\mu}\left(\beta_l^+;\frac{\alpha_l^+(n-x)}{n^\frac{1}{2\mu}}\right) \\= \frac{1}{x^\frac{1}{2\mu}} \left(H_{2\mu}\left(\beta_l^+;\frac{\alpha_l^+(n-x)}{x^\frac{1}{2\mu}}\right) - H_{2\mu}\left(\beta_l^+;\frac{\alpha_l^+(n-x)}{n^\frac{1}{2\mu}}\right)\right) + H_{2\mu}\left(\beta_l^+;\frac{\alpha_l^+(n-x)}{n^\frac{1}{2\mu}}\right)\left(\frac{1}{x^\frac{1}{2\mu}}-\frac{1}{n^\frac{1}{2\mu}}\right).
	\end{multline}
	We want to prove generalized Gaussian bounds for the two terms on the right hand side of \eqref{lem:ComportementPrincGaussienne4}. Applying the mean value inequality and \eqref{inH}, we have that there exist two constants $C,c>0$ such that for all $x\in[\frac{n}{2},2n]$ 
	$$\left|\frac{1}{x^\frac{1}{2\mu}} \left(H_{2\mu}\left(\beta_l^+;\frac{\alpha_l^+(n-x)}{x^\frac{1}{2\mu}}\right) - H_{2\mu}\left(\beta_l^+;\frac{\alpha_l^+(n-x)}{n^\frac{1}{2\mu}}\right)\right)\right|\leq \frac{C}{n^\frac{1}{2\mu}} |n-x| \left|\frac{1}{x^\frac{1}{2\mu}}-\frac{1}{n^\frac{1}{2\mu}}\right| \exp\left(-c\left(\frac{|n-x|}{n^\frac{1}{2\mu}}\right)^\frac{2\mu}{2\mu-1}\right).$$
	Since $x\in\left[\frac{n}{2},2n\right]$, we also have using the mean value inequality that
	\begin{equation}\label{lem:ComportementPrincGaussienne5}
		\left|\frac{1}{x^\frac{1}{2\mu}}-\frac{1}{n^\frac{1}{2\mu}}\right| \leq \frac{|n-x|}{2\mu}\sup_{t\in[x,n]}\frac{1}{|t|^{1+\frac{1}{2\mu}}}\leq \frac{2^\frac{1}{2\mu}}{\mu}\frac{|n-x|}{ n^{1+\frac{1}{2\mu}}}.
	\end{equation}
	Therefore, since the function $y\mapsto y^2\exp\left(-\frac{c}{2}y^\frac{2\mu}{2\mu-1}\right)$ is bounded, there exist two new constants $C,c>0$ such that for all $x\in[\frac{n}{2},2n]$ 
	\begin{equation}\label{lem:ComportementPrincGaussienne6}
		\left|\frac{1}{x^\frac{1}{2\mu}} \left(H_{2\mu}\left(\beta_l^+;\frac{\alpha_l^+(n-x)}{x^\frac{1}{2\mu}}\right) - H_{2\mu}\left(\beta_l^+;\frac{\alpha_l^+(n-x)}{n^\frac{1}{2\mu}}\right)\right)\right|\leq \frac{C}{n} \exp\left(-c\left(\frac{|n-x|}{n^\frac{1}{2\mu}}\right)^\frac{2\mu}{2\mu-1}\right).
	\end{equation}
	
	We have thus proved generalized Gaussian bounds for the first term of the right hand side in \eqref{lem:ComportementPrincGaussienne4}. We now focus on the second term. Using \eqref{inH}, \eqref{lem:ComportementPrincGaussienne5} and the fact that, for any constant $c>0$, the function $y\mapsto y\exp\left(-cy^\frac{2\mu}{2\mu-1}\right)$ is bounded, we have that there exist two constants $C,c>0$ such that
	\begin{equation}\label{lem:ComportementPrincGaussienne7}
		\left|H_{2\mu}\left(\beta_l^+;\frac{\alpha_l^+(n-x)}{n^\frac{1}{2\mu}}\right)\left(\frac{1}{x^\frac{1}{2\mu}}-\frac{1}{n^\frac{1}{2\mu}}\right)\right| \leq \frac{C}{n}\exp\left(-c\left(\frac{|n-x|}{n^\frac{1}{2\mu}}\right)^\frac{2\mu}{2\mu-1}\right).
	\end{equation}
	
	Combining \eqref{lem:ComportementPrincGaussienne4}, \eqref{lem:ComportementPrincGaussienne6} and \eqref{lem:ComportementPrincGaussienne7}, we have proved generalized Gaussian bounds for the difference 
	$$ \frac{1}{x^\frac{1}{2\mu}} H_{2\mu}\left(\beta_l^+;\frac{\alpha_l^+(n-x)}{x^\frac{1}{2\mu}}\right) -\frac{1}{n^\frac{1}{2\mu}} H_{2\mu}\left(\beta_l^+;\frac{\alpha_l^+(n-x)}{n^\frac{1}{2\mu}}\right) .$$
	With \eqref{lem:ComportementPrincGaussienne3}, we easily conclude the proof of \eqref{lem:ComportementPrincGaussienneRes2}.
	
	$\bullet$ \textbf{Proof of \eqref{lem:ComportementPrincGaussienneRes3}:}
	
	$\blacktriangleright$ We observe that \eqref{in:varphi} implies that
	\begin{equation*}
		\left|\int_{\Gamma_{in}^\pm(\eta)} \frac{\exp\left(n\tau+x\alpha_l^+\varphi_l^+(\tau)\right)}{\tau}d\tau\right| \leq \frac{1}{r_\varepsilon(\eta)}\int_{-\eta}^\eta \exp\left((n-x)t + x A_Rt^{2\mu} - x A_Ir_\varepsilon(\eta)^{2\mu}\right)dt.
	\end{equation*}
	Using \eqref{condEta2Reformulee} since $t\in[-\eta,\eta]$ and $x\in\left[\frac{n}{2},2n\right]$, we have that
	\begin{equation*}
		(n-x)t + x A_Rt^{2\mu} - x A_Ir_\varepsilon(\eta)^{2\mu}\leq -\frac{A_Ir_\varepsilon(\eta)^{2\mu}}{4}n.
	\end{equation*}
	Using the observations above and Cauchy's formula, we have thus proved that for all $n\in\N\backslash\lc0\rc$, $x\in\left[\frac{n}{2},2n\right]$ and paths $\Gamma\in X$
	\begin{equation}\label{lem:ComportementPrincGaussienne:Res3:1}
		\left|\int_{\Gamma_{in}(\eta)} \frac{\exp\left(n\tau+x\alpha_l^+\varphi_l^+(\tau)\right)}{\tau}d\tau-\int_{\Gamma_{in}^0(\eta)} \frac{\exp\left(n\tau+x\alpha_l^+\varphi_l^+(\tau)\right)}{\tau}d\tau\right| \leq \frac{4\eta}{r_\varepsilon(\eta)} \exp\left(-\frac{A_Ir_\varepsilon(\eta)^{2\mu}}{4}n\right).
	\end{equation}
	
	$\blacktriangleright$ Since $x\geq \frac{n}{2}\geq \frac{1}{2}$, we observe that 
	\begin{multline*}
		\int_{r_\varepsilon(\eta)}^{+\infty} \frac{\exp\left((n-x)\eta +xA_R\eta^{2\mu}-xA_It^{2\mu}\right)}{|\eta+it|}dt \\ \\ \leq  \frac{1}{\eta}\underset{<+\infty}{\underbrace{\int_{r_\varepsilon(\eta)}^{+\infty} \exp\left(- \frac{A_I t^{2\mu}}{8}\right)dt}} \exp\left((n-x)\eta +xA_R\eta^{2\mu}-x\frac{3}{4}A_Ir_\varepsilon(\eta)^{2\mu}\right) .
	\end{multline*}
	Furthermore, using \eqref{condEta2Reformulee} since $x\in\left[\frac{n}{2},2n\right]$, we have that
	$$ \exp\left((n-x)\eta +xA_R\eta^{2\mu}-x\frac{3}{4}A_Ir_\varepsilon(\eta)^{2\mu}\right) \leq \exp\left(-\frac{A_Ir_\varepsilon(\eta)^{2\mu}}{8}n\right) .$$
	Therefore, if we introduce the path 
	$$\Gamma_{in}^\infty(\eta):=\lc \eta+ it, t\in\R\rc$$
	then, using \eqref{in:varphi}, the integral 
	$$\int_{\Gamma_{in}^\infty(\eta)} \frac{\exp\left(n\tau+x\alpha_l^+\varphi_l^+(\tau)\right)}{\tau}d\tau$$
	is defined and we have that
	\begin{multline}\label{lem:ComportementPrincGaussienne:Res3:2}
		\left|\int_{\Gamma_{in}^0(\eta)} \frac{\exp\left(n\tau+x\alpha_l^+\varphi_l^+(\tau)\right)}{\tau}d\tau-\int_{\Gamma_{in}^\infty(\eta)} \frac{\exp\left(n\tau+x\alpha_l^+\varphi_l^+(\tau)\right)}{\tau}d\tau\right| \\ \leq \frac{2}{\eta}\int_{r_\varepsilon(\eta)}^{+\infty} \exp\left(- \frac{A_I t^{2\mu}}{8}\right)dt \exp\left(-\frac{A_Ir_\varepsilon(\eta)^{2\mu}}{8}n\right) .
	\end{multline}
	
	$\blacktriangleright$ We observe that using the change of variables $t= -\frac{|\alpha_l^+|}{x^\frac{1}{2\mu}}u$, we have
	\begin{align*}
		\frac{1}{2i\pi} \int_{\Gamma_{in}^\infty(\eta)} \frac{\exp\left(n\tau+x\alpha_l^+\varphi_l^+(\tau)\right)}{\tau}d\tau & = \frac{1}{2i\pi} \int_{-\infty}^{+\infty} \frac{\exp\left(i(n-x)(t-i\eta) -x\frac{\beta^+_l}{{\alpha_l^+}^{2\mu}}(\eta+it)^{2\mu}\right)}{t-i\eta}dt\\
		& = - \frac{1}{2i\pi} \int_{-\infty}^{+\infty} \frac{\exp\left(i\frac{-|\alpha_l^+|(n-x)}{x^\frac{1}{2\mu}}\left(u+i\frac{x^\frac{1}{2\mu}\eta}{|\alpha_l^+|}\right) -\beta^+_l\left(u+i\frac{x^\frac{1}{2\mu}\eta}{|\alpha_l^+|}\right) ^{2\mu}\right)}{u+i\frac{x^\frac{1}{2\mu}\eta}{|\alpha_l^+|}}du.
	\end{align*}
	Furthermore, we can prove that 
	\begin{equation*}
		\forall s\in]0,+\infty[,\forall x\in\R,\quad -\frac{1}{2i\pi}  \int_{-\infty}^{+\infty} \frac{\exp\left(ix\left(u+is\right) -\beta^+_l\left(u+is\right) ^{2\mu}\right)}{u+is}du = E_{2\mu}\left(\beta^+_l;x\right).
	\end{equation*}
	The proof is done in \cite[(5.65)]{CoeuIBVP}. Therefore, 
	$$\frac{1}{2i\pi} \int_{\Gamma_{in}^\infty(\eta)} \frac{\exp\left(n\tau+x\alpha_l^+\varphi_l^+(\tau)\right)}{\tau}d\tau = E_{2\mu}\left(\beta^+_l;\frac{-|\alpha_l^+|(n-x)}{x^\frac{1}{2\mu}}\right).$$
	Combining this observation with \eqref{lem:ComportementPrincGaussienne:Res3:1}, \eqref{lem:ComportementPrincGaussienne:Res3:2} and \eqref{Expo-Gaussian}, we have that there exist two positive constants $C,c$ such that for all $n\in \N\backslash\lc0\rc$ and $x\in \left[\frac{n}{2},2n\right]$
	$$ \left|\frac{1}{2i\pi} \int_{\Gamma_{in}(\eta)} \frac{\exp\left(n\tau+x\alpha_l^+\varphi_l^+(\tau)\right)}{\tau}d\tau - E_{2\mu}\left(\beta^+_l;\frac{-|\alpha_l^+|(n-x)}{x^\frac{1}{2\mu}}\right)\right|\leq \frac{C}{n^\frac{1}{2\mu}}\exp\left(-c\left(\frac{|n-x|}{n^\frac{1}{2\mu}}\right)^\frac{2\mu}{2\mu-1}\right).$$
	
	$\blacktriangleright$ We notice that $\partial_x E_{2\mu}\left(\beta_l^+;\cdot\right) = -H_{2\mu}\left(\beta_l^+;\cdot\right)$. Therefore, we have using the mean value inequality and \eqref{inH} that there exist two positive constants $C,c$ such that for all $n\in\N\backslash\lc0\rc$ and $x\in\left[\frac{n}{2},2n\right]$
	$$\left|E_{2\mu}\left(\beta^+_l;\frac{-|\alpha_l^+|(n-x)}{x^\frac{1}{2\mu}}\right)-E_{2\mu}\left(\beta^+_l;\frac{-|\alpha_l^+|(n-x)}{n^\frac{1}{2\mu}}\right)\right| \leq C |n-x| \left|\frac{1}{x^\frac{1}{2\mu}}-\frac{1}{n^\frac{1}{2\mu}}\right| \exp\left(-c\left(\frac{|n-x|}{n^\frac{1}{2\mu}}\right)^\frac{2\mu}{2\mu-1}\right).$$ 
	Using \eqref{lem:ComportementPrincGaussienne5} and the fact that $y\mapsto y^2\exp\left(-\frac{c}{2}y^\frac{2\mu}{2\mu-1}\right)$ is bounded, we have that there exist two new positive constants $C,c$ such that for all $n\in\N\backslash\lc0\rc$ and $x\in\left[\frac{n}{2},2n\right]$
	$$\left|E_{2\mu}\left(\beta^+_l;\frac{-|\alpha_l^+|(n-x)}{x^\frac{1}{2\mu}}\right)-E_{2\mu}\left(\beta^+_l;\frac{-|\alpha_l^+|(n-x)}{n^\frac{1}{2\mu}}\right)\right| \leq \frac{C}{n^{1-\frac{1}{2\mu}}}\exp\left(-\frac{c}{2}\left(\frac{|n-x|}{n^\frac{1}{2\mu}}\right)^\frac{2\mu}{2\mu-1}\right) . $$
	This allows us to conclude the proof of \eqref{lem:ComportementPrincGaussienneRes3}.

		\subsection*{Proof of the limits \eqref{limites_termes_mass} in Section \ref{subsec:ConcluTh1}}
		
		We recall that the linear operator $\widehat{\Pi}$ is the projector on $\mathrm{Span}\left(\sum_{j\in\Z}V(j)\right)$ along the vector space spanned by $\rg_1^-,\hdots,\rg_{I-1}^-,\rg_{I+1}^+,\hdots,\rg_d^+$.
		
		\textbf{$\bullet$ Proof of \eqref{limites_termes_mass1}:}
		
		Let us consider $l^\prime\in\lbrace1,\hdots,I\rbrace$, $j_0\in\Z$ and $\textbf{e}\in\C^d$. Since the sequence $V$ decays exponentially at $\pm\infty$ fast (see \eqref{decExpoV}) and the linear operator $\widehat{\Pi}$ is continuous, we have that:
		\begin{equation}\label{limSumV}
			\widehat{\Pi}\left(\sum_{j=j_0-nq}^{j_0+np}V(j)\right) \underset{n\rightarrow+\infty}\rightarrow \sum_{j\in\Z}V(j).
		\end{equation}
		We also observe that $\alpha_{l^\prime}^+=\nug \lambg_{l^\prime}^+<0$ and thus using the limit \eqref{eq:E_en_-infty} of the function $E_{2\mu}$ and the definition \eqref{def:E+} of the matrix $E_{l^\prime}^\pm(n,j_0)$, we have: 
		\begin{equation}\label{limEl+}
			E_{l^\prime}^\pm(n,j_0) = E_{2\mu}\left(\beta_{l^\prime}^+; \frac{n\alpha^+_{l^\prime}+j_0}{n^\frac{1}{2\mu}}\right){\lg^+_{l^\prime}}^T\underset{n\rightarrow+\infty}{\rightarrow} {\lg^+_{l^\prime}}^T.
		\end{equation}
		Thus, combining \eqref{limSumV} and \eqref{limEl+}, we obtain \eqref{limites_termes_mass1}:
		$$C^{E,+}_{l^\prime} E_{l^\prime}^\pm(n,j_0) \textbf{e} \widehat{\Pi}\left(\sum_{j=j_0-nq}^{j_0+np}V(j)\right) \underset{n\rightarrow+\infty}{\rightarrow} C^{E,+}_{l^\prime} {\lg^+_{l^\prime}}^T \textbf{e}\sum_{j\in\Z}V(j).$$
		
		\textbf{$\bullet$ Proof of \eqref{limites_termes_mass2}:}
		
		Using \eqref{limSumV}, we immediately obtain \eqref{limites_termes_mass2}:
		$$P_U(j_0)\textbf{e}\widehat{\Pi}\left(\sum^{j_0+np}_{j=j_0-nq}V(j)\right)\underset{n\rightarrow+\infty}\rightarrow P_U(j_0)\textbf{e}\sum_{j\in\Z}V(j).$$
		
		\textbf{$\bullet$ Proof of \eqref{limites_termes_mass3}:}
		
		We consider $l\in\lbrace1,\hdots,I\rbrace$, $j_0\in\N$ and $\textbf{e}\in\C^d$. Using the definition \eqref{def:S+} of the matrix $S_l^+(n,j_0,j)$ and the inequality \eqref{inH} of the function $H_{2\mu}$, there exists a constant $c$ such that for all $n\in\N\backslash\lbrace0\rbrace$ such that:
		\begin{equation}\label{limGauss}
			\sum_{j=j_0-nq}^{j_0+np} \widehat{\Pi}\left(S^+_{l}(n,j_0,j)\textbf{e}\right) =\sum_{j=j_0-nq}^{j_0+np} \ind_{j\geq0} O\left(\frac{1}{n^\frac{1}{2\mu}}\exp\left(-c\left(\frac{\left|n-\left(\frac{j-j_0}{\alpha^+_l}\right)\right|}{n^\frac{1}{2\mu}}\right)^\frac{2\mu}{2\mu-1}\right)\right). 
		\end{equation}
		Since $\alpha^+_l=\nug\lambg^+_l<0$, we have that:
		$$\forall n\in\N\backslash\lbrace0\rbrace,\forall j_0\in\N,\forall j\in\N,\quad n\geq -\frac{2j_0}{\alpha^+_l} \; \Rightarrow \; \frac{n}{2}\leq n-\left(\frac{j-j_0}{\alpha^+_l}\right).  $$
		Therefore, for $n\in\N\backslash\lbrace0\rbrace$ large enough such that $n\geq -\frac{2j_0}{\alpha^+_l}$, we have:
		\begin{align}\label{limGauss2}
			\begin{split}
				\sum_{j=j_0-nq}^{j_0+np} \ind_{j\geq0} O\left(\frac{1}{n^\frac{1}{2\mu}}\exp\left(-c\left(\frac{\left|n-\left(\frac{j-j_0}{\alpha^+_l}\right)\right|}{n^\frac{1}{2\mu}}\right)^\frac{2\mu}{2\mu-1}\right)\right)  & = \sum_{j=j_0-nq}^{j_0+np} \ind_{j\geq0} O\left(\frac{1}{n^\frac{1}{2\mu}}\exp\left(-\frac{c}{2^\frac{2\mu}{2\mu-1}}n\right)\right) \\&= O\left(n^\frac{2\mu-1}{2\mu}\exp\left(-\frac{c}{2^\frac{2\mu}{2\mu-1}}n\right)\right).
			\end{split}
		\end{align}
		Thus, combining \eqref{limGauss} and \eqref{limGauss2}, we obtain \eqref{limites_termes_mass3}:
		$$\sum_{j=j_0-nq}^{j_0+np} \widehat{\Pi}\left(S^+_{l}(n,j_0,j)\textbf{e}\right) \underset{n\rightarrow+\infty}\rightarrow 0.$$
		
		\textbf{$\bullet$ Proof of \eqref{limites_termes_mass4}:}
		
		We consider $j_0\in\N$ and $\textbf{e}\in\C^d$. Using the expressions \eqref{def:ResTh} of the fast decaying residual $\Rc(n,j_0,j)$ and applying the projector $\widehat{\Pi}$, there exists a constant $c>0$ such that for $n\in\N\backslash\lbrace0\rbrace$ and $j\in\Z$:
		
		\begin{subequations}\label{limRes}
			$\bullet$ For $j\geq0$ such that $j-j_0\in\lc-nq,\ppp,np\rc$:
			\begin{align}
				\begin{split}
					&\widehat{\Pi}\left(\Rc(n,j_0,j)\textbf{e}\right) = O(e^{-cn}) \\
					& +\sum_{l=1}^I O\left(\frac{1}{n^\frac{1}{2\mu}}\exp\left(-c\left(\frac{\left|n-\left(\frac{j-j_0}{\alpha_l^+}\right)\right|}{n^\frac{1}{2\mu}}\right)^\frac{2\mu}{2\mu-1}\right)\right) +\sum_{l=I+1}^d O\left(\frac{e^{-c|j|}}{n^\frac{1}{2\mu}}\exp\left(-c\left(\frac{\left|n-\left(\frac{j-j_0}{\alpha_l^+}\right)\right|}{n^\frac{1}{2\mu}}\right)^\frac{2\mu}{2\mu-1}\right)\right)\\
					& + \sum_{l^\prime=1}^I\sum_{l=I+1}^d
					O\left(\frac{e^{-c|j|}}{n^\frac{1}{2\mu}}\exp\left(-c\left(\frac{\left|n-\left(\frac{j}{\alpha_l^+}-\frac{j_0}{\alpha_{l^\prime}^+}\right)\right|}{n^\frac{1}{2\mu}}\right)^\frac{2\mu}{2\mu-1}\right)\right) + \sum_{l^\prime=1}^I O\left(\frac{e^{-c|j|}}{n^\frac{1}{2\mu}}\exp\left(-c\left(\frac{\left|n+\frac{j_0}{\alpha_{l^\prime}^+}\right|}{n^\frac{1}{2\mu}}\right)^\frac{2\mu}{2\mu-1}\right)\right),
				\end{split}\label{limRes1}
			\end{align}
			
			$\bullet$ For $j<0$ such that $j-j_0\in\lc-nq,\ppp,np\rc$:
			\begin{align}
				\begin{split}
					\widehat{\Pi}\left(\Rc(n,j_0,j)\textbf{e}\right) =& O(e^{-cn}) +\sum_{l^\prime=1}^I\sum_{l=1}^{I-1} O\left(\frac{e^{-c|j|}}{n^\frac{1}{2\mu}}\exp\left(-c\left(\frac{\left|n-\left(\frac{j}{\alpha_l^-}-\frac{j_0}{\alpha_{l^\prime}^+}\right)\right|}{n^\frac{1}{2\mu}}\right)^\frac{2\mu}{2\mu-1}\right)\right) \\
					& + \sum_{l^\prime=1}^I O\left(\frac{e^{-c|j|}}{n^\frac{1}{2\mu}}\exp\left(-c\left(\frac{\left|n+\frac{j_0}{\alpha_{l^\prime}^+}\right|}{n^\frac{1}{2\mu}}\right)^\frac{2\mu}{2\mu-1}\right)\right).
				\end{split}\label{limRes2}
			\end{align}
		\end{subequations}
		
		We also claim that we can prove the following limits : 
		\begin{subequations}\label{limTermesRes}
			\begin{align}
				&\sum_{j=j_0-nq}^{j_0+np} O\left(e^{-cn}\right) = O(ne^{-cn}) \underset{n\rightarrow+\infty}\rightarrow 0,\label{limTermesRes1}\\
				\forall l\in\lbrace1,\hdots,I\rbrace, \quad& \sum_{j=j_0-nq}^{j_0+np}\ind_{j\geq0}O\left(\frac{1}{n^\frac{1}{2\mu}}\exp\left(-c\left(\frac{\left|n-\left(\frac{j-j_0}{\alpha_l^+}\right)\right|}{n^\frac{1}{2\mu}}\right)^\frac{2\mu}{2\mu-1}\right)\right) \underset{n\rightarrow+\infty}\rightarrow 0,\label{limTermesRes2}\\
				\forall l^\prime\in\lbrace1,\hdots,I\rbrace,\quad & \sum_{j=j_0-nq}^{j_0+np} O\left(\frac{e^{-c|j|}}{n^\frac{1}{2\mu}}\exp\left(-c\left(\frac{\left|n+\frac{j_0}{\alpha_{l^\prime}^+}\right|}{n^\frac{1}{2\mu}}\right)^\frac{2\mu}{2\mu-1}\right)\right)  \underset{n\rightarrow+\infty}\rightarrow 0,\label{limTermesRes3}\\
				\forall l\in\lbrace I+1,\hdots,d\rbrace, \quad & \sum_{j=j_0-nq}^{j_0+np}\ind_{j\geq0}O\left(\frac{e^{-c|j|}}{n^\frac{1}{2\mu}}\exp\left(-c\left(\frac{\left|n-\left(\frac{j-j_0}{\alpha_l^+}\right)\right|}{n^\frac{1}{2\mu}}\right)^\frac{2\mu}{2\mu-1}\right)\right) \underset{n\rightarrow+\infty}\rightarrow 0,\label{limTermesRes4}\\
				\forall l^\prime\in\lbrace1,\hdots,I\rbrace,\forall l\in\lbrace I+1,\hdots,d\rbrace,\quad & \sum_{j=j_0-nq}^{j_0+np}\ind_{j\geq0} O\left(\frac{e^{-c|j|}}{n^\frac{1}{2\mu}}\exp\left(-c\left(\frac{\left|n-\left(\frac{j}{\alpha_l^+}-\frac{j_0}{\alpha_{l^\prime}^+}\right)\right|}{n^\frac{1}{2\mu}}\right)^\frac{2\mu}{2\mu-1}\right)\right)  \underset{n\rightarrow+\infty}\rightarrow 0,\label{limTermesRes5}\\
				\forall l^\prime\in\lbrace1,\hdots,I\rbrace,\forall l\in\lbrace 1,\hdots,I-1\rbrace,\quad & \sum_{j=j_0-nq}^{j_0+np}\ind_{j<0} O\left(\frac{e^{-c|j|}}{n^\frac{1}{2\mu}}\exp\left(-c\left(\frac{\left|n-\left(\frac{j}{\alpha_l^-}-\frac{j_0}{\alpha_{l^\prime}^+}\right)\right|}{n^\frac{1}{2\mu}}\right)^\frac{2\mu}{2\mu-1}\right)\right)   \underset{n\rightarrow+\infty}\rightarrow 0.\label{limTermesRes6}
			\end{align}
		\end{subequations}
		
		Combining \eqref{limRes} and \eqref{limTermesRes}, we immediately obtain the expected limit \eqref{limites_termes_mass4}:
		$$\sum_{j=j_0-nq}^{j_0+np} \widehat{\Pi}\left(\Rc(n,j_0,j)\textbf{e}\right) \underset{n\rightarrow+\infty}\rightarrow 0.$$
		There just remains to actually prove \eqref{limTermesRes}. The first limit \eqref{limTermesRes1} is immediate. The second one \eqref{limTermesRes2} has already been proved via \eqref{limGauss2}. The limit \eqref{limTermesRes3} can be easily obtained via the following calculations for $ l^\prime\in\lbrace1,\hdots,I\rbrace$:
		$$ \sum_{j=j_0-nq}^{j_0+np} O\left(\frac{e^{-c|j|}}{n^\frac{1}{2\mu}}\exp\left(-c\left(\frac{\left|n+\frac{j_0}{\alpha_{l^\prime}^+}\right|}{n^\frac{1}{2\mu}}\right)^\frac{2\mu}{2\mu-1}\right)\right)= O\left(\frac{1}{n^\frac{1}{2\mu}}\exp\left(-c\left(\frac{\left|n+\frac{j_0}{\alpha_{l^\prime}^+}\right|}{n^\frac{1}{2\mu}}\right)^\frac{2\mu}{2\mu-1}\right)\right) \underset{n\rightarrow+\infty}\rightarrow 0.$$
		
		The last limits \eqref{limTermesRes4}-\eqref{limTermesRes6} can be obtained using a similar proof. We will thus prove \eqref{limTermesRes5} and leave the proof of the two remaining limits \eqref{limTermesRes4} and \eqref{limTermesRes6} for the interested reader.  We consider $l^\prime\in\lbrace1,\hdots I\rbrace$, $l\in\lbrace I+1,\hdots,d\rbrace$ and $j_0\in\N$. We observe that $\alpha^+_{l^\prime}<0$ and $\alpha_l^+>0$. Let us consider $n\in\N\backslash\lbrace0\rbrace$ large enough so that:
		\begin{equation}\label{inj0}
			j_0\leq -\frac{n\alpha^+_{l^\prime}}{2}.
		\end{equation}
		We separate our analysis in two parts:
		\begin{subequations}\label{limTermesRes5:Sep}
			\begin{itemize}
				\item For $j\in\N$ such that $j\leq\frac{n\alpha^+_l}{4}$, since $n$ is supposed to be large enough so that \eqref{inj0} is verified, we have that:
				$$n-\left(\frac{j}{\alpha^+_l}-\frac{j_0}{\alpha^+_{l^\prime}}\right)\geq \frac{n}{4}. $$
				As a consequence, we have in this case:
				\begin{equation}
					O\left(\frac{e^{-c|j|}}{n^\frac{1}{2\mu}}\exp\left(-c\left(\frac{\left|n-\left(\frac{j}{\alpha_l^+}-\frac{j_0}{\alpha_{l^\prime}^+}\right)\right|}{n^\frac{1}{2\mu}}\right)^\frac{2\mu}{2\mu-1}\right)\right) = O\left(\frac{1}{n^\frac{1}{2\mu}}\exp\left(-\frac{c}{4^\frac{2\mu}{2\mu-1}}n\right)\right).
				\end{equation}
				\item For $j\in\N$ such that $j\geq\frac{n\alpha^+_l}{4}$, we immediately have that:
				\begin{equation}
					O\left(\frac{e^{-c|j|}}{n^\frac{1}{2\mu}}\exp\left(-c\left(\frac{\left|n-\left(\frac{j}{\alpha_l^+}-\frac{j_0}{\alpha_{l^\prime}^+}\right)\right|}{n^\frac{1}{2\mu}}\right)^\frac{2\mu}{2\mu-1}\right)\right) = O\left(\frac{1}{n^\frac{1}{2\mu}}\exp\left(-\frac{c\alpha^+_l}{4}n\right)\right).
				\end{equation}
			\end{itemize}
		\end{subequations}
		Combining the two cases in \eqref{limTermesRes5:Sep}, we easily obtain \eqref{limTermesRes5} by observing that there exists a constant $\tilde{c}>0$ such that:
		$$\sum_{j=j_0-nq}^{j_0+np}\ind_{j\geq0} O\left(\frac{e^{-c|j|}}{n^\frac{1}{2\mu}}\exp\left(-c\left(\frac{\left|n-\left(\frac{j}{\alpha_l^+}-\frac{j_0}{\alpha_{l^\prime}^+}\right)\right|}{n^\frac{1}{2\mu}}\right)^\frac{2\mu}{2\mu-1}\right)\right)  = O\left(n^\frac{2\mu-1}{2\mu}e^{-\tilde{c}n}\right)\underset{n\rightarrow +\infty}\rightarrow 0. $$

	\printbibliography

@article {BeckHupkesSandstedeZumbrun,
	AUTHOR = {Beck, M. and Hupkes, H. J. and Sandstede, B.
	and Zumbrun, K.},
	TITLE = {Nonlinear stability of semidiscrete shocks for two-sided
	schemes},
	JOURNAL = {SIAM J. Math. Anal.},
	FJOURNAL = {SIAM Journal on Mathematical Analysis},
	VOLUME = {42},
	YEAR = {2010},
	NUMBER = {2},
	PAGES = {857--903},
	ISSN = {0036-1410,1095-7154},
	MRCLASS = {35L65 (34A33 34D09 34K06 35L67 37L60 65M06 65M12)},
	MRNUMBER = {2644362},
	MRREVIEWER = {Kayyunnapara\ Thomas\ Joseph},
	DOI = {10.1137/090775634},
	URL = {https://doi.org/10.1137/090775634},
}

@article {BenzoniHuotRousset,
	AUTHOR = {Benzoni-Gavage, S. and Huot, P. and Rousset, F.},
	TITLE = {Nonlinear stability of semidiscrete shock waves},
	JOURNAL = {SIAM J. Math. Anal.},
	FJOURNAL = {SIAM Journal on Mathematical Analysis},
	VOLUME = {35},
	YEAR = {2003},
	NUMBER = {3},
	PAGES = {639--707},
	ISSN = {0036-1410,1095-7154},
	MRCLASS = {35L65 (35B35 35L67 65M06 76L05)},
	MRNUMBER = {2048404},
	DOI = {10.1137/S0036141002418054},
	URL = {https://doi.org/10.1137/S0036141002418054},
}

@unpublished{CoeuLLT,
	TITLE = {{Local Limit Theorem for Complex Valued Sequences}},
	AUTHOR = {Coeuret, L.},
	URL = {https://hal.science/hal-03463375},
	NOTE = {preprint},
	YEAR = {2022},
	MONTH = Nov,
	HAL_ID = {hal-03463375},
	HAL_VERSION = {v2},
}

@article{CoeuIBVP,
	TITLE = {{Tamed stability of finite difference schemes for the transport equation on the half-line}},
	AUTHOR = {Coeuret, L.},
	URL = {https://hal.science/hal-04059973},
	JOURNAL = {{Mathematics of Computation}},
	PUBLISHER = {{American Mathematical Society}},
	YEAR = {2023},
	DOI = {10.1090/mcom/3901},
	HAL_ID = {hal-04059973},
	HAL_VERSION = {v1},
}

@book {Coppel,
	AUTHOR = {Coppel, W. A.},
	TITLE = {Dichotomies in stability theory},
	SERIES = {Lecture Notes in Mathematics, Vol. 629},
	PUBLISHER = {Springer-Verlag, Berlin-New York},
	YEAR = {1978},
	PAGES = {ii+98},
	ISBN = {3-540-08536-X},
	MRCLASS = {34A30 (34DXX)},
	MRNUMBER = {0481196},
	MRREVIEWER = {G. R. Sell},
}

@article {C-F,
	AUTHOR = {Coulombel, J.-F. and Faye, G.},
	TITLE = {Generalized {G}aussian bounds for discrete convolution powers},
	JOURNAL = {Rev. Mat. Iberoam.},
	FJOURNAL = {Revista Matem\'{a}tica Iberoamericana},
	VOLUME = {38},
	YEAR = {2022},
	NUMBER = {5},
	PAGES = {1553--1604},
	ISSN = {0213-2230},
	MRCLASS = {42A85 (35K25 60F99 65M12 65N12)},
	MRNUMBER = {4502075},
	DOI = {10.4171/rmi/1338},
	URL = {https://doi.org/10.4171/rmi/1338},
}

@article{C-FIBVP,
	AUTHOR = {Coulombel, J.-F. and Faye, G.},
	TITLE = {Sharp stability for finite difference approximations of
	hyperbolic equations with boundary conditions},
	JOURNAL = {IMA J. Numer. Anal.},
	FJOURNAL = {IMA Journal of Numerical Analysis},
	VOLUME = {43},
	YEAR = {2023},
	NUMBER = {1},
	PAGES = {187--224},
	ISSN = {0272-4979,1464-3642},
	MRCLASS = {65M06},
	MRNUMBER = {4565578},
	DOI = {10.1093/imanum/drab088},
	URL = {https://doi.org/10.1093/imanum/drab088},
	}

@article {D-S,
	AUTHOR = {Diaconis, P. and Saloff-Coste, L.},
	TITLE = {Convolution powers of complex functions on {$\mathbb{Z}$}},
	JOURNAL = {Math. Nachr.},
	FJOURNAL = {Mathematische Nachrichten},
	VOLUME = {287},
	YEAR = {2014},
	NUMBER = {10},
	PAGES = {1106--1130},
	ISSN = {0025-584X},
	MRCLASS = {42A85 (35K25 60F99)},
	MRNUMBER = {3231528},
	MRREVIEWER = {George A. Willis},
	DOI = {10.1002/mana.201200163},
	URL = {https://doi.org/10.1002/mana.201200163},
}

@Article{Godillon,
	Author = {P. {Godillon}},
	Title = {{Green's function pointwise estimates for the modified Lax-Friedrichs scheme.}},
	FJournal = {{M2AN. Mathematical Modelling and Numerical Analysis. ESAIM, European Series in Applied and Industrial Mathematics}},
	Journal = {{M2AN, Math. Model. Numer. Anal.}},
	ISSN = {0764-583X},
	Volume = {37},
	Number = {1},
	Pages = {1--39},
	Year = {2003},
	Publisher = {EDP Sciences, Les Ulis; Soci\'et\'e de Math\'ematiques Appliqu\'ees et Industrielles (SMAI), Paris},
	Language = {English},
	DOI = {10.1051/m2an:2003022},
	MSC2010 = {35L65 35Q35 35A35},
	Zbl = {1038.35036}
}

@phdthesis{Godillon_these,
	TITLE = {{Stabilit{\'e} des profils de chocs dans les syst{\`e}mes de lois de conservation}},
	AUTHOR = {Lafitte-Godillon, P.},
	URL = {https://tel.archives-ouvertes.fr/tel-00396376},
	SCHOOL = {{Ecole normale sup{\'e}rieure de lyon - ENS LYON}},
	YEAR = {2001},
	MONTH = Dec,
	TYPE = {Theses},
	HAL_ID = {tel-00396376},
	HAL_VERSION = {v1},
}

@article {Jennings,
	AUTHOR = {Jennings, G.},
	TITLE = {Discrete shocks},
	JOURNAL = {Comm. Pure Appl. Math.},
	FJOURNAL = {Communications on Pure and Applied Mathematics},
	VOLUME = {27},
	YEAR = {1974},
	PAGES = {25--37},
	ISSN = {0010-3640},
	MRCLASS = {39A10 (65N20)},
	MRNUMBER = {338594},
	MRREVIEWER = {R. Kodn\'{a}r},
	DOI = {10.1002/cpa.3160270103},
	URL = {https://doi.org/10.1002/cpa.3160270103},
}

@book {Kato,
	AUTHOR = {Kato, T.},
	TITLE = {Perturbation theory for linear operators},
	SERIES = {Classics in Mathematics},
	NOTE = {Reprint of the 1980 edition},
	PUBLISHER = {Springer-Verlag, Berlin},
	YEAR = {1995},
	PAGES = {xxii+619},
	ISBN = {3-540-58661-X},
	MRCLASS = {47A55 (46-00 47-00)},
	MRNUMBER = {1335452},
}

@Article{Kreiss,
	Author = {Kreiss, H.-O.},
	Title = {Stability theory for difference approximations of mixed initial boundary value problems. {I}},
	FJournal = {Mathematics of Computation},
	Journal = {Math. Comput.},
	ISSN = {0025-5718},
	Volume = {22},
	Pages = {703--714},
	Year = {1968},
	Language = {English},
	DOI = {10.2307/2004572},
	zbMATH = {3313828},
	Zbl = {0197.13704}
}

@article {Liu-Xin,
	AUTHOR = {Liu, J.-G. and Xin, Z. P.},
	TITLE = {{$L^1$}-stability of stationary discrete shocks},
	JOURNAL = {Math. Comp.},
	FJOURNAL = {Mathematics of Computation},
	VOLUME = {60},
	YEAR = {1993},
	NUMBER = {201},
	PAGES = {233--244},
	ISSN = {0025-5718,1088-6842},
	MRCLASS = {35L65 (35L67 65M12)},
	MRNUMBER = {1159170},
	MRREVIEWER = {Benoit\ Perthame},
	DOI = {10.2307/2153163},
	URL = {https://doi.org/10.2307/2153163},
}

@article {Liu-Xin2,
	AUTHOR = {Liu, J.-G. and Xin, Z. P.},
	TITLE = {Nonlinear stability of discrete shocks for systems of
	conservation laws},
	JOURNAL = {Arch. Rational Mech. Anal.},
	FJOURNAL = {Archive for Rational Mechanics and Analysis},
	VOLUME = {125},
	YEAR = {1993},
	NUMBER = {3},
	PAGES = {217--256},
	ISSN = {0003-9527},
	MRCLASS = {35L65 (35B35 65M12 76L05 76M20)},
	MRNUMBER = {1245071},
	MRREVIEWER = {Anders\ Szepessy},
	DOI = {10.1007/BF00383220},
	URL = {https://doi.org/10.1007/BF00383220},
}

@article {Majda-Ralston,
	AUTHOR = {Majda, A. and Ralston, J.},
	TITLE = {Discrete shock profiles for systems of conservation laws},
	JOURNAL = {Comm. Pure Appl. Math.},
	FJOURNAL = {Communications on Pure and Applied Mathematics},
	VOLUME = {32},
	YEAR = {1979},
	NUMBER = {4},
	PAGES = {445--482},
	ISSN = {0010-3640},
	MRCLASS = {35L65 (35L60)},
	MRNUMBER = {528630},
	MRREVIEWER = {Ronald i Perna},
	DOI = {10.1002/cpa.3160320402},
	URL = {https://doi.org/10.1002/cpa.3160320402},
}

@article {MasciaZumbrun,
	AUTHOR = {Mascia, C. and Zumbrun, K.},
	TITLE = {Pointwise {G}reen's function bounds and stability of
	relaxation shocks},
	JOURNAL = {Indiana Univ. Math. J.},
	FJOURNAL = {Indiana University Mathematics Journal},
	VOLUME = {51},
	YEAR = {2002},
	NUMBER = {4},
	PAGES = {773--904},
	ISSN = {0022-2518,1943-5258},
	MRCLASS = {35L60 (35A08 35B45 47N20 76L05)},
	MRNUMBER = {1947862},
	MRREVIEWER = {Vladimir\ Tulovsky},
	DOI = {10.1512/iumj.2002.51.2212},
	URL = {https://doi.org/10.1512/iumj.2002.51.2212},
}

@article {Michelson,
	AUTHOR = {Michelson, D.},
	TITLE = {Discrete shocks for difference approximations to systems of
	conservation laws},
	JOURNAL = {Adv. in Appl. Math.},
	FJOURNAL = {Advances in Applied Mathematics},
	VOLUME = {5},
	YEAR = {1984},
	NUMBER = {4},
	PAGES = {433--469},
	ISSN = {0196-8858,1090-2074},
	MRCLASS = {65M10 (35L65)},
	MRNUMBER = {766606},
	MRREVIEWER = {Reza\ Malek-Madani},
	DOI = {10.1016/0196-8858(84)90017-4},
	URL = {https://doi.org/10.1016/0196-8858(84)90017-4},
}

@article {Michelson2,
	AUTHOR = {Michelson, D.},
	TITLE = {Stability of discrete shocks for difference approximations to
	systems of conservation laws},
	JOURNAL = {SIAM J. Numer. Anal.},
	FJOURNAL = {SIAM Journal on Numerical Analysis},
	VOLUME = {40},
	YEAR = {2002},
	NUMBER = {3},
	PAGES = {820--871},
	ISSN = {0036-1429,1095-7170},
	MRCLASS = {65M06 (35B35 35L65 35L67 65M12)},
	MRNUMBER = {1949395},
	MRREVIEWER = {Gholam-Ali\ Zakeri},
	DOI = {10.1137/S0036142900377577},
	URL = {https://doi.org/10.1137/S0036142900377577},
}

@article {Newman,
	AUTHOR = {Newman, D. J.},
	TITLE = {A simple proof of {W}iener's {$1/f$} theorem},
	JOURNAL = {Proc. Amer. Math. Soc.},
	FJOURNAL = {Proceedings of the American Mathematical Society},
	VOLUME = {48},
	YEAR = {1975},
	PAGES = {264--265},
	ISSN = {0002-9939},
	MRCLASS = {42A28 (43A20)},
	MRNUMBER = {365002},
	MRREVIEWER = {H. Burkill},
	DOI = {10.2307/2040730},
	URL = {https://doi.org/10.2307/2040730},
}

@book {Rob,
	AUTHOR = {Robinson, D. W.},
	TITLE = {Elliptic operators and {L}ie groups},
	SERIES = {Oxford Mathematical Monographs},
	NOTE = {Oxford Science Publications},
	PUBLISHER = {The Clarendon Press, Oxford University Press, New York},
	YEAR = {1991},
	PAGES = {xii+558},
	ISBN = {0-19-853591-0},
	MRCLASS = {58G11 (22E30 35J99 47D40 47F05 47N99 58G05)},
	MRNUMBER = {1144020},
	MRREVIEWER = {Palle E. T. Jorgensen},
}

@article {R-S,
	AUTHOR = {Randles, E. and Saloff-Coste, L.},
	TITLE = {On the convolution powers of complex functions on {$\mathbb{Z}$}},
	JOURNAL = {J. Fourier Anal. Appl.},
	FJOURNAL = {The Journal of Fourier Analysis and Applications},
	VOLUME = {21},
	YEAR = {2015},
	NUMBER = {4},
	PAGES = {754--798},
	ISSN = {1069-5869},
	MRCLASS = {42A85 (60F99)},
	MRNUMBER = {3370010},
	MRREVIEWER = {George A. Willis},
	DOI = {10.1007/s00041-015-9386-1},
	URL = {https://doi.org/10.1007/s00041-015-9386-1},
}

@InCollection{Serre,
	Author = {D. {Serre}},
	Title = {{Discrete shock profiles: Existence and stability}},
	BookTitle = {Hyperbolic systems of balance laws. Lectures given at the C.I.M.E. summer school, Cetraro, Italy, July 14--21, 2003},
	ISBN = {978-3-540-72186-4},
	Pages = {79--158},
	Year = {2007},
	Publisher = {Berlin: Springer},
	Language = {English},
	MSC2010 = {65M06 65M12 35L65 35L67},
	Zbl = {1128.65068}
}

@article {Smyrlis,
	AUTHOR = {Smyrlis, Y. S.},
	TITLE = {Existence and stability of stationary profiles of the {LW}
	scheme},
	JOURNAL = {Comm. Pure Appl. Math.},
	FJOURNAL = {Communications on Pure and Applied Mathematics},
	VOLUME = {43},
	YEAR = {1990},
	NUMBER = {4},
	PAGES = {509--545},
	ISSN = {0010-3640,1097-0312},
	MRCLASS = {65M99 (35L65 65M50)},
	MRNUMBER = {1047334},
	MRREVIEWER = {Claude\ Carasso},
	DOI = {10.1002/cpa.3160430405},
	URL = {https://doi.org/10.1002/cpa.3160430405},
}

@article {Thomee,
	AUTHOR = {Thom\'{e}e, V.},
	TITLE = {Stability of difference schemes in the maximum-norm},
	JOURNAL = {J. Differential Equations},
	FJOURNAL = {Journal of Differential Equations},
	VOLUME = {1},
	YEAR = {1965},
	PAGES = {273--292},
	ISSN = {0022-0396},
	MRCLASS = {65.65 (35.99)},
	MRNUMBER = {176240},
	MRREVIEWER = {P. R\'{o}zsa},
	DOI = {10.1016/0022-0396(65)90008-2},
	URL = {https://doi.org/10.1016/0022-0396(65)90008-2},
}

@article {Ying,
	AUTHOR = {Ying, L.},
	TITLE = {Asymptotic stability of discrete shock waves for the
	{L}ax-{F}riedrichs scheme to hyperbolic systems of
	conservation laws},
	JOURNAL = {Japan J. Indust. Appl. Math.},
	FJOURNAL = {Japan Journal of Industrial and Applied Mathematics},
	VOLUME = {14},
	YEAR = {1997},
	NUMBER = {3},
	PAGES = {437--468},
	ISSN = {0916-7005,1868-937X},
	MRCLASS = {65M06 (35L65)},
	MRNUMBER = {1475142},
	MRREVIEWER = {Riccardo\ Fazio},
	DOI = {10.1007/BF03167392},
	URL = {https://doi.org/10.1007/BF03167392},
}

@article {ZH,
	AUTHOR = {Zumbrun, K. and Howard, P.},
	TITLE = {Pointwise semigroup methods and stability of viscous shock
	waves},
	JOURNAL = {Indiana Univ. Math. J.},
	FJOURNAL = {Indiana University Mathematics Journal},
	VOLUME = {47},
	YEAR = {1998},
	NUMBER = {3},
	PAGES = {741--871},
	ISSN = {0022-2518},
	MRCLASS = {35L67 (35B35 35L65 47D06 47N20)},
	MRNUMBER = {1665788},
	MRREVIEWER = {Roberto Natalini},
	DOI = {10.1512/iumj.1998.47.1604},
	URL = {https://doi.org/10.1512/iumj.1998.47.1604},
}

@Unpublished{Coeuret2024,
  author   = {Coeuret, Lucas},
  title    = {Nonlinear orbital stability of stationary discrete shock profiles for scalar conservation laws},
  month    = sep,
  year     = {2024},
  url      = {https://hal.science/hal-04712769},
}

@Unpublished{Coulombel2024,
  author   = {Coulombel, Jean-François and Faye, Grégory},
  title    = {Nonlinear orbital stability of stationary shock profiles for the {Lax}-{Wendroff} scheme},
  month    = nov,
  year     = {2024},
  url      = {https://hal.science/hal-04790143},
}

@Book{Godlewski,
  author    = {Godlewski, Edwige and Raviart, Pierre-Arnaud},
  publisher = {Springer-Verlag, New York},
  title     = {Numerical approximation of hyperbolic systems of conservation laws},
  year      = {[2021] \copyright2021},
  edition   = {Second},
  isbn      = {978-1-0716-1342-9; 978-1-0716-1344-3},
  series    = {Applied Mathematical Sciences},
  volume    = {118},
  doi       = {10.1007/978-1-0716-1344-3},
  mrclass   = {65M06 (35A35 35L65 65-02 76-02 76Mxx 76N10)},
  mrnumber  = {4331351},
  pages     = {xiii+840},
  url       = {https://doi.org/10.1007/978-1-0716-1344-3},
}

@Article{Hedstrom1975,
  author    = {Hedstrom, G. W.},
  journal   = {Mathematics of Computation},
  title     = {Models of {Difference} {Schemes} for \$u\_t + u\_x = 0\$ by {Partial} {Differential} {Equations}},
  year      = {1975},
  issn      = {0025-5718},
  number    = {132},
  pages     = {969--977},
  volume    = {29},
  doi       = {10.2307/2005735},
  publisher = {American Mathematical Society},
  url       = {https://www.jstor.org/stable/2005735},
}
\end{document}